\documentclass[11pt,twoside]{article}
\usepackage{amsmath}
\usepackage{amssymb}
\usepackage{enumerate,epsfig}
\usepackage{epstopdf}
\usepackage{graphicx}
\newtheorem{theorem}{Theorem}[section]
\newtheorem{lemma}[theorem]{Lemma}

\newtheorem{remark}[theorem]{Remark}
\newtheorem{prop}[theorem]{Proposition}
\newtheorem{defi}[theorem]{Definition}

\newenvironment{proof}{\paragraph{Proof} \phantom{9}}{\hfill$\Box$\bigskip}
\newcommand{\R}{ {\mathbb R} }
\newcommand{\N}{{\mathbb N}}
\newcommand{\eps}{{\varepsilon}}
\renewcommand{\div}{\mbox{div}\,}
\makeatletter
\@addtoreset{equation}{section}
\makeatother

\newcommand{\cqfd}{{\unskip\kern 6pt\penalty 500
\raise -2pt\hbox{\vrule\vbox to 6pt{\hrule width 6pt
\vfill\hrule}\vrule}\par}}
\newcommand{\ep}{{\varepsilon}}
\newcommand{\ind}{{\mathbb I}}

\title{Global Existence of Weak Solutions for Compresssible Navier--Stokes Equations:
         Thermodynamically unstable pressure and anisotropic viscous stress tensor}
\author{Didier Bresch
\footnote{LAMA CNRS UMR 5127, University of Savoie Mont-Blanc, Bat. Le Chablais, Campus scientifique, 73376 Le Bourget du Lac, France. D.~{\sc Bresch} is partially supported by the ANR- 13-BS01-0003-01 project DYFICOLTI. Email: didier.bresch@univ-smb.fr },
Pierre--Emmanuel Jabin
\footnote{CSCAMM and Dept. of Mathematics, University of Maryland,
College Park, MD 20742, USA. P.--E.~{\sc Jabin} is partially supported by NSF Grant 1312142 and by NSF Grant RNMS (Ki-Net) 1107444. Email: pjabin@cscamm.umd.edu}
}
\date{}
\begin{document}
\maketitle

\abstract
We prove global existence of appropriate weak solutions for the compressible Navier--Stokes equations for more general stress tensor than those covered by {\sc P.--L. Lions} and E. {\sc Feireisl}'s theory. More precisely we focus on more {\it general pressure laws} which are {\it not thermodynamically stable}; we are also able to handle some  {\it anisotropy in the viscous stress tensor}. 
  To give answers to these two longstanding problems, we revisit the classical compactness theory on the density  by obtaining precise quantitative regularity estimates: This requires a more precise analysis of the structure of the equations combined to a novel approach to the compactness of the continuity equation.
  These two cases open the theory to important physical applications, for instance to describe solar events (virial pressure law), geophysical flows (eddy viscosity) or biological situations (anisotropy).  

\medskip

\noindent {\bf Keywords.} Compressible Navier-Stokes, global--weak solutions, 
transport equation,  propagation of regularity,  non--monotone pressure laws, 
anisotropic viscous stress,  vacuum state, non-local terms.

\medskip

\noindent {\bf AMS classifications.}  35Q30, 35D30, 54D30, 42B37, 35Q86, 92B05.

\vfill
\eject

\tableofcontents
\section{Introduction}
The question of global in time existence of solutions to fluid dynamics' models goes back to the pioneering work by {\sc J. Leray} (1933) where he introduced the concept of weak (turbulent) solutions to the Navier--Stokes systems describing the motion of an {\sc incompressible} fluid; this work has become the basis of the underlying mathematical theory up to present days.
  The theory for viscous {\sc compressible} fluids in a barotropic regime has, in comparison, been developed more recently in the monograph by {\sc P.--L. Lions} \cite{Li2} (1993-1998), later extended by E. ~{\sc Feireisl}  and collaborators \cite{FeNoPe} (2001) and has been since then a very active field of study. 

When changes in temperature are not taken into account, the barotropic Navier-Stokes system reads
\begin{equation}\label{Barocomp0}
\left\{
\begin{array}{rl}
& \partial_t \rho + {\rm div} (\rho u) =0,\\
& \partial_t (\rho u) + {\rm div}(\rho u\otimes u) - {\rm div} \, \sigma  =\rho f,
\end{array}
\right. 
\end{equation} 
where $\rho$ and $u$ denote respectively the density and the velocity field. 
The stress tensor $\sigma$ of a general fluid obeys Stokes' law $\sigma = {\cal S} - \,  P \,{\rm Id}$
where $P$ is a scalar function termed pressure (depending on the density in the compressible barotropic setting or being an unknown in the incompressible setting) and ${\cal S}$ denotes the viscous stress tensor which characterizes the measure of resistance of the fluid to flow. 

Our approach also applies to the Navier-Stokes-Fourier system, as explained in section \ref{withtemperature}, which is considered more physically relevant. But our main purpose here is to explain how the new regularity method that we introduce can be applied to a wide range of Navier-Stokes like models and not to focus on a particular system. For this reason, we discuss the main features of our new theory on the simpler \eqref{Barocomp0}.  We indicate later in the article after the main ideas how
to extend the result to some Navier-Stokes-Fourier systems.
 
    In comparison with Leray's work on incompressible flows, which is nowadays relatively "simple" at least from the point of view of modern functional analysis (and in the linear viscous stress tensor case), the mathematical theory of weak solutions to compressible fluids is quite involved, bearing many common aspects   with the theory of nonlinear conservation laws.

Our focus is on the global existence of weak solutions. For this reason we will not refer to the important question of existence of strong solutions or the corresponding uniqueness issues.

\medskip

 Several important problems about global existence of weak solutions for compressible flows remain open. We consider in this article the following  questions
\begin{itemize}
\item General pressure laws, in particular without any monotonicity assumption;
\item Anisotropy in the viscous stress tensor which is especially important in geophysics.
\end{itemize}
In the current {\sc Lions-Feireisl} theory,  the pressure law $P$ is often assumed to be of the form $P(\rho)= a\rho^\gamma$  but this can be generalized, a typical example being
\begin{align} 
&\nonumber  P\in {\cal C}^1([0,+\infty)), \quad P(0)=0 \hbox{ with }  \\
&a\rho^{\gamma-1} - b \le P'(\rho) \le  \frac{1}{a} \rho^{\gamma-1} + b \hbox{ with } \gamma>d/2 \label{nonmonopressure1}
\end{align}
for some constants $a>0 , b\ge 0$: See   {\sc B. Ducomet, E.~Feireisl, H. Petzeltova, I. Straskraba} \cite{DuFePeSt} or {\sc E.~Feireisl} \cite{Fe1} for slightly more general assumptions.
However it is always required that {\em $P(\rho)$ be increasing} after a certain critical 
value of~$\rho$.

\medskip

\noindent This monotonicity of $P$ is connected to several well known difficulties
\begin{itemize}
\item The monotonicity of the pressure law is required for the stability of the thermodynamical equilibrium. Changes in monotonicity in the pressure are typically connected to intricate phase transition problems. 
\item At the level of compressible Euler, {\em i.e.} when ${\cal S}=0$, non-monotone pressure laws may lead to a loss of hyperbolicity in the system, possibly leading to corrected systems (as by {\sc Korteweg} in particular). 
\end{itemize}

In spite of these issues, we are able to show that compressible Navier-Stokes systems like \eqref{Barocomp0} are {\em globally well posed without monotonicity assumptions on the pressure law}; instead only rough growth estimates are required. This allows to {\em consider for the first time several famous physical laws such as modified virial expansions}.  
 
\medskip

As for the pressure law, the theory initiated in the multi-dimensional setting by {\sc P.--L.~Lions} and {\sc  E.~Feireisl} requires that the stress tensor has the very specific form
\[
 \sigma = 2 \mu D(u) + \lambda {\rm div} u\,  {\rm Id}  - P(\rho)\,  {\rm Id}
\]
with $D(u)= (\nabla u + \nabla u^T)/2$,  $\mu$ and $\lambda$ such that
$\lambda + 2 {\mu}/{d}  \ge 0$. The coefficients $\lambda$ and $\mu$ do not need to be constant but require some explicit regularity,
 see for instance \cite{FeNo} for temperature dependent coefficients.
 
Unfortunately several physical situations involve some anisotropy in the stress tensor; geophysical flows for instance use and require such constitutive laws, see for instance  \cite{TeZi} and \cite{BrDeGe} with eddy viscosity in turbulent flows.
   
 We present in this article the first results able to handle {\it more general viscous stress tensor}  of the form  
\[
\sigma = A(t) \nabla u  + \lambda {\rm div} u \, {\rm Id}  - P(\rho)\, {\rm Id} 
\]
with a $d\times d$ symmetric matrix $A$ with regular enough coefficients. The matrix $A$ can incorporate anisotropic phenomena in the fluid. 
  Note that our result also applies to the case
\[
\sigma = A(t) D(u)  + \lambda {\rm div} u \, {\rm Id}  - P(\rho)\, {\rm Id}
\] 
 where $D(u)= (\nabla u + \nabla u^T)/2$ still.
 
 \medskip
 
Our new results therefore significantly expand the reach of the current theory for compressible Navier-Stokes and make it more robust with respect to the large variety of laws of state and stress tensors that are used. 
This is achieved through a complete revisiting of the classical compactness theory by obtaining {\it quantitative regularity estimates}. The idea is inspired by estimates obtained for nonlinear continuity equations in \cite{BeJa}, though with a different method than the one introduced here. Those estimates correspond to critical spaces, also developed and used for instance in works by J.~{\sc Bourgain}, H.~{\sc Br\'ezis} and P.~{\sc Mironescu} and by A.C. ~{\sc Ponce}, see \cite{BoBrMi} and \cite{Po}.  

   Because of the weak regularity of the velocity field, the corresponding norm of the critical space cannot be propagated. Instead the norm has to be modified by  weights based on a auxiliary function which solves a kind of dual equation adapted to the compressible Navier--Stokes system under consideration.  After proving appropriate properties of the weights, we can prove compactness on 
 the density.

\medskip

The plan of the article is as follows
\begin{itemize}
\item Section \ref{classicaltheory} presents the classical theory by P.~--~L.~{\sc Lions} and E.~{\sc Feireisl}, with the basic energy estimates. It explains why the classical proof of compactness does not seem able to handle the more general equations of state that concern us here. We also summarize the basic physical discussions on pressure laws and stress tensors choices which motivates our study. This section can be skipped by readers which are already familiar with the state of the art.
\item    In Section \ref{Eqsmainresults}, we present the equations and the corresponding main results concerning global existence of weak solutions for non-monotone pressure law  and then for anisotropic viscous stress tensor. Those are given in the barotropic setting.
\item  Section \ref{sketchnewcompactness} is devoted to an introduction to our new method. We give our quantitative compactness criterion and  
  we show the basic ideas in the simplex context of linear uncoupled transport equations 
  and a very rough sketch of proof in the compressible Navier--Stokes setting.
\item  Section \ref{secstability} states the stability results which constitute the main contribution of the paper. 
\item  Section \ref{usefulsection} states technical lemmas which are needed in the main proof and are based on classical harmonic analysis tools: Maximal and square functions properties, translation
of operators. 
\item  Section \ref{secrenormalized} and \ref{proofstabilityres} constitutes the heart of the proof. Section \ref{secrenormalized} is devoted to renormalized equation with definitions and properties of the weights.
   Section \ref{proofstabilityres} is devoted to the proof of the stability results of section \ref{secstability} both concerning more general pressure laws and concerning the anisotropic stress tensor.  
\item Section \ref{sec-proof-mainresults} concerns the construction of the approximate solutions. It uses the stability results of section \ref{secstability} to conclude the proof of the existence theorems of section \ref{Eqsmainresults}.
\item The extension for non-monotone pressure laws (with respect to density) to the Navier--Stokes--Fourier system is discussed in Section \ref{withtemperature}. It contains a discussion of the state of the art complementing section \ref{classicaltheory} in that case. It follows some steps
already included in the book \cite{FeNo} but also ask for a careful check of the estimates at each
level of approximation.
\item  We present in Section \ref{otherextensions} some models occurring in other contexts where the new mathematical techniques presented here could be useful in the future. 
\item Section \ref{notations} is a list of some of the notations that we use.
\item Section \ref{Besov}  is an appendix recalling basic facts on Besov spaces which are used in the article.
\end{itemize}

\medskip

\section{Classical theory by {\sc E. Feireisl} and  {\sc P.--L.~Lions}, open problems and physical considerations\label{classicaltheory}}
We consider for the moment compressible fluid dynamics in a general domain $\Omega$ which can be the whole space $\R^d$, a periodic box or a bounded smooth domain with adequate boundary conditions. We do not precise the boundary conditions and instead leave those various choices open as they may depend on the problem and we want to insist in this section on the common difficulties and approaches. We will later present our precise estimates in the periodic setting for simplicity.
\subsection{A priori estimates \label{basicaprioriest}}
We collect the main physical {\it a priori} estimates for very general barotropic systems
on $\R_+\times \Omega$ 
\begin{equation}\label{Barocompgeneral}
\left\{
\begin{array}{rl}
& \partial_t \rho + {\rm div} (\rho u) =0,\\
& \partial_t (\rho u) + {\rm div}(\rho u\otimes u) -{\cal D}\,u + \nabla P(\rho) =\rho f,
\end{array}
\right. 
\end{equation}
where ${\cal D}$ is only assumed to be a negative differential operator in divergence form on $u$ s.t.
\begin{equation}
\int_\Omega u\cdot {\cal D}\,u\,dx\sim -\int_\Omega |\nabla u|^2\,dx,\label{generalstress}
\end{equation}
and for any $\phi$ and $u$
\begin{equation}
\int_\Omega \phi\cdot {\cal D}\,u\,dx\leq C\,\|\nabla\phi\|_{L^2}\,\|\nabla u\|_{L^2}. \label{generalstress2}
\end{equation}
The following estimates form the basis of the classical theory of existence of weak solutions and we will use them in our own proof. We only give the formal derivation of the estimates at the time being.

First of all, the total energy of the fluid is dissipated. This energy is the sum of the kinetic energy and the potential energy (due to the compressibility) namely
\[
E(\rho,u)=\int_\Omega \left(\rho \frac{|u|^2}{2}+\rho e(\rho)\right)\,dx,
\]
where 
$$e(\rho)= \int_{\rho_{\rm ref}}^\rho P(s)/s^2 ds$$
 with $\rho_{\rm ref}$ a constant reference density. Observe that formally from \eqref{Barocompgeneral}
\[
\partial_t \bigl(\rho\frac{|u|^2}{2}\bigr)+\div\left(\rho\,u \frac{|u|^2}{2}\right)-u\cdot {\cal D}u+u\cdot \nabla P(\rho)=\rho\,f\cdot u, 
\]
and thus
\[
\frac{d}{dt} \int_\Omega \rho\frac{|u|^2}{2}-\int_\Omega u\cdot {\cal D}\,u-\int_\Omega P(\rho)\,\div\,u=\int_\Omega \rho f\cdot u.
\]
On the other hand, by the definition of $e$, the continuity equation on $\rho$ implies that
\[
\partial_t (\rho e(\rho))
+\div(\rho e(\rho)\,u)+P(\rho)\,\div u=0.
\]
Integrating and combining with the previous equality leads to the energy equality
\begin{equation}
\frac{d}{dt}E(\rho,u)-\int_\Omega u\cdot {\cal D}\,u=\int_\Omega \rho f\cdot u.
\label{energyeq}
\end{equation}
Let us quantify further the estimates which follow from \eqref{energyeq}. Assume that $P(\rho)$ behaves roughly like $\rho^\gamma$ in the following weak sense 
\begin{equation}
C^{-1}\,\rho^\gamma-C\leq P(\rho)\leq C\,\rho^\gamma+C,\label{quantpress}
\end{equation}
then $\rho e(\rho)$ also behaves like $\rho^\gamma$. Note that \eqref{quantpress} does not imply any monotonicity on $P$ which could keep oscillating. One could also work with even more general assumption than \eqref{quantpress}: Different exponents $\gamma$ on the left--hand side and the right--hand side for instance... But for simplicity we use \eqref{quantpress}.

Assuming that $f$ is bounded (or with enough integrability), one now deduces from \eqref{energyeq} the following uniform bounds
\begin{equation}\begin{split}
&\sup_t \int_\Omega \rho\,|u|^2\,dx\leq C+E(\rho^0,u^0),\\
& \sup_t \int_\Omega \rho^\gamma\,dx\leq C,\\
&\int_0^T\int_\Omega |\nabla u|^2\,dx\leq C.
\end{split}\label{energyestimates}
\end{equation}
We can now improve on the integrability of $\rho$, as it was first observed by {\sc P.-L. Lions}. Choose any smooth, positive $\chi(t)$ with compact support,  and test the momentum equation by $\chi\,g=\chi\,{\cal B}\,\rho^a$ where ${\cal B}$ is a linear operator (in $x$) s.t.
\[
\div\, g=(\rho^a-\overline{\rho^a}),\quad \|\nabla g\|_{L^p}\leq C_p\,\|\rho^a-\overline{\rho^a}\|_{L^p}, 
\quad  \|{\cal B}\,\phi\|_{L^p}\leq C_p\,\|\phi\|_{L^p}, \ \forall\ 1<p<\infty,
\]
where we denote by $\overline{\rho^a}$ the average of $\rho^a$ over $\Omega$. 
Finding $g$ is straightforward in the whole space but more delicate in bounded domain as the right boundary conditions must also be imposed. This is where {\sc E. Feireisl} {\em et al.} introduce the {\sc BOGOVSKI} operator.
   We obtain that
\[\begin{split}
\int\chi(t)\int_\Omega \rho^a\,P(\rho)\,dx\,dt\le &\int\chi(t)\int_\Omega g\,(\partial_t(\rho\,u)+\div(\rho\,u\otimes u)-{\cal D}\,u-\rho\,f)\,dx\,dt\\
& + \int \chi(t) \int_\Omega \overline{\rho^a} P(\rho).
\end{split}\]
By \eqref{quantpress}, the left--hand side dominates
\[
\int\chi(t)\int_\Omega \rho^{a+\gamma}\,dx\,dt.
\]
It is possible to bound the terms in the right--hand side. For instance by \eqref{generalstress2}
\[\begin{split}
-\int\chi(t)\, g\,{\cal D}\,u\,dx\,dt&\leq C\,\|\nabla u\|_{L^2([0,\ T],\;L^2(\Omega))}\, \|\chi\,\nabla g\|_{L^2([0,\ T],\;L^2(\Omega))}\\
&\leq C\,\|\nabla u\|_{L^2([0,\ T],\;L^2(\Omega))}\, \|\chi\,(\rho^a- \overline {\rho^a})\|_{L^2([0,\ T],\;L^2(\Omega))},
\end{split}
\]
by the choice of $g$. Given the bound \eqref{energyestimates} on $\nabla u$, 
this term does not pose any problem if $2\,a<a+\gamma$. Next
\begin{equation}
\int\chi\, g\,\partial_t(\rho\,u)\,dx\,dt=-\int(g\,\chi'(t)+\chi(t)\,{\cal B}\,(\partial_t (\rho^a- \overline{\rho^a})))\,\rho\,u\,dx\,dt.\label{gainrhotheta}
\end{equation}
The first term in the right--hand side is easy to bound; as for the second one, the continuity equation implies
\begin{equation}\begin{split}
\int\chi(t)\,{\cal B}\,(\partial_t (\rho^a-\overline{\rho^a}))\,\rho\,u\,dx\,dt=-\int\chi\,\left[{\cal B}\,(\div(u\, \rho^a))\right]\,\rho\,u\\
 - \int \chi \left[(a-1)\,{\cal B}\,\bigl(\rho^a\,\div\,u 
- \overline{{\rm div} (u \rho^a) + (a-1) \rho^a \div u)}\bigr)\right]\,\rho\,u.
\end{split} 
\end{equation} 
Using the properties of ${\cal B}$ and the energy estimates \eqref{energyestimates}, it is possible to control those terms as well as the last one in \eqref{gainrhotheta}, provided $a\leq 2\gamma/d -1$ and $\gamma>d/2$ which leads to
\begin{equation}
\int_0^T\int_\Omega \rho^{\gamma+a}\,dx\,dt\leq C(T, E(\rho^0,u^0)).\label{gainintegrability}
\end{equation}
\subsection{Heuristic presentation of the method by {\sc E.~Feireisl} and {\sc P.~--L.~Lions}\label{classicalapproach}}
Let us explain, briefly and only heuristically the main steps to prove global existence of weak solutions in the barotropic case with constant viscosities and power $\gamma$ pressure law. Our purpose is to highlight why a specific form of the pressure or of the stress tensor is needed in the classical approaches. We also refer for such a general presentation of the theory to the book by {\sc A. Novotny} and {\sc I. Straskraba} \cite{NoSt}, the monograph Etats de la Recherche edited by 
{\sc D. Bresch} \cite{Br} or the book by {\sc P.~Plotnikov} and {\sc J. Sokolowski} \cite{PlSo}.

Let us first consider the simplest model with constant viscosity coefficients $\mu$ and $\lambda$.
 In that case, the compressible Navier--Stokes equation reads on $\R_+\times \Omega$
\begin{equation}\label{Barocomp}
\left\{
\begin{array}{rl}
& \partial_t \rho + {\rm div} (\rho u) =0,\\
& \partial_t (\rho u) + {\rm div}(\rho u\otimes u) - \mu \Delta u - (\lambda+\mu) \nabla {\rm div} u + \nabla P(\rho) =\rho f,
\end{array}
\right. 
\end{equation} 
with $P(\rho)= a\rho^\gamma$. For simplicity, we work in a smooth, Lipschitz regular, bounded domain $\Omega$ with 
homogeneous Dirichlet boundary conditions on the velocity
\begin{equation}
u\vert_{\partial\Omega} =0.\label{Dirichlet}
\end{equation}   
A key concept for the existence of weak solutions is  the notion of renormalized solution to the continuity equation, as per the theory for linear transport equations by {\sc R.J. DiPerna} and {\sc P.--L. Lions}. 
Assuming $\rho$ and $u$ are smooth and satisfy the continuity equation, for all $b\in {\cal C}([0,+\infty))$,  one may multiply the equation by $b'(\rho)$ to find that $(\rho,u)$ also 
\begin{equation}
\partial_t b(\rho) + {\rm div}(b(\rho) u) + (b'(\rho)\rho - b(\rho)){\rm div} u = 0.\label{renormalizedeq}
\end{equation}
This leads to the following definition
\begin{defi}\label{defi}
For any $T\in (0,+\infty)$, $f$, $\rho_0$, $m_0$ satisfying some technical assumptions,
we say that a couple $(\rho,u)$ is a weak renormalized solution with bounded energy
if it has the following properties
\begin{itemize} 
\item $\rho \in L^\infty(0,T;L^\gamma(\Omega))\cap {\cal C}^0([0,T], L^\gamma_{\rm weak}(\Omega))$,
$\rho\ge 0$ {\it a.e.} in $(0,T)\times\Omega$, $\rho|_{t=0}= \rho_0$ {\it a.e.} in $\Omega$;
\item $u\in L^2(0,T;H_0^1(\Omega))$, 
$\rho |u|^2 \in L^\infty(0,T;L^1(\Omega))$, $\rho u$ is continuous in time with value in the weak topology of $L^{2\gamma/(\gamma+1)}_{\rm weak}(\Omega)$,
$(\rho u)|_{t=0}= m_0$ {\it a.e.} $\Omega$;
\item $(\rho,\,u)$ extended by zero out of $\Omega$ solves the mass and momentum equations in $\R^d$,
in ${\cal D}'((0,T)\times R^d)$;
\item For any smooth $b$ with appropriate monotony properties, $b(\rho)$ solves the renormalized Eq. \eqref{renormalizedeq}.
\item
For almost all $\tau \in (0,T)$, $(\rho,u)$ satisfies the energy inequality
\[
E(\rho,u)(\tau) +\int_0^\tau \int_\Omega (\mu |\nabla u|^2  + (\lambda+\mu)\, |{\rm div} u|^2) \le E_0 + \int_0^\tau\int_\Omega \rho f\cdot u.
\]
\end{itemize}
In this
inequality, $E(\rho,u)(\tau) = \int_\Omega (\rho|u|^2 / 2 + \rho e(\rho))(\tau)$, with
$e(\rho)= \int_{\rho_{\rm ref}}^\rho P(s)/s^2 ds$ ($\rho_{\rm ref}$ being any constant reference density), denotes the total
energy at time $\tau$ and $E_0=  \int_\Omega |m_0|^2 / 2\rho_0 + \rho_0 e(\rho_0)$
denotes the initial total energy.
\end{defi}
 Assuming $P(\rho) = a \rho^\gamma$ (in that case $e(\rho)$ may equal to $a\rho^{\gamma-1}/(\gamma-1)$), the theory developed by {\sc P.--L. Lions} to prove the global existence of renormalized weak solution
with bounded energy asks for some limitation on the adiabatic constant $\gamma$ namely
$\gamma > 3d/(d+2)$.  {\sc E.~Feireisl} {\it et al.} have generalized this approach in order to cover the range
$\gamma > 3/2$ in dimension 3 and more generally $\gamma>d/2$ where $d$ 
is the space dimension. 

We present the initial proof due to {\sc P.--L. Lions} and indicate quickly at the end how it was improved by
{\sc E.~Feireisl} {\it et al.}. The method relies on the construction of a sequence of approximate solution, derivation of {\em a priori} estimates and passage to the limit which requires delicate compactness estimate. We skip for the time being the construction of such an approximate sequence, see for instance the book by A.~{\sc Novotny} and I. {\sc straskraba} for details.

The approximate sequence, denoted by $(\rho_k,\;u_k)$, should satisfy the energy inequality leading to a first uniform {\it a priori bound}, using that $\mu>0$ and $\lambda+2\,\mu/d>0$
\[
\sup_t \int_\Omega (\rho_k\,|u_k|^2 / 2 + a \rho^\gamma_k /(\gamma-1))\,dx+
\mu \int_0^t\int_{\Omega} |\nabla u_k|^2\,dx\,dt\leq C,
\]
for some constant independent of $n$. 

For $\gamma>d/2$, we also have the final {\it a priori} estimate \eqref{gainintegrability} explained in the previous subsection namely
\[
\int_0^\infty \int_\Omega \rho_k^{\gamma+a} \le C(R,T) \hbox{ for } a \le \frac{2}{d}\gamma -1.
\]
When needed for clarification,  we denote by $\overline U$  the weak limit of a general sequence $U_k$
(up to a subsequence). Using the energy estimate and the extra integrability property proved on the density, and by extracting subsequences, one obtains the following convergence
\begin{align*}
&\rho_k \rightharpoonup \rho \hbox{ in } {\cal C}^0([0,T]; L^\gamma_{\rm weak}(\Omega)),\\
&\rho_k^\gamma \rightharpoonup \overline {\rho^\gamma} \hbox{ in } L^{(\gamma+a)/\gamma}((0,T)\times\Omega),\\
&\rho_k u_k \rightharpoonup \rho u \hbox{ in } {\cal C}^0([0,T]; L^{2\gamma/(\gamma+1)}(\Omega)),\\
&\rho_k u_k^iu_k^j \rightharpoonup \rho u^i\, u^j \hbox{ in } {\cal D}'((0,T)\times\Omega) \hbox{ for } i,j=1,2,3.
\end{align*}
The convergence of the nonlinear terms $\rho_k\,u_k$ and $\rho_k\,u_k\otimes u_k$ uses  the compactness in time of  $\rho_k$ deduced from the uniform estimate on  $\partial_t \rho_k$ given by the continuity equation and  the compactness in time of $\sqrt {\rho_k}u_k$ deduced form the uniform estimate on   $\partial_t(\rho_k u_k)$ given by the momentum equation. This is combined with the $L^2$ estimate on $\nabla u_k$. 

Consequently, the extensions by zero to $(0,T)\times R^3\slash\Omega$ of the functions
 $\rho, u, \overline{\rho^\gamma}$, denoted again  $\rho, u, \overline{\rho^\gamma}$,
satisfy the system in $\R_+\times \R^d$
\begin{equation}\label{Barocompaverage}
\left\{
\begin{array}{rl}
& \partial_t \rho + {\rm div} (\rho u) =0,\\
& \partial_t (\rho u) + {\rm div}(\rho u\otimes u) - \mu \Delta u - (\lambda+\mu) \nabla {\rm div} u + a \nabla  \overline{\rho^\gamma}=\rho f.
\end{array}
\right. 
\end{equation}
The difficulty consists in proving that $(\rho,u)$ is a renormalized weak solution 
with bounded energy and the main point is showing that $\overline {\rho^\gamma} = \rho^\gamma$ {\it a.e.} 
in $(0,T)\times\Omega$. 

This requires compactness on the density sequence which cannot follow from the previous {\em a priori} estimates only.
Instead {\sc P.--L.~Lions} uses a weak compactness of  the sequence $\{a \rho_k^\gamma - {\lambda+2\mu} {\rm div} u_k\}_{k\in N^*}$ which is
 usually called the viscous effective flux. This property was previously identified in one space dimension by {\sc D. Hoff} and  {\sc D. Serre}.  
 More precisely, we have the following property for all function $b\in {\cal C}^1([0,+\infty))$ satisfying some increasing properties
 at infinity
\begin{align} 
&\nonumber  \lim_{k\to +\infty} \int_0^T \int_\Omega (a \rho_k^\gamma - (2\mu+\lambda){\rm div} u_k)b(\rho_k) \varphi dx dt\\
&\qquad\qquad\qquad =
 \int_0^T \int_\Omega (a\overline{\rho^\gamma} - (2\mu+\lambda){\rm div} u)\overline{ b(\rho)}\varphi dx dt \label{divcurlequiv}
\end{align}
where the over-line quantities design the weak limit of the corresponding quantities and $\varphi \in {\cal D}((0,T)\times \Omega).$ Note that such a property is reminiscent of compensated compactness as the weak limit of a product is shown to be the product of the weak limits.
   In particular the previous property implies that
\begin{equation}\label{rel}
\overline{\rho{\rm div} u} - {\rho} {\rm div} u 
     = \frac{\overline{P(\rho)}\rho - \overline{P(\rho)\,\rho}}{2\mu+\lambda}
\end{equation}
Taking the divergence of the momentum equation,
we get the relation
\[
 \Delta [(2\mu+\lambda) {\rm div} u_k - P(\rho_k)] = 
    {\rm div} [\partial_t(\rho_k u_k) + {\rm div}(\rho_k u_k \otimes u_k)].
\]
Note that here the
 form of the stress tensor has been strongly used.
From this identity,  {\sc P.--L. Lions} proved the property \eqref{divcurlequiv} based on harmonic analysis due to {\sc  R. Coifman} and {\sc Y.~Meyer} (regularity properties of commutators) and takes the observations by {\sc D.  Serre} made in the one-dimensional case into account. The proof by  
{\sc E.~Feireisl} is based on div-curl Lemma introduced by {\sc F. Murat} and {\sc L.~Tartar}.  

   To simplify the remaining calculations, we assume $\gamma \ge 3d/(d+2)$ and in that case due to the extra integrability on the density, we get that $\rho_k \in L^2((0,T)\times\Omega).$ 
This lets us choose $b(s) = s \log s$ in the renormalized formulation for $\rho_k$ and $\rho$ and take the difference of  the two equations. Then pass to the  limit $n\to +\infty$ and use the identity of weak compactness on the effective flux  to replace terms with divergence of velocity by terms with density
using \eqref{rel}, leading to
\[
\partial_t(\overline{\rho\log\rho}- \rho\log \rho) + {\rm div} ((\overline{\rho\log\rho} - \rho\log \rho)u)
= \frac{1}{2\mu+\lambda} (\overline{P(\rho)}\rho - \overline{P(\rho)\rho}).
\]
Observe that the monotonicity of the pressure  $P(\rho)=a\rho^\gamma$ implies that
\[
\overline{P(\rho)}\rho - \overline{P(\rho)\rho}\leq 0.
\] 
This is the one point where the monotonicity assumption is used.
It allows to show that the defect measure for the sequence of density satisfies
\[
 {\rm dft} [\rho_k -\rho](t) = \int_\Omega \overline{\rho\log\rho}(t) - \rho\log\rho (t) \, dx\leq {\rm dft} [\rho_k -\rho](t=0).
\]
On the other hand, the
strict-convexity of the function $s\mapsto s\log s$, $s\ge 0$ implies that
$ {\rm dft} [\rho_k -\rho]\geq 0$.  If initially this quantity vanishes, it then vanishes at every later time.

Finally the commutation of the weak convergence with a strictly convex function yields {\em the strong convergence} of the density $\rho_k$ in $L^1_{\rm loc}$. Combined with the uniform bound of $\rho_k$ in $L^{\gamma+a} ((0,T)\times \Omega)$, we get
the strong convergence of the pressure term $\rho_k^\gamma$.

This concludes the proof  in the case $\gamma\ge 3d/(d+2)$.
The proof of {\sc E. Feireisl} works even if the density is not {\it a priori} square integrable.
For that {\sc E. Feireisl} observes that it is possible to control the amplitude of the possible oscillations
on the density in a norm $L^p$ with $p>2$ allowing to use an effective flux property with some truncature. Namely he introduced the following oscillation measure
\[
{\rm osc}_p [\rho_k-\rho] = \sup_{n\ge1}\,  [\limsup_{k\to +\infty} \|T_n(\rho_k)-T_n(\rho)\|_{L^p((0,T)\times\Omega} ],
\]
where $T_n$ are cut-off functions defined as
\[
T_n(z) = n T\bigl(\frac{z}{n}\bigr), \qquad n\ge 1
\]
with $T\in {\cal C}^2(R)$
\[
 T(z) = z \hbox{ for } z\le 1, \qquad
    T(z) = 2 \hbox{ for } z\ge 3, \qquad
    T \hbox{ concave on } R.
\]
 The existence result can  then obtained up to $\gamma >d/2$: See again the review by A.~{\sc Novotny} and I.~{\sc Straskraba} \cite{NoSt}.
 
  To the author's knowledge there exists few extension of the previous study to more general
pressure laws or more general stress tensor. Concerning a generalization of the pressure law, 
as explained in the introduction there exist the works by {\sc B.~Ducomet, E. Feireisl, H. Petzeltova, 
I. Straskraba} \cite{DuFePeSt} and  {\sc E. Feireisl} \cite{Fe1} where the hypothesis  imposed on the pressure $P$ imply that
 $$P(z)= r_3(z) - r_4(z)$$
 where $r_3$ is non-decreasing in $[0,+\infty)$ with $r_4\in {\cal C}^2([0,+\infty))$ satisfying 
 $r_4\ge 0$ and $r_4(z)\equiv 0$ when  $z\ge Z$ for a certain $Z\ge0$.
 The form is used to show that it is possible respectively  to continue to control the amplitude of the oscillations ${\rm osc}_p[\rho_k - \rho]$ and then to show that the defect measure vanishes if initially it vanishes.  The two papers \cite{Fe1} and \cite{DuFePeSt} we refer to allow to consider for instance
the two important cases : Van der Waals equation of state and  some cold nuclear equations 
of state with finite number of monomial (see the subsection on the physical discussion).
\subsection{The limitations of the {\sc Lions-Feireisl} theory} 
The previous heuristical part makes explicit the difficulty in extending the global existence result for more general non-monotone pressure law or for non-isotropic stress tensor. First of all the key point in the previous approach was 
\[
\overline{P(\rho)}\,\rho-\overline{P(\rho)\,\rho}\leq 0.
\]
This property is intimately connected to the monotonicity of $P(\rho)$ or of $P(\rho)$ for $\rho\geq \rho_c$ with truncation operators as in \cite{Fe1} or \cite{DuFePeSt}. Non-monotone pressure terms cannot satisfy such an inequality and are therefore completely outside the current theory.

\medskip

The difficulty with anisotropic stress tensor is that we are losing the other key relation in the previous proof namely \eqref{rel}.
   For non-isotropic stress tensor with an additional vertical component and  power pressure law for instance, we get instead the following relation
\[
\overline{\rho {\rm div} u} - \rho {\rm div} u \le  
a \frac{\overline\rho \overline{A_\mu\rho^\gamma}-\overline {\rho A_\mu\rho^\gamma}}{\mu_x+\lambda}
\] 
with some non-local anisotropic operator $A_\mu= (\Delta - (\mu_z-\mu_x)\partial_z^2)^{-1} \partial_z^2 $ where $\Delta$ is the total Laplacian in terms of $(x,z)$ with variables $x=(x_1,\cdots ,x_{d-1})$, $z=x_d$.

Unfortunately, we are again losing the structure and in particular the sign of the right--hand side, as observed in particular in \cite{BrDeGe}. Furthermore even small anisotropic perturbations of an isotropic stress tensor cannot be controlled in terms of the defect measure introduced by {\sc E.~Feireisl} and collaborators: Remark the non-local behavior in the right--hand side due to the term $A_\mu$. For this reason, the anisotropic case seems to fall completely out the theory developed by {\sc P.--L. Lions} and {\sc E.~Feireisl}.
 
 \medskip

Those two open questions are the main objective of this monograph.
\subsection{Physical discussions on pressure laws and stress tensors \label{discusslaws}} 
The derivation of the compressible Navier--Stokes system from first principles is delicate and goes well beyond the scope of this manuscript. In several respects the system is only an approximation and this should be kept in mind in any discussion of the precise form of the equations which should allow for some uncertainty. 

\subsubsection{Equations of state} 
It is in general a non straightforward  question to decide what kind of pressure law should be used depending on the many possible applications: mixtures of fluids, solids, and even the interior of stars. Among possible equation of state, one can find several well known laws such as Dalton's law of partial pressures (1801), ideal gas law (Clapeyron 1834), Van der Waals equation of state (Van De Waals 1873),  virial equation of state (H. Hamerlingh Onnes 1901). 

In general the pressure law $P(\rho,\vartheta)$ can depend on both the density $\rho$ and the temperature $\vartheta$. We present the temperature dependent system in full later but include the temperature already in the present discussion to emphasize its relevance and importance.

Let us give some important examples of equations of state
\begin{itemize}
\item State equations are barotropic if $P(\rho)$ depends only on the density. As explained  in the book by E. {\sc Feireisl}  \cite{Fei} (see pages 8--10 and 13--15),
the simplest example of a barotropic flow is an isothermal flow where the temperature is assumed to be constant. If both conduction of heat and its generation by dissipation of mechanical energy can be neglected then the temperature is uniquely determined as a function of the density (if initially the entropy is constant) yielding a barotropic state equation for the pressure $P(\rho)=a\rho^\gamma$ with $a>0$ and 
 $\displaystyle \gamma={(R+c_v)}/{c_v}>1$. Another barotropic flow was discussed in \cite{DuFePeSt}.
\item The classical Van der Waals equation reads
\[
(P+a\,\rho^2)\,(b-\rho)=c\,\rho\,\vartheta,
\]
where $a,\;b,\;c$ are constants. The pressure law is non-monotone if the temperature is below a critical value, $\vartheta<\vartheta_c$, but it satisfies \eqref{nonmonopressure1}. In compressible fluid dynamics, the Van des Waals equation of state is sometimes simplified by neglecting specific volume changes and becomes
\[
(P+a)\,(b-\rho)=c\,\rho\,\vartheta,
\]
with similar properties.
\item Using finite-temperature Hartree-Fock theory, it is possible to obtain a temperature dependent equation of state of the following form
\begin{equation}
P(\rho, \vartheta) = a_3(1 + \sigma)\rho^{2+\sigma} - a_0\rho^2 + k \vartheta \sum_{n\ge 1} B_n \rho^n,\label{PG}
\end{equation}
where $k$ is the Boltzmann's constant, and where the last expansion (a simplified virial series) converges rapidly because of the rapid decrease of the $B_n$. 
\item  Equations of state can include other physical mechanism. A good example is found in the article \cite{DuFePeSt} where radiation comes into play: a photon assembly is superimposed to the nuclear matter background. If this radiation is in quasi-local thermodynamical equilibrium with the (nuclear) fluid, the resulting mixture nucleons+photons can be described by a one-fluid heat-conducting Navier--Stokes system, provided one adds to the equation of state  a Stefan-Boltzmann contribution of black-body type
\[
P_R(\vartheta) = a\vartheta^4 \hbox{ with } a>0,
\]
and provided one adds a corresponding contribution to the energy equation. The corresponding models are more complex and do not satisfy \eqref{nonmonopressure1} in general. 

\item In the context of the previous example, a further simplification can be introduced leading to the so-called Eddington's standard model.
   This approximation assumes that the ratio between the total pressure $P(\rho, \vartheta) = P_G(\rho, \vartheta) + P_R(\vartheta)$ and the radiative pressure $P_R(\vartheta)$ is a pure constant
\[
\frac{P_R(\vartheta)}{P_G(\rho,\vartheta)+P_R(\vartheta)} = 1 - \beta,
\]
where $0 < \beta < 1$ and $P_G$ is given for instance by \eqref{PG}. Although crude, this model is in good agreement with more sophisticated models, in particular for the sun.

One case where this model leads to a pressure law satisfying \eqref{nonmonopressure1} is when one keeps only the low order term into the virial expansion. Suppose that $\sigma = 1$ and let us plug the expression of the two pressure laws in this relation,
\[
2a_3\rho^3  - a_0 \rho^2 + k B_1 \vartheta \rho = \frac{\beta}{1-\beta} \frac{a}{3}\vartheta^4.
\]
By solving this algebraic equation to leading order (high temperature), one gets 
\[
\vartheta \approx \bigl( \frac{6 a_3(1-\beta)}{a\beta} \bigr) \rho^{3/4},
\]
leading to the pressure law
\[
P(\rho,\vartheta) = \frac{2 a_3}{\beta} \rho^3 - a_0\rho^2 
                                  + k B_1\bigl( \frac{6 a_3(1-\beta)}{a\beta} \bigr) \rho^{7/4},
\]
which satisfies \eqref{nonmonopressure1} because of the
constant coefficients.

However in this approximation, only the higher order terms were kept. Considering non-constant coefficient or keeping all the virial sum in the 
 pressure law was out of the scope of  \cite{DuFePeSt} and leads to precisely the type of non- monotone pressure laws that we consider in the present work. 
\item The virial 
equation of state  for heat conducting Navier--Stokes equations can be derived from statistical physics and reads 
\[
P(\rho,\vartheta) = \rho\, \vartheta \Bigl(\sum_{n\ge 0} B_n(\vartheta) \rho^n\Bigr)
\]
with $B_0={\rm cte}$ and the coefficients $B_n(\vartheta)$ have to be specified for $n\ge1$. 

While the full virial pressure law is beyond the scope of this article, we can handle truncated virial with appropriate assumptions or pressure laws of the type $P(\rho,\vartheta) = P_e(\rho)+ P_{th}(\rho,\vartheta)$
\item Pressure laws can also incorporate many other type of phenomena. Compressible fluids may include or model biological agents which have their own type of interactions. In addition, as explained later, our techniques also apply to other types of ``momentum'' equations. The range of possible pressure laws is then even wider. 
\end{itemize}

Based on these examples, the possibilities of pressure laws are many. Most are not monotone and several do not satisfy \eqref{nonmonopressure1}, proving the need for a theory able to handle all sort of behaviors.
\subsubsection{Stress tensors} 
One finds a similar variety of stress tensors as for pressure laws. We recall that we denote $D(u)=(\nabla u+(\nabla u)^T)/2$.
\begin{itemize}
\item The isotropic stress tensor with constant coefficients
\[
{\cal D} =   \mu\, \Delta u + (\lambda+\mu)\,\nabla\, \div u,
\]
which is the classical example that can be handled by the {\sc Lions-Feireisl} theory:
See for instance \cite{Li2}, \cite{NoSt} and \cite{PlSo} with $\gamma>d/2$. 
See also the recent interesting  work by {\sc P.I. Plotnikov} and {\sc W. Weigant} 
(see \cite{PlWe})  in the two-dimensional in space case with $\gamma=1$.
\item Isotropic stress tensors with non constant coefficients better represent the physics of the fluid however. Those coefficients can be temperature $\vartheta$ dependent
\[
{\cal D} =   2\,\div\,(\mu(\vartheta)\, D(u)) + \nabla\,(\lambda(\vartheta) \div u).
\]
Provided adequate non-degeneracy conditions are made on $\mu$ and $\lambda$, this case can still be efficiently treated by the {\sc Lions-Feireisl} theory under some assumptions on the pressure law. See for instance \cite{Fei} or \cite{FeNo}.

\item  The coefficients of the isotropic stress tensors may also depend on the density
\[
{\cal D} =  2\,\div\,(\mu(\rho,\vartheta)\, D(u)) + \nabla\,(\lambda(\rho, \vartheta)\, \div u).
\]
This is a very difficult problem in general. The almost only successful insight in this case can be found in \cite{BD, BrDeGe, BrDeZa, VaYu, LiXi} with no dependency with respect to the temperature. Those articles require a very special form of $\mu(\rho)$ and $\lambda(\rho)$ and  without such precise assumptions, almost nothing is known. 
 Note also the very nice paper concerning global existence of strong solutions in two-dimension by {\sc A.~Kazhikhov} and {\sc V.A. Vaigant} where
$\mu$ is constant but $\lambda= \rho^\beta$ with $\beta\ge 3$, see \cite{KaVa}.
\item Geophysical flow cannot in general be assumed to be isotropic but instead some directions have different behaviors; this can be due to gravity in large scale fluids for instance. A nice example is found in the Handbook written by {\sc R. Temam}
and {\sc M. Ziane}, where the eddy-viscous term ${\cal D}$ is given by
\[
{\cal D} =   \mu_h \Delta_x u + \mu_z \partial_z^2 u + (\lambda+\mu)\, \nabla \div u,
\]
with $\mu_h \not = \mu_z$. While such an anisotropy only requires minor modifications for the incompressible Navier--Stokes system, it is not compatible with the {\sc Lions-Feireisl} approach,
see for instance \cite{BrDeGe}.  
\end{itemize}

\section{Equations and main results: The barotropic case\label{Eqsmainresults}}
We will from now on work on the torus $\Pi^d$. This is only for simplicity in order to avoid discussing boundary conditions or the behavior at infinity. The proofs would easily extend to other cases as mentioned at the end of the paper.

\subsection{Statements of the results: Theorem \ref{MainResultPressureLaw} and  \ref{MainResultAniso}}
We present in this section our main existence results. As usual for global existence of weak solutions to nonlinear PDEs, one has to prove stability estimates for sequences of approximate solutions and construct such approximate
sequences. The main contribution in this paper and the major part of the proofs concern the stability procedure and more precisely the compactness of the density.

\bigskip

\noindent {\bf I) Isotropic compressible Navier--Stokes equations with general pressure.}
Let us consider the isotropic compressible Navier--Stokes equations  
\begin{equation}\label{BaroPressureLaw}
\left\{
\begin{array}{rl}
& \partial_t \rho + {\rm div} (\rho u) =0,\\
& \partial_t (\rho u) + {\rm div}(\rho u\otimes u) - \mu \Delta u - (\lambda+\mu) \nabla {\rm div} u + \nabla P(\rho) =\rho f,
\end{array}
\right. 
\end{equation} 
with ${2}\,\mu/d+\lambda$, a pressure law $P$ which is continuous on $[0,+\infty)$, $P$ locally Lipschitz on $(0,+\infty)$ with $P(0)=0$ such that there exists $C>0$ with
\begin{equation}\label{gammacontrol}
C^{-1} \rho^\gamma - C  \le P(\rho) \le C \rho^\gamma + C
\end{equation}
and for all $s\ge 0$
\begin{equation}
|P'(s)|\leq \bar P s^{\tilde \gamma-1}.
\label{nonmonotonehyp}
\end{equation}
\medskip

One then has global existence
\begin{theorem} \label{MainResultPressureLaw} 
Assume that the initial data $u_0$ and $\rho_0\ge 0$ with $\int_{\Pi^d} \rho_0 = M_0>0$  satisfies  the bound
\[
E_0= \int_{\Pi^d} \bigl(\frac{|(\rho u)_0|^2}{2\rho_0} + \rho_0 e(\rho_0)\bigr)\,dx <+\infty.
\]
   Let the pressure law $P$  satisfies
{\rm (\ref{gammacontrol})} and {\rm (\ref{nonmonotonehyp})} with 
\begin{equation}
\gamma>\, \bigl(\max(2,\tilde \gamma) +1\bigr)\, \frac{d}{d+2}.\label{hypgammaiso}
\end{equation} 
   Then there exists a global  weak solution of the compressible Navier--Stokes system 
{\rm (\ref{BaroPressureLaw})} in the sense of Definition \ref{defi}. 
\end{theorem}

\bigskip

\noindent {\bf Remark.} Let us note that the solution satisfies the explicit regularity estimate 
\[
\int_{\Pi^{2d}} \ind_{\rho_k(x,t)\geq \eta}\,\ind_{\rho_k(y,t)\geq \eta}\,K_h(x-y)\,\chi(\delta\rho_k)
(t)\leq \frac{C\,\|K_h\|_{L^1}}{\eta^{1/2}\,
|\log h|^{\theta/2}},
\]
for some $\theta>0$ where $K_h$ is defined in proposition \ref{propcomp}, 
$\delta\rho_k$ and $\chi$ are defined in Section \ref{proofstabilityres}, see \eqref{defchi}.

\bigskip

\noindent {\bf II) A non-isotropic compressible Navier--Stokes equations.}
We consider an example of non-isotropic compressible Navier--Stokes equations
\begin{equation}\label{BaroAnisotropic}
\left\{
\begin{array}{rl}
& \partial_t \rho + {\rm div} (\rho u) =0, \\
& \partial_t (\rho u) + {\rm div}(\rho u\otimes u) - \div\,(A(t)\, \nabla u)  
- (\mu+\lambda) \nabla {\rm div} u + \nabla P(\rho) =0,
\end{array}
\right. 
\end{equation} 
with $A(t)$ a given smooth and symmetric matrix, satisfying
\begin{equation}
A(t)=\mu\,{\rm Id}+\delta A(t),\quad \mu>0,\quad \frac{2}d\,\mu+\lambda-\|\delta A(t)\|_{L^\infty}>0.
\end{equation}
where $\delta A$ will be a perturbation around $\mu \, {\rm Id}$.
 We again take $P$ locally Lipschitz on $[0,+\infty)$ with $P(0)=0$  but require it to be monotone 
 after a certain point
\begin{equation}
C^{-1}\,\rho^{\gamma-1}-C\leq P'(\rho)\leq C\,\rho^{\gamma-1}+C.\label{strictmonotone}
\end{equation}
with $\gamma>d/2$.
The second main result that we obtain is 
\begin{theorem} \label{MainResultAniso}
Assume that the initial data $u_0$ and $\rho_0\ge 0$ with $\int_{\Pi^d} \rho_0 = M_0>0$  satisfies  the bound
\[
E_0= \int_{\Pi^d} \bigl(\frac{|(\rho u)_0|^2}{2\rho_0} + \rho_0 e(\rho_0)\bigr)\,dx <+\infty.
\]
Let  the pressure $P$ satisfy \eqref{strictmonotone} with 
\[
\gamma > \frac{d}2 \left[\left(1+\frac{1}d\right) + \sqrt{1+\frac1{d^2}}\right]. \label{hypgammanoniso}
\] 
  There exists a universal constant $C_\star>0$
 such that if 
 \[
\|\delta A\|_{\infty} \le C_\star\,(2\mu+\lambda),
\]
  then there exists a global weak solution  of the compressible  Navier--Stokes equation  
  in the sense of Definition \ref{defi} replacing the isotropic energy inequality by the
  following anisotropic energy
\[
E(\rho,u)(\tau) +\int_0^\tau \int_\Omega (\nabla_x u^T\,A(t)\,\nabla u
    + (\mu+\lambda)\, |{\rm div} u|^2) \le E_0.
\]
\end{theorem}

\bigskip

\noindent {\bf Remark.} Let us note that the constraint on $\gamma$ corresponds to
the constraint on p:  $p>\gamma + \gamma/(\gamma-1)$ where $p$ is the extra
integrability property on $\rho$.

\subsection{Important comments/comparison with previous results}
The choice was made to focus on explaining the new method instead of trying to write results as general as possible but at the cost of further burdening the proofs. For this reason, Theorems \ref{MainResultPressureLaw} and \ref{MainResultAniso} are only two examples of what can be done. 

\medskip

We explain how to apply our new method to the Navier--Stokes--Fourier system (with an additional equation for temperature) in section \ref{withtemperature}. The Navier--Stokes--Fourier system is physically more relevant than the barotropic case and as seen from the discussion in subsection \ref{discusslaws}, it exhibits even more examples of non-monotone pressure laws.

\medskip

\noindent {I) \it Possible extensions.}
   In section \ref{otherextensions}, we also present applications to various other important models, in particular in the Bio-Sciences where the range of possible pressure laws (or what plays their role such as chemical attraction/repulsion) is wide. 
  But there are many other possible extensions; for instance \eqref{gammacontrol} could be replaced with a more general
\[
C^{-1}\,\rho^{\gamma_1}-C\leq P(\rho)\leq C\,\rho^{\gamma_2}+C,
\]
with different exponents $\gamma_1\neq\gamma_2$. While the proofs would essentially remain the same, the assumption \eqref{hypgammaiso} would then have to be replaced and would involve $\gamma_1$ and $\gamma_2$.
   Similarly, it is possible to consider spatially dependent stress tensor $A(t,x)$ in Theorem \ref{MainResultAniso}. This introduces additional terms in the proof but those can easily be handled as long as $A$ is smooth by classical methods for pseudo differential operators. 

\medskip

\noindent {II) \it Comparison with previous results.}

\smallskip

\noindent {II-1) \it Non-monotone pressure laws.}
     Theorem \ref{MainResultPressureLaw} is the first result to allow for completely non-monotone pressure laws. Among many important previous contributions, we refer to \cite{DuFePeSt, Fe1, CHT, Li2} and \cite{Fei, FeNo, NoSt, ChTr} for the Navier--Stokes--Fourier system,  which are our main point of comparison. All of those require $P'>0$ after a certain point and in fact typically a condition like \eqref{strictmonotone}. The removal of the key assumption of monotonicity has important consequences:
\begin{itemize}
\item From the physical and modeling point of view, it opens the possibility of working with a wider range of equations of state as discussed in subsection \ref{discusslaws} and it makes the current theory on viscous, compressible fluids more robust to perturbation of the model.
\item Changes of monotonicity in $P$ can create and develop oscillations in the density $\rho$ (because some ``regions'' of large density become locally attractive). It was a major question whether such  oscillations remain under control at least over bounded time intervals. This shows that the stability for bounded times is very different from uniform in time stability as $t\rightarrow +\infty$.     
    Only the latter requires  assumptions of thermodynamical nature such as the monotonicity of $P$.
\item Obviously well posedness for non monotone $P$ could not be obtained as is done here for the  compressible Euler system. As can be seen from the proofs, the viscous stress tensor in the compressible Navier-Stokes system has precisely the critical scaling to control the oscillations created by the non-monotonicity. This implies for instance that in phase transition phenomena, the transition occurs smoothly precisely at the scale of the viscosity.
\item Our results could have further consequences for instance to show convergence of numerical schemes (or for other approximate systems). Typical numerical schemes for compressible Navier--Stokes raise issues of oscillations in the density which are reminiscent of the ones faced in this article. The question of convergence of numerical schemes to compressible Navier--Stokes is an important and delicate subject in its own, going well beyond the scope of this short comment. We refer for instance to the works by {\sc R. Eymard, T. Gallou\"et, R. Herbin, J.--C. Latch\'e} and 
{\sc T.K. Karper}: See for instance  \cite{GHL, EGHL2} for the simpler Stokes case, \cite{EGHL, GGHL, GHL2} for Navier--Stokes and more recently to the work \cite{CEGH, Ka}.
\end{itemize} 

Concerning the requirement on the growth of the pressure at $\infty$, that is on the coefficient $\gamma$ in \eqref{hypgammaiso}
\begin{itemize}
\item In the typical case where $\tilde \gamma=\gamma$, \eqref{hypgammaiso} leads to the same constraint as in {\sc P.--L.~Lions} \cite{Li2} for a similar reason: The need to have $\rho\in L^2$ to make sense of $\rho\,\div u$. It is  worse than the $\gamma>d/2$ required for instance in \cite{Fei}. In $3d$, we hence need $\gamma>9/5$ versus only $\gamma>3/2$ in \cite{Fei}.
\item It may be possible to improve on \eqref{hypgammaiso} while still using the method introduced here but propagating compactness on appropriate truncation of $\rho$; for instance by writing an equivalent of Lemma \ref{renor} on $\phi(\rho(x))-\phi(\rho(y))$ as in the multi-dimensional setting by {\sc E. Feireisl}. This possibility was left to future works. Note that the requirement on $\gamma>d/2$ comes from   the need to gain integrability as per \eqref{gainintegrability} along the strategy presented in subsection \ref{basicaprioriest}. Our new method still relies on this estimate and therefore has no hope, on its own, to improve on the condition $\gamma>d/2$.   
\item In the context of general pressure laws, and even more so for Navier--Stokes--Fourier, assumption \eqref{hypgammaiso} is not a strong limitation. Virial-type pressure laws, where $P(\rho)$ is a polynomial expansion, automatically satisfy it for instance as do many other examples discussed in subsection \ref{discusslaws}. 
\end{itemize} 

\bigskip
\noindent {II-2)  \it Anisotropic stress tensor.}
Theorem \ref{MainResultAniso} is so far the only result of global existence of weak solutions which is able to handle anisotropy in the stress tensor. It applies for instance to eddy-viscous tensor mentioned above for geophysical flows
\[
{\cal D} =   \mu_h \Delta_x u + \mu_z \partial_z^2 u + (\lambda+\mu)\, \nabla \div u,
\]
where $\mu_h \not = \mu_z$ and corresponding to 
\begin{equation}
A_{ij}=\mu_h\,\delta_{ij}\quad\mbox{for}\ i,\,j=1,\,2,\qquad A_{33}=\mu_z,\qquad A_{ij}=0\quad\mbox{otherwise}.\label{example1A}
\end{equation}
This satisfies the assumptions of Theorem \ref{MainResultAniso} provided
$|\mu_h-\mu_x|$ is not too large, which is usually the case in the context of geophysical flows. 

Additional examples of applications are given in section \ref{otherextensions} but we wish to emphasize here that it is also possible to have a fully 
symmetric anisotropy, namely $\div(A\,D\,u)$ with $D(u)=\nabla u+\nabla u^T$ in the momentum equation. This is the equivalent of the anisotropic case in linear elasticity and it is also an important case for compressible fluids. Note that it leads to a different form of the stress tensor. With the above choice of $A$, Eq. \ref{example1A}, one would instead obtain
\[
 \div(A\,D\,u) =   \mu_h \Delta_x u + \mu_z \partial_z^2 u + \mu_z\,\nabla \partial_z u_z+(\lambda+\mu)\, \nabla \div u.
\]
Accordingly we choose Theorem \ref{MainResultAniso} with the non-symmetric anisotropy $\div(A\,\nabla u)$ as it corresponds to the eddy-viscous term by {\sc  R. Temam} and {\sc M. Ziane} mentioned above. But the extension to the symmetric anisotropy is possible although it introduces some minor complications. For instance one cannot simply obtain $\div u$ by solving a scalar elliptic system but one has to solve a vector valued one instead; please see the remark just after \eqref{divu0noniso} and at the end of the proof of Theorem \ref{MainResultAniso} in section \ref{sec-proof-mainresults}.

Ideally one would like to obtain an equivalent of Theorem \ref{MainResultAniso} assuming only uniform elliptic bounds on $A(t)$ and much lower bound on $\gamma$. Theorem \ref{MainResultAniso} is a first attempt in that direction, which can hopefully later be improved. 

However the reach of Theorem \ref{MainResultAniso} should not be minimized because non isotropy in the stress tensor appears to be a level of difficulty above even non-monotone pressure laws. Losing the pointwise relation between $\div u$ and $P(\rho(x))$ is {\em a major hurdle},  as it can also be seen from the proofs later in the article.  Instead one has to work with
\[
\div u(t,x)=P(\rho(t,x))+L\,P(\rho)+\,\mbox{effective pressure},
\] 
with $L$ a non local operator of order $0$. The difficulty is to control appropriately this non local term so that its contribution can eventually be bounded by the dissipation due to the local pressure term.

\bigskip
 
\noindent {\bf Notations.}
    For simplicity, in the rest of the article, $C$ will denote a numerical constant whose value may change from line to line. It may depend on some uniform estimates on the sequences of functions considered (as per bounds \eqref{bounduk} or \eqref{boundrhok} for instance) but it will never depend on the sequence under consideration (denoted with index $k$) or the scaling parameters $h$ or $h_0$.

\section{Sketch of the new compactness method \label{sketchnewcompactness}}
The standard compactness criteria used in the compressible Navier--Stokes framework is the 
Aubin--Lions--Simon Lemma to get compactness on  the terms $\rho u$ and $\rho u\otimes u$.
 A more complex trick is used  to get the strong convergence of the density. More precisely it
 combines  extra integrability estimates on the density and the effective flux property (a kind of weak compactness) and then a convexity-monotonicity tool to conclude. 

\medskip

   Here we present a tool which will be the cornerstone in our study to prove compactness
on the density and which will be appropriate to cover more general equation of state or
stress tensor form. 

In order to give the main idea of the method, we present it first in this section for the well known case of linear transport equations, {\em i.e.} assuming that $u$ is given. We then give a rough sketch of the main ideas we will use in the rest of the article. This presents the steps we will follow for proofs  in the more general setting.
\subsection{The compactness criterion}
We start by a well known result providing compactness of a sequence
\begin{prop} \label{propcomp}
Let $\rho_k$ be a sequence uniformly bounded in some $L^p((0,T)\times \Pi^d)$ with $1\leq p<\infty$. Assume that ${\cal K}_h$ is a sequence of positive, bounded functions s.t.
\[\begin{split}
&i.\quad \forall \eta>0,\quad \sup_h \int_{|x|\geq \eta} {\cal K}_h(x)\,dx<\infty,\qquad \mbox{supp}\, K_h\in B(0,R),\\
&ii.\quad \|{\cal K}_h\|_{L^1(\Pi^d)}\longrightarrow +\infty.
\end{split}\]
If $\partial_t \rho_k \in L^q([0,\ T]\times W^{-1,q}(\Pi^d))$ with $q\geq 1$ uniformly in $k$ and
\[
\limsup_k \Bigl[ \frac{1}{\|{\cal K}_h\|_{L^1}}\,\limsup_{t\in [0,T]} \int_{\Pi^{2d}} {\cal K}_h(x-y) \,|\rho_k(t,x)-\rho_k(t,y)|^p\,dx\,dy\Bigr] \longrightarrow 0,\quad \mbox{as}\ h\rightarrow 0
\]
then $\rho_k$ is compact in $L^p([0,\ T]\times\Pi^d)$. Conversely if $\rho_k$ is compact in $L^p([0,\ T]\times \Pi^d)$ then the above quantity converges to $0$ with $h$.
\label{kernelcompactness}
\end{prop}
For the reader's convenience, we just quickly recall that the compactness in space is connected to the classical approximation by convolution.
Denote $\bar {\cal K}_h$ the normalized kernel
\[
\bar {\cal K}_h=\frac{{\cal K}_h}{\|{\cal K}_h\|_{L^1}}.
\] 
Write
\[
\begin{split}
\|\rho_k - \bar {\cal K}_h\star_x \rho_k \|_{L^p}^p 
   &\le\frac{1}{ \|{\cal K}_h\|_{L^1}^p}\int_{\Pi^d}\Bigl(\int_{\Pi^d} 
    {\cal K}_h(x-y) |\rho_k(t,x)-\rho_k(t, y)|dx \Bigr)^p \, dy\\
&\le
  \frac{1}{ \|{\cal K}_h\|_{L^1}}\int_{\Pi^{2d}} {\cal K}_h(x-y) |\rho_k(t,x)-\rho_k(t, y)|^p dx \, dy
\end{split}
\]
which converges to zero as $h\rightarrow 0$ uniformly in $k$ by assumption. 
On the other--hand for a fixed $h$, the sequence $\overline K_h\star_x u_k$ in $k$ is compact in $x$. This
complete the compactness in space. 
Concerning the compactness in time, we just have to couple everything and 
use the uniform bound on $\partial_t \rho_k$ as per the usual {\sc Aubin-Lions-Simon} Lemma. 
\bigskip

%

\noindent 
{\it The ${\cal K}_h$ functions.} In this paper we choose 
\[
K_h(x) =\frac{1}{(h+|x|)^a},\quad \mbox{for}\ |x|\leq 1/2. 
\]
with some $a>d$ and $K_h$ non negative,
independent of $h$ for $|x|\geq 2/3$, $K_h$ constant outside $B(0,3/4)$ and periodized such as to belong in
$C^\infty(\Pi^d\setminus B(0,3/4))$. For convenience we denote
\[
\overline K_h(x)=\frac{K_h(x)}{\|K_h\|_{L^1}}.
\] 

This is enough for linear transport equations but for compressible Navier--Stokes we also need 
\[
{\cal K}_{h_0}(x)=\int_{h_0}^1 \overline K_h(x)\,\frac{dh}h.
\]
We will see that
\[
\frac{C^{-1}}{(h_0+|x|)^d}\leq {\cal K}_{h_0}(x)\leq \frac{C}{(h_0+|x|)^d}.
\]
\subsection{Compactness for linear transport equation\label{seclineartransport}}
Consider a sequence of solutions $\rho_k$, on the torus $\Pi^d$ (so as to avoid any discussion of boundary conditions or behavior at infinity) to
\begin{equation}
\partial_t\rho_k+\div(\rho_k\,u_k)=0,\label{lineartransport}
\end{equation}
where $u_k$ is assumed to satisfy for some $1<p\leq \infty$
\begin{equation}
\sup_k \|u_k\|_{L^p_tW^{1,p}_x}<\infty, \label{bounduklinear}
\end{equation}
with $\div u_k$ compact in $x$, {\em i.e.}
\begin{equation}
\limsup_k \eps_k(h)=\frac{1}{\|K_h\|_{L^1}}\int_0^T\,\int_{\Pi^{2d}} K_h(x-y)\,|\div u_k(t,x)-\div u_k(t,y)|^p\longrightarrow 0,\label{compactestdivuk}
\end{equation}
as $h\rightarrow 0$. The condition on the divergence is replaced by bounds on $\rho_k$
\begin{equation}
\frac{1}{C}\leq \inf_x\rho_k\leq \sup_x\rho_k\leq C,\quad \forall\;t\in[0,\ T].\label{lowsuprhok}
\end{equation}
One then has the well known
\begin{prop}
Assume $\rho_k$ solves \eqref{lineartransport} with the bounds \eqref{bounduklinear}, \eqref{compactestdivuk} and \eqref{lowsuprhok}; assume moreover that the initial data $\rho_k^0$ is compact. Then $\rho_k$ is locally compact and more precisely
\[
\int_{\Pi^{2d}} K_h(x-y)\,|\rho_k(t,x)-\rho_k(t,y)|\,dx\,dy\leq C\,\frac{\|K_h\|_{L^1}}{|\log (h+\eps_k(h)+\tilde \eps_k(h))|},
\]
where
\[
\tilde \eps_k(h)=\frac{1}{\|K_h\|_{L^1}}\int_{\Pi^{2d}} K_h(x-y)\,|\rho_k^0(x)-\rho_k^0(y)|\,dx\,dy.
\]
\label{proplineartransport}
\end{prop}
This type of results for {\em non Lipschitz} velocity fields $u_k$ was first obtained by {\sc R.J. Di Perna} and P.--L. {\sc Lions} in \cite{DL} with the introduction of renormalized solutions for $u_k\in W^{1,1}$ and appropriate bounds on $\div\,u_k$. This was extended to $u_k\in BV$, first
by F. {\sc Bouchut} in \cite{Bo} in the kinetic context (see also M. {\sc Hauray} in \cite{Ha2}) and then by L. {\sc Ambrosio} in \cite{Am} in the most general case. We also refer to C. {\sc Le Bris} and P.--L. {\sc Lions} in  \cite{LL, Li2}, and to the nice lecture notes written by {\sc C. De Lellis}  in \cite{DeL}.
  In general $u_k\in BV$ is the optimal regularity as shown by N. {\sc Depauw}  in \cite{DeP}. This can only be improved with specific additional structure, such as provided by low dimension, see \cite{ABC, BoDe, CCR, CR, Ha1},  Hamiltonian properties \cite{CJ,JaMa}, or as a singular integral \cite{BoCr}.

Of more specific interest for us are the results which do not require bounds on $\div u_k$ (which are not available for compressible Navier--Stokes) but replace them by bounds on $\rho_k$, such as \eqref{lowsuprhok}. The compactness in Prop. \ref{proplineartransport} was first obtained in \cite{ADM}. 

Explicit regularity estimates of $\rho_k$ have first been derived by {\sc G.~Crippa} and {\sc C.~De Lellis} in \cite{CD} (see also \cite{Ja} for the $W^{1,1}$ case). These are based on explicit control on the characteristics. While it is quite convenient to work on the characteristics in many settings, this is not the case here, in particular due to the coupling between $\div\,u_k$ and $p(\rho_k)$.  

   In many respect the proof of Prop. \ref{proplineartransport} is an equivalent of the method of G.~{\sc Crippa} and {\sc C. De Lellis} in \cite{CD} at the PDE level, instead of the ODE level. Its interest will be manifest later in the article when dealing with the full Navier--Stokes system. The idea of controlling the compactness of solutions to transport equations through estimates such as provided by Prop. \ref{propcomp} was first introduced in \cite{BeJa} but relied on a very different method.

\begin{proof}
One does not try to propagate directly 
\[
\int_{\Pi^{2d}} K_h(x-y)\,|\rho_k(t,x)-\rho_k(t,y)|\,dx\,dy.
\]
Instead one introduces the weight $w(t,x)$ solution to the auxiliary equations
\begin{equation}
\partial_t w+u_k\cdot\nabla w=-\lambda \,M\,|\nabla u_k|,\qquad w\vert_{t=0} =1,\label{auxiliaryweight}
\end{equation}
where $M\,f$ denotes the maximal function of $f$ (recalled in Section \ref{usefulsection}) and $\lambda$ is constant to be chosen large enough. 

\medskip

\noindent {\em First step: Propagation of a weighted regularity.} We propagate
\[
R(t)=\int_{\Pi^{2d}} K_h(x-y)\,|\rho_k(t,x)-\rho_k(t,y)|\,w(t,x)\,w(t,y)\,dx\,dy.
\]
Using \eqref{lineartransport}, we obtain that
\[\begin{split}
&\partial_t |\rho_k(x)-\rho_k(y)|+\div_x\,(u_k(x)\,|\rho_k(x)-\rho_k(y)|)
+\div_y\,(u_k(y)\,|\rho_k(x)-\rho_k(y)|)\\
&\quad\leq\frac{1}{2}(\div u_k(x)+\div u_k(y))\,|\rho_k(x)-\rho_k(y)|\\
&\qquad-\frac{1}{2}(\div u_k(x)-\div u_k(y))\,(\rho_k(x)+\rho_k(y))\,s_k,
\end{split}\]
where $s_k=sign\,(\rho_k(x)-\rho_k(y))$. We refer to subsection \ref{renormequation} for the details of this calculation, which is rigorously justified for a fixed $k$ through the theory of renormalized solutions in \cite{DL}.
From this equation on $|\rho_k(x)-\rho_k(y)|$, we deduce
\[\begin{split}
\frac{d}{dt}R(t)&=\int_{\Pi^{2d}} \nabla K_h(x-y)\,(u_k(x)-u_k(y))\,|\rho_k(t,x)-\rho_k(t,y)|\,w(t,x)\,w(t,y)\\
&\ -\frac{1}{2}\int_{\Pi^{2d}}  K_h(x-y)\,(\div u_k(x)-\div u_k(y))\,(\rho_k(x)-\rho_k(y))\,s_k\,w(x)\,w(y)\\
&\ +\int_{\Pi^{2d}}  K_h(x-y)\,|\rho_k(x)-\rho_k(y)|\,\left(\partial_t w+u_k\cdot\nabla w+\frac{1}{2}\,\div u_k\,w\right)\,w(y)\\
&\ +symmetric.
\end{split}
\]
Observe that by \eqref{auxiliaryweight}, $w\leq 1$ and therefore by \eqref{lowsuprhok} and the definition of $\eps_k$ in \eqref{compactestdivuk}, the second term in the right--hand side is easily bounded
\[
\int_{\Pi^{2d}}  K_h(x-y)\,(\div u_k(x)-\div u_k(y))\,(\rho_k(x)-\rho_k(y))\,s_k\,w(x)\,w(y)\leq C\,\eps_k(h).
\]
For the first term, one uses the well known inequality (see \cite{Stein, Stein2} or section \ref{usefulsection})
\[
|u_k(x)-u_k(y)|\leq C_d\,|x-y|\,(M\,|\nabla u_k|(x)+M\,|\nabla u_k|(y)),
\]
combined with the remark that from the choice of $K_h$
\[
|\nabla K_h(x-y)|\,|x-y|\leq C\,K_h(x-y).
\]
Therefore
\[\begin{split}
&\int_{\Pi^{2d}} \nabla K_h(x-y)\,(u_k(x)-u_k(y))\,|\rho_k(t,x)-\rho_k(t,y)|\,w(t,x)\,w(t,y)\\
&\leq C\,\int_{\Pi^{2d}} K_h(x-y)\,(M\,|\nabla u_k|(x)+M\,|\nabla u_k|(y))\,|\rho_k(x)-\rho_k(y)|\,w(x)\,w(y),
\end{split}
\]
and combining everything
\[\begin{split}
&\frac{d}{dt}R(t)\leq C\,\|K_h\|_{L^1}\,\eps_k(h)\\
&+\int_{\Pi^{2d}}  K_h\,|\rho_k(x)-\rho_k(y)|\,\big(\partial_t w+u_k\cdot\nabla w+C\,(\div u_k+M\,|\nabla u_k|)\,w\big)\,w(y)\\
&+symmetric.
\end{split}
\]
Since $\div\,u_k\leq d\,|\nabla u_k|\leq d\,M\,|\nabla u_k|$, by taking the constant $\lambda$ large enough in \eqref{auxiliaryweight},
\[
\partial_t w+u_k\cdot\nabla w+C\,(\div u_k+M\,|\nabla u_k|)\,w\leq 0,
\]
and hence 
\begin{equation}
R(t)\leq \int_{\Pi^{2d}}  K_h\,|\rho_k^0(x)-\rho_k^0(y)|+C\,t\,\|K_h\|_{L^1}\,\eps_k(h)\leq C\,\|K_h\|_{L^1}\,(\eps_k(h)+\tilde \eps_k(h)).\label{boundRt}
\end{equation}

\medskip

\noindent  {\em Second step: property of the weight.} We need to control the measure of the set where the weight $w$ is small. Obviously if $w$ were to vanish everywhere then the control of $R(t)$ would be trivial but of very little interest. From Eq. \eqref{auxiliaryweight}
\[
\partial_t (\rho_k\,|\log w|)+\div\,(\rho_k |\log w|)=\lambda\,\rho_k\,M\,|\nabla u_k|.
\]
And thus 
\begin{equation}
\frac{d}{dt}\int_{\Pi^d} |\log w|\,\rho_k\,dx=\lambda \,\int_{\Pi^d}\rho_k\,M\,|\nabla u_k|\,dx\leq 
\lambda\,\|\rho_k\|_{L^{p^*}}\,\|M\,|\nabla u_k|\|_{L^p}\leq C,
\label{linearlog}
\end{equation}
by \eqref{bounduklinear}, \eqref{lowsuprhok} and the fact that the maximal function is bounded on $L^p$ for $p>1$.

\medskip

\noindent  {\em Third step: Conclusion of the proof.} 
\noindent Using again \eqref{lowsuprhok}
\[
|\{x,\ w(t,x)\leq \eta\}|\leq \frac{C}{|\log \eta|}\int_{\Pi^d} |\log w|\,\rho_k\,dx\leq \frac{C}{|\log \eta|}.
\]
Thus
\[\begin{split}
&\int_{\Pi^{2d}} K_h(x-y)\,|\rho_k(t,x)-\rho_k(t,y)|\\
&\qquad=\int_{w(x)>\eta,\ w(y)> \eta} K_h\,|\rho_k(t,x)-\rho_k(t,y)|\\
&+
\int_{w(x)\leq \eta\ or\ w(y)\leq \eta} K_h(x-y)\,|\rho_k(t,x)-\rho_k(t,y)|,
\end{split}\]
and so
\[\begin{split}
&\int_{\Pi^{2d}} K_h(x-y)\,|\rho_k(t,x)-\rho_k(t,y)|\leq \frac{1}{\eta^2}\int K_h\,|\rho_k(t,x)-\rho_k(t,y)|\,w(t,x)\,w(t,y)\\
&\qquad\qquad\qquad\qquad+\frac{C}{|\log \eta|}\,\|K_h\|_{L^1}\\
&\qquad\leq  C\,\|K_h\|_{L^1}\,\left(\frac{\eps_k(h)+\tilde\eps_k(h)}{\eta^2}+\frac{1}{|\log \eta|}\right),
\end{split}
\]
which by minimizing in $\eta$ finishes the proof.
\end{proof}

\subsection{A rough sketch of the extension to compressible Navier--Stokes}
The aim of this subsection is to provide a rough idea of how to extend the previous method. We will only consider the case of general pressure laws and assume that the stress tensor is isotropic.
  When now considering the compressible Navier--Stokes system, the divergence $\div u_k$ is not given anymore but has to be calculated from $\rho_k$ through a relation of the kind
\begin{equation}
\div u_k=P(\rho_k)+R_k,\label{divukPsimp}
\end{equation}
where $R_k$ includes the force applied on the fluid, the effective stress tensor... To keep things as simple as possible here, we temporarily assume that
\begin{equation}
\begin{split}
&\sup_k \|R_k\|_{L^\infty}<\infty,\\
&\limsup_k \eps_k(h)=\frac{1}{\|K_h\|_{L^1}}\int_0^T\,\int_{\Pi^{2d}} K_h(x-y)\,|R_k(t,x)-R_k(t,y)|^p\longrightarrow 0.
\end{split}\label{assumptRk}
\end{equation}
We do not assume monotonicity on the pressure $P$ but simply the control
\begin{equation}
|P'(\rho)|\leq C\,\rho^{\gamma-1}.\label{nonmonotonerough}
\end{equation}
A modification of the previous proof then yields
\begin{prop}
Assume $\rho_k$ solves \eqref{lineartransport}  and the bounds 
\[
\sup_k \|\rho_k\|_{L^\infty_t\,L^1_x}<\infty,\qquad \sup_k \|\rho_k\|_{L^p_{t,x}}<\infty\quad\mbox{with}\ p\geq\gamma+1.
\]
Assume moreover that $\sup_k \|u_k\|_{L^2_t H^1_x}<\infty$, that \eqref{divukPsimp} holds with the bounds \eqref{assumptRk} on $R_k$ and \eqref{nonmonotonerough} on $P$. Then $\rho_k$ is locally compact away from vacuum and more precisely
\[
\int_{\rho_k(x)\geq \eta,\;\rho_k(y)\geq \eta} K_h(x-y)\,|\rho_k(t,x)-\rho_k(t,y)|\,dx\,dy\leq C_\eta\,\frac{\|K_h\|_{L^1}}{|\log (h+(\eps_k(h))+\tilde \eps_k(h))|},
\]
where
\[
\tilde \eps_k(h)=\frac{1}{\|K_h\|_{L^1}}\int_{\Pi^{2d}} K_h(x-y)\,|\rho_k^0(x)-\rho_k^0(y)|\,dx\,dy.
\]
\label{proproughsketch}
\end{prop}
Unfortunately Prop. \ref{proproughsketch} is only a rough and unsatisfactory attempt for the following reasons
\begin{itemize}
\item The main problem with Prop. \ref{proproughsketch} is that it does not imply compactness on the sequence $\rho_k$ because it only controls oscillations of $\rho_k$ for large enough values but we do not have any lower bounds on $\rho_k$. In fact not only can $\rho_k$ vanish but for weak solutions, vacuum could even form: That is there may be a set of non vanishing measure where $\rho_k=0$. 
This comes from the fact that the proof only give an estimate on
\[
\int_{\Pi^{2d}}K_h(x-y)\,|\rho_k(t,x)-\rho_k(t,y)|\,w(x)\,w(y)\,dx\,dy,
\]    
but since there is no lower bound on $\rho_k$ anymore, estimates like \eqref{linearlog} only control the set where $w(x)\,w(y)$ is small and both $\rho_k(x)$ and $\rho_k(y)$ are small. Unfortunately $|\rho_k(t,x)-\rho_k(t,y)|$ could be large while only one of $\rho_k(x)$ and $\rho_k(y)$ is small (and hence  $w(x)\,w(y)$ is small as well).
\item The solution is to work with $w(x)+w(y)$ instead of $w(x)\,w(y)$. Now the sum  $w(x)+w(y)$ can only be small if $|\rho_k(t,x)-\rho_k(t,y)|$ is small as well, meaning that a bound on
\[
\int_{\Pi^{2d}}K_h(x-y)\,|\rho_k(t,x)-\rho_k(t,y)|\,(w(x)+w(y))\,dx\,dy,
\]    
together with estimates like \eqref{linearlog} would control the compactness on $\rho_k$. Unfortunately this leads to various additional difficulties because some terms are now not localized at the right point.  For instance one has problems estimating the commutator term in $\nabla K_h\cdot(u_k(x)-u_k(y))$ or one cannot directly control terms like $\div u_k(x)\,w(y)$ by the penalization which would now be of the form $M\,|\nabla u_k|(x)\,w(x)$. Some of these problems are solved by using more elaborate harmonic analysis tools, others require a more precise analysis of the structure of the equations. Those difficulties are even magnified for anisotropic stress tensor which add even trickier non-local terms.
\item The integrability assumption on $\rho_k$, $p>\gamma+1$ is not very realistic and too demanding. If $p= \gamma (1+2/d) - 1$ as for the compressible Navier-Stokes
equations with power law $P(\rho)= a \rho^\gamma$, then this requires $\gamma > d$.
 Improving it creates important difficulty in the interaction with the penalization. It forces us to modify the penalization and prevent us from getting an inequality like \eqref{linearlog} and in fact only modified inequalities can be obtained, of the type
\[
\sup_k\int_{\Pi^d} |\log w(t,x)|^\theta\,\rho_k(t,x)\,dx<\infty.
\]
\item The bounds \eqref{assumptRk} that we have assumed for simplicity on $R_k$ cannot be deduced from the equations. The effective pressure is not bounded in $L^\infty$ and it is not {\em a priori} compact (it will only be so at the very end as a consequence of $\rho_k$ being compact). Instead we will have to establish regularity bounds on the effective pressure when integrated against specific test functions; but in a manner more precise than the existing {\sc Lions-Feireisl} theory, see Lemma \ref{Drhou} later.
\item This is of course only a stability result, in order to get existence one has to work with appropriate approximate system. This will be the subject of  Section \ref{sec-proof-mainresults}. 
\end{itemize}
\begin{proof} One now works with a different equation for the weight
\begin{equation}
\partial_t w+u_k\cdot\nabla w=-\lambda\,\left(M\,|\nabla u_k|+\rho_k^\gamma\right),
\quad w\vert_{t=0} =1,\label{auxiliaryweight2}
\end{equation}
where $M\,f$ is again the maximal function of $f$. 

\bigskip

\noindent {\em First step: Propagation of some weighted regularity} The beginning of the first step essentially remains the same as in the proof of Prop. \ref{lineartransport}: One propagates
\[
R(t)=\int_{\Pi^{2d}} K_h(x-y)\,|\rho_k(t,x)-\rho_k(t,y)|\,w(t,x)\,w(t,y)\,dx\,dy.
\]
The initial calculations are nearly identical. The only difference is that we do not have \eqref{compactestdivuk} any more so we simply keep the term with $\div u_k(x)-\div u_k(y)$ for the time being. We thus obtain
\begin{equation}
\begin{split}
\frac{d}{dt} R(t)\leq &  \int_{\Pi^{2d}} K_h(x-y)\,(\div u_k(x)-\div u_k(y))\\
& \hskip3cm      (\rho_k(t,x)+\rho_k(t,y))\,s_k\,w(t,x)\,w(t,y)\\
&
 -\lambda\,\int_{\Pi^{2d}}(\rho_k^\gamma(x)+\rho_k^\gamma(y))\, |\rho_k(t,x)-\rho_k(t,y)|\,w(t,x)\,w(t,y),
\end{split}\label{R}
\end{equation}
by taking the additional term in Eq. \eqref{auxiliaryweight2} into account. This is of course where the coupling between $u_k$ and $\rho_k$ comes into play, here only through the simplified equation \eqref{divukPsimp}. Thus
\begin{equation}
\begin{split}
&\int_{\Pi^{2d}} K_h(x-y)\,(\div u_k(x)-\div u_k(y))\,(\rho_k(t,x)+\rho_k(t,y))\,s_k\,w(t,x)\,w(t,y)\\
&\ =\int_{\Pi^{2d}} K_h(x-y)\,(P(\rho_k(x))-P(\rho_k(y))) \\
& \hskip6cm \,(\rho_k(t,x)+\rho_k(t,y))\,s_k\,w(t,x)\,w(t,y)\\
& +\int_{\Pi^{2d}} K_h(x-y)\,(R_k(x)-R_k(y)) \,(\rho_k(t,x)+\rho_k(t,y)) \,s_k\,w(t,x)\,w(t,y).
  \label{divv}\\
\end{split}
\end{equation}
By the uniform $L^p$ bound on $\rho_k$ and the estimate \eqref{assumptRk}, one has
\begin{equation}
\begin{split}
&\int_0^t\int_{\Pi^{2d}} K_h(x-y)\,(R_k(x)-R_k(y)) \,(\rho_k(t,x)+\rho_k(t,y))\,s_k\,w(t,x)\,w(t,y)\\
&\qquad\qquad\leq C\,\|K_h\|_{L^1}\,(\eps_k(h))^{1-1/p}. \label{Rkcomp}
\end{split}
\end{equation}
Using now \eqref{nonmonotone}, it is possible to bound
\[
\begin{split}
|P(\rho_k(x))-P(\rho_k(y))|&\leq |\rho_k(t,x)-\rho_k(t,y)|\,\int_0^1 |P'(s\,\rho_k(x)+(1-s)\,\rho_k(y))|\,ds\\
&\leq C\,(\rho_k^{\gamma-1}(x)+\rho_k^{\gamma-1}(y))\,|\rho_k(t,x)-\rho_k(t,y)|,
\end{split}
\]
leading to
\[\begin{split}
&\int_{\Pi^{2d}} K_h(x-y)\,(P(\rho_k(x))-P(\rho_k(y))) \,(\rho_k(t,x)+\rho_k(t,y))\,s_k\,w(t,x)\,w(t,y)\\
&\ \leq C\,\int_{\Pi^{2d}} K_h(x-y)\,(\rho_k^{\gamma-1}(x)+\rho_k^{\gamma-1}(y)) \, (\rho_k(t,x)+\rho_k(t,y))\\
&\hskip5cm |\rho_k(t,x)-\rho_k(t,y)|\,w(t,x)\,w(t,y)\\
&\ \leq C\,\int_{\Pi^{2d}} K_h(x-y)\,(\rho_k^{\gamma}(x)+\rho_k^{\gamma}(y)) \,
|\rho_k(t,x)-\rho_k(t,y)|\,w(t,x)\,w(t,y).
\end{split}
\]
Using now this estimate, the Equality \eqref{divv}, the compactness \eqref{Rkcomp} and by taking $\lambda$ large enough  one finds from \eqref{R}
\[
R(t)\leq R(0)+C\,\|K_h\|_{L^1}\,(\eps_k(h))^{1-1/p}.
\]

\medskip

\noindent{\em Second step and third steps: Property of the weight and conclusion.} The starting point is again the same and gives
\[
\int_{\Pi^d} |\log w(t,x)|\,\rho_k(t,x)\,dx\leq C,
\]
with $C$ independent of $k$ but where we now need $\rho_k\in L^p$ with $p\geq \gamma+1$ because of the additional term in Eq. \eqref{auxiliaryweight2}. Now
\[
\begin{split}
&\int_{\rho_k(x)\geq \eta,\ \rho_k(y)\geq\eta} K_h(x-y)\,|\rho_k(t,x)-\rho_k(t,y)|\\
&\quad=\int_{\rho_k(x)\geq \eta,\ \rho_k(y)\geq\eta,\ w(x)\geq\eta',\ w(y)\geq\eta' } K_h(x-y)\,|\rho_k(t,x)-\rho_k(t,y)|\\
&\qquad+\|K\|_{L^1}\,\int_{\rho_k(x)\geq \eta,\ w(x)\leq\eta'} (1+\rho_k(x)).
\end{split}
\]
On the one hand,
\[\begin{split}
&\int_{\rho_k(x)\geq \eta,\ w(x)\leq\eta'} (1+\rho_k(x))\leq (\frac{1}\eta+1) \int_{w(x)\leq\eta'} \rho_k(x)\\
&\qquad\qquad\leq (\frac{1}\eta+1)\frac{1}{|\log \eta'|}\,\int |\log w(x)|\,\rho_k(x)\leq  (\frac{1}\eta+1)\frac{C}{|\log \eta'|}.
\end{split}\]
On the other hand
\[\begin{split}
&\int_{\rho_k(x)\geq \eta,\ \rho_k(y)\geq\eta,\ w(x)\geq\eta',\ w(y)\geq\eta' } K_h(x-y)\,|\rho_k(t,x)-\rho_k(t,y)|\\
&\quad\leq \frac{1}{(\eta')^2}\,\int_{\Pi^{2d}}\,K_h(x-y)\,|\rho_k(t,x)-\rho_k(t,y)|\,w(t,x)\,w(t,y).
\end{split}
\]
Therefore
\[\begin{split}
&\int_{\rho_k(x)\geq \eta,\ \rho_k(y)\geq\eta} K_h(x-y)\,|\rho_k(t,x)-\rho_k(t,y)| \leq
 (\frac{1}\eta+1) \frac{C}{|\log \eta'|}\\
&\hskip3cm+\frac{C}{(\eta')^2}\,\|K_h\|_{L^1}\,\left(\tilde\eps_k(h)+(\eps_k(h))^{1-1/p}\right),
\end{split}
\]
which concludes the proof by optimizing in $\eta'$.
\end{proof}
\section{Stability results\label{secstability}}
\subsection{The equations}
\subsubsection {General pressure law.}
   Let $(\rho_k, u_k)$ solve
\begin{equation}
\partial_t \rho_k+\div(\rho_k\, u_k)=\alpha_k\Delta\rho_k.\label{continuity}
\end{equation}
and 
\begin{equation}
\mu_k(t,x)\, \div u_k=P_k(t,x,\rho_k)+\Delta^{-1}\,\div(\partial_t(\rho_k\,u_k)+\div(\rho_k\,u_k\otimes u_k))+F_k.\label{divu0}
\end{equation}

\noindent {\it Important remark.}
 Note that we allow here a possible explicit dependence on $t$ and $x$ in $P_k$. This does not really affect our stability results and it will allow later to use directly the stability estimates for the case with temperature. We will write some comments for reader's who are only interested by the hypothesis mentioned in the barotropic  theorem.

\smallskip

   The viscosity of the fluid is assumed to be bounded from below and above
\begin{equation}
\exists \bar\mu,\qquad \frac{1}{\bar \mu}\leq \mu_k(t,x)\leq \bar\mu\label{elliptic}
\end{equation}
We consider  the following control on the density (for $p>1$) 
\begin{equation}
\sup_k \left[
\|\rho_k\|_{L^\infty([0,\ T],\ L^1(\Pi^d) \cap L^\gamma(\Pi^d))}+\|\rho_k\|_{L^p([0,\ T]\times\Pi^d)}\right]<\infty.\label{boundrhok}
\end{equation}
and the following control for $u_k$
\begin{equation}
\sup_k \Bigl[ \|\rho_k\,|u_k|^2\|_{L^\infty([0,\ T],\ L^1(\Pi^d))} + 
    \|\nabla u_k\|_{L^2(0,T;L^2(\Pi^d))} \Bigr] <+ \infty. \label{bounduk}
\end{equation}
   We also need some control on the time derivative of $\rho_k u_k$ through
\begin{equation}
\exists \bar p>1,\quad\sup_k\;\left\|\partial_t(\rho_k u_k)\right\|_{L^2_t W^{-1,\bar p}_x}<\infty,
\label{rhout}
\end{equation}
and on the time derivative of $\rho_k$ namely
\begin{equation}
\exists  q>1,\quad\sup_k\;\left\|\partial_t \rho_k\right\|_{L^q_t W^{-1,q}_x}<\infty.
\label{rhorho_t}
\end{equation}

\noindent {\it Remark.} Note  that  usually (see for instance \cite{Fei}--\cite{Li2}), \ref{rhout} and \ref{rhorho_t} are   consequences of the  momentum equation and the mass equation using the uniform estimates given by the energy  estimates and the extra integrability on the density.

\bigskip

Concerning the equation of state, we will consider that for every $y,\;x$, $P_k(t,x,s)$ continuous in $s$ on $[0,+\infty)$ and positive,  $P_k(s=0)=0$, $P_k$ locally Lipschitz in $s$ on $(0,+\infty)$ with  one of the two following cases:
\begin{itemize}
\item {\rm i)} Pressure laws with a quasi-monotone property:

  There exists $,\bar P,\; \rho_0$ independent of $t,\;x$ such that: \\
  If $s\geq \rho_0$, $P_k(t,x,s)$ is a function $\tilde P_k(t,x)$ plus a function independent of $t,\;x$ Ê
  
 \smallskip
  
  and
\begin{equation}
\begin{split}
 &\partial_s (\bigl[P_k(t,x,s)- \tilde P_k(t,x)\bigr]/s)\geq 0  \hbox{ for all } s\geq \rho_0, 
     \qquad \lim_{s\to +\infty} P_k(t,x,s)=+\infty,\\
 & |P_k(t,x,r)-P_k(t,y,s)| \\
 & \hskip2cm \leq \bar P\,|r-s|+Q_k(t,x,y)\hbox{ for all } r,\;s\leq \rho_0,\ x,\;y\in\Pi^d,\\
&\limsup_{h\rightarrow 0}\,\sup_k\int_0^T\int_{\Pi^{2d}} \frac{K_h(x-y)}{\|K_h\|_{L^1}}\,
    \Bigl(|\tilde P_k(t,x)-\tilde P_k(t,y)| \\
& \hskip5cm +Q_k(t,x,y)\Bigr)\,dx\,dy\,dt=0. 
\end{split}\label{monotone}
\end{equation}

\item {\rm ii)} Non-monotone pressure laws (with very general Lipschitz pressure laws);

   There exists $\bar P>0$, $\tilde\gamma>0$ and $\tilde P_k(t,x)$ in $L^2([0,\ T]\times\Pi^d)$, 
   $Q_k$ and $R_k$ in $L^1([0,\ T]\times\Pi^{2d})$ such  that for all $t,\;x,\;y$ 
\begin{equation}\begin{split}
&|P_k(t,x,\rho_k(t,x))-P_k(t,y,\rho_k(t,y))|\leq Q_k(t,x,y)\\
&+ \left[\bar P\,\left((\rho_k(t,x))^{\tilde \gamma-1}+(\rho_k(t,x))^{\tilde \gamma-1}\right)+\tilde P_k(t,x)+\tilde P_k(t,y)\right]\,|\rho_k(t,x)-\rho_k(t,y)|,\\
&P_k(t,x,\rho_k(t,x))\leq \bar P\,(\rho_k(t,x))^{\tilde\gamma}+R_k(t,x) \hbox{ with } R_k\ge 0,\\
& \sup_k(\|\tilde P_k\|_{L^2([0,\ T]\times\Pi^d)}\!+\!\|R_k\|_{L^1([0,\ T]\times\Pi^d)})<\infty,\\
&\limsup_{h\rightarrow 0}\sup_k\int_0^T\int_{\Pi^{2d}} \frac{K_h(x-y)}{\|K_h\|_{L^1}}\,
  \Bigl(|\tilde P_k(t,x)-\tilde P_k(t,y)| \\
 & \hskip2cm +|R_k(t,x) - R_k(t,y)|+Q_k(t,x,y)\Bigr)\,dx\,dy\,dt=0. \\
\end{split}\label{nonmonotone}
\end{equation}
\end{itemize}

\noindent {\it Two important remarks.}
 
\noindent 1 --   It is important to note that the general hypothesis on pressure laws will be used to conclude in the heat-conducting compressible Navier-Stokes case: see Section \ref{withtemperature}. 
     The pressure law will include for instance a radiative part (namely a part depending only on the temperature as in \cite{FeNo})  and a pressure law in density with coefficients depending on temperature (see the comments in the section under consideration). 
     Note that compactness in space for the temperature will be obtained directly from the thermal flux:
 This property will be strongly used to check  hypothesis for the pressure under consideration
 especially the last one
\begin{equation}
\begin{split}
 \limsup_{h\rightarrow 0}\sup_k\int_0^T\int_{\Pi^{2d}} \frac{K_h(x-y)}{\|K_h\|_{L^1}} 
 &\Bigl( |\tilde P_k(t,x)-\tilde P_k(t,y)| \\
 & + |R_k(t,x) - R_k(t,y)| +Q_k(t,x,y)\Bigr )=0.
 \end{split}
\end{equation}
  Of course, the first step will be to show that with the pressure under consideration
estimates on the velocity and the density mentioned in the stability process are obtained.

\medskip   
\noindent 2 --  In the basic case where $P_k$ does not depend explicitly on $t$ or $x$
(namely $\tilde P_k \equiv 0$,  $Q_k\equiv 0$ and $R_k\equiv 0$) then \eqref{nonmonotone} reduces to the very simple condition
\[
|P_k(r)-P_k(s)|\leq \bar P\,r^{\tilde\gamma-1} |r-s|.
\]
Note that this assumption is satisfied if $P$ is locally Lipschitz on $(0,+\infty)$
with
$$|P'(s)| \le \bar P s^{\tilde \gamma -1}$$
namely with the hypothesis mentioned in Theorem 3.1.
Readers who are interested by the barotropic case with this hypothesis are invited to 
choose $\tilde P_k \equiv 0$,  $Q_k\equiv 0$ and $R_k\equiv 0$ in the proof. 
Hypothesis on the pressure law are used  in the subsection named "The coupling
with the pressure law''.

\medskip

\noindent {\it Remark.}  Note that i) with the lower bound $P(\rho) \ge C^{-1} \rho^\gamma - C$
provides the same assumptions than in the article \cite{Fe1} by E. {\sc Feireisl}. Point
i) will be used to construct approximate solutions in the non-monotone case. 

\medskip

\subsubsection {A non-isotropic stress tensor.}  
In that case, assume $(\rho_k, u_k)$ solve \eqref{continuity} with $\alpha_k=0$
\[
\partial_t \rho_k+\div(\rho_k\, u_k)= 0.
\]
and 
\begin{equation} \nonumber 
\hskip-5cm \div u_k= \nu_k\, P_k(\rho_k) 
        + \nu_k\,a_\mu A_\mu \,  P_k(\rho_k)
\end{equation}
\begin{equation} 
\hskip1cm +\nu_k\,(\Delta -a_\mu\, E_k)^{-1}\,\div(\partial_t(\rho_k\,u_k)+\div(\rho_k\,u_k\otimes u_k))\label{divu0noniso}
\end{equation}
where $\displaystyle A_\mu = (\Delta -a_\mu\,E_k)^{-1}\, E_k$. 

\bigskip

\noindent {\it Remark.} If one considers a symmetric anisotropy, $\div (A\,D\,u)$ in Theorem \ref{MainResultAniso}, then instead of \eqref{divu0noniso} we have the more complicated formula 
\begin{equation}
\begin{split}
\div u_k= &\nu_k\, P_k(\rho_k) 
        + \nu_k\,a_\mu A_\mu \,  P_k(\rho_k)\\
&\quad +\nu_k\,\div\,(\Delta\,I -a_\mu\, E_k)^{-1}\,(\partial_t(\rho_k\,u_k)+\div(\rho_k\,u_k\otimes u_k)),
\end{split}\label{divu0nonisosym}
\end{equation}
where $\displaystyle A_\mu = (\Delta\,I -a_\mu\,E_k)^{-1}\cdot \tilde E_k$. But now $E_k$ and $\tilde E_k$ may be different and are vector-valued operators, so that in particular $(\Delta\,I -a_\mu\,E_k)^{-1}$ means inverting a vector valued elliptic system. Except for the formulation there would however be no actual difference in the rest of the proof.

\medskip

Coming back to \eqref{divu0noniso}, we assume ellipticity on $\nu_k$
\begin{equation}
0<\underline\nu\leq \nu_k\leq \overline\nu<\infty.\label{elliptic2}
\end{equation}
We assume that $E_k$ is a given operator (differential or integral) s.t. \begin{itemize}
\item $(\Delta -a_\mu\,E_k)^{-1}\,\Delta$ is bounded on every $L^p$ space,
\item $A_\mu=(\Delta -a_\mu\,E_k)^{-1}\,E_k$ is bounded of norm less than $1$ on every $L^p$ space and can be represented by a convolution with a singular integral denoted by $A_\mu$ still
\[
A_\mu \,f=A_\mu\star_x f,\quad |A_\mu(x)|\leq \frac{C}{|x|^d},\quad \int A_\mu(x)\,dx=0.
\] 
\end{itemize}
Note here that to make more apparent the smallness of the non isotropic part, we explicitly scale it with $a_\mu$.
    We consider again  the control \eqref{boundrhok} on the density  but for 
    $p>\gamma^2/(\gamma-1)$,
and the bound \eqref{bounduk} for $u_k$.
   We also need the same controls: \eqref{rhout} on the time derivative of $\rho_k u_k$ and
\eqref{rhorho_t} on the time of the $\rho_k$.

\medskip

The main idea in this part is to investigate the compactness for an anisotropic viscous stress obtained as the perturbation of the usual isotropic viscous stress tensor, namely 
$-\div(A\,\nabla u)+(\lambda+\mu)\, \nabla {\rm div} u$ assuming $A=\mu\,Id+\delta A$ and 
$a_\mu = \|\delta A\|\le \varepsilon$ for some small enough  $\varepsilon$.  

\subsection{The main stability results: Theorems \ref{maincompactness}, \ref{maincompactness2} and \ref{maincompactness3}}

\subsubsection{General pressure laws}
The main step in that case is to prove the two compactness results
\begin{theorem} \label{maincompactness}
Assume that $\rho_k$ solves \eqref{continuity}, $u_k$ solves \eqref{divu0} with the bounds \eqref{elliptic}, \eqref{bounduk}, \eqref{rhout}, \eqref{rhorho_t}, and that $\mu_k$ and $F_k$ are compact in $L^1$. Moreover
\begin{itemize}
\item[i)] If $\alpha_k>0$, we assume the estimate \eqref{boundrhok} on $\rho_k$ with $\gamma>3/2$ and $p>2$ and quasi-monotonicity on $P_k$ through \eqref{monotone}.
\item[ii)] if $\alpha_k=0$, it is enough to assume \eqref{boundrhok} with $\gamma>3/2$ and $p>\max(2,\tilde \gamma)$ and only \eqref{nonmonotone} on $P_k$.
\end{itemize}
Then the sequence $\rho_k$ is compact in $L^1_{loc}$. 
\end{theorem}

\bigskip

We also provide a complementary result which is a more precise rate of compactness
 away from the vacuum namely
 
\begin{theorem} \label{maincompactness2}
Assume again that $\rho_k$ solves \eqref{continuity} with $\alpha_k=0$, $u_k$ solves \eqref{divu0} with the bounds \eqref{elliptic}, \eqref{bounduk}, \eqref{rhout}, \eqref{rhorho_t} and that $\mu_k$ and $F_k$ are compact in $L^1$. Assume that \eqref{boundrhok} holds with $\gamma>d/2$ and $p>\max(2,\tilde\gamma)$ and that $P_k$ satisfies \eqref{nonmonotone}. Then there exists $\theta>0$ and a continuous function $\eps$ with $\eps(0)=0$, depending only on $\mu_k$ and $F_k$ s.t.
\[
\int_{\Pi^{2d}} \ind_{\rho_k(x)\geq \eta}\,\ind_{\rho_k(y)\geq \eta}\,K_h(x-y)\,\chi(\delta\rho_k)\leq \frac{C\,\|K_h\|_{L^1}}{\eta^{1/2}\,
|\log (\eps(h)+h^\theta)|^{\theta/2}}.
\]
\end{theorem}

For instance if $\tilde P_k$, $R_k$, $\mu_k$ and $F_k$ are uniformly in $W^{s,1}$ for $s>0$, then 
for some constant $C>0$
\[
\int_{\Pi^{2d}} \ind_{\rho_k(x)\geq \eta}\,\ind_{\rho_k(y)\geq \eta}\,K_h(x-y)\,\chi(\delta\rho_k)\leq C\,\frac{\|K_h\|_{L^1}}{\eta^{1/2}\,
 |\log h|^{\theta/2}}.
\]
Since those results depend on the regularity of $\mu_k$ and $F_k$, we denote  $\eps_0(h)$ a continuous function with $\eps_0(0)=0$ s.t. 
\begin{equation}
\begin{split}
&\int_0^T\!\int_{\Pi^{2d}} K_h(x-y)\,\Big(|F_k(t,x)-F_k(t,y)|+|\mu_k(t,x)-\mu_k(t,y)|\\
&+|\tilde P_k(t,x)-\tilde P_k(t,y)|+|R_k(t,x)- R_k(t,y)|+|Q_k(t,x,y)|\Big)\leq \eps_0(h)\,\|K_h\|_{L^1}.
\end{split}\label{compactFmu}
\end{equation}


\subsubsection{Non isotropic stress tensor}
In that case our result reads
 \begin{theorem}  \label{maincompactness3}
Assume that $\rho_k$ solves \eqref{continuity}, $u_k$ solves \eqref{divu0noniso} with the bounds \eqref{bounduk}, \eqref{rhout}, \eqref{rhorho_t} and \eqref{elliptic2} together with all the assumptions on $E_k$ below \eqref{divu0noniso}. Assume as well that $P_k$ satisfies \eqref{monotone} and that \eqref{boundrhok} with $\gamma>d/2$ and $p>\gamma^2/(\gamma-1)$. There exists a universal constant $C_*> 0$ s.t. if 
\[
a_\mu\leq C_*,
\]
then $\rho_k$ is compact in $L^1_{\rm loc}$. 
\end{theorem}

\bigskip

\noindent {\bf Remarks.}
Theorems \ref{maincompactness}, \ref{maincompactness2}, \ref{maincompactness3} are really the main contributions of this article. For instance, deducing Theorems \ref{MainResultPressureLaw} and \ref{MainResultAniso} follows usual and straightforward approximation procedures. 

As such the main improvements with respect to the existing theory can be seen in the fact that point ii) in Theorem \ref{maincompactness} does not require monotonicity on $P_k$ and in the fact that Theorem \ref{maincompactness3} does not require isotropy on the stress tensor.

Our starting approximate system involves diffusion, $\alpha_k\neq 0$, in the continuity equation \eqref{continuity}. As can be seen from point i) of Theorem \ref{maincompactness}, our compactness result in that case requires an isotropic stress tensor and a pressure $P_k$ which is monotone after a certain point by \eqref{monotone}. This limitation is the reason why we have to consider also approximations $P_k$ and $E_k$ of the pressure and the stress tensor. While it may superficially appear that we did not improve the existing theory in that case with diffusion, we want to point out that
\begin{itemize}
\item We could not have used {\sc P.--L. Lions}' approach because this requires strict monotonicity: $P_k'>0$ everywhere. Instead any non-monotone pressure $P$ satisfying \eqref{nonmonotone} can be approximated by $P_k$ satisfying \eqref{monotone} simply by considering $P_k=P+\eps_k\,\rho^{\bar\gamma}$ as long as $\bar\gamma>\tilde\gamma$ and thus without changing the requirements on $\gamma$.
\item {\sc E. Feireisl} {\it et al} can handle ``quasi-monotone'' pressure laws satisfying \eqref{monotone} together with diffusion but they require higher integrability on $\rho_k$ for this: $p\geq 4$ in \eqref{boundrhok}. This in turn leads to a more complex approximation procedure.
\end{itemize}
%
\section{Technical lemmas and renormalized solutions\label{usefulsection}}
%
\subsection{Useful technical lemmas}
 Let us recall the well known inequality, which we used in subsection \ref{seclineartransport} and will use several times in the following (see \cite{Stein} for instance)
\begin{equation}
|\Phi(x)-\Phi(y)|\leq C\,|x-y|\,(M|\nabla\Phi|(x)+M|\nabla \Phi|(y)),\label{maximalineq}
\end{equation}
where $M$ is the localized maximal operator
\begin{equation}
M\,f(x)=\sup_{r\leq 1} \frac{1}{|B(0,r)|}\,\int_{B(0,r)} f(x+z)\,dz.\label{defmaximal}
\end{equation}
As it will be seen later, there is a technical difficulty in the proof, which would lead us to try (and fail) to control $M|\nabla u_k|(y)$ by $M|\nabla u_k|(x)$. Instead we have to be more precise than \eqref{maximalineq} in order to avoid this. To deal with such problems,
we use more sophisticated tools.
    First one has 
\begin{lemma} There exists $C>0$ s.t. for any $u\in W^{1,1}(\Pi^d)$, one has
\[
|u(x)-u(y)|\leq C\,|x-y|\,(D_{|x-y|}u (x)+D_{|x-y|}u (y)), 
\]
where we denote
\[
D_h u(x)=\frac{1}{h}\,\int_{|z|\leq h} \frac{|\nabla u(x+z)|}{|z|^{d-1}}\,dz.
\]\label{diffulemma}
\end{lemma}
\begin{proof}
A full proof of such well known result can for instance be found in \cite{Ja} in a more general setting namely $u\in BV$. The idea is simply to consider trajectories $\gamma(t)$ from $x$ to $y$ which stays within the ball of diameter $|x-y|$ to control
\[
|u(x)-u(y)|\leq \int_0^1 \gamma'(t)\cdot\nabla u(\gamma(t))\,dt.
\]
And then to average over all such trajectories with length of order $|x-y|$. 
Similar calculations are also present for instance in \cite{FaKeSe}.
\end{proof}
 
Note that this result implies the estimate \eqref{maximalineq} as
\begin{lemma}
\label{compareDhmax}
 There exists $C>0$, for any $u\in W^{1,p}(\Pi^d)$ with $p\geq 1$
\[
D_h\,u(x)\leq C\,M|\nabla u|(x).
\]\
\end{lemma}
\begin{proof}
Do a dyadic decomposition and define $i_0$ s.t. $2^{-i_0-1}< h\leq 2^{-i_0}$
\[\begin{split}
D_h\,u(x)&\leq \frac{1}{h}\sum_{i\geq i_0} \int_{2^{-i-1}<|z|\leq 2^{-i}} \frac{|\nabla u(x+z)|}{|z|^{d-1}}\,dz\\
&\leq \sum_{i\geq i_0} \frac{2^{(i+1)\,(d-1)}}{h}  \int_{2^{-i-1}<|z|\leq 2^{-i}} |\nabla u(x+z)|\,dz\\
&\leq 2^{d-1}\sum_{i\geq i_0} |B(0,1)|\,\frac{2^{-i}}{h} M\,|\nabla u|(x)
\leq C\,M\,|\nabla u|(x).
\end{split}\]
\end{proof}

The key improvement in using $D_h$ is that small translations of the operator $D_h$ are actually easy to control
\begin{lemma} \label{shiftDulemma}
For any $1<p<\infty$, there exists $C>0$ s.t. for any $u\in H^1(\Pi^d)$
\begin{equation}
\int_{h_0}^{1} \int_{\Pi^d} \overline K_h(z)\,\|D_{|z|}\,u(.)-D_{|z|}\,u(.+z)\|_{L^p}\,dz\,\frac{dh}{h}\leq C\,\,\|u\|_{B^1_{p,1}},\label{optimalbesov}
\end{equation}
where the definition and basic properties of the Besov space $B^1_{p,1}$ are recalled in section \ref{Besov}. As a consequence
\begin{equation}
\int_{h_0}^{1} \int_{\Pi^d} \overline K_h(z)\,\|D_{|z|}\,u(.)-D_{|z|}\,u(.+z)\|_{L^2}\,dz\,\frac{dh}{h}\leq C\,|\log h_0|^{1/2}\,\|u\|_{H^1}.\label{squarefunction}
\end{equation}
It is also possible to disconnect the shift from the radius in $D_r u$ and obtain for instance
\begin{equation}
\int_{h_0}^{1} \int_{\Pi^{2d}} \overline K_h(z)\,\overline K_h(w)\|D_{|z|}\,u(.)-D_{|z|}\,u(.+w)\|_{L^2}\,dz\, dw\,\frac{dh}{h}\leq C\,|\log h_0|^{1/2}\,\|u\|_{H^1}.\label{squarefunctiondisc}
\end{equation}
\end{lemma} 
We can in fact write a more general version of Lemma \ref{shiftDulemma} for any kernel 
\begin{lemma} \label{shiftNlemma}
For any $1<p<\infty$, any family $N_r\in W^{s,1}(\Pi^d)$ for some $s>0$ s.t. 
\[\begin{split}
&\sup_{|\omega|\leq 1}\sup_r r^{-s}\,\int_{\Pi^d} |z|^s\,|N_r(z)-N_r(z+r\,\omega)|\,dz<\infty,\\
& \sup_r \left(\|N_r\|_{L^1}+ r^s\,\|N_r\|_{W^{s,1}}\right)<\infty,
\end{split}\]
there exists $C>0$ s.t. for any $u\in L^p(\Pi^d)$ 
\begin{equation}
\int_{h_0}^{1} \int_{\Pi^d} \overline K_h(z)\,\|N_{h}\star u(.)-N_{h}\star u(.+z)\|_{L^p}\,dz\,\frac{dh}{h}\leq C\,|\log h_0|^{1/2}\,\|u\|_{L^p}.\label{squarefunctiongen}
\end{equation}
\end{lemma}
We will mostly use the specific version in Lemma \ref{shiftDulemma} but will need the more general Lemma \ref{shiftNlemma} to handle the anisotropic case in Lemma \ref{divanisotropic}.
   Both lemmas are in fact a corollary of a classical result
\begin{lemma} \label{shiftconvolution}
For any $1<p<\infty$, any family $L_r$ of kernels satisfying for some $s>0$
\begin{equation}
\int L_r=0,\quad \sup_r \left(\|L_r\|_{L^1}+ r^s\,\|L_r\|_{W^{s,1}}\right)\leq C_L,\quad \sup_r r^{-s}\,\int |z|^s\,|L_r(z)|\,dz\leq C_L.
\end{equation}
Then there exists $C>0$ depending only on $C_L$ above s.t. for any $u\in L^p(\Pi^d)$
\begin{equation}
\int_{h_0}^{1} \,\|L_r \star u\|_{L^p}\,\frac{dr}{r}\leq C\,\,\|u\|_{B^0_{p,1}}.\label{optimalbesovconvolution}
\end{equation}
As a consequence for $p\leq 2$
\begin{equation}
\int_{h_0}^{1}  \|L_r\star\,u\|_{L^p}\,\frac{dr}{r}\leq C\,|\log h_0|^{1/2}\,\|u\|_{L^p}.\label{squarefunctionconvolution}
\end{equation}
\end{lemma}
Note that by a simple change of variables in $r$, one has for instance
for any fixed power $l$
\[
\int_{h_0}^{1}  \|L_{r^l}\star\,u\|_{L^p}\,\frac{dr}{r}\leq C_l\,|\log h_0|^{1/2}\,\|u\|_{L^p}.
\]

\noindent {\it Remark.} The bounds \eqref{squarefunction} and \eqref{squarefunctionconvolution} could also be obtained by straightforward application of the so-called square function, see the book written by E.M. {\sc Stein} \cite{Stein}. We instead use Besov spaces as this yields the interesting and optimal inequalities \eqref{optimalbesov}-\eqref{optimalbesovconvolution} as an intermediary step.

\bigskip
\noindent {\bf Proof of Lemma \ref{shiftDulemma} and Lemma \ref{shiftNlemma} assuming  Lemma \ref{shiftconvolution}.} First of all observe that $D_h\,u=N_h\star u$ with
\[
N_h=\frac{1}{h\,|z|^{d-1}}\,\ind_{|z|\leq h|}
\]
which satisfies all the assumptions of Lemma \ref{shiftNlemma}. Therefore the proofs of Lemmas \ref{shiftDulemma} and \ref{shiftNlemma} are identical, just by replacing $D_h$ by $N_h\star$. Hence we only give the proof of Lemma \ref{shiftDulemma}, 
Calculate
\[
\int_{h_0}^1 \overline K_h(z)\,\frac{dh}{h}\leq\int_{h_0}^1 \frac{C\,h^{\nu-d}}{(h+|z|)^\nu}\,\frac{dh}{h} \leq \frac{C}{(|z|+h_0)^d}. 
\]
Note also for future use that the same calculation provides
\begin{equation}
\int_{h_0}^1 \overline K_h(z)\,\frac{dh}{h}\geq  \frac{1}{C\,(|z|+h_0)^d}.\label{averageKh} 
\end{equation}
Therefore, using spherical coordinates
\[\begin{split}
&\int_{h_0}^{1} \int_{\Pi^d} \overline K_h(z)\,\|D_{|z|}\,u(.)-D_{|z|}\,u(.+z)\|_{L^p}\,dz\,dh\\
&\qquad\leq C\,\int_{S^{d-1}}\int_{h_0}^1 \|D_{r}\,u(.)-D_{r}\,u(.+r\,\omega)\|_{L^p}\,\frac{dr}{r+h_0}\,d\omega.\\
\end{split}\]
Denote
\[\displaystyle 
L_\omega(x)=\frac{\ind_{|x|\leq 1}}{|x|^{d-1}}-\frac{\ind_{|x-\omega|\leq 1}}{|x-\omega|^{d-1}},\quad L_{\omega,r}(x)=r^{-d}\,L_\omega(x/r),
\]
and remark that $L_{\omega}\in W^{s,1}$ with a norm uniform in $\omega$ and with support in the unit ball. Moreover
\[
D_ru(x)-D_ru(x+r\omega)=\int |\nabla u|(x-r\,z)\,L_\omega(z)\,dz=L_{\omega,r}\star |\nabla u|.
\]
We hence apply Lemma \ref{shiftconvolution} since the family $L_{\omega,r}$ satisfies the required  hypothesis
and we get
\[
\int_{h_0}^1 \|L_{\omega,r}\star \nabla\,u\|_{L^p}\,\frac{dr}{r}\leq C\,\|u\|_{B^1_{1,p}},
\]
with a constant $C$ independent of $\omega$ and so
\[
\begin{split}
&\int_{h_0}^{1} \int_{\Pi^d} \overline K_h(z)\,\|D_{|z|}\,u(.)-D_{|z|}\,u(.+z)\|_{L^p}\,dz\,dh\\
&\qquad\leq C\,\int_{S^{d-1}}\int_{h_0}^1 \|L_{\omega,r}\star \nabla u\|_{L^p}\,\frac{dr}{r}\,d\omega\leq C\,\int_{S^{d-1}}\|u\|_{B_{1,p}^1}\,d\omega,
\end{split}
\]
yielding \eqref{optimalbesov}. The bound \eqref{squarefunction} is deduced in the same manner. The proof of the bound \eqref{squarefunctiondisc} follows the same steps; the only difference is that the average over the sphere is replaced by a smoother integration against the weight $1/(1+|w|)^a$.
$\square$

\bigskip

\bigskip
\noindent {\bf Proof of Lemma \ref{shiftconvolution}.}  First remark that $L_r$ is not smooth enough to be used as the basic kernels $\Psi_k$ in the classical Littlewood-Paley decomposition (see section \ref{Besov}) as in particular the Fourier transform of $L_r$ is not necessarily compactly supported.
   We use instead the Littlewood-Paley decomposition of $u$. Denote 
\[
U_k=\Psi_k\star u.
\]
The kernel $L_r$ has $0$ average and so
\[
L_r\star U_k=\int_{\Pi^d} L_r(x-y)\,(U_k(y)-U_k(x))\,dy.
\]
Therefore 
\[\begin{split}
\|L_{r}\star U_k\|_{L^p}&\leq\int_{\Pi^d} L_r(z)\,\|U_k(.)-U_k(.+z)\|_{L^p}\,dz\\
&\leq \int_{\Pi^d} L_r(z)\,|z|^s\,\|U_k\|_{W^{s,p}}\,dz,\\
\end{split}
\]
yielding by the assumption on $L_r$, for $k< |\log_2 r|$ 
\begin{equation}
\|L_{r}\star U_k\|_{L^p}\leq C\,r^{s}\,2^{k\,s}\,\|U_k\|_{L^p},\label{ksmall}
\end{equation} 
by Prop. \ref{propLP}. Note that $C$ only depends on $\int |z|^s\, L_r(z)\,dz$.

\medskip

\noindent We now use similarly that $L_r\in W^{s,1}$ and deduce for $k\geq |\log_2 r|$ by Prop. \ref{propLP},
\begin{equation}
\|L_{r}\star U_k\|_{L^p}\leq \|L_{r}\|_{W^{s,1}}\,\|U_k\|_{W^{-s,p}}\leq C\,r^{-s}\,2^{-k\,s}\,\|U_k\|_{L^p},\label{klarge}
\end{equation} 
 where $C$ only depends on $\sup_r r^s\,\|L_r\|_{W^{s,1}}$.
From the decomposition of $f$
\[\begin{split}
&\int_{h_0}^1 \|L_r\star u\|_{L^p}\,\frac{dr}{r}=\sum_{k=0}^\infty \int_{h_0}^1 \|L_{r}\star U_k\|_{L^p} \frac{dr}{r}\\
&\qquad\leq C\,\sum_{k=0}^\infty \|U_k\|_{L^p} \Bigg(\ind_{k\leq |\log_2 h_0|}\int_{h_0}^{2^{-k}} r^{s}\,2^{k\,s}\frac{dr}{r}
+\int_{\max(h_0,2^{-k})}^1 r^{-s}\,2^{-k\,s}\frac{dr}{r}\Bigg),  
\end{split}
\]
by using \eqref{ksmall} and \eqref{klarge}. This shows that 
\begin{equation}\begin{split}
\int_{h_0}^1 \|L_r\star u\|_{L^p}\,\frac{dr}{r}\leq &C\,\sum_{k\leq |\log_2 h_0|} \|U_k\|_{L^p}+C\,\sum_{k>|\log_2 h_0|} \frac{2^{-k\,s}}{h_0^s}\,\|U_k\|_{L^p}.\label{interBesov}
\end{split}\end{equation}
Now simply bound
\[\begin{split}
\sum_{k\leq |\log_2 h_0|} \|U_k\|_{L^p}
+\sum_{k>|\log_2 h_0|} \frac{2^{-ks}}{h_0^s}\,\|U_k\|_{L^p}&\leq C\,\sum_{k=0}^\infty 
2^k\,\|U_k\|_{L^p}\\
&=C\,\|u\|_{B^0_{p,1}},
\end{split}
\]
which gives \eqref{optimalbesovconvolution}.

\bigskip

\noindent Next remark that
\[\begin{split}
\sum_{k>|\log_2 h_0|} \frac{2^{-ks}}{h_0^s}\,\|U_k\|_{L^p}&
\leq  C \sup_k \|U_k\|_{L^p}
\leq C\,\|u\|_{B^0_{p,\infty}}.
\end{split}\]
Therefore \eqref{interBesov} combined with Lemma \ref{truncatedbesov} yields
\[
\int_{h_0}^1 \|L_r\star u\|_{L^p}\,\frac{dr}{r}\leq C\,\sqrt{|\log_2 h_0|}\,\|u\|_{L^p}+C\,\|u\|_{B^0_{p,\infty}},
\] 
which gives \eqref{squarefunctionconvolution} by Prop. \ref{propLP}.$\square$

\bigskip

\noindent Finally we emphasize that 
\begin{lemma}
The kernel
\[
{\cal K}_{h_0}(z)=\int_{h_0}^1 \overline K_h(z)\,\frac{dh}{h} 
\]
also satisfies $i)$ and $ii)$ of Prop. {\rm \ref{kernelcompactness}}. 
\label{lemmakernelaveraged}
\end{lemma}
\begin{proof}
This is a straightforward consequence of using \eqref{averageKh}.
\end{proof}
\subsection{A brief presentation of renormalized solutions}
Many steps in our proofs manipulate solutions to the transport equation, either under the conservative form
\begin{equation}
\partial_t \rho+\div(\rho\,u)=0,\label{conservative}
\end{equation}
or under the advective form
\begin{equation}
\partial_t w+u\cdot\nabla w=F.\label{advective}
\end{equation}
We will also consider the particular form of \eqref{advective}
\begin{equation}
\partial_t w+u\cdot\nabla w=f\,w,\label{advective2}
\end{equation}
which can directly be obtained from \eqref{advective} by taking $F=f\,w$.

However since $u$ is not Lipschitz, we do not have strong solutions to these equations and one should in principle be careful with using them. Those manipulations can be justified using the theory of renormalized solutions as introduced in \cite{DL}. Instead of having to justify every time, we briefly explain in this subsection how one may proceed. The reader more familiar with the theory of renormalized solutions may safely skip most of the presentation below.

Assume for the purpose of this subsection that $u$ is a given vector field in $L^2_t H^1_x$. The basic idea behind the renormalized solution is the commutator estimate
\begin{lemma}
Assume that $\rho\in L^2_{t,x}$ and $w\in L^2_{t,x}$. Consider any convolution kernel $L\in C^1$, compactly supported in some $B(0,r)$ with $\int_{\Pi^d} L\,dx=1$. Then
\[\begin{split}
&\|\div(L_\eps\star_x(\rho\,u)-u\,L_\eps\star u)\|_{L^1_{t,x}}\longrightarrow 0\quad\mbox{as}\ \eps\rightarrow 0,\\
&\|(L_\eps\star_x(u\cdot\nabla_x w)-u\cdot\nabla_x L_\eps\star w\|_{L^1_{t,x}}\longrightarrow 0\quad\mbox{as}\ \eps\rightarrow 0.
\end{split}\] \label{lemmacommutator} 
\end{lemma}
The proof of Lemma \ref{lemmacommutator} is straightforward and can be found in \cite{DL}. Note however that the techniques we introduce here could also be used, a variant of Prop. \ref{proplineartransport}, to make the estimates even more explicit.
   From Lemma \ref{lemmacommutator}, one may simply prove
\begin{lemma}
Assume that $\rho\in L^2_{t,x}$ is a solution in the sense of distribution to \eqref{conservative}. Assume $w\in L^2_{t,x}$, with $F\in L^1$, a solution in the sense of distribution to \eqref{advective}. Then for any $\chi\in W^{1,\infty}(\R)$, one has in the sense of distribution that
\[
\begin{split}
&\partial_t \chi(\rho)+\div(\chi(\rho)\,u)=(\chi(\rho)-\rho \chi'(\rho))\,\div u,\\
&\partial_t \chi(w)+u\cdot\nabla \chi(w)=F\,\chi'(w).
\end{split}
\]
Finally if in addition $\rho\in L^{p_1},\; w\in L^{p_2},\;u\in L^{p_3}$ with $1/p_1+1/p_2+1/p_3\leq 1$ and $F\in L^q_{t,x}$ with $1/p_1+1/q\leq 1$, then in the sense of distribution for any $\chi\in W^{1,\infty}(\R)$
\[
\partial_t (\rho\,\chi(w))+\div(\rho\,\chi(w)\,u)=F\,\chi'(w)\,\rho.
\]\label{lemmarenormalized}
\end{lemma}
Of course Lemma \ref{lemmarenormalized} applies to \eqref{advective2} in the exact same manner just replacing $F$ by $f\,w$, provided that $f\in L^p$ and $w\in L^{p^*}$ with $1/p^*+1/p=1$ (so that $F\in L^1$) and $f\,w\in L^q_{t,x}$.

\bigskip

Lemma \ref{lemmarenormalized} can be used to justify most of our manipulations later on. Remark that all terms in the equation make sense in ${\cal D}'$: For instance $u\cdot\nabla w=\div(u\,w)-w\,\div u$ which is well defined since $u$, $\div u$ and $w$ belong to $L^2$.
  The proof of Lemma \ref{lemmarenormalized} is essentially found in \cite{DL} and consists simply in writing approximate equations on $L_\eps\star \rho$, $L_\eps\star w$, perform the required manipulation on those quantities and then simply pass to the limit in $\eps$.

As a straightforward consequence, we can easily obtain uniqueness for \eqref{conservative}. Consider two solutions $\rho_1,\;\rho_2\in L^2_{t,x}$ to \eqref{conservative} with same initial data. Apply the previous lemma to $\rho=\rho_1-\rho_2$ and $\chi(\rho)=|\rho|$ and simply integrate the equation over $\Pi^d$ to find
\[
\frac{d}{dt}\int_{\Pi^d} |\rho_1(t,x)-\rho_2(t,x)|\,dx=0.
\]
Thus
\begin{lemma}
For a given $\rho^0\in L^2_x$, there exists at most one solution $\rho\in L^2_{t,x}$ to \eqref{conservative}.\label{uniqueconservative}
\end{lemma}
The uniqueness for the dual problem \eqref{advective} or \eqref{advective2} is however more delicate and in particular the previous strategy cannot work unless $\div u\in L^\infty_x$. The estimates are now slightly different from \eqref{advective} or \eqref{advective2} and we present them for \eqref{advective2} as we use more this form later on.

\medskip

If one considers two solutions $w_1$ and $w_2$ to \eqref{advective2} and a solution $\rho$ to \eqref{conservative}, one has
\[
\frac{d}{dt}\int_{\Pi^d} \rho(t,x)\,|w_1(t,x)-w_2(t,x)|\,dx=\int_{\Pi^d} f\,\rho(t,x)\,|w_1(t,x)-w_2(t,x)|\,dx,
\]
leading to
\begin{lemma}
Assume that
\begin{itemize}
\item $\rho\in L^2_{t,x}$ solves \eqref{conservative}.
\item $w_1$ and $w_2$ are two solutions in $L^2_{t,x}$ to \eqref{advective2} with $w_1(t=0)=w_2(t=0)$ for a given $f\in L^p_{t,x}$ and $w_i\in L^{p^*}$, $i=1,\;2$, with $1/p^*+1/p=1$.
\item $\rho\in L^{p_1},\; w\in L^{p_2},\;u\in L^{p_3}$ with $1/p_1+1/p_2+1/p_3\leq 1$ and $f\,w\in L^q_{t,x}$ with $1/p_1+1/q\leq 1$.
\item Finally either $f\in L^\infty$ or $f\leq 0$ and $\rho\geq 0$.
\end{itemize}
Then $w_1=w_2$ $\rho$ a.e. 
\label{uniqueadvective}
\end{lemma}
Of course if $\rho>0$ everywhere then Lemma \ref{uniqueadvective} provides the uniqueness of the solution to \eqref{advective}. But in general $\rho$ could vanish on a set of non zero measure (this is the difficult vacuum problem for compressible Navier-Stokes). In that case one cannot expect in general uniqueness for \eqref{advective}.

We will use the same strategy of integrating against a solution $\rho$ to the conservative equation \eqref{conservative} to obtain some bounds on $\log w$ for $w$ a solution to \eqref{advective}. 
\begin{lemma}
Assume that
\begin{itemize}
\item $\rho\geq 0$ in $L^2_{t,x}$ solves \eqref{conservative}.
\item $w_1$ is a solution in $L^2_{t,x}$ to \eqref{advective2} with $0\leq w_1\leq 1$, $w_1(t=0)=w^0$ for a given $f\in L^p_{t,x}$ and $w\in L^{p^*}$, $i=1,\;2$, with $1/p^*+1/p=1$.
\item $\rho\in L^{p_1},\; w\in L^{p_2},\;u\in L^{p_3}$ with $1/p_1+1/p_2+1/p_3\leq 1$ and $f\,w\in L^q_{t,x}$ with $1/p_1+1/q\leq 1$.
\end{itemize}
Then one has for any $0\leq \theta\leq 1$
\[\begin{split}
\int_{\Pi^d} (1+|\log w(t,x)|)^\theta\,\rho(t,x)\,dx&\leq \int_{\Pi^d} |\log w^0|^\theta\,\rho(t,x)\,dx\\+&\theta\,\int_0^t\int_{\Pi^d} |f(s,x)|\,(1+|\log w(s,x)|)^{\theta-1}\,\rho(s,x)\,dx\,ds.
\end{split}
\]\label{lemmalog}
\end{lemma}
The lemma is proved simply by applying Lemma \ref{lemmarenormalized} (the last point) to a sequence $\chi_\eps(w)=(1+|\log (\eps+w)|)^\theta$, as for a fixed $\eps>0$, $\chi_\eps$ is Lipschitz. One then integrates in $t$ and $x$ and finally passes to the limit $\eps\rightarrow 0$ by the monotone convergence theorem.

Note that the $log$ transform allows to derive \eqref{advective} from \eqref{advective2} but requires in addition $\log w\in L^2$ while Lemma \ref{lemmalog} does not require any a priori estimates on $\log w$.

\bigskip

Let us finish this subsection by briefly mentioning the existence question. This does not use renormalized per se, although as we saw using the solutions once they are obtained requires the theory.

For uniqueness, the conservative form was well behaved and the advective form delicate. Hence for existence, things are reversed. Unless $\div u\in L^\infty$ it is not possible to have a general existence result for \eqref{conservative}. In general a solution to \eqref{conservative} with only $\div u\in L^2$ may concentrate, forming Dirac masses for instance.

But it is quite simple to obtain a general existence result for \eqref{advective}
\begin{lemma}
Assume that $w^0\in L^\infty(\Pi^d)$ and either that $f\in L^\infty(\Pi^d)$ or that $f\leq 0$, $f\in L^1_{t,x}$ and $w^0\geq 0$. Then there exists $w\in L^\infty([0,\ T]\times \Pi^d)$ for any $T>0$ solution to \eqref{advective} in the sense of distributions. \label{lemmaexistence}
\end{lemma} 
\begin{proof}
Consider a sequence $u_n\in C^\infty$ s.t. $u_n$ converges to $u$ in $L^2_t H^1_x$. Define the solution $w_n$ to
\[
\partial_t w_n+u_n\cdot\nabla_x w_n=f\,w_n,\quad w_n(t=0)=w^0.
\]
This solution $w_n$ is easy to construct by using the characteristics flow based on $u_n$. Now if $f\in L^\infty$ then
\[
\|w_n(t,.)\|_{L^\infty_x}\leq \|w^0\|_{L^\infty_x}\,e^{t\,\|f\|_{L^\infty_{t,x}}}.
\]
In the other case if $w^0\geq 0$ then $w_n\geq 0$. Furthermore if $f\leq 0$ then 
\[
\|w_n(t,.)\|_{L^\infty_x}\leq \|w^0\|_{L^\infty_x}.
\]
So in both cases $w_n$ is uniformly bounded in $L^\infty([0,\ T]\times \Pi^d)$ for any $T>0$. Extracting a subsequence, still denoted by $w_n$ for simplicity, $w_n$ converges to $w$ in the weak-* topology of $L^\infty([0,\ T]\times \Pi^d)$.

It only remains to pass to the limit in $u_n\cdot \nabla_x w_n=\div(u_n\,w_n)-w_n\,\div u_n$ which follows from the strong convergence in $L^2$ of $u_n$ and $\div u_n$. Similarly one may pass to the limit in $f\,w_n$.  
\end{proof}
\section{Renormalized equation and weights \label{secrenormalized}}
We explain here the various renormalizations of the transport equation satisfied by $\rho_k $, We then define the
weights we will consider  and give their properties. 
\subsection{Renormalized equation\label{renormequation}}
We explain in this subsection how to obtain the equation satisfied by various quantities that we will need and of the form $Z_k^{x,y} \chi(\rho_k^x-\rho_k^y)$ where $Z_k^{x,y}$ is chosen as $Z_k^{x,y} = K_h(x-y) W_{k,h}^{x,y}$ with $W_{k,h}^{x,y}=W_{k,h}(t,x,y)$ appropriate weights.
\begin{lemma} \label{renor}
Assume $\rho_k$ solves \eqref{continuity} with \eqref{boundrhok} for $p>2$ and that $u_k$ satisfies \eqref{bounduk}.  Denote for convenience 
\[\begin{split}
&\rho_k^x=\rho_k(t,x),\quad \rho_k^y=\rho(t,y),\quad u_k^x=u_k(t,x),\quad u_k^y=u_k(t,y),\\
& \delta\rho_k = \rho_k^x - \rho_k^y,\quad \bar \rho_k = \rho_k^x + \rho_k^y.
\end{split}\]
Then for any $\chi\in W^{2,\infty}$
\[
\begin{split}
& \int_{\Pi^{2d}} \bigl[ K_{h}(x-y) W_{k,h}^{x,y}\,  \chi(\delta\rho_k) \bigr] (t) 
-  \int_{\Pi^{2d}} \bigl[ K_{h}(x-y) W_{k,h}^{x,y}\,  \chi(\delta\rho_k) \bigr](0)  \\
& \qquad    
+\int_0^t \int_{\Pi^{2d}} \bigl(\chi'(\delta\rho_k) \delta\rho_k 
      - \chi(\delta\rho_k)\bigr)\bigl(\div_x u_k^x
      + \div_y u_k^y \bigr) K_{h}(x-y) W_{k,h}^{x,y} \\
&-\int_0^t\int_{\Pi^{2d}} \chi(\delta\rho_k)\bigl[
      u_k^x\cdot\nabla_x K_h(x-y)  + u_k^y\cdot\nabla_y  K_h(x-y)\\
&\qquad\qquad
   \hskip4cm   + \alpha_k (\Delta_x+\Delta_y) K_h(x-y)\bigr] W_{k,h}^{x,y}  \\
&\qquad -  2 \int_0^t\int_{\Pi^{2d}} \alpha_k\,\nabla K_h(x-y)\, \chi(\delta\rho_k)\, \left[\nabla_x W_{k,h}^{x,y} -\nabla_y W_{k,h}^{x,y}\right]   \\
&\qquad - 2\,\alpha_k\,\int_0^t\int_{\Pi^{2d}} K_h(x-y)\,\chi(\delta\rho_k) \left[\Delta_x\,W_{k,h}^{x,y}+\Delta_y\, W_{k,h}^{x,y}\right]
\end{split}
\]
\[
\begin{split}
&\quad = \int_0^t\int_{\Pi^{2d}} \chi(\delta\rho_k)[\partial_t W_{k,h}^{x,y} 
      + u_k^x\cdot\nabla_x W_{k,h}^{x,y} + u_k^y\cdot\nabla_y W_{k,h}^{x,y}\\
&\qquad\qquad      - \alpha_k (\Delta_x+\Delta_y) W_{k,h}^{x,y}]K_h(x-y)  \\
&\quad
    - \frac{1}{2}\int_0^t\int_{\Pi^{2d}} \chi'(\delta\rho_k) K_h(x-y)W_{k,h}^{x,y} \bigl( \div_x u_k^x - \div_y u_k^y) \bar\rho_k  \\
&\quad 
    + \frac{1}{2} \int_0^t\int_{\Pi^{2d}} \chi'(\delta\rho_k) K_h(x-y)W_{k,h}^{x,y}(\div_x u_k^x + \div_y u_k^y) \delta\rho_k \\
& \qquad - \alpha_k \int_0^t\int_{\Pi^{2d}}   \chi''(\rho_k) 
   K_h(x-y)W_{k,h}^{x,y}   \bigl(|\nabla_x\delta\rho_k|^2 + |\nabla_y\delta\rho_k|^2\bigr).
     \\
\end{split}
\]
\end{lemma}
\begin{proof}
The result essentially relies on a doubling of variable argument and straight forward algebraic calculations. 

Since $\rho_k$ solves \eqref{continuity}, one has that  
\[
\partial_t \rho_k^x + {\rm div}_x(\rho_k^x u_k^x) = \alpha_k \Delta_x \rho_k^x,
\]
\[
\partial_t \rho_k^y + {\rm div}_y(\rho_k^y u_k^y) = \alpha_k \Delta_y \rho_k^y.
\]
Recalling $\delta\rho_k=\rho_k^x-\rho_k^y$, and using that $\rho_k\in L^p_{t,x}$ with $p>2$ and hence $\rho_k \div\,u_k$ is well defined, one  can check that
\[
\partial_t \delta\rho_k+\div_x(u_k^x \,\delta\rho_k)+\div_y (u_k^y\,\delta\rho_k)=\alpha_k(\Delta_x+\Delta_y)\,\delta\rho_k -\rho_k^y\,\div_x u_k^x+\rho_k^x\,\div_y u_k^y.
\]
Then, recalling the notation $\bar \rho_k = \rho_k^x+\rho_k^y$, we observe that
\[
\begin{split}
&-\rho_k^y\,\div_x u_k^x+\rho_k^x \,\div_y u_k^y
=\frac{1}{2}\Big(\div_x u_k^x\,\rho_k^x-\div_x u_k^x\,\rho_k^y\div_y u_k^y\,\rho_k^x-\div_y u_k^y\rho_k^y\\
&\qquad-\div_x u_k^x\,\rho_k^x
-\div_x u_k^x\,\rho_k^y
+\div_y u_k^y\,\rho_k^x+\div_y u_k^y\,\rho_k^y\Big)\\
&\quad =\frac{1}{2}(\div_x u_k^x+\div_y u_k^y)\,\delta\rho_k-\frac{1}{2}(\div u_k^x-\div u_k^y)\,\bar\rho_k.\\
\end{split}\]
Consequently, we can write
\begin{equation}
\begin{split}
&\partial_t \delta\rho_k+\div_x(u_k^x\,\delta\rho_k)+\div_y (u_k^y\,\delta\rho_k)=\alpha_k\,(\Delta_x+\Delta_y)\,\delta\rho_k\\
&\qquad+\frac{1}{2}(\div_x u_k^x+\div_y u_k^y)\,\delta\rho_k-\frac{1}{2}(\div u_k^x-\div u_k^y)\,\bar\rho_k,
\end{split}\label{deltarho}
\end{equation}
We now turn to the renormalized equation that means the equation satisfied by $\chi(\delta_k)$
for a nonlinear function $s\mapsto \chi(s)$. Formally the equation can be obtained by multiplying \eqref{deltarho} by $\chi'(\delta\rho_k)$. If $\alpha_k=0$ and $\rho_k$ is not smooth then the formal calculation can be justified following {\sc Di~Perna--Lions} techniques using regularizing by convolution and the estimate \eqref{bounduk}, {\em i.e.} $u_k\in L^2_t H^1_x$. Then 
\[
\begin{split}
& \partial_t \chi(\delta\rho_k) + \div_x(u_k^x \chi(\delta\rho_k))
   + \div_y(u_k^y \chi(\delta\rho_k))  \\
&  \qquad\qquad  + \bigl(\chi'(\delta\rho_k) \delta\rho_k - \chi(\delta\rho_k)\bigr)\bigl(\div_x u_k^x
      + \div_y u_k^y \bigr)\\
&\quad =   
\alpha_k (\Delta_x+\Delta_y) \chi(\delta\rho_k) 
    - \frac{1}{2} \chi'(\delta\rho_k) \bigl( \div_x u_k^x - \div_y u_k^y) \bar\rho_k  \\
&   \qquad  + \frac{1}{2}  \chi'(\delta\rho_k)(\div_x u_k^x + \div_y u_k^y) \delta\rho_k 
 - \alpha_k  \chi''(\rho_k) \bigl(|\nabla_x\delta\rho_k|^2 + |\nabla_y\delta\rho_k|^2\bigr).
\end{split}\]
For any $V_k^x$, $V_k^y$ and smooth enough $Z_{k,h}^{x,y}$, one has in the sense of distributions
\[
Z_{k,h}^{x,y}\Delta_x V_k^x = \div_x(Z_{k,h}^{x,y}\nabla_x V_k^x) - \div_x (V_k^x\nabla_x Z_k^{x,y}) 
+ V_k^x\Delta_x Z_{k,h}^{x,y}.$$
$$  Z_{k,h}^{x,y}\Delta_y V_k^y = \div_y(Z_{k,h}^{x,y}\nabla_y V_k^y) - \div_y (V_k^y\nabla_y Z_{k,h}^{x,y}) 
+ V_k^y\Delta_y Z_{k,h}^{x,y}.
\]
Consequently, we get the following equation for $Z_{k,h}^{x,y} \chi(\delta\rho_k)$: 
\[
\begin{split}
& \partial_t [Z_{k,h}^{x,y} \chi(\delta\rho_k)] + \div_x(u_k^x \chi(\delta\rho_k)Z_{k,h}^{x,y})
   + \div_y(u_k^y \chi(\delta\rho_k)Z_{k,h}^{x,y})  \\
& \qquad    + \bigl(\chi'(\delta\rho_k) \delta\rho_k - \chi(\delta\rho_k)\bigr)\bigl(\div_x u_k^x
      + \div_y u_k^y \bigr)Z_{k,h}^{x,y} \\
&\qquad -\chi(\delta\rho_k)[\partial_t Z_{k,h}^{x,y} 
      + u_k^x\cdot\nabla_x Z_{k,h}^{x,y} + u_k^y\cdot\nabla_y Z_{k,h}^{x,y}
      - \alpha_k (\Delta_x+\Delta_y) Z_{k,h}^{x,y}]=r.h.s,  \\
\end{split}
\]
with
\[
\begin{split}
&r.h.s. = 
    - \frac{1}{2} \chi'(\delta\rho_k) Z_{k,h}^{x,y} \bigl( \div_x u_k^x - \div_y u_k^y) \bar\rho_k  
+ \frac{1}{2}  \chi'(\delta\rho_k)Z_{k,h}^{x,y}(\div_x u_k^x + \div_y u_k^y) \delta\rho_k \\
& \qquad\quad   - \alpha_k \chi''(\rho_k) 
   Z_{k,h}^{x,y}  \bigl(|\nabla_x\delta\rho_k|^2 + |\nabla_y\delta\rho_k|^2\bigr)
     \\
&\  +2\,\alpha_k\, Z^{x,y}_{k,h}\,\left(\Delta_x+\Delta_y\right)\chi(\delta\rho_k)+ \alpha_k \bigl[ \div_x ( Z_{k,h}^{x,y}\nabla_x \chi(\delta\rho_k)) - \div_x(\chi(\delta\rho_k)\nabla_x  Z_{k,h}^{x,y}) \\
&\qquad\qquad +  \div_y (Z_{k,h}^{x,y}\nabla_y \chi(\delta\rho_k)) 
     - \div_y(\chi(\delta\rho_k)\nabla_y  Z_{k,h}^{x,y}) \bigr].\\
\end{split}\]
Integrating in time and space and performing the required integration by parts, we get the desired equality writing $Z_{k,h}^{x,y} = K_h(x-y) W_{k,h}^{x,y}$.
\end{proof}

%
%
\subsection{The weights: Choice and properties\label{weightsec}}
 In this subsection, we choose the PDEs satisfied by the weights, we state and then prove some of their properties.
\subsubsection{Basic considerations}
We define weights $w(t,x)$ which satisfy
\begin{equation}
\partial_t w+u_k(t,x)\cdot \nabla_x w=-D\,w+\alpha_k\Delta_{x} w,\qquad w(t=0,x)=w^0(x)\label{eqw}
\end{equation}
for some appropriate penalization $D$ depending on the case under consideration.
   Note that $w$ depends on $k$, but we do not precise the index to keep notations simple. 
The choice of $D$ will be based on the need to control ``bad'' terms when looking at the propagation
of the weighted quantity.
The choice will  also have to ensure that the weights are not too small, too often.

\subsubsection{Isotropic viscosity, general pressure laws.}

\noindent {\it The case with $\alpha_k>0$ and monotone pressure.}
The simplest choice for the penalization $D$ to define $w_0$  is
\begin{equation}
D_0=\lambda\,M\,|\nabla u_k|,\label{dk0}
\end{equation}
with  $\lambda$ a fixed constant (chosen later on) and $M$ the localized maximal operator as defined by \eqref{defmaximal}.
   We choose accordingly in that case
\[
w_0|_{t=0}=w_0^0 \equiv1.
\]

\bigskip
\noindent {\it The case $\alpha_k=0$ and non-monotone pressure.}
In the absence of diffusion in \eqref{continuity} ($\alpha_k=0$) and when the pressure term $P_k$ is non-monotone for instance, one needs to add a term $\rho_k^{\tilde\gamma}$ in the penalization. This would lead to very strong assumptions, in particular on the exponent $p$ in \eqref{boundrhok} (and hence $\gamma$) as explained after Prop. \ref{proproughsketch}. It is possible to obtain better results using that $\rho \in L^{p}$ for some $p>2$,  by taking the more refined
\begin{equation}\begin{split}
\frac{D_1}{\lambda}=&  \rho_k \, |\div u_k|
+|\div u_k|+M\,|\nabla u_k|+\rho_k^{\tilde\gamma}+\tilde P_k\,\rho_k+R_k,
\end{split}\label{dk1}
\end{equation}
for the general compactness result.
  We take for simplicity
\[
w_1|_{t=0}=w_1^0 \equiv \exp\left(-\lambda \sup \rho_k^0\right).
\] 
   The reason for the first term in $D_1$ compared to $D_0$ is to ensure that $w_1\leq e^{-\lambda \rho_k}$ which helps compensates the penalization in $\rho_k^{\tilde\gamma}$ to get the
 property on $\rho_k |\log w_1|^\theta$ for some $\theta>0$.  
  The three last terms are needed to respectively counterbalance: Additional divergence terms in the propagation quantity compared to $w_0$, the same $M\,|\nabla u_k|$ as for $w_0$ and the $\rho_k^{\tilde\gamma}$ for terms coming from the pressure.

\subsubsection{Anisotropic stress tensor.} 

The choice for the penalization, denoted $D_a$ in this case and leading to
the weight $w_a$, is now 
\begin{equation}
 \frac{D_a}{\lambda}  =  M |\nabla u_k| 
    +      \overline {K_h}\star (|{\rm div} u_k| +  \bigl|A_\mu \rho_k^\gamma\bigr|).
\label{dk00} 
\end{equation}
Note that the second term in the penalization is used to control the non local part of the pressure terms. As initial condition, we choose accordingly 
$$w_a\vert_{t=0} = w_a^0 \equiv 1.$$

\subsubsection{The forms of  the weights.}  

Two types of weights $W$ are used
\[
 W(t,x,y)=w(t,x)+w(t,y),\quad\mbox{or}\quad
 W(t,x,y)=w(t,x)\,w(t,y).
\]
The first one will provide compactness and will be used with \eqref{dk0} or \eqref{dk1}. The second, used with \eqref{dk1} gives better explicit regularity estimates but far from vacuum and is considered
for the sake of completeness.
   Therefore one defines
\begin{equation}\begin{split}
&W_0(t,x,y)=w_0(t,x)+w_0(t,y),\quad W_1(t,x,y)=w_1(t,x)+w_1(t,y),\\ &W_2(t,x,y)=w_1(t,x) w_1(t,y),\quad W_a(t,x,y)=w_a(t,x)+w_a(t,y).
\end{split}\end{equation}
As for the penalization, we use the notation $W$ when the particular choice is not relevant and $W_i$, $i=0,\,1,\,2$ or $a$ otherwise. For all choices, one has
\begin{equation}
\partial_t W+u_k(x)\cdot \nabla_x W+u_k(y)\cdot \nabla_y W=-Q+\alpha_k\Delta_{x,y} W.\label{W}
\end{equation}
The term $Q$ depends on the choices of penalizations and weights with the four possibilities
\[\begin{split}
&Q_0=D_0\,w_0(t,x)+D_0\,w_0(t,y),\qquad Q_1= D_1\,w_1(t,x)+D_1\,w_1(t,y),\\
&Q_2=(D_1(t,x)+D_1(t,y))\,w_1(t,x)\,w_1(t,y),  \qquad Q_a=D_a\,w_a(t,x)+D_a\,w_a(t,y).
\end{split}\]

\bigskip


\subsubsection{The weight properties}
We summarize the main estimates on the weights previously defined
\begin{prop} Assume that $\rho_k$ solves \eqref{continuity} with the bounds 
\eqref{bounduk} on $u_k$and \eqref{boundrhok} with $p>\max(2,\;\tilde\gamma)$. Assume that $\tilde P_k,\;R_k$ are given by \eqref{nonmonotone} and in particular $\tilde P_k$ is uniformly bounded in $L^2_{t,x}$ and $R_k$ in $L^1_{t,x}$.
Then there exists weights $w_0$, $w_1$, $w_a$ which satisfy Equation \eqref{eqw} with 
initial data respectively 
$$w^0\vert_{t=0} = 1, 
 \quad w_1\vert_{t=0} = \exp(-\lambda\, \sup \rho_k^0),
 \quad w_a\vert_{t=0} = 1$$
 and $D_0, D_1, D_a$ respectively  given by \eqref{dk0}, \eqref{dk1} and \eqref{dk00}
 such that
\begin{itemize}
\item[i)] For any $t$,$x$:
\begin{equation} \label{propoweights}
0\leq w_0(t,x)\leq 1, \quad 0\leq w_a(t,x)\leq 1,\quad 0\leq w_1(t,x)\leq e^{\displaystyle -\lambda\,\rho_k(t,x)}.
\end{equation}
\item[ii)] One has
\[
\int \rho_k(t,x)\,|\log w_0(t,x)|\,dx\leq C\,(1+\lambda),
\]
if $\alpha_k=0$ and $p> \max(2,\,\tilde\gamma)$ then similarly there exists $\theta>0$ s.t.
\[
\int  \rho_k(t,x)\,|\log w_1(t,x)|^\theta\,dx\leq C_\lambda.
\]
while finally if $p\ge \gamma+1$
\begin{equation}
\label{estimwa}
\int  \rho_k(t,x)\,|\log w_a(t,x)|\,dx\leq C(1+\lambda).
\end{equation}

\item[iii)] 
 For any $\eta$, we  have the two estimates
\[
\int_{\Pi^d} \rho_k\,\ind_{\displaystyle \overline K_h\star w_0\leq \eta}\,dx\leq C\,\frac{1+\lambda}{|\log \eta|},
\]
and if $p\ge \gamma+1$
\[
\int_{\Pi^d} \rho_k\,\ind_{\displaystyle \overline K_h\star w_a\leq \eta}\,dx\leq C\,\frac{1+\lambda}{|\log \eta|}.
\]
\item[iv)] Denoting  $w_{a,h} = \overline K_h \star w_a$, if $p>\gamma$, we have for some $0<\theta<1$
\[\begin{split}
\int_{h_0}^1\int_0^t&\Big\| \overline K_h \star \bigl(\overline K_h \star (|\div u_k|+ |A_\mu \rho_k^\gamma|) w_a\bigr) -\\
&\qquad
     \bigl(\overline K_h \star (|\div u_k|+ |A_\mu \rho_k^\gamma|)\bigr)\,\omega_{a,h}\Big\|_{L^q}\,dt\,\frac{dh}{h}\leq C\,|\log h_0|^\theta,    
\end{split} 
\]
with $q=\min(2,p/\gamma).$ 

\end{itemize}
\label{weightprop}
\end{prop}

\begin{remark} Part i) tells us that $w_i$ is small at the right points (in particular when $\rho_k$ is large). On the other hand we want $w_i$ to be small only on a set of small mass otherwise one obviously does not control much. This is the role of part ii). We use part iv) to regularize weights
in the anisotropic case. Part iii) is also used to get control under the form given in
\ref{propcomp} from the estimates with weights.
\end{remark}

\begin{remark} Even when $\alpha_k>0$, it would be possible to define $D_1$  in order to have a bound like $w_1\leq e^{-\lambda\rho_k(t,x)^{q-1}}$.
For instance take 
\[\begin{split}
\frac{D_1}{\lambda}=&\rho_k^{q-1}\,(q-1)\,|\div u_k|
+\frac{\alpha_{k,\lambda}}{\lambda} \,|\nabla\log w_1|^2\\
&+|\div u_k|+M\,|\nabla u_k|+L\,\rho_k^\gamma+\tilde P_k\,\rho_k+R_k,
\end{split}
\]
with
\[
\alpha_{k,\lambda}=\alpha_k\left(1-\frac{\bar\alpha_\lambda}{1+\lambda\,\rho_k^{q-1}}\right) , \qquad \bar \alpha_\lambda=\frac{q-2}{(q-1)}.
\]
However one needs $q\leq p/2$ and $q>2$ which already forces $p>4$. Moreover the main difficulty when $\alpha_k>0$ comes from the proof of Lemma \ref{boundQh} which forces us to work with $\overline K_h\star w_0$ and not $w_0$. Because of that any pointwise inequality between $w(x)$ and $\rho_k(x)$ is mostly useless. 
\end{remark}

\begin{proof}

\noindent {\it Point {\rm i)}.}  This point focuses on the construction of the weights satisfying the  
bounds~\eqref{propoweights}.

\smallskip

\noindent {\it Construction of $w_0$.} The construction of the weights $w_0$ is classical  since it satisfies a parabolic equation; moreover since $D_0$ is positive then one has $0\leq w_0\leq 1$.

 \smallskip
   
\noindent Ê{\it Construction of $w_a$:}    Choosing $w_a\vert_{t=0} =1$ and $f=-D_a \le 0$ and noticing that $D_a\in L^2$ and we hence easily construct (see Lemmas  \ref{lemmaexistence},
\ref{lemmarenormalized} and \ref{uniqueconservative}) $w_a$ such that $0\leq w_a\leq 1$ solution of
$$\partial_t w_a + u_k\cdot\nabla w_a + D_a\, w_a = 0,
    \qquad w_a\vert_{t=0} = 1.$$
Note that $D_a\ge |{\rm div} u_k|$ because $M\,|\nabla u_k|\geq |\nabla u_k|\geq |{\rm div} u_k|$ 
and $w_a$ also solves
\[
\partial_t w_a + {\rm div} (u_k w_a) + (D_a - {\rm div} u_k)\,  w_a = 0,
\]
so that by the maximum principle, we also have $w_a\leq \rho_k$ where $\rho_k>0$. This means that we can actually uniquely define $w_a$ by imposing $w_a=0$ if $\rho_k=0$.
 
\medskip

\noindent Ê{\it Construction of $w_1$:}   Choosing $w_1\vert_{t=0} = \exp(- \lambda \sup \rho_k^0)$ and $f=-D_1 \le 0$ and noticing that $D_1\in L^1$ 
and $D_1 \ge |{\rm div} u_k|$ we again construct $w_1$ just as $w_a$ such that 
$0\leq w_1\leq 1$ with $w_1(t,x)=0$ where $\rho_k(t,x)=0$ and solution to
$$\partial_t w_1 + {\rm div} (u_k w_1) + (D_1 - {\rm div} u_k)\,  w_1 = 0,
    \qquad w_1\vert_{t=0} = \exp(- \lambda \sup \rho_k^0).$$
Note that using renormalization technique on the mass equation
$$ \partial_t [\exp (-\lambda \rho_k]
    + {\rm div} \bigl( u_k  [ \exp (-\lambda \rho_k)]\bigr)\\
  + [ - \lambda   \rho_k {\rm div} u_k   
      - {\rm div} u_k  ]  \exp (-\lambda \rho_k) 
    = 0.$$
Subtracting the two equations we get the following equation on 
$g=w_1 - \exp(-\lambda \rho_k)$:
$$ \partial_t g + {\rm div} (u_k g)
   + (D_1 - {\rm div} u_k) g  = - (D_1 + \lambda  \rho_k{\rm div} u_k)
     \exp (-\lambda \rho_k) 
$$
Recall now that $D_1 \ge | {\rm div} u_k|$ and 
$D_1 \ge - \lambda  \rho_k {\rm div} u_k$, thus  using the maximum principle, we get 
$$w_1 \le e^{-\lambda \rho_k}$$
recalling that we have $w_1=0$ where $\rho_k=0$.

\bigskip

\noindent {\it Point {\rm ii)}.}  By point i), $w\leq 1$,  hence $|\log w|=-\log w$ and from \eqref{eqw}, denoting $|\log w_i|=A_i$,
\begin{equation}\begin{split}
\partial_t (A_i)+u_k\cdot \nabla_x(A_i) -D_i &=-\frac{\alpha_k}{w_i}\Delta\,w_i\\
&=\alpha_k \Delta A_i -\alpha_k\,|\nabla A_i|^2.
\end{split}\label{logw}\end{equation}
For $A_0$, we directly apply
\[
\partial_t (\rho_k\,A_0)+\nabla_x\cdot (\rho_k\,
A_0)=D_0\,\rho_k+\alpha_k\Delta(\rho_k\,A_0)-2\alpha_k\,\nabla \rho_k\cdot \nabla A_0-\alpha_k \rho_k |\nabla A_0|^2,
\]
and integrate to find
\[\label{A0} \begin{split}
\int \rho_k\,A_0\,dx=&C+\int_0^t\int D_0(t,x)\,\rho_k\,dx\,ds-\alpha_k\int_0^t\int \rho_k |\nabla A_0|^2\,dx\,ds\\
&-2\alpha_k\,\int\nabla\rho_k\cdot\nabla A_0\,dx\,ds.
\end{split}\]
Simply bound
\[
\int \nabla \rho_k\,\cdot\nabla A_0\leq \frac{1}{4}\int \rho_k |\nabla A_0|^2+\int \frac{|\nabla\rho_k|^2}{\rho_k}.
\]
On the other hand, using renormalization techniques, 
\[
\frac{d}{dt}\int \rho_k\,\log \rho_k\,dx=-\int \rho_k\,\div u_k\,dx-\alpha_k\int \frac{|\nabla\rho_k|^2}{\rho_k}\,dx. 
\]
However as $p\geq 2$ then $\rho_k\,\div u_k$ is bounded uniformly in $L^1_{t,x}$ and $\rho_k\,\log \rho_k$ in $L^\infty_t(L^1_x)$. This implies, from the previous equality, that
\[
\alpha_k\int_0^t\int \frac{|\nabla\rho_k|^2}{\rho_k}\,dx\,ds\leq C,
\]
and consequently
\[
-2\alpha_k\int_0^t\int \nabla \rho_k\,\cdot\nabla A_i\,dx\,ds\leq \frac{\alpha_k}{2}\int_0^t\int \rho_k |\nabla A_i|^2\,dx\,ds+C.
\]
Using this in the equality on $\displaystyle \int \rho_k\,A_0$ given previously, we get
\[
\int \rho_k\,A_0\,dx=C+\int_0^t\int D_0(t,x)\,\rho_k\,dx\,ds
     -\frac{\alpha_k}{2}\int_0^t\int \rho_k |\nabla A_0|^2\,dx\,ds.
\]
In that case, we know that $\|D_0\|_{L^2}\leq C\,\lambda$ and since $p\geq 2$ we get 
\begin{equation}
\label{rhoA} \qquad 
\int \rho_k\,|\log w_0|\,dx+\alpha_k\int_0^t\int\rho_k\,|\nabla A_0|^2\,dx\,ds \leq C(1+\lambda).
\end{equation}

\bigskip

Concerning $w_a$, the estimate is similar, even simpler as $\alpha_k=0$, to get
$$\int \rho_k\,|\log w_a|\,dx  \le C(1+\lambda)$$
assuming $p\ge \gamma + 1$. 
Indeed $M |\nabla u_k|$ and $\overline K_h\star |\div u_k|$ are bounded in $L^2$ by \eqref{bounduk}. Finally
\[\begin{split}
\int_0^T\!\! \int_{\Pi^d}\rho_k\,\overline K_h\star(|A_\mu\,\rho_k^\gamma|)&\leq 
\left(\int_0^T\!\! \int_{\Pi^d}\rho_k^{\gamma+1}\right)^{{1}/{(\gamma+1)}}  \left(\int_0^T\!\! \int_{\Pi^d}|A_\mu \,\rho_k^\gamma|^{(\gamma+1)/\gamma}\right)^{{\gamma}/{(\gamma+1)}}\\
& \leq C\,\int_0^T \int_{\Pi^d}\rho_k^{\gamma+1},
\end{split}
\] 
since $A_\mu$ is continuous on any $L^q$ space for $1<q<+\infty$. 
The right--hand side is bounded assuming $p>\gamma +1$.

\bigskip

For $\omega_1$, the estimate is as before a bit different. 
We now assume that $\alpha_k=0$, define $\tilde A_1=(1+A_1)^\theta$ and obtain from \eqref{logw}
\[
\partial_t \tilde A_1+u_k\cdot\nabla \tilde A_1=\theta\,\frac{D_1}{(1+A_1)^{1-\theta}}.
\]
Integrating and recalling $A_1\geq \lambda\,\rho_k$, $M\,|\nabla u_k|$ and $\tilde P_k$ by \eqref{nonmonotone} are uniformly bounded in $L^2$ and $R_k$ in $L^1$ 
\[\begin{split}
&\int \rho_k\,\tilde A_1\,dx\leq C+C\,\int_0^t\int \frac{1+\rho_k^{2}}{1+\rho_k^{(1-\theta)}}\,|\div u_k|\,dx\,ds\\
&\qquad+C\,\int_0^t\int \frac{\rho_k}{1+\rho_k^{(1-\theta)}}\,(M\,|\nabla u_k|+\tilde P_k\,\rho_k+R_k)\,dx\,ds\\
&\qquad+C\,\int_0^t\int \frac{\rho_k^{\tilde\gamma+1}}{1+\rho_k^{(1-\theta)}}\,dx\,ds\leq C,
\end{split}\]
for some $\theta>0$ depending on $p-\max(2,\;\tilde\gamma)$.
This gives the desired control regarding $\rho |\log w_1|^\theta$ for an exponent $\theta$ small enough.

\bigskip

\noindent {Point \rm iii)}
 Estimate \eqref{rhoA} will not be enough in the proof and we will need to control the mass of $\rho_k$ where $\overline K_h\star w_0$ is small. Denote 
\[
\Omega_{h,\eta}=\{x\in \Pi^d,\ \overline K_h\star w_0(x)\leq \eta\},\qquad \tilde \Omega_{h,\eta}=\{x\in\Omega_{h,\eta},\ w_0(x)\geq \sqrt{\eta}\}.
\]
The time $t$ is fixed during this argument and for simplicity we omit it.

   One cannot easily estimate $|\Omega_{h,\eta}|$ directly but it is straightforward to bound $|\tilde\Omega_{h,\eta}|$. 
Assume $x\in \Omega_{h,\eta}$ {\em i.e.} $\overline K_h\star w_0(x)\leq \eta$. From the expression of $K_h$, if $|\delta|\leq h$
\[
\overline K_h(z+\delta)\leq \frac{\|K_h\|_{L^1}^{-1}}{(h+|z+\delta|)^a} \leq C\,\overline K_h(z), 
\]
we deduce that for any $y\in B(x,h)$
\[
\overline K_h\star w_0 (y)\leq C\,\eta.
\]
Now cover $\tilde \Omega_{\eta,h}$ by $\bigcup_i C_i^h$ with  $C_i^h$ disjoint hyper-cubes of diameter $h/C$. For any $i$, denote $\tilde\Omega_{\eta,h}^i=\tilde\Omega_{\eta,h}\cap C_i^h$. 

If $\tilde\Omega_{\eta,h}^i\neq \emptyset$ then $\overline K_h\star w_0(x)\leq C\,\eta$ on the whole $C_i^h$. In that case
\[\begin{split}
C\eta\,h^{d}&\geq \int_{C_i^h} \overline K_h \star w_0\geq \int_{C_i^h}\int_{\tilde \Omega_{\eta,h}^i} \overline K_h(x-y)\,w_0(y)\,dy\,dx\\
&\geq \frac{\sqrt{\eta}}{C}\,|\tilde \Omega_{\eta,h}^i|.
\end{split}\] 
We conclude that $|\tilde \Omega_{\eta,h}^i|\leq C\,\sqrt{\eta}\,h^{d}$. Summing over the cubes, we deduce that one has
\[
|\tilde\Omega_{\eta,h}|\leq C\,\sqrt{\eta}.
\]
Finally
\[
\int_{\Omega_{\eta,h}} \rho_k\,dx\leq \int_{\tilde \Omega_{\eta,h}}\rho_k\,dx+\frac{2}{|\log \eta|}\,\int_{\Pi^d} \rho_k\,|\log w_0|\,dx\leq C\,\eta^{1/2-1/2\gamma}+  \frac{C}{\,|\log \eta|},
\]
since $\rho\in L^\infty_t(L^\gamma_x)$ for some $\gamma>1$. This is the desired bound.

  The same bound may be obtained on the quantity $\rho_k \ind_{\overline K_h\star w_a\le \eta}$ in a similar way when $p>\gamma+1$ because of bound (\ref{estimwa}) on $\rho_k|\log w_a|$.
 
\bigskip

\noindent {Point \rm iv)} Denote to simplify $f=|\div u_k|+|A_\mu\,\rho^\gamma_k|$, then by the definition of $w_{a,h}$
\[\begin{split}
\int_{h_0}^1\int_0^t&\Big\| \overline K_h \star \bigl(\overline K_h \star f\, w_a\bigr) -
     \bigl(\overline K_h \star f \bigr)\,\omega_{a,h}\Big\|_{L^q}\,dt\,\frac{dh}{h}\\
&\leq \int_{h_0}^1\int_0^t\overline K_h(z)\Big\| \bigl(\overline K_h \star f(.+z) -
\overline K_h \star f(.) \bigr)\,\omega_{a}(.+z)\Big\|_{L^q}\,dt\,\frac{dh}{h}\\
&\leq \int_0^t\int_{h_0}^1\overline K_h(z)\Big\| \bigl(\overline K_h \star f(t,.+z)- \overline K_h \star f(t,.) \bigr)\Big\|_{L^q}\,\frac{dh}{h}\,dt\\
&\leq C\,|\log h_0|^{1/2}\,\int_0^t \|f(t,.)\|_{L^q_x}\,dt\leq C\,|\log h_0|^{1/2},    
\end{split} 
\]
by a direct application of Lemma \ref{shiftNlemma} with $N_h=\bar K_h$ provided $f$ is uniformly bounded in $L^1_t L^q_x$ which is guaranteed by $q\leq \min (2,\;p/\gamma)$.



\end{proof}
%
\section{Proof of Theorems \ref{maincompactness}, \ref{maincompactness2} and \ref{maincompactness3} \label{proofstabilityres}}
 We start with the propagation of regularity
 on the transport equation in terms of the regularity of $\div u_k$, more precisely $\div_x u_k^x - \div_y u_k^y$.
We prove in the second subsection  some estimates on the effective pressure. This allows to
 write a Lemma in the third subsection controlling $\div_x u_k^x - \div_y u_k^y$ and then to close the loop in the fourth 
 subsection thus concluding the proof.
%
\subsection{The propagation of regularity on the transport equation}
     This subsection uses only Eq. \eqref{continuity} on $\rho_k$ without yet specifying the coupling between $\div u_k$ and $\rho_k$ (for instance through \eqref{divu0}). Recall that we denote
\[
\delta \rho_k=\rho_k(t,x)-\rho_k(t,y),\qquad \bar\rho_k=\rho_k(t,x)+\rho_k(t,y).
\]
Choose any $C^2$ function $\chi$ s.t.
\begin{equation}
\left|\chi(\xi)-\frac{1}{2}\chi'(\xi)\,\xi\right|\leq \frac{1}{2} \chi'(\xi)\,\xi,
\quad 
\chi'(\xi) \,\xi\leq C\,\chi(\xi)\leq C\,|\xi|.\label{defchi}
\end{equation}
It is for instance possible to take $\chi(\xi)=\xi^2$ for $|\xi|\leq 1/2$ and $\chi(\xi)=|\xi|$ for $|\xi|\geq 1$. Note however that the first inequality on the right forces $\chi(\xi)\geq |\xi| /C$. 

\medskip

\noindent Similarly for the anisotropic viscous term, for some $\ell>0$, choose any $\chi_a\in C^1$ 
\begin{equation}\begin{split}
&\left|\chi_a(\xi)-\frac{1}{2}\chi_a'(\xi)\,\xi\right|\leq \frac{1-\ell}{2} \chi_a'(\xi)\,\xi,
\quad 
\chi_a'(\xi) \,\xi\leq C\,\chi_a(\xi)\leq C\,|\xi|^{1+\ell},\\
&
(\xi^\gamma+\tilde\xi^\gamma)\,
\bigl(- \chi_a'(\xi-\tilde\xi)(\xi-\tilde\xi) + 2\chi_a(\xi-\tilde\xi)\bigr) \\
& \hskip4cm \geq- (\xi^\gamma-\tilde\xi^\gamma)\,\frac{\ell-1}{\ell}\chi_a'(\xi-\tilde\xi)(\xi+\tilde\xi).
\end{split}
\label{defchi0}
\end{equation}
Note that it is possible to simply choose $\chi_a=|\xi|^{1+\ell}$. But to unify the notations and the calculations with the other terms involving $\chi$, we use the abstract $\chi_a$.

The properties on this non-linear function will be strongly used to characterize the effect of the pressure law in the contribution of $\div_x u_k(x) - \div_y u_k(y)$ in the anisotropic case. The form of $\chi_a$ and the choice of $\ell$  will have to be determined very precisely so that the corresponding which will be exactly counterbalanced by the $\lambda$ terms in Lemma \ref{boundQhAniso}.

\medskip

   We can write two distinct Lemmas respectively concerning the non monotone pressure law case and the anisotropic tensor case.
\begin{lemma} Assume that $\rho_k$ solves \eqref{continuity} with estimates \eqref{boundrhok} and \eqref{bounduk} on $u_k$.

{\rm i)} With diffusion, $\alpha_k>0$, if $p> 2$,  $\exists \eps_{h_0}(k)\rightarrow 0$ as $k\rightarrow \infty$ for a fixed $h_0$ 
\[\begin{split}
&\int_{h_0}^1\int_{\Pi^{4d}} \overline K_h(x-z)\,\overline K_h(y-w)\,W_{0}(t,z,w)\,K_h(x-y)\,\chi(\delta\rho_k)\,dx\,dy\,dz\,dw\,\frac{dh}{h}\\
&\ \leq C\,(\eps_{h_0}(k)+|\log h_0|^{1/2})\\
&-\frac{1}{2}\,\int_{h_0}^1 \int_0^t\int_{\Pi^{4d}} K_h(x\!-\!y)\,(\div u_k(x)\!-\!\div u_k(y))\,\chi'_0\,\bar\rho_k\,W_{0}\,\overline K_h(x\!-\!z)\,\overline K_h(y\!-\!w)\, \frac{dh}{h} \\
& - \frac{1}{2}\,\int_{h_0}^1\int_0^t\int_{\Pi^{4d}}K_h(x-y)\,(\div u_k(x)+\div u_k(y))\\
&\qquad\qquad\qquad\qquad\qquad (\chi'(\delta\rho_k)\,\delta\rho_k-2\,\chi(\delta\rho_k))\,\,W_{0}\,\overline K_h(x-z)\,\overline K_h(y-w)\frac{dh}{h} \\
&-  \frac{\lambda}{2} \,  \int_{h_0}^1\int_0^t\int_{\Pi^{3d}} K_h(x-y) \chi(\delta\rho_k)\, \overline K_h(x-z)\,M|\nabla u_k|(z)\, w_0(z)\,\frac{dh}{h},
\end{split}\]
where we recall that $W_0=w_0(t,x)+w_0(t,y)$.

\medskip

{\rm  ii)} Without diffusion, $\alpha_k=0$, if $p\geq 2$, 
\[
\begin{split}
&\int_{h_0}^1\int_{\Pi^{2d}}\overline K_h(x-y)\,\chi(\delta\rho_k)\, (w_1(t,x)+w_1(t,y))\,dx\,dy\,\frac{dh}{h}\leq C\, |\log h_0|^{1/2}\,\|u_k\|_{H^1}\\
& -2\,\lambda\,\int_{h_0}^1\int_0^t\int_{\Pi^{2d}} \overline K_h(x-y)\,(\rho_k^{\tilde\gamma}(x)+\tilde P_k(x)\,\rho_k(x)+R_k)\,w_1(x)\,  \chi(\delta\rho_k)\,\frac{dh}{h}\\ 
&-2\,\int_0^t\int_{h_0}^1\int_{\Pi^{2d}} \overline K_h(x-y)\,(\div u_k(x)-\div u_k(y))\,\Big(\frac{1}{2}\chi'(\delta\,\rho_k)\,\bar\rho_k\\
&\quad +\chi(\delta\rho_k)-\frac{1}{2}\chi'(\delta\,\rho_k)\,\delta\rho_k \Big)\,w_1(x)\,\frac{dh}{h}.\end{split}
\]
For the derivation of explicit regularity estimates, we have as well the version with the product weight, namely
\[
\begin{split}
\int_{\Pi^{2d}}K_h(x-y)\,&\chi(\delta\rho_k)\,w_1(t,x)\,w_1(t,y)\,dx\,dy\leq C\\
-\int_0^t\int_{\Pi^{2d}} K_h(x-y)\,&(\div u_k(x)-\div u_k(y))\,\chi'(\delta\,\rho_k)\,\bar\rho_k\, w_1(x)\,w_1(y)\\
-\lambda\,\int_0^t\int_{\Pi^{2d}} K_h(x-y)\,&\Big(\rho_k^{\tilde\gamma}(x)+\tilde P_k(x)\,\rho_k(x)+ R_k(x)+\rho_k^{\tilde\gamma}(y)\\
&\qquad+\tilde P_k(y)\,\rho_k(y)+R_k(y)\Big)\,  \chi(\delta\rho_k)\,w_1(x)\,w_1(y).
\end{split}
\]
\label{boundQh}
\end{lemma}
    For convenience, we write separately the result we will use in the anisotropic case
\begin{lemma} Assume that $\rho_k$ solves \eqref{continuity} with estimates \eqref{bounduk}-\eqref{boundrhok}. Without diffusion, $\alpha_k=0$, assume \eqref{defchi0} on $\chi_a$ with $p> \gamma+\ell+1$, and denote $w_{a,h}=\overline K_h\star w_a$,
\[
\begin{split}
&\int_{h_0}^1 \int_{\Pi^{2d}} \frac{\overline K_{h}(x-y)}{h} (w_{a,h}(x) + w_{a,h}(y))\,  \chi_a(\delta\rho_k) (t)  \\
&
\leq \int_{h_0}^1 \int_{\Pi^{2d}} \frac{\overline K_{h}(x-y)}{h} (w_{a,h}(x) + w_{a,h}(y))\,  \chi_a(\delta\rho_k)|_{t=0}  
+C\,|\log h_0|^{\theta} +  I + II-\Pi_a,  \\
\end{split}
\]
with the dissipation term 
\[
 \Pi_a=   \lambda \int_0^t \int_{h_0}^1 \int_{\Pi^{2d}} w_{a,h}(x)\,
     \chi_a(\delta \rho_k)\,\overline K_h \star (|\div u_k|+|A_\mu \rho^\gamma|)(x)\,
    \bar{K_h} \frac{dh}{h} ,
\]
while 
\[\begin{split}
I= - \frac{1}{2}\int_{h_0}^1\int_0^t\int_{\Pi^{2d}} \frac{\overline K_h(x-y)}{h}\,(\div u_k(x)-\div u_k(y))&
        \chi_a'(\delta\rho_k) \,\bar \rho_k\\
&(w_{a,h}(x)+w_{a,h}(y)),
\end{split}\]
and 
\[\begin{split}
II=- \frac{1}{2}\int_{h_0}^1 \int_0^t\int_{\Pi^{2d}} 
&\frac{\overline K_h(x-y)}{h}\,(\div u_k(x)+\div u_k(y))\\
& \hskip2cm  (\chi'_a(\delta\rho_k) \delta\rho_k - 2 \chi_a(\delta\rho_k))\, 
(w_{a,h}(x)+ w_{a,h}(y)).
\end{split}\]
\label{boundQhAniso}
\end{lemma}

\bigskip

\noindent {\bf Remark.} We emphasize that the $\lambda$ terms in relations  i) and ii) of Lemma \ref{boundQh} come from the penalization in the definition of the weights $w_0$ and $w_1$. They will help to  counterbalance 
terms  coming from the contribution
by $\div_x u_k(x)-\div_y u_k(y)$.  Similarly the nonlocal term $\Pi_a$ 
follows from the definition of the weights $w_a$. 
bad terms coming from the viscosity anisotropy.


\begin{proof}\

\medskip

\noindent {\bf Case \rm i)}.
Denote
\[
W_{0,h}(t,x,y)=\int_{\Pi^{2d}}\overline K_h(x-z)\,\overline K_h(y-w)\,W_0(t,z,w)\,dz\,dw,
\]
and let us use $\chi$ in the renormalized equation from Lemma \ref{renor}. We get 
\[\begin{split}
&\int_{\Pi^d} W_{0,h}(t,x,y)\,K_h(x-y)\,\chi(\delta\rho_k) \,dx\,dy=A+B+D+E\\
&\ -\frac{1}{2}\int_0^t\int K_h(x-y)\,(\div u_k(x)-\div u_k(y))
   \,\chi'(\delta\rho_k)\,\bar\rho_k\,W_{0,h}\,dx\,dy,\\
& - \frac{1}{2}\int_0^t\int K_h(x-y)\,(\div u_k(x)+\div u_k(y))\,(\chi'(\delta\rho_k)\,\delta\rho_k-2\,\chi(\delta\rho_k))\,W_{0,h}\,dx\,dy,
\end{split}\]
with, by the symmetry of $K_h$, $\overline K_h$ and $W_{0,h}$, and in particular since $\nabla_x W_{0,h}(x,y)=\nabla_y W_{0,h}(y,x)$
\[\begin{split}
&A=\int_0^t\int_{\Pi^{2d}} (u_k(t,x)-u_k(t,y))\cdot \nabla K_h(x-y)\,\,\chi(\delta\,\rho_k)
        \,W_{0,h}\,dx\,dy\,dt,\\
&B=2\int_0^t\int_{\Pi^{2d}} K_h(x-y)\,(\partial_t W_{0,h}(x)+u_k\cdot\nabla_x W_{0,h}
             -\alpha_k \Delta_x W_{0,h})\,\chi(\delta\,\rho_k)\,dx\,dy\,dt,\\
&D=
    2\,\alpha_k\int_0^t\int_{\Pi^{2d}} \chi(\delta\,\rho_k)\,\left[\Delta_x\,K_h(x-y)\,W_{0,h} + 2\,K_h(x-y)\,\Delta_x\,W_{0,h}\right]\,  dx\,dy\,dt,\\
&E=-2\,\alpha_k\int_0^t\int_{\Pi^{2d}}K_h\,W_{0,h}\,\chi''(\delta\,\rho_k)\,|\nabla_x\delta\rho_k|^2\,dx\,dy\,dt.\\
\end{split}
\]
Now note that by the convexity of $\chi$, $E\leq 0$. Then simply bound using \eqref{defchi}, 
\[
D\leq 8\,\alpha_k\,\,h^{-2}\,\|K_h\|_{L^1}\,\|\rho\|_{L^{1}}\leq C\,\alpha_k\,\,h^{-2}\,\|K_h\|_{L^1},
\]
leading us to choose $\eps_{h_0}(k)=\alpha_k\,\int_{h_0}^1 h^{-2}\,\frac{dh}{h}$.

As for $B$, using Eq. \eqref{eqw}
\[
B=B_1-2\int K_h(x-y)\,\chi(\delta\rho_k)\,\tilde K_h\star_{x,y}\,R_0\,dx\,dy\,dt,
\]
with
\[\begin{split}
&B_1=2\int K_h(x-y)\,\chi(\delta\rho_k)\,(u_k(x)-u_k(z))\cdot
\nabla_x\overline K_h(x-z)\\
&\qquad\quad\overline K_h(y-w)\,W_0(t,z,w)\,dx\,dy\,dz\,dw\,dt.
\end{split}
\]
We recall that $R_0(x,y)=D_0(x)\,w_0(x)+D_0(y)\,w_0(y)$ with $D_0=\lambda\,\overline K_h\star (M\,|\nabla u_k|)$ and we thus only have to bound $B_1$.
By Lemma \ref{diffulemma}, we have 
\[\begin{split}
&B_1\le C\,\int K_h(x-y)\,\chi(\delta\rho_k)\,(D_{|x-y|} u_k(x)+D_{|x-y|} u_k(z))\\
&\qquad\quad|x-z|\,
|\nabla\overline K_h(x-z)|\overline K_h(y-w)\,W_0(t,z,w)\\
&\ \leq C \int K_h(x-y)\,\chi(\delta\rho_k)\,(D_{|x-y|} u_k(x)+D_{|x-y|} u_k(z))
\,\overline K_h(x-z)\overline K_h(y-w)\,W_0,
\end{split}
\]
as $|x|\,|\nabla K_h|\leq C\,K_h$. 

Next recalling that $W_0(t,z,w)=w_0(t,z)+w_0(t,w)$, by symmetry
\[\begin{split}
&B_1\le C\,\int K_h(x-y)\,\chi(\delta\rho_k)\,D_{|x-y|} u_k(z)
\,\overline K_h(x-z)\overline K_h(y-w)\,w_0(t,z)\\
&\ +C\, \int K_h(x-y)\,\chi(\delta\rho_k)\,(D_{|x-y|} u_k(x)+D_{|x-y|} u_k(w)-2D_{|x-y|} u_k(z))
\\
&\qquad\qquad\qquad\qquad\qquad\qquad\qquad\overline K_h(x-z)\overline K_h(y-w)\,\,w_0(t,z).\\
\end{split}
\]
Since $D_{|x-y|} u_k(z)\leq C\,M|\nabla u_k|(z)$, for $\lambda$ large enough, the first term may be bounded by 
\[
-\frac{\lambda}{2}\,\int K_h(x-y)\,\chi(\delta\rho_k)\,\overline K_h\star_x(M|\nabla u_k|\,w_0)\,dx\,dy\,dt.
\]
Use the uniform bound on $\|\rho_k\|_{L^p}$ with $p>2$, to find 
\[\begin{split}
&\int K_h(x-y)\,\chi(\delta\rho_k)\,(D_{|x-y|} u_k(x)+D_{|x-y|} u_k(w)-2D_{|x-y|} u_k(z))\tilde K_h\,w_0(t,z)\\
&\ \leq C\,\int \left\|D_{|r|} u_k(.)+D_{|r|} u_k(.+r+u)-2D_{|r|} u_k(.+v)\right\|_{L^2}\,K_h(r)\,\overline K_h(u)\,\overline K_h(v),
\end{split}
\]
where we used that $w=x+(y-x)+(w-y)$. We now use Lemma \ref{shiftDulemma} and more precisely the inequality \eqref{squarefunctiondisc} to obtain
\[\begin{split}
&\int_{h_0}^1\int \overline K_h(x-y)\,\chi(\delta\rho_k)\,(D_{|x-y|} u_k(x)+D_{|x-y|} u_k(w)-2D_{|x-y|} u_k(z))\tilde K_h\,w_0\,\frac{dh}{h}\\
&\ \leq C\,|\log h_0|^{1/2}\,\int_0^t\|u_k(t,.)\|_{H^1}\,dt\leq C\,|\log h_0|^{1/2}.
\end{split}
\]
Therefore we have that
\[\begin{split}
&\int_{h_0}^1 B\,\frac{\|K_h\|_{L^1}^{-1}}{h}\,dh\leq C\,\eps_{h_0}(k)+C\,|\log h_0|^{1/2}\\
&\ -\frac{3\,\lambda}{4}\int_{h_0}^1\int_0^t\int \overline K_h(x-y)\,\chi(\delta\rho_k)\,M|\nabla u_k|(z)
\,\overline K_h(x-z)\,w_0(z)\,dz\,dx\,dy\,dt\,\frac{dh}{h}.
\end{split}\]
The computations are similar for $A$ and we only give the main steps. We have,
using again Lemma \ref{diffulemma}  
\[\begin{split}
& \int_{\Pi^{2d}} \nabla K_h(x-y)\cdot (u_k(t,x)-u_k(t,y))\,\chi(\delta\rho_k)\,W_{0,h} \\
&\qquad  \le  C\,\int_{\Pi^{2d}} K_h(x-y)\,(D_{|x-y|}\,u_k(x)+D_{|x-y|}\,u_k(y))\,\chi(\delta\,\rho_k)\,W_{0,h}.
\end{split}\]
By decomposing $W_{0,h}$, we can write just as for $B_1$
\[\begin{split}
& \int_{\Pi^{2d}} \nabla K_h(x-y)\cdot (u_k(t,x)-u_k(t,y))\,\chi(\delta\rho_k)\,W_{0,h} \\
&\ \leq C\,\int_{\Pi^{2d}} K_h(x-y)\,M\,|\nabla u_k|(z)\,\chi(\delta\,\rho_k)\,\overline K_h(x-z)\,w_{0}(z)\\
&+C\,\int_{\Pi^{2d}} K_h(x-y)\, (D_{|x-y|}\,u_k(x)+D_{|x-y|}\,u_k(y)+D_{|x-y|}\,u_k(w)-3\,D_{|x-y|}\,u_k(z))\\
&\qquad\qquad\qquad\qquad\qquad\qquad\qquad \chi(\delta\,\rho_k)\,\overline K_h(x-z)\,\overline K_h(y-w)\,w_0(z).
\end{split}
\]
The first term in the r.h.s. can again be bounded by 
\[
-\frac{\lambda}{2}\,\int K_h(x-y)\,\chi(\delta\rho_k)\,\overline K_h\star_x(M|\nabla u_k|\,w_0)\,dx\,dy\,dt.
\]
The second term in the r.h.s. is now integrated in $h$ and controlled as before thanks to the bound \eqref{squarefunctiondisc} in Lemma \ref{shiftDulemma} and the uniform $L^p$ bound on $\rho_k$ and $H^1$ on $u_k$.
This leads to 
\[\begin{split}
&\int_{h_0^1} A\,\|K_h\|_{L^1}^{-1}\,\frac{dh}{h}\leq C\,|\log h_0|^{1/2}\\
&\ +\frac{\lambda}{4}\,\int_{h_0}^1\int_0^t\int_{\Pi^{2d}} \overline K_h(x-y)\,M|\nabla u_k|(z)\,\chi(\delta\,\rho_k)\,\overline K_h(x-z)\,w_{0}(z)\,dx\,dy\,dz\,dt\,\frac{dh}{h}.
\end{split}\]
Now summing all the contributions we get
\[\begin{split}
&\int_{h_0}^1 (A+B+D+E)\,\|K_h\|_{L^1}^{-1}\,\frac{dh}{h} \leq C\,\eps_{h_0}(k)+C\,|\log h_0|^{1/2} \\
&\qquad\qquad -\frac{\lambda}{2}\, \int_{h_0}^1\int_0^t \int K_h(x-y) \chi(\delta\rho_k)  \overline K_h \star  (M|\nabla u_k|\, w_0).
\end{split}\]
Note that indeed $\eps_{h_0}(k)\rightarrow 0$ as $k\rightarrow\infty$ for a fixed $h_0$. This concludes the proof in that first case. 

\bigskip

\noindent {\bf Case {\rm ii)}}. In this part, we assume $\alpha_k=0$. We may not assume that $\rho_k$ is smooth anymore. However by \cite{DL} since $\rho_k\in L^2$ and so is $\nabla u_k$, one may use the renormalized relation with $\varphi = \chi$ and choose $W_{k,h}^{x,y} = W_i$. 
We then can use the identity given in Lemma \ref{renor}. 
Denoting
\[
\tilde \chi(\xi)=\frac{1}{\chi(\xi)}\,\left(\chi(\xi)-\frac{1}{2}\, \chi'(\xi)\, \xi\right),
\]
we get for  $i=1,\,2$ 
\[\begin{split}
\int_{\Pi^{2d}} K_h(x-y)\,\chi(\delta\rho_k)\,W_i(t,x,y)\,dx\,dy \leq  A_i+ B_i+D_i,
\end{split}\]
where by the symmetry in $x$ and $y$
\[\begin{split}
A_1&=\int_0^t\int_{\Pi^{2d}} (u_k(t,x)-u_k(t,y))\cdot \nabla K_h(x-y)
   \,\chi(\delta\,\rho_k)\,W_1\,dx\,dy\,dt\\
&-\lambda\,\int_0^t\int_{\Pi^{2d}} K_h(x-y)\,(M\,|\nabla u_k|(x)\,w_1(x)+M\,|\nabla u_k|(y)\,w_1(y))\,\chi(\delta\,\rho_k),
\end{split}
\]
while
\[\begin{split}
&A_2=\int_0^t\int_{\Pi^{2d}} (u_k(t,x)-u_k(t,y))\cdot \nabla K_h(x-y)\,\chi(\delta\,\rho_k)\,W_2\,dx\,dy\,dt\\
&-\lambda\,\int_0^t\int_{\Pi^{2d}} K_h(x-y)\,(M\,|\nabla u_k|(x)+M\,|\nabla u_k|(y))\,w_1(x)\,w_1(y))\,\chi(\delta\,\rho_k).
\end{split}
\]
And
\[
\begin{split}
B_1=2\int_0^t\int_{\Pi^{2d}} K_h(x-y)\,\big(&\partial_t w_1(x)+u_k(x)\cdot\nabla_x w_1+2\div_x u_k(x)\,\tilde\chi(\delta\rho_k)\,w_1\\
&+\lambda\,M|\nabla u_k|(x)\,w_1\big)\,\chi(\delta\,\rho_k)\,dx\,dy\,dt,\\
\end{split}
\]
while
\[
\begin{split}
B_2=2\int_0^t\int_{\Pi^{2d}} K_h(x-y)\,\big(&\partial_t w_1(x)+u_k(x)\cdot\nabla_x w_1+\div_x u_k(x)\,\tilde\chi(\delta\rho_k)\,w_1\\
&+\lambda\,M|\nabla u_k|(x)\,w_1\big)\,w_1(y)\,\chi(\delta\,\rho_k)\,dx\,dy\,dt.\\
\end{split}
\]
Finally
\[
\begin{split}
&D_1=-2\,\int_0^t\int_{\Pi^{2d}} K_h(x-y)\,(\div u_k(x)-\div u_k(y))\,\Big(\frac{1}{2}\chi'(\delta\,\rho_k)\,\bar\rho_k\\
&\quad +\chi(\delta\rho_k)-\frac{1}{2}\chi'(\delta\,\rho_k)\,\delta\rho_k \Big)\,w_1(x)\,dx\,dy\,dt,\\
\end{split}
\]
and
\[
\begin{split}
D_2=-\int_0^t\int_{\Pi^{2d}} K_h(x-y)\,&(\div u_k(x)-\div u_k(y))\,\chi'(\delta\,\rho_k)\,\bar\rho_k\\
& w_1(x)\,w_1(y)\,dx\,dy\,dt.\\
\end{split}
\]
Note that we have split in several parts a null contribution in terms of the maximal function
namely the ones with $M|\nabla u_k|^2$.
Notice also the additional terms in $D_1$ that come from cross products such as $\div u_k(y)$\,$\,\tilde \chi\,\chi\,w_1(x)$ which would pose problems in $B_1$. 

\medskip

The contributions $D_1$ and $D_2$ are already under the right form. Using Eq. \eqref{eqw} with \eqref{dk1}, one may directly bound
\[
\begin{split}
B_1\leq -2\,\lambda\,\int_0^t\int_{\Pi^{2d}}K_h(x-y)\,(\rho_k^{\tilde\gamma}(t,x)+\tilde P_k(x)\,\rho_k(x)+R_k(x))\,w_1(t,x)\,\chi(\delta\,\rho_k),\\
\end{split}
\]
and
\[
\begin{split}
B_2\leq -2\,\lambda\,\int_0^t\int_{\Pi^{2d}}K_h(x-y)\,(\rho_k^{\tilde\gamma}(t,x)+\tilde P_k(x)\,\rho_k(x)+R_k(x))\,w_1(x)\,w_1(y)\,\chi(\delta\,\rho_k),
\end{split}
\]
giving the desired result by symmetry of the expression in $x$ and $y$.

The term $A_2$ is straightforward to handle as well. Use \eqref{maximalineq} to get
\[\begin{split}
&A_2\leq \int_0^t\int_{\Pi^{2d}} |\nabla K_h(x-y)|\,|x-y|\,(M\,|\nabla u_k|(x)+M\,|\nabla u_k|(y))\,\chi(\delta\,\rho_k)\,W_2\\
&-\lambda\,\int_0^t\int_{\Pi^{2d}} K_h(x-y)\,(M\,|\nabla u_k|(x)+M\,|\nabla u_k|(y))\,w_1(x)\,w_1(y))\,\chi(\delta\,\rho_k).
\end{split}
\]
Since $|x|\,|\nabla K_h|\leq C K_h$, by taking $\lambda$ large enough, one obtains
\[
A_2\leq 0.
\]
The term $A_1$ is more complex because it has no symmetry. By Lemma \ref{diffulemma}
\[\begin{split}
A_1&\leq 
C\,\int_0^t\int_{\Pi^{2d}} |\nabla K_h(x-y)|\,|x-y|\,(D_{|x-y|}\,u_k(x)+D_{|x-y|}\,u_k(y))\,\chi(\delta\,\rho_k)\\
&\qquad\qquad\qquad w_1(t,x)\,dx\,dy\,dt\\
&-\lambda\,\int_0^t\int_{\Pi^{2d}} K_h(x-y)\,M\,|\nabla u_k|(x)\,w_1(t,x)\,\chi(\delta\,\rho_k)
 + \hbox{similar terms in}\ w_1(t,y).
\end{split}
\]
 The {\em key problem} here is the $D_h\,u(y)\,w_1(x)$ term which one has to control by the term $M\,|\nabla u|(x)\,w_1(x)$. This is where integration over $h$ and the use of Lemma \ref{shiftDulemma} is needed (the other term in $w_1(y)$ is dealt with in a symmetric manner). For that we will
 add and subtract an appropriate quantity to see the quantity 
 $D_{|x-y|} u_k (x) - D_{|x-y|} u_k(y)$.

\bigskip

   By the definition of $K_h$, $|z|\,|\nabla K_h(z)|\leq C\,K_h(z)$ and by \eqref{defchi}, $\chi(\delta\rho_k)\leq C\,(\rho_k(x)+\rho_k(y))$ with $\rho_k\in L^2$ uniformly moreover $w_1$ in uniformly
bounded. Hence using Cauchy-Schwartz and denoting $z=x-y$
\[\begin{split}
\int_{h_0}^1 \frac{A_1}{\|K_h\|_{L^1}}\frac{dh}{h}&\leq
2 C\,\int_{h_0}^1\int_0^t \int_{\Pi^d} \overline K_h(z)\,\|D_{|z|}\,u_k(.)-D_{|z|}u_k(.+z)\|_{L^2}\, \frac{dh}{h}\\
&+4\,C\,\int_0^t\int_{\Pi^{2d}} {\cal K}_{h_0} (x-y)
\,D_{|x-y|}\,u_k(x)\,\chi(\delta\,\rho_k)\,W_1\\
&-2 \lambda\,\int_0^t\int_{\Pi^{2d}}  {\cal K}_{h_0} (x-y)\,M\,|\nabla u_k|(x)\,w_1(x)\,\chi(\delta\,\rho_k)\\
&\leq 2 C\,\int_{h_0}^1\int_{\Pi^d} \overline K_h(z)\,\|D_{|z|}\,u_k(.)-D_{|z|}u_k(.+z)\|_{L^2}\, \frac{dh}{h},
\end{split}
\]
by taking $\lambda$ large enough since Lemma \ref{compareDhmax} bounds $D_{|x-y|}\,u_k(x)$ by $M\,|\nabla u_k|(x)$. 

\bigskip

\noindent Finally using Lemma \ref{shiftDulemma}
\[
\int_{h_0}^1 \frac{A_1}{\|K_h\|_{L^1}}\frac{dh}{h}\leq C\,|\log h_0|^{1/2}\,\|u_k\|_{H^1}.
\]
Summing up $A_i+B_i+D_i$, and integrating against $\displaystyle \frac{dh}{\|K_h\|_{L^1}\,h}$ for $i=1$, concludes the proof.

%
\end{proof}

\begin{proof} \hskip-5pt{\bf of Lemma \ref{boundQhAniso}.}
In this part, we again assume $\alpha_k=0$ and still use \cite{DL} to obtain the renormalized relation of Lemma \ref{renor} with $\varphi = \chi_a$ and $W_{k,h}^{x,y} =\overline K_h\star  W_a = w_{a,h}+w_{a,h}$. With this exception the proof follows the lines of point i) in the previous Lemma \ref{boundQh}, so we only sketch it here.

From Lemma \ref{renor}, we get
\[
\begin{split}
&\int_{h_0}^1 \int_{\Pi^{2d}} \frac{\overline K_{h}(x-y)}{h} (w_{a,h}(x) + w_{a,h}(y))\,  \chi_a(\delta\rho_k) (t)  \\
& 
\leq \int_{h_0}^1 \int_{\Pi^{2d}} \frac{\overline K_{h}(x-y)}{h} (w_{a,h}(x) + w_{a,h}(y))\,  \chi_a(\delta\rho_k)|_{t=0}   +
  A + B  + D +  I + II + \Pi_a,  \\
\end{split}
\]
with the terms 
\[
 A=   \int_0^t \int_{\Pi^{2d}} (u_k(t,x)-u_k(t,y))\cdot \nabla\overline  K_h(x-y) \chi_a(\delta\rho_k) 
        (w_{a,h}(x)+ w_{a,h}(y)),
\]
while
\[\begin{split}
 B
 & = 2\,\Big[ \int_0^t \int_{\Pi^{3d}}
         K_h(x-y) \chi_a(\delta\rho_k) (u_k(x) - u_k(z))\cdot \nabla_x \overline K_h(x-z)
                w_a(t,z)  \\                
&   \hskip4cm  -  \lambda 
\int_0^t \int_{\Pi^{2d}} K_h(x-y) \chi_a(\delta\rho_k) \overline K_h \star (M|\nabla u_k| w_a)\Big],
\end{split}\]
and
\[\begin{split}
 D =  
 &-  \lambda \Bigl[ \int_0^t \int_{h_0}^1 \int_{\Pi^{2d}} 
     \chi_a(\delta \rho_k)\,\overline K_h \star ((|\div u_k|+|A_\mu \rho^\gamma|)w_{a}\,)(x)\,
     \overline K_h(x-y) \frac{dh}{h}    \\
  & -  \int_0^t \int_{h_0}^1 \int_{\Pi^{2d}} w_{a,h}(x)\,
     \chi_a(\delta \rho_k)\,\overline K_h \star (|\div u_k|+|A_\mu \rho^\gamma|)(x)\,
    \overline{K_h}(x-y) \frac{dh}{h} 
  \Bigr].
\end{split}\]
The dissipation term is under the right form
\[
 \Pi_a=  - \lambda \int_0^t \int_{h_0}^1 \int_{\Pi^{2d}} w_{a,h}(x)\,
     \chi_a(\delta \rho_k)\,\overline K_h \star (|\div u_k|+|A_\mu \rho^\gamma|)(x)\,
    \overline{K_h}(x-y) \frac{dh}{h} ,
\]
and so are by symmetry
\[\begin{split}
I= - \frac{1}{2}\int_{h_0}^1\int_0^t\int_{\Pi^{2d}} &\frac{\overline K_h(x-y)}{h}\,(\div u_k(x)-\div u_k(y))
        \chi_a'(\delta\rho_k) \,\bar \rho_k\,w_{a,h}(x),
\end{split}\]
and 
\[\begin{split}
II=- \frac{1}{2}\int_{h_0}^1 \int_0^t\int_{\Pi^{2d}} 
&\frac{\overline K_h(x-y)}{h}\,(\div u_k(x)+\div u_k(y))\\
& \hskip3cm  (\chi'_a(\delta\rho_k) \delta\rho_k - 2 \chi_a(\delta\rho_k))\, w_{a,h}(x).
\end{split}\]
The terms $A$ and $B$ are treated exactly as in case i) of Lemma \ref{boundQh}; they only require the higher integrability $p>\gamma+1+\ell$. 

The only additional term is hence $D$ which is required in order to write the dissipation term $\Pi_a$ in the right form. $D$ is bounded directly by point iv)
 in Lemma \ref{weightprop}. Thus 
\[
A+B+D\leq C\,|\log h_0|^\theta,
\]
for some $0<\theta<1$ which concludes the proof.
\end{proof}

%
\subsection{The control on the effective flux}
Before coupling the previous estimate with the equation on $\div u_k$, we start with a lemma which will be used in every situation as it controls the regularity properties of
\[
D\rho u_k=\Delta^{-1}\,\div\,\left(\partial_t(\rho_k\,u_k)+\div(\rho_k\,u_k\otimes u_k)\right), 
\]
per
\begin{lemma}
Assume that $\rho_k$ solves \eqref{continuity}, that \eqref{bounduk}-\eqref{rhout} hold and \eqref{boundrhok} with $\gamma>d/2$. Assume moreover that $\Phi\in L^\infty([0,\ T]\times \Pi^{2d})$ and that 
\[\begin{split}
C_\phi&:=\left\|\int_{\Pi^d}\overline K_h(x-y)\,\Phi(t,x,y)\,dy\right\|_{ W^{1,1}_t\,W^{-1,1}_x}\\&
+\left\|\int_{\Pi^d}\overline K_h(x-y)\,\Phi(t,x,y)\,dx\right\|_{ W^{1,1}_t\,W^{-1,1}_y}
<\infty,
\end{split}\] 
then there exists $\theta>0$ s.t.
\[\begin{split}
&\int_0^t \int_{\Pi^d} K_h(x-y)\,\Phi(t,x,y)\,(D\rho u_k(t,y)-D\rho u_k(t,x))\,dx\,dy\,dt\\
&\qquad\qquad\leq C\,\|K_h\|_{L^1}\,h^\theta\,\left(\|\Phi\|_{L^\infty_{t,x}}+ C_\Phi\right).
\end{split}\]
\label{Drhou}
\end{lemma}

\begin{proof}  This proof is divided in four steps: The first one concerns a control 
on $\rho_k |u_k- u_{k,\eta}|$ where $u_{k,\eta}$ is a regularization of $u_k$ defined
later-on; The second step concerns the proof of an estimate for $\Phi_x= \bigl(\overline K_h \star \Phi\bigr)(t,x)$ and $\Phi_y= \bigl(\overline K_h \star \Phi\bigr)(t,y)$ in~$L^2_t L_x^{\bar p'}$ with $\bar p'= \bar p/(\bar p-1)$; The third step  concerns a control with respect to $h$ when $\Phi_x$, $\Phi_y$ in $L^2_t L_x^{\bar p'}
\cap W^{1,+\infty}_t W^{-1,\infty-0}$ with $\bar p'= \bar p/(\bar p-1)$; The last term is
the end of the proof obtained by interpolation.

\bigskip

\noindent {i) \it A control on $\rho_k |u_k- u_{k,\eta}|$ where $u_{k,\eta}$ 
                          is a regularization in space and time defined later-on.}
Choose a kernel ${\cal K}_\eta\in C^{\infty}_c(\R_+\times\R_+)$ s.t. ${\cal K}_\eta(t,s)=0$ if $|t-s|\geq \eta$, for smoothing in time. We still denote, with a slight abuse of notation
\[
{\cal K}_\eta\star_t u_k (t)=\int_{\R_+} {\cal K}_\eta(t,s)\, u_k (s)\,ds.
\]
Now since $u_k$ is uniformly bounded in $L^2_t H^1_x\subset L^2_t L^6_x$ in dimension $d=3$ and in $L^2_t L^q_x$ for any $q<\infty$ in dimension $d=2$, one has 
\[\begin{split}
&\int \rho_k(t,x)\,\frac{(u_k(t,x)-u_k(s,x))^2}{1+|u_k(t,x)|+|u_k(s,x)|}\,{\cal K}_\eta(t,s)\,dt\,ds\,dx\\
&\quad \leq \int \rho_k(t,x)\,(u_k(t,x)-u_k(s,x))\,{\cal K}_\eta(t,s)\,\overline K_{\eta'}\star_x\frac{u_k(t,.)-u_k(s,.)}{1+|u_k(t,.)|+|u_k(s,.)|}\\
&\qquad+C\,(\eta')^\theta\,\|\rho_k\|_{L^\infty_t L^\gamma_x},
\end{split}
\]
with $\theta>0$ if $\gamma>d/2$ for $d=2,\;3$.  Note that
\[
\|\partial_t \rho_k\|_{L^1_t W^{-1,1}_x}\leq C,
\]
and by interpolation, as $\gamma>d/2$ and thus $\gamma>2d/(d+2)$, there exists $\theta>0$ s.t.
\[
\|\partial_t^\theta \rho_k\|_{L^2_t H^{-1}_x}\leq C.
\] 
Thus
\[\begin{split}
&\int \rho_k(t,x)\,\frac{(u_k(t,x)-u_k(s,x))^2}{1+|u_k(t,x)|+|u_k(s,x)|}\,{\cal K}_\eta(t,s)\,dt\,ds\,dx\\
& \leq\int(\rho_k(t)\,u_k(t,x)-\rho_k(s)\,u_k(s,x))\,{\cal K}_\eta(t,s)\,\overline K_{\eta'}\star\frac{u_k(t,.)-u_k(s,.)}{1+|u_k(t,.)|+|u_k(s,.)|}\\
&\qquad+C\,(\eta')^\theta+C\,\frac{\eta^\theta}{\eta'^d}\,\left\|u_k
\right\|_{L^2_t H^1_x}.\end{split}\]
Using \eqref{rhout}, one deduces that
\[
\begin{split}
&\int \rho_k(t,x)\,\frac{(u_k(t,x)-u_k(s,x))^2}{1+|u_k(t,x)|+|u_k(s,x)|}\,{\cal K}_\eta(t,s)\,dt\,ds\,dx\\
&\leq C\eta'^\theta+C\,\frac{\eta^\theta}{\eta'^d}+C\,\frac{\eta}{\eta'^d}.
\end{split}
\]
Optimizing in $\eta'$ and interpolating (using again $\gamma>d/2$), one finally gets that for some $\theta>0$
\[
\int \rho_k(t,x)\,(u_k(t,x)-u_k(s,x))^2\,{\cal K}_\eta(t,s)\,dt\,ds\,dx\leq C\eta^\theta.
\]
Define for some $\nu>0$
\[
u_{k,\eta}=\overline K_{\eta^\nu}\star_x\,{\cal K}_\eta\star_t u_k. 
\]
Using the regularity in $x$ of $u_{k,\eta}$, one has that
\[
\int \rho_k(t,x)\,|u_k(t,x)-u_{k,\eta}(t,x)|^2\,dx\,dt\leq C\,\eta^\theta,
\]
and finally by \eqref{bounduk}, 
there exists $\theta>0$ s.t.
\begin{equation}
\left\|\rho_k\,(u_k-u_{k,\eta})\right\|_{L^{1+0}_t L^{1+0}_x}\leq C\,\eta^{\theta}.\label{ut-us}
\end{equation}

\medskip

\noindent {ii) \it The case where $\Phi_x$, $\Phi_y$ is only in $L^2_t L_x^{\bar p'}$ with $\bar p'= \bar p/(\bar p-1)$.} We recall that $\bar p$ is the exponent in \eqref{rhout}.
Denote
\[\begin{split}
I\Phi=&\int_0^t \int_{\Pi^{2d}} K_h(x-y)\,\Phi(t,x,y)\,(D\rho u_k(t,y)-D\rho u_k(t,x))\,dx\,dy\,dt\\
\end{split}\]
which can be seen as a linear form on $\Phi$. Recall as well 
\[
\Phi_x=\int_{\Pi^d} \overline K_h(x-y)\,\Phi(t,x,y)\,dy,\quad \Phi_y=\int_{\Pi^d} \overline K_h(x-y)\,\Phi(t,x,y)\,dx.
\]
By \eqref{rhout}, $D\rho u_k$ is uniformly bounded in $L^2_t L^{\bar p}_x$. Therefore
\begin{equation}
|I\,\Phi|\leq C\,\|K_h\|_{L^1}\,\left(\|\Phi_x\|_{L^2_t L^{\bar p'}_x}+\|\Phi_y\|_{L^2_t L^{\bar p'}_x}\right),\quad \frac{1}{\bar p'}=1-\frac{1}{\bar p}>0.\label{IPhi0}
\end{equation}

\medskip

\noindent {iii) \it The case $\Phi_x$, $\Phi_y$ in $L^2_t L_x^{\bar p'}
\cap W^{1,+\infty}_t W^{-1,\infty-0}$ with $\bar p'= \bar p/(\bar p-1)$.}
 Denote
\[
\tilde C_\Phi=\|\Phi_x\|_{L^2_t L^{\bar p'}_x}+\|\Phi_y\|_{L^2_t L^{\bar p'}_x}+ \|\Phi_x\|_{W^{1,\infty}_t W^{-1,\infty-0}}+ \|\Phi_y\|_{W^{1,\infty}_t W^{-1,\infty-0}},
\]
and 
\[
R_1=\Delta^{-1}\,\div \rho_k\,(u_k-u_{k,\eta}).
\]
Observe that by \eqref{ut-us} and integration by part in time
\[
\int_0^t\int_{\Pi^d} \Phi_x\,\partial_t R_1\,dx\,dt\leq \tilde C_\Phi\,\eta^\theta\,\|K_h\|_{L^1}.
\]
The same procedure can be performed with $\div (\rho_k\,u_k\otimes u_k)$. Denoting
\[
D\rho u_{k,\eta}=\Delta^{-1}\,\div\,\left(\partial_t(\rho_k\,u_{k,\eta})+\div (\rho_k\,u_k\otimes u_{k,\eta})\right),
\]
one then has
\[\begin{split}
I\,\Phi \leq &\tilde C_\Phi\,\|K_h\|_{L^1}\,\eta^\theta\\
&+\int_0^t\!\! \int_{\Pi^{2d}} K_h(x-y)\,\Phi(t,x,y)\,(D\rho u_{k,\eta}(t,y)-D\rho u_{k,\eta}(t,x))\,dx\,dy\,dt.
\end{split}
\]
However using \eqref{continuity}
\[
\partial_t(\rho_k\,u_{k,\eta})+\div (\rho_k\,u_k\otimes u_{k,\eta})=\rho_k( \partial_t u_{k,\eta}+u_k\cdot \nabla u_{k,\eta})+\alpha_k\,u_{k,\eta}\,\Delta \rho_k.
\]
For some exponent $\kappa$
\[
\left\|\Delta^{-1}\div\,\left(\rho_k( \partial_t u_{k,\eta}+u_k\cdot \nabla u_{k,\eta})\right)\right\|_{L^1_t W^{1,1}_x}\leq C\,\eta^{-\kappa},
\]
and
\[
\alpha_k \left\|\Delta^{-1}\div\,\left(\alpha_k\,u_{k,\eta}\,\Delta \rho_k\right)\right\|_{L^2_t\,L^2_x}\leq C\,\eta^{-\kappa}\,\sqrt{\alpha_k}.
\]
Therefore
\[\begin{split}
&\int_0^t\!\! \int_{\Pi^d} K_h(x-y)\,\Phi(t,x,y)\,(D\rho u_{k,\eta}(t,y)-D\rho u_{k,\eta}(t,x))\,dx\,dy\,dt\\
&\quad\leq \tilde C_\Phi\,\eta^{-\kappa}\,\|K_h\|_{L^1}\,(h+\sqrt{\alpha_k})^{1-0}.
\end{split}\]
Finally
\[
\begin{split}
& I\,\Phi\leq
\tilde C_\Phi\,\|K_h\|_{L^1}\,\left(\eta^\theta+\eta^{-\kappa}\,(h+\sqrt{\alpha_k})\right),\\ 
\end{split}
\]
and by optimizing in $\eta$, there exists $\theta>0$ s.t.
\begin{equation}
I\,\Phi\leq C\,\|K_h\|_{L^1}\,h^\theta\,\left(\|\Phi_x\|_{W^{1,\infty}_t W^{-1,\infty-0}}+ \|\Phi_y\|_{W^{1,\infty}_t W^{-1,\infty-0}}\right).\label{IPhi1}
\end{equation}

\medskip

\noindent {iv) \it Interpolation between the two inequalities \eqref{IPhi0} and \eqref{IPhi1}.}
 For any $s\in (0,1)$ there exists $\theta>0$ s.t.
\[\begin{split}
I\,\Phi\leq C\,\|K_h\|_{L^1}\,h^\theta\,\Big(&\|\Phi_x\|_{L^2_t L^{\bar p'}_x}+\|\Phi_y\|_{L^2_t L^{\bar p'}_x}\\
& + \|\Phi_x\|_{W^{s,q+0}_t W^{-s,r+0}_x}+ \|\Phi_y\|_{W^{1,q+0}_t W^{-1,r+0}_x}\Big),
\end{split}\]
with
\[
\frac{1}{q}=\frac{1-s}{2},\quad \frac{1}{r}=\frac{1-s}{\bar p'}.
\]
On the other hand, if for example $\Phi_x$ belongs to $L^\infty_{t,x}$ and to $W^{1,1}_t W^{-1,1}_x$ then by interpolation $\Phi_x$ is in $W^{s,1/s-0}_t\,W^{-s,1/s-0}_x$. Hence $C_\phi$ controls the $W^{s,q+0}_t W^{-s,r+0}_x$ norm provided
\[
s<1/q=\frac{1-s}{2},\quad s<1/r=\frac{1-s}{\bar p'}.
\]
This is always possible by taking $s$ small enough (but strictly positive) as $\bar p>1$ and hence $\bar p'<\infty$. This concludes the proof.

Note finally that the interpolations between Sobolev spaces are not exact which is the reason for the $1/s-0$ or $q+0$ and $r+0$ (one would have to use Besov spaces instead, see for instance \cite{LP}).
\end{proof}
%
\subsection{The coupling with the pressure law}
We handle all weights at the same time. For convenience, we denote
\[
\chi_1(t,x,y)=\frac{1}{2}\chi'(\delta\rho_k)\,\bar\rho_k+\chi(\rho_k)-\frac{1}{2}\chi'(\delta\rho_k)\,\delta\rho_k.
\]
In the case without viscosity, one has
\begin{lemma}
Assume that $\rho_k$ solves \eqref{continuity} with $\alpha_k=0$, that \eqref{rhout},  \eqref{bounduk}, \eqref{boundrhok} with $\gamma>d/2$ and $p>2$ hold. Assume moreover that $u$ solves \eqref{divu0} with $\mu_k$ compact in $L^1$ and satisfying \eqref{elliptic}, $F_k$ compact in $L^1$, $P_k$ satisfying \eqref{nonmonotone}. \\
i) Then there exists a continuous function $\eps(.)$ with $\eps(0)=0$ s.t.
\[\begin{split}
& -\int_0^t\int_{\Pi^{2d}} K_h(x-y)\,(\div u_k(x)-\div u_k(y))\,\chi_1\,w_1(x) \leq C\,\|K_h\|_{L^1}\,\eps(h)\\
&\quad+C\,\int_0^t\int_{\Pi^{4d}}K_h(x-y)\,\left(1+\rho_k^{\tilde \gamma}(x)+\tilde P_k(x)\,\rho_k(x)+R_k(x)\right)\,\chi(\delta\rho_k)\,w_{1}(x).
\end{split}\]
ii) There exist $\theta>0$ and a continuous function $\eps$ with $\eps(0)=0$, depending only on $p$ and the smoothness of $\mu_k$ and $F_k$, s.t.
\[\begin{split}
& -\int_0^t\int_{\Pi^{2d}} K_h(x-y)\,(\div u_k(x)-\div u_k(y))\,\chi'(\delta\rho_k)\,\bar\rho_k\,w_1(x)\,w_1(y)\\
&\quad \leq C\,\|K_h\|_{L^1}\,\left(\eps(h)+h^\theta\right)+C\,\int_0^t\int_{\Pi^{4d}}K_h(x-y)\, \Big(1+\rho_k^{\tilde \gamma}(x)+\tilde P_k(x)\,\rho_k(x)\\
&\hspace{3cm}+R_k(x)+\rho_k^{\tilde \gamma}(y)+\tilde P_k(y)\,\rho_k(y)+R_k(y)\Big)\,
\chi(\delta\rho_k)\,w_{1}(x)\,w_1(y).
\end{split}\]
For instance if $\mu_k$ and $F_k$ belong to $W^{s,1}$ for some $s>0$ then one may take $\eps(h)=h^\theta$ for some $\theta>0$.
\label{divnoviscosity}
\end{lemma}
While in the case with viscosity
\begin{lemma}
Assume that $\rho_k$ solves \eqref{continuity}, that \eqref{rhout}, \eqref{bounduk}, \eqref{boundrhok} with $\gamma>d/2$ and $p>2$ hold. Assume moreover that $u$ solves \eqref{divu0} with $\mu_k$ compact in $L^1$ and satisfying \eqref{elliptic}, $F_k$ compact in $L^1$, $P_k$ satisfying \eqref{monotone}. Then there exists a continuous function $\eps(.)$ with $\eps(0)=0$ and depending only on the smoothness of $\mu_k$ and $F_k$ s.t.
\[\begin{split}
& -\frac{1}{2} \int_0^t\int_{\Pi^{4d}} K_h(x-y)\,(\div u_k(x)-\div u_k(y))\,\chi'(\delta\rho_k)\,\bar\rho_k\,W_{0}\,\overline K_h(x-z)\,\overline K_h(y-w)\\   
& - \frac{1}{2} \int_0^t \int_{\Pi^{4d}} K_h(x-y)(\div u_k(x) + \div u_k(y))
             (\chi'(\delta\rho_k)\delta\rho_k - 2 \chi(\delta\rho_k))\\
&\qquad\qquad\qquad\qquad\qquad\qquad\qquad W_0(z,w)
             \overline K_h(x-z) \overline K_h(y-w)\\  
& - \frac{\lambda}{2}\, \int_0^t\int_{\Pi^{3d}} K_h(x-y) \chi(\delta\rho_k)\,  \overline K_h(x-z)\,M|\nabla u_k|\, w_0(z)   \\ 
&\quad \leq C\,\|K_h\|_{L^1}\,\eps(h)
+C\,\int_0^t\int_{\Pi^{4d}}K_h(x-y)\,\chi(\delta\rho_k)\,W_{0}\,\overline K_h(x-z)\,\overline K_h(y-w).     
\end{split}\]\label{divwithviscosity}
\end{lemma}

\begin{proof} The computations are very similar for i) and ii) in Lemma \ref{divnoviscosity} and for Lemma \ref{divwithviscosity}. For simplicity, in order to treat the proofs together as much as possible,
we denote
\[\begin{split}
&G_1(t,x,y)=\chi_1(t,x,y)\,w_1(x),\quad G_2(t,x,y)=\chi'(\delta\rho_k)\,\bar\rho_k\,w_1(t,x)\,w_1(t,y),\\
& G_0(t,x,y)=\frac{1}{2} \chi'_0(\delta\rho_k)\,\bar\rho_k\,\int_{\Pi^{2d}}W_0(t,z,w)\,\overline K_h(x-z)\,\overline K_h(y-w)\,dz\,dw.
\end{split}\]
The first step is to truncate: Denote $I_k^L(t,x,y)=\phi(\rho_k(t,x)/L)\,\phi(\rho_k(t,y)/ L)$ for some smooth and compactly supported $\phi$, 
\[\begin{split}
& -\int_0^t\int_{\Pi^{2d}} K_h(x-y)\,(\div u_k(x)-\div u_k(y))\,G_i\\
&\leq C\,\|K_h\|_{L^1}\,L^{-\theta_0}-\int_0^t\int_{\Pi^{2d}} K_h(x-y)\,(\div u_k(x)-\div u_k(y))\,G_i\,I_k^L.\\
\end{split}\]
Here for $i=0,\,1,\,2$, $G_i\leq C\,(\rho_k(t,x)+\rho_k(t,y))$ (even $G_2\leq 2$) and consequently, as $\div u_k\in L^2$ uniformly, only $p>2$ is required with $\theta_0=(p-2)/2>0$. 

Introduce an approximation $\mu_{k,\eta}$ of $\mu_k$, satisfying \eqref{elliptic} and s.t.
\begin{equation}\begin{split}
&\|\mu_{k,\eta}\|_{W^{2,\infty}_{t,x}}\leq C\,\eta^{-2}, \quad \|\mu_{k,\eta}-\mu_{k}\|_{L^1}\leq \eps_0(\eta),\\
& \int_0^T\int_{\Pi^{2d}} K_h(x-y)\,|\mu_{k,\eta}(t,x)-\mu_{k,\eta}(t,y)|\,dx\,dy\,dt\leq \,\|K_h\|_{L^1}\,\eps_0(h),
\end{split}\label{muketa}\end{equation}
from \eqref{compactFmu}.
  Use \eqref{divu0} to decompose 
\[\begin{split}
&-\int K_h(x-y)\,(\div u_k(x)-\div u_k(y))\,G_i\,I_k^L\,dx\,dy\\
&\qquad=2\,A_i+2\,B_i+2\,E_i,
\end{split}\]
with 
\[
A_i=-\int K_h(x-y)\,(P_k(\rho_k(x))-P_k( \rho_k(y)))\,G_i\,\frac{I_k^L}{\mu_{k,\eta}(x)}\,dx\,dy,
\]
and
\[
B_i=\int K_h(x-y)\,\tilde F_k(t,x,y)\,G_i\,\frac{I_k^L}{\mu_{k,\eta}(x)}\,dx\,dy,
\]
where
\[\begin{split}
\tilde F_k(t,x,y)=&F_k(t,x)-F_k(t,y)+\mu_k(y)\,\mu_{k,\eta}(x)\,\div u_k(t,y)\,\left(\frac{1}{\mu_{k,\eta}(x)}- \frac{1}{\mu_k(y)}\right)\\
&-\mu_k(x)\,\mu_{k,\eta}(x)\,\div u_k(t,x)\,\left(\frac{1}{\mu_{k,\eta}(x)}- \frac{1}{\mu_k(x)}\right).
\end{split}\]
Finally
\[
E_i=\int K_h(x-y)\,(D\rho u_k(t,y)-D\rho u_k(t,x))\,G_i\,\frac{I_k^L}{\mu_{k,\eta}(x)}\, dx\,dy,
\]
with as before
\[
D\rho u_k=\Delta^{-1}\,\div(\partial_t(\rho_k\,u_k)+\div(\rho_k\,u_k\otimes u_k)).
\]
For $B_i$ by the compactness of $F_k$, $\mu_k$, estimates \eqref{muketa} and \eqref{compactFmu}, and by \eqref{elliptic}
\begin{equation}
B_i\leq C\,L\,\int_0^t \int_{\Pi^{2d}} K_h(x-y)\,|\tilde F_k(t,x,y)|\leq C\,L\,(\eps_0(h)+\eps_0(\eta))\,\|K_h\|_{L^1}.
\label{boundBiallcase}
\end{equation}
Note that again $|G_i|\leq C\,(\rho_k(x)+\rho_k(y))$ for $i=0,\,1,\,2$.

For $E_i$, we use Lemma \ref{Drhou} by defining simply
\[
\Phi_i(t,x,y)=G_i\,I_k^L(t,x,y)\,\frac{1}{\mu_{k,\eta}(t,x)}.
\]
By \eqref{elliptic},
\[
\|\Phi_i\|_{L^\infty_{t,x}}\leq C\,L.
\]
As for the time derivative of $\Phi$, for $i=1,\,2$, $G_i$
is a combination of functions of $\rho_k(t,x)$, $\rho_k(t,y)$ and $w_i$ which all satisfy the same transport equation (with different right-hand sides). By \eqref{continuity},  
\[\begin{split}
&\partial_tG_i+\div_x\left(u_k(x)\,G_i\right)+\div_y\left(u_k(y)\,G_i\right)
=f_{1,i}\,\div_x u_k(x)+f_{2,i}\,\div_y u_k(y)\\
&\quad+f_{3,i}\,D_i(x)\,+f_{4,i}\,D_i(y),
\end{split}\]
where the $D_i$ are the penalizations introduced in Section \ref{weightsec} and the $f_{n,i}$ are again combinations of functions of $\rho_k(t,x)$, $\rho_k(t,y)$ and $w_i$. Every $f_{n,i}$ contains as a factor $\phi(\rho_k(x)/L)$ or a derivative of $\phi$ and thus
\[
\|f_{n,i}\|_{L^\infty}\leq C\,L,\qquad\forall n,\;i.
\]
Finally by the smoothness of $\mu_{k,\eta}$
\[
\begin{split}
&\partial_t\Phi_i+\div_x\left(u_k(x)\,\Phi_i\right)
+\div_y\left(u_k(y)\,\Phi_i\right)
=\tilde f_{1,i}\,\div_x u_k(x)+\tilde f_{2,i}\,\div_y u_k(y)\\
&\quad+\tilde f_{3,i}\,D_i(x)\,+\tilde f_{4,i}\,D_i(y)+\Phi_i\,g_\eta,
\end{split}
\]
and it is easy to check that the constant $C_{\Phi_i}$ as defined in Lemma \ref{Drhou} is bounded by $C\,L\,\eta^{-1}$. 

\medskip

     The case $i=0$ is slightly more complicated as $W_0$ is integrated against $\overline K_h$ so the equation on $\Phi_0$ involves non local terms and we have to take into account extra terms
 as mentioned in the statement of Lemma 8.4.
By \eqref{eqw}, denoting $w_{0,h}=\overline K_h\star w_0$
\[
\partial_t w_{0,h}+u_k(t,x)\cdot\nabla w_{0,h}-\alpha_k \Delta_x w_{0,h}=-\overline K_h\star (D_0\,w_0)+R_h-\overline K_h\star(\div u_k\,w_0),
\]
with
\[
R_h=\int_{\Pi^d} \nabla\overline K_h(x-z)\cdot(u_k(t,x)-u_k(t,z))\,w_0(t,z)\,dz.
\]
Remark that $R_h$ is uniformly bounded in $L^2_{t,x}$ by usual commutator estimates.

Finally as $\mu_{k,\eta}$ is smooth in time, one has
\[\begin{split}
&\partial_t\Phi_0+\div_x\left(u_k(x)\,\Phi_0\right)+ \div_y\left(u_k(y)\,\Phi_0\right)
-\alpha_k\,( \Delta_x+\Delta_y)\, \Phi_0\\
& =f_{1,0}\,\div_x u_k(t,x)+f_{2,0}\,\div_y u_k(t,y)+\alpha_k\,\left( f_{3,0}\,|\nabla_x\rho_k(x)|^2+f_{4,0}\,|\nabla_x\rho_k(y)|^2\right)\\
&\ -\frac{\Phi_\rho}{\mu_{k,\eta}}\,\left(\overline K_h\star (D_0\,w_0)+R_h-\overline K_h\star(\div u_k\,w_0)\right)-2\,\frac{\alpha_k}{\mu_{k,\eta}}\,\nabla_x\Phi_\rho\cdot\nabla_x w_{0,h}\\
&\ +\Phi_0\,g_\eta ,
\end{split}\]
where $\Phi_\rho=\delta\rho_k\,\bar\rho_k\,I_k^L(t,x,y)$,
$g_\eta$ is a function involving first and second derivatives of $\mu_{k,\eta}$ in $t$ and $x$ and $\nabla u_k$. The $f_{j,0}$ are combinations of functions of $\rho_k(t,x)$ and $\rho_k(t,y)$, multiplied by $w_{0,h}$, and involving $\phi(\rho_k(x)/L)$, $\phi'(\rho_k(x)/L)$, or $\phi''(\rho_k(x)/L)$ and the corresponding term with $\rho_k(y)$. 

By the $L^\infty$ bounds on $\Phi_\rho$, $w_0$, each $f_{j,0}$ and by $\eqref{bounduk}$, one obtains
\[
\left\|\partial_t\int \overline K_h(x-y)\, \Phi_0(t,x,y)\,dy\right\|_{L^1_t\,W^{-1,1}_x}\leq C\,L\,\eta^{-1}.
\]   
Therefore $C_\phi\leq C\,L\,\eta^{-1}$. Thus for all three cases, Lemma \ref{Drhou} yields 
\begin{equation}
E_i\leq C\,L\,\eta^{-1}\,\|K_h\|_{L^1}\,h^\theta,\label{boundEiallcase}
\end{equation}
for some $\theta>0$.

\bigskip

{\em Proof of Lemma \ref{divwithviscosity}: The term $A_0$.} The terms $A_i$ are where lies the main difference between Lemmas \ref{divnoviscosity} and \ref{divwithviscosity} as $P_k$ is not monotone in the first case and monotone after a certain threshold in the second.  For this reason we now proceed separately for  Lemma \ref{divwithviscosity} and Lemma \ref{divnoviscosity}. In the case with diffusion for Lemma \ref{divwithviscosity}, there also exist extra terms to handle, namely $J+I$ with
\[
J=-\frac{\lambda}{2} \,  \int_0^t\int_{\Pi^{3d}} K_h(x-y) \chi(\delta\rho_k) \overline K_h(x-z) \,M|\nabla u_k|(z)\, w_0(z),
\] 
and
\[\begin{split}
 I= 
& - \frac{1}{2}\int_0^t\int_{\Pi^{4d}}K_h(x-y)\,(\div u_k(x)+\div u_k(y))\\
&\qquad\qquad\qquad\qquad\qquad (\chi'(\delta\rho_k)\,\delta\rho_k-2\,\chi(\delta\rho_k))\,\,W_{0}\,\overline K_h(x-z)\,\overline K_h(y-w).
\end{split}\]
We decompose this last term in a manner similar to what we have just done, first of all by introducing the truncation of $\rho_k$
\[
\begin{split}
I\leq &C\,\|K_h\|_{L^1}\,L^{-\theta_0}\\
& - \frac{1}{2}\int_0^t\int_{\Pi^{4d}}K_h(x-y)\,(\div u_k(x)+\div u_k(y))\\
&\qquad\qquad\qquad\qquad\qquad (\chi'(\delta\rho_k)\,\delta\rho_k-2\,\chi(\delta\rho_k))\,I_k^L\,W_{0}\,\overline K_h(x-z)\,\overline K_h(y-w),
\end{split}
\]
with again $\theta_0=(p-2)/2$. Now introduce the $\mu_k$
 \[\begin{split}
& I\leq  C\,\|K_h\|_{L^1}\,L^\theta
 - \frac{1}{2}\int_0^t\int_{\Pi^{4d}}K_h(x-y)\,(\mu_k(t,x)\div u_k(x)+\mu_k(t,y)\div u_k(y))
    \\
&\qquad\qquad\qquad\qquad\qquad   \frac{I_k^L}{\mu_{k,\eta}(t,x)}\, (\chi'(\delta\rho_k)\,\delta\rho_k-2\,\chi(\delta\rho_k))\,\,W_{0}\,\overline K_h(x-z)\,\overline K_h(y-w) \\
&+\frac{1}{2} \int_0^t\int_{\Pi^{4d}}K_h(x-y) H_k(x,y) (\chi'\,\delta\rho_k-2\,\chi(\delta\rho_k))\,I_k^L\,W_{0}\,\overline K_h(x-z)\,\overline K_h(y-w) \\
\end{split}\]
where
\[
H_k(t,x,y) 
 = \mu_k(x) \div u_k(x) \bigl(\frac{1}{\mu_{k,\eta}(x)} - \frac{1}{\mu_{k}(x)} \bigr)
   - \mu_k(y) \div u_k(y) \bigl(\frac{1}{\mu_{k}(y)} - \frac{1}{\mu_{k,\eta}(x)} \bigr).
\]
By the compactness of $\mu_k$, one has that
\[\begin{split}
&\int_0^t\int_{\Pi^{4d}}K_h(x-y) H_k(x,y) (\chi'\,\delta\rho_k-2\,\chi(\delta\rho_k))\,\,W_{0}\,\overline K_h(x-z)\,\overline K_h(y-w)\\
&\qquad \leq \|K_h\|_{L^1}\,\eps_0(h)\,\|u_k\|_{L^2_t\,H^1_x}\,\|\rho_k\|_{L^{2}_{t,x}}\leq C\,\eps_0(h)\,\|K_h\|_{L^1}.
\end{split}
\]
This implies that
 \[\begin{split}
& I\leq  
 - \frac{1}{2}\int_0^t\int_{\Pi^{4d}}K_h(x-y)\,(\mu_k(t,x)\div u_k(x)+\mu_k(t,y)\div u_k(y))
    \\
&\qquad\qquad\qquad\qquad   \frac{I_k^L}{\mu_{k,\eta}(t,x)}\, (\chi'(\delta\rho_k)\,\delta\rho_k-2\,\chi(\delta\rho_k))\,\,W_{0}\,\overline K_h(x-z)\,\overline K_h(y-w)\\
&\qquad+C\,\|K_h\|_{L^1}(\,L^{-\theta}+\eps_0(h)). \\
\end{split}
\]
Using \eqref{divu0} or namely that 
$\mu_k \div u_k = D\rho\,u_k +F_k+ P_k(\rho_k)$, the quantity $A_0+I+J$ may be written
\begin{equation}
A_0+I+J\leq  C\,\|K_h\|_{L^1}(\,L^{-\theta}+\eps_0(h))+I_1+I_2,\label{boundFdiff}
\end{equation}
with
\[
\begin{split}
&I_1=A_0- \frac{1}{2}\int_0^t\int_{\Pi^{4d}}K_h(x-y)\,(
      P_k(x,\rho_k(x)) +  P_k(y,\rho_k(y)))\frac{I_k^L}{\mu_{k,\eta}(t,x)}\\
&\qquad\qquad\qquad\qquad\qquad (\chi'(\delta\rho_k)\,\delta\rho_k-2\,\chi(\delta\rho_k))\,\,W_{0}\,\overline K_h(x-z)\,\overline K_h(y-w), \\
\end{split}\]
and
 \[\begin{split}
&I_2= - \frac{1}{2}\int_0^t\int_{\Pi^{4d}}K_h(x-y)\,
   (D\rho\,u_k(x)+F_k(x) + D\rho\,u_k(y)+F_k(y))\frac{I_k^L}{\mu_{k,\eta}(t,x)}\\
&\qquad\qquad\qquad\qquad\qquad (\chi'(\delta\rho_k)\,\delta\rho_k-2\,\chi(\delta\rho_k))\,\,W_{0}\,\overline K_h(x-z)\,\overline K_h(y-w) \\
&-  \frac{\lambda}{2} \,   \int_0^t\int_{\Pi^{3d}} K_h(x-y) \chi(\delta\rho_k) \overline K_h(x-z) \, M|\nabla u_k|(z)\, w_0(z).
\end{split}
\]
In this case with diffusion, because $P_k$ is essentially monotone, the term $A_0$ is mostly dissipative and helps control the rest. More precisely
 \[\begin{split}
 I_1= -\frac{1}{2} \int_0^t \int_{\Pi^{2d}} \Bigl[K_h(x-y) \,
 & \bigl[ (P_k(x,\rho_k(x))-P_k(y,\rho_k(y)))\chi'(\delta\rho_k) \overline\rho_k \\
+ &(P_k(x,\rho_k(x))+P_k(y,\rho_k(y)))(\chi'(\delta\rho_k)\delta\rho_k - 2 \chi(\delta\rho_k))
    \bigr]  \\
& \frac{I_k^L}{\mu_{k,\eta}(t,x)}  \int_{\Pi^{2d}}\,\,W_{0}\,\overline K_h(x-z)\,\overline K_h(y-w) dzdw
  \Bigr]dx dy.
  \end{split}\]
As $P_k\geq 0$ and by \eqref{defchi}, $\chi'(\delta\rho_k)\delta\rho_k - 2 \chi(\delta\rho_k)\geq -\chi'(\delta\rho_k)\delta\rho_k$, thus
\[\begin{split}
&(P_k(x,\rho_k(x))-P_k(y,\rho_k(y)))\chi' \bar\rho_k + (P_k(x,\rho_k(x))+P_k(y,\rho_k(y)))(\chi'\delta\rho_k - 2 \chi(\delta\rho_k))\\
&\quad\geq \chi'(\delta\rho_k)\,\left[(P_k(x,\rho_k(x))-P_k(y,\rho_k(y)))\,\bar\rho_k- (P_k(x,\rho_k(x))+P_k(y,\rho_k(y)))\,\delta\rho_k\right].
\end{split}\]
Without loss of generality, we may assume that $\rho_k(x)\geq \rho_k(y)$ and hence $\chi'(\delta\rho_k)\geq 0$. Developing
\[\begin{split}
&(P_k(x,\rho_k(x))-P_k(y,\rho_k(y)))\,\bar\rho_k- (P_k(x,\rho_k(x))+P_k(y,\rho_k(y)))\,\delta\rho_k\\
&\qquad=2\,P_k(x,\rho_k(x))\,\rho_k(y) -2\,P_k(y,\rho_k(y))\,\rho_k(x).
\end{split}\]
We now use the quasi-monotonicity \eqref{monotone} of $P_k(x,s)/s$. First of all if $\rho_0\leq \rho_k(y)\leq \rho_k(x)$  then necessarily $P_k$ depends only on $\rho_k(x)$ or $\rho_k(y)$ plus $\tilde P_k$. Thus 
\begin{equation}\label{estim1}
P_k(x,\rho_k(x))\,\rho_k(y) -\,P_k(y,\rho_k(y))\,\rho_k(x)\geq -|\tilde P_k(t,x)-\tilde P_k(t,y)|.
\end{equation}

If $\rho_k(y)\leq \rho_0$, by \eqref{monotone} $P_k(x,s)\rightarrow \infty$ as $s\rightarrow\infty$ while $P_k(y,\rho_k(y))$ is bounded. Hence there exists
 $\bar \rho$ large enough with respect to $\rho_0$, s.t. if $\rho_k(x)\geq \bar\rho$, then again 
 $$P_k(x,\rho_k(x))\,\rho_k(y) -\,P_k(y,\rho_k(y))\,\rho_k(x)\geq 0.$$

The only case where one does not have the right sign is hence where both $\rho_k(x)$ and $\rho_k(y)$ are bounded by $\bar\rho$ and $\rho_0$. Therefore using the local regularity of $P_k$ given by \eqref{monotone}
\begin{eqnarray} 
&
\nonumber 
(P_k(x,\rho_k(x))-P_k(y,\rho_k(y)))\,\bar\rho_k- (P_k(x,\rho_k(x))+P_k(y,\rho_k(y)))\,\delta\rho_k\\
&\hspace{4cm} \label{estim2}\geq -\bar P\,|\delta\rho_k|-Q_k(t,x,y).
\end{eqnarray}
Introducing these estimates (\ref{estim1}) and (\ref{estim2}) in $I_1$ yields 
\begin{equation}\begin{split}
I_1&\le \bar P \int_0^t\!\!\int_{\Pi^{2d}} K_h(x-y)\, (|\tilde P_k(x)-\tilde P_k(y)|+Q_k+|\delta\rho_k|)\, 
  \frac{|\chi'(\delta\rho_k)|\,I_k^L}{\mu_{k,\eta}(t,x)}  \\
&\hspace{2cm}\int_{\Pi^{2d}}\!\!\,W_{0}\,\overline K_h(x\!-z)\,\overline K_h(y\!-w) \\
&\leq \eps_0(h)\,\|K_h\|_{L^1}+\bar P\,\int_0^t\int_{\Pi^{4d}} K_h(x-y)
\chi(\delta\rho_k)\,W_{0}\,\overline K_h(x-z)\,\overline K_h(y-w).
\end{split}\label{boundF1diff}
\end{equation}
Turning now to $I_2$, we observe that $\mu_k\,M|\nabla u_k| \ge \mu_k\,\div u_k \ge D\rho\,u_k+F_k$ and that $\chi(\delta_k)\geq (2\chi(\delta_k)-\chi'(\delta_k)\,\delta_k)/C$. Therefore for $\lambda$ large enough, using $W_0(z,w)=w_0(z)+w_0(w)$ and the symmetry
\[\begin{split}
&I_2\leq - \frac{1}{2}\int_0^t\int_{\Pi^{4d}}K_h(x-y)\,
   \left(D\rho\,u_k(x)-D\rho\,u_k(z)+F_k(x)-F_k(z)\right.\\
&\quad\left. + D\rho\,u_k(y)-D\rho\,u_k(z)+F_k(y)-F_k(z)\right)\frac{I_k^L}{\mu_{k,\eta}(t,x)} (\chi'(\delta\rho_k)\,\delta\rho_k-2\,\chi(\delta\rho_k))\\
&\qquad\qquad\qquad\qquad\qquad\qquad\qquad\qquad\qquad\qquad
\,W_{0}\,\overline K_h(x-z)\,\overline K_h(y-w). \\
\end{split}
\]
The differences in the $F_k$ are controlled by the compactness of $F_k$ and the differences in the $D\rho\,u_k$ by Lemma \ref{Drhou} as for the terms $E_i$.
Hence, finally
\begin{equation}
I_2\leq C\,\|K_h\|_{L^1}\,(\eps_0(h)+L\,\eta^{-1}\,h^\theta).
\label{boundF2diff}
\end{equation}

\medskip

{\em Conclusion for Lemma \ref{divwithviscosity}.} We sum up the contributions from $B_0$ in \eqref{boundBiallcase}, $E_0$ in \eqref{boundEiallcase}, $A_0+I+J$ in \eqref{boundFdiff} together with the bounds on $I_1$ in \eqref{boundF1diff} and $I_2$ in \eqref{boundF2diff} to obtain
\[\begin{split}
& -\frac{1}{2}\int_0^t\int_{\Pi^{2d}} K_h(x-y)\,(\div u_k(x)-\div u_k(y))\,G_0\\
& -\frac{1}{2}\int_0^t\int_{\Pi^{4d}} K_h(x-y)\,(\div u_k(x)+\div u_k(y))\,(\chi'(\delta\rho_k)\,\delta\rho_k-2\chi(\delta\rho_k))\\
&\qquad\qquad\qquad\qquad\qquad W_0(z,w)\,\overline K_h(x-z)\,\overline K_h(y-w)\\
&\quad-\frac{\lambda}{2}\int_0^t\int_{\Pi^{3d}} K_h(x-y)\,\chi(\delta\rho_k)\,\overline K_h(x-z)\,M\,|\nabla u_k|(z)\,w_0(z)\\
&\qquad \leq C\,\|K_h\|_{L^1}\left(L^{-\theta_0}+L\,(\eps_0(h)+\eps_0(\eta))+ L\,\eta^{-1}\,h^\theta\right)\\
&\qquad+C\,\int_0^t\int_{\Pi^{4d}}K_h(x-y)\,
\chi(\delta\rho_k)\,W_0(z,w)\, \overline K_h(x-z)\,\overline K_h(y-w).
\end{split}\]
Just optimizing in $L$ and $\eta$ leads to the desired $\eps(h)$ and concludes the proof of Lemma \ref{divwithviscosity}.

\bigskip

{\em Proof of Lemma \ref{divnoviscosity}: The term $A_i$ with $i=1,\;2$}. 
It now remains to analyze more precisely the terms $(P_k(x,\rho_k(x))-P_k(y,\rho_k(y)))\,G_i$
 for i=1,2 concerning the case without diffusion but with non monotone pressure.  We will split the study in three cases but remark that now the possible dependence of $P_k$ in terms of $x$ affects the estimates. For this reason, we carefully write explicitly this dependence.

\medskip

\noindent {\it Case \rm 1).}  The case $(P_k(x,\rho_k(x))-P_k(y,\rho_k(y)))\delta\rho_k\ge 0$.
Since $G_2$ obviously have the same sign as $\delta\rho_k$, one simply has
\[\begin{split}
(P_k(x,\rho_k(x))-P_k(y,\rho_k(y)))\,G_i&
\geq 0,
\end{split}\]
for $i=2$. In the same case, for $G_1$
\[\begin{split}
&(P_k(x,\rho_k(x))-P_k(y,\rho_k(y)))\,G_1\\
&=(P_k(x,\rho_k(x))-P_k(y,\rho_k(y)))\,\left(\frac{1}{2}\chi'(\delta\rho_k)\,\bar\rho_k+\chi(\delta \rho_k)-\frac{1}{2}\chi'(\delta\rho_k)\,\delta\rho_k\right)\,w_1(x)\\
&\geq |P_k(x,\rho_k(x))-P_k(y,\rho_k(y))|\,\left(\frac{1}{2}|\chi'(\delta\rho_k)|\,\bar\rho_k-\left|\chi(\delta\rho_k) -\frac{1}{2}\chi'(\delta\rho_k)\,\delta\rho_k\right|\right)\,w_1(x)\\
&\geq 0,
\end{split}\]
by \eqref{defchi} as $\bar\rho_k\geq |\delta\rho_k|$. Therefore in that case the terms have the right sign and can be dropped.

\bigskip

\noindent {\it Case \rm 2).}  The case $(P_k(x,\rho_k(x))-P_k(y,\rho_k(y)))\delta\rho_k< 0$ and  $\rho_k(y) \le \rho_k(x)/2$ or 
$\rho_k(y)\ge 2\,\rho_k(x)$ for some constant $C$. 

\smallskip

%
%
\noindent -- 
For $i=1$, first assume that $P_k(x,\rho_k(x))\geq P_k(y,\rho_k(y))$ while $\rho_k(y)\geq 2\,\rho_k(x)$.
\[
(P_k(x, \rho_k(x))-P_k(y,\rho_k(y)))\,G_1\geq -\,P_k(\rho_k(x)) \,(|\chi'(\delta\rho_k)|\,\bar \rho_k+\chi(\delta\rho_k))\,w_1(x).
\]
Now observe that since $\rho_k(y)\geq 2\,\rho_k(x)$ then
\[
|\chi'(\delta\rho_k)|\,\bar \rho_k\leq \frac{3}{2}\,|\chi'(\delta\rho_k)|\,\rho_k(y)\leq 3\,|\chi'(\delta\rho_k)|\,|\delta\,\rho_k|\leq C\,\chi(\delta\rho_k),
\]
by \eqref{defchi}. Therefore in that case by \eqref{nonmonotone}
\[
(P_k(x, \rho_k(x))-P_k(y,\rho_k(y)))\,G_1\geq -C\,(\rho_k(x)^{\tilde\gamma}+R_k(x))\, \,\chi(\delta\rho_k)\,w_1(x),
\]
%
%
Note that the result is not symmetric in $x$ and $y$: We also have to check also 
$P_k(x, \rho_k(x))\leq P_k(y, \rho_k(y))$ and $\rho_k(y)\leq \rho_k(x)/2$.
Then simply bound since now $\rho_k(y)\leq \rho_k(x)$
\[\begin{split}
(P_k(x, \rho_k(x))-P_k(y, \rho_k(y)))\,G_1&\geq -C\,((\rho_k(y))^{\tilde\gamma}+R_k(y))\, \,\chi(\delta\rho_k)\,w_1(x)\\
&\geq -C\,((\rho_k(x))^{\tilde\gamma}+R_k(y))\, \,\chi(\delta\rho_k)\,w_1(x).
\end{split}\]
In both cases, one finally obtains
\[\begin{split}
(P_k(x, \rho_k(x))-P_k(y, \rho_k(y)))\,G_1\geq &-C\,((\rho_k(x))^{\tilde\gamma}+R_k(x))\, \,\chi(\delta\rho_k)\,w_1(x)\\
&-|R(x)-R(y)|\,\bar\rho_k\,w_1(x).
\end{split}\]

\medskip

\noindent -- For $i=2$. The calculations are similar (simpler in fact) for $G_2$ and this lets us deduce that if $P_k(x,\rho_k(x))- P_k(y,\rho_k(y))$ and $\rho_k(x)-\rho_k(y)$ have different signs but $\rho_k(y)\leq \rho_k(x)/2$ or $\rho_k(y)\geq 2\,\rho_k(x)$ then 
\[\begin{split}
(P_k(x,\rho_k(x))-P_k(y,\rho_k(y)))\,G_2\geq &-C\,((\rho_k(x))^{\tilde\gamma}+\tilde P_k(x)\,\rho_k(x)\\
&+(\rho_k(y))^{\tilde\gamma}+\tilde P_k(y)\,\rho_k(y))\, \,\chi(\delta\rho_k)\,w_1(x)\,w_1(y).
\end{split}\]

\bigskip

\noindent {\it Case \rm 3).}  For $i=1,\, 2$, the situation  where $P_k(x,\rho_k(x))- P_k(y,\rho_k(y))$ and $\rho_k(x)-\rho_k(y)$ have different signs but  $\rho_k(x)/2\leq \rho_k(y)\leq 2\,\rho_k(x)$. 
Then one bluntly estimates using the Lipschitz bound on $P_k$ given by \eqref{nonmonotone}
\[\begin{split}
\left|(P_k(x,\rho_k(x))-P_k(y,\rho_k(y)))\right|\leq C\,(&(\rho_k(x))^{\tilde\gamma-1}+\tilde P_k(x)\\
&+(\rho_k(y))^{\tilde\gamma-1}+\tilde P_k(y))\,|\delta\rho_k|+Q_k. 
\end{split}\]
Bounding now the $G_i$ by \eqref{defchi},
\[\begin{split}
(P_k(x, \rho_k(x))-P_k(y, \rho_k(y)))\,G_2\leq &C\,((\rho_k(x))^{\tilde\gamma}+\tilde P_k(x)\,\rho_k(x)
+(\rho_k(y))^{\tilde\gamma}\\
&+\tilde P_k(y)\,\rho_k(y))\,\chi(\delta\rho_k)\,w_1(x)\,w_1(y)+Q_k\,\bar\rho_k\,w_1(x),
\end{split}
\]
and
\[\begin{split}
(P_k(x, \rho_k(x))-P_k(y, \rho_k(y)))\,G_1&\leq C\,((\rho_k(x))^{\tilde\gamma}+\tilde P_k(x)\,\rho_k(x)+(\rho_k(y))^{\tilde\gamma}\\
&\quad+\tilde P_k(y)\,\rho_k(y)) \,\chi(\delta\rho_k)\,w_1(x)+Q_k\,\bar\rho_k\\
&\leq C\,(1+\rho_k^{\tilde \gamma}(x)+\tilde P_k(x)\,\rho_k(x))\,\chi(\delta\rho_k)\,w_1(x)\\
&\qquad+(Q_k+|P_k(x)-P_k(y)|\,\bar\rho_k)\,\bar\rho_k\,w_1(x),
\end{split}\]
as $\rho_k(x)$ and $\rho_k(y)$ are of the same order. 
From Prop. \ref{weightprop}, point $i$, we know that $w_1(x)\leq e^{-\lambda\,\rho_k(x)^{p-1}}$. One the other hand, we are precisely in the case where $\rho_k(x)$ and $\rho_k(y)$ are of the same order. Hence $\bar\rho_k^l\,w_1$ is uniformly bounded for any $l>0$.
Hence we finally obtain in this case
\[\begin{split}
(P_k(x, \rho_k(x))-P_k(y, \rho_k(y)))\,G_2\leq C\,(&(\rho_k(x))^{\tilde\gamma}+\tilde P_k(x)\,\rho_k(x)
+(\rho_k(y))^{\tilde\gamma}\\
&+\tilde P_k(y)\,\rho_k(y))\,\chi(\delta\rho_k)\,w_1(x)\,w_1(y)+Q_k,
\end{split}
\]
and
\[\begin{split}
(P_k(x, \rho_k(x))-P_k(y, \rho_k(y)))\,G_1&
\leq C\,(1+\rho_k^{\tilde \gamma}(x)+\tilde P_k(x)\,\rho_k(x))\,\chi(\delta\rho_k)\,w_1(x)\\
&\qquad+Q_k+|P_k(x)-P_k(y)|.
\end{split}\]

\medskip

From the analysis of these three cases, one has that
\[\begin{split}
A_1\leq &C\, \int K_h(x-y)\,(1+\rho_k^{\tilde \gamma}(x)+\tilde P_k(x)\,\rho_k(x)+R_k(x))\,\chi(\delta\rho_k)\,w_1(x)\,dx\,dy\,dt\\
&+C\, \int K_h(x-y)\,(Q_k+|\tilde P_k(x)-\tilde P_k(y)|+|R_k(x)- R_k(y)|)\,dx\,dy\,dt.
\end{split}
\]
Therefore by the compactness properties of $P_k$ and the estimate on $Q_k$ in the assumption \eqref{monotone}
\begin{equation}\begin{split}
A_1\leq &C\, \int K_h(x-y)\,(1+\rho_k^{\tilde \gamma}(x)+\tilde P_k(x)\,\rho_k(x)+R_k(x))\,\chi(\delta\rho_k)\,w_1(x)\,dx\,dy\,dt\\
&+\|K_h\|_{L^1}\,\eps_0(h),
\end{split}\label{estimateA1}
\end{equation}
and
\begin{equation}\begin{split}
A_2\leq &C\, \int K_h(x-y)\,(1+\rho_k^{\tilde \gamma}(x)+\tilde P_k(x)\,\rho_k(x)+R_k(x)+\rho_k^{\tilde \gamma}(y)\\
&+\tilde P_k(y)\,\rho_k(y)+R_k(y))
\,\chi(\delta\rho_k)\,w_1(x)\, w_1(y)\,dx\,dy\,dt+\|K_h\|_{L^1}\,\eps_0(h).\label{estimateA2}
\end{split}\end{equation}

\bigskip

\noindent {\it Conclusion of the proof of Lemma \ref{divnoviscosity}.} Summing up every term, namely \eqref{boundBiallcase}-\eqref{boundEiallcase} and \eqref{estimateA1}-\eqref{estimateA2}, we eventually find
that
\[\begin{split}
& -\int_0^t\int_{\Pi^{2d}} K_h(x-y)\,(\div u_k(x)-\div u_k(y))\,G_1\\
&\quad \leq C\,\|K_h\|_{L^1}\left(L^{-\theta}+L\,(\eps_0(h)+\eps_0(\eta))+ L\,\eta^{-1}\,h^\theta\right)\\
&\qquad+C\,\int_0^t\int_{\Pi^{2d}}K_h(x-y)\,(1+\rho_k(x)^{\tilde\gamma}+\tilde P_k(x)\,\rho_k(x)+R_k)\,
\chi(\delta\rho_k)\,w_1(x),
\end{split}\]
while
\[\begin{split}
& -\int_0^t\int_{\Pi^{2d}} K_h(x-y)\,(\div u_k(x)-\div u_k(y))\,G_2\\
&\quad \leq C\,\|K_h\|_{L^1}\left(L^{-\theta}+L\,(\eps_0(h)+\eps_0(\eta))+ L\,\eta^{-1}\,h^\theta\right)\\
&\qquad+C\,\int_0^t\int_{\Pi^{2d}}K_h(x-y)\,\Big(1 +\rho_k(x)^{\tilde\gamma} +\tilde P_k(x)\,\rho_k(x)+R_k(x)+\rho_k(x)^{\tilde\gamma}\\
&\hspace{4cm}+\tilde P_k(y)\,\rho_k(y)+R_k(y)\Big)\,
\chi(\delta\rho_k)\,w_1(x)\,w_1(y).
\end{split}\]
To conclude the proof of Lemma \ref{divnoviscosity}, one optimizes in $\eta$ and $L$. Just remark that since the inequalities depend polynomially in $L$ and $\eta$ then the result depends on $\eps_0^\theta$ for some $\theta$.
\end{proof}

\subsection{Conclusion of the proofs of Theorems \ref{maincompactness} and \ref{maincompactness2}.}
We combine Lemma \ref{boundQh} with Lemma \ref{divwithviscosity} or \ref{divnoviscosity} and we finally use Prop. \ref{weightprop}. Let us summarize the required assumptions.
   In all cases one assumes that $\rho_k$ solves \eqref{continuity} and that $\div u_k$ is coupled with $\rho_k$ through \eqref{divu0}; bounds are assumed on the viscosity as per \eqref{elliptic}, on the time derivative of $\rho_k\,u_k$ per \eqref{rhout}, on $u_k$ per \eqref{bounduk}. Finally the viscosity $\mu_k$ and the force term $F_k$ are assumed to be compact in $L^1$.
 Moreover
\begin{itemize}
\item In the case with diffusion, $\alpha_k>0$, one assumes that the pressure term $P_k$ satisfies \eqref{monotone} and the bounds \eqref{boundrhok} on $\rho_k$ with $\gamma>d/2$ and $p>2$.
\item In the case without diffusion, $\alpha_k=0$, one needs only \eqref{nonmonotone} on the pressure $P_k$ and the bounds \eqref{boundrhok} on $\rho_k$ with $\gamma>d/2$ and $p>2$. Moreover for Prop. \ref{weightprop}, it is necessary that $p\geq \tilde \gamma$ (in general $\tilde\gamma=\gamma<p$ so this is not a big issue).
\end{itemize}

\bigskip

Then one obtains by taking $\lambda$ large enough and using a simple Gronwall lemma
\begin{equation}
\begin{split}
&\int_{h_0}^1\int_{\Pi^{4d}} \overline K_h(x-z)\,\overline K_h(y-w)\,(w_0(t,z)+w_0(t,w))\,K_h(x-y)\,\chi(\delta\rho_k)\,dx\,dy\,dz\,dw \frac{dh}{h}\\
& = \int_{\Pi^{4d}} \overline K_h(x-z)\,\overline K_h(y-w)\,(w_0(t,z)+w_0(t,w))\,{\cal K}_{h_0}(x-y)\,\chi(\delta\rho_k)\,dx\,dy\,dz\,dw  \\
&\qquad\leq C\, \Bigl(|\log h_0|^{1/2} + \ep_{h_0}(k)+\int_{h_0}^1 \eps(h)\frac{dh}{h}\Bigr),\\
\end{split}\label{fw0}\end{equation}
and
\begin{equation}
\begin{split}
&\int_{h_0}^1\int_{\Pi^{2d}} (w_1(t,x)+w_1(t,y))\,\overline K_h(x-y)\,\chi(\delta\rho_k)\,dx\,dy\frac{dh}{h}\\
&=\int_{\Pi^{2d}} (w_1(t,x)+w_1(t,y))\,{\cal K}_{h_0}(x-y)\,\chi(\delta\rho_k)\,dx\,dy\\
&\qquad\leq C\,\left(|\log h_0|^{1/2}+\int_{h_0}^1 \eps(h)\,\frac{dh}{h}\right),\\
\end{split}\label{fw1}\end{equation}
with finally
\begin{equation}
\begin{split}
&\int_{\Pi^{2d}} w_1(t,x)\,w_1(t,y)\,K_h(x-y)\,\chi(\delta\rho_k)\,dx\,dy\\
&\qquad\leq C\,\|K_h\|_{L^1}\,\left(h^\theta+\eps(h)\right),\\
\end{split}\label{fw2}\end{equation}
where $\eps$ depends only on the smoothness of $\mu_k$ and $F_k$ and of $p>2$.

\bigskip

The key point in all three cases is to be able to remove the weights from those estimates. For that, one uses point ii) of Prop. \ref{weightprop}. 

\medskip

\noindent {\it The case with $w_0(x) + w_0(y)$.} Denote $\omega_\eta=\{\overline K_h\star w_0(t,x)\leq \eta\}\subset [0,\ T]\times \Pi^d$ and remark that 
\[\begin{split}
\int_{\Pi^{2d}} {\cal K}_{h_0}(x-y)\,\chi(\delta\rho_k)&=\int_{h_0}^1\int_{\Pi^{2d}} {\overline K}_h(x-y)\,\chi(\delta\rho_k)\,\frac{dh}{h}\\
&=\int_{h_0}^1\int_{x\in\omega_\eta^c\ or\ y\in\omega_\eta^c} {\overline K}_h(x-y)\,\chi(\delta\rho_k)\,\frac{dh}{h}\\
&+\int_{h_0}^1\int_{x\in\omega_\eta\ and\ y\in\omega_\eta} {\overline K}_h(x-y)\,\chi(\delta\rho_k)\,\frac{dh}{h}.\\
\end{split}\]
Now
\[\begin{split}
&\int_{h_0}^1\int_{x\in\omega_\eta^c\ or\ y\in\omega_\eta^c}  \overline K_h(x-y)\,\chi(\delta\rho_k)\,\frac{dh}{h}\\
&\qquad\leq \frac{1}{\eta}\, \int_{h_0}^1 \int_{\Pi^{2d}}  \overline K_h(x-y) (\overline K_h\star w_{0}(x)+\overline K_h\star w_{0}(y))\, \chi(\delta\rho_k)
\frac{dh}{h},
\end{split}
\]
while by point iii) in Prop. \ref{weightprop}, using that $\rho\in L^p((0,T)\times \Pi^d)$ with $p>2$ and recalling that $\chi(\xi)\leq C\,|\xi|$
\[
\begin{split}
\int_{h_0}^1\int_{x\in\omega_\eta\ and\ y\in\omega_\eta} {\overline K}_h(x-y)\,\chi(\delta\rho_k)\,\frac{dh}{h}&\leq 2\,\int_{h_0}^1\int_{\Pi^{2d}} {\overline K}_h(x-y)\,\rho_k \,\ind_{\overline K_h\star w_0\le \eta}\,\frac{dh}{h}\\
&\leq \frac{C\,|\log h_0|}{|\log \eta|^{1/2}}.
\end{split}
\]
Therefore combining this with \eqref{fw0}, one obtains
\[\begin{split}
\int_{\Pi^{2d}} {\cal K}_{h_0}(x-y)\,\chi(\delta\rho_k)\,dx\,dy&\leq C\,\left(\frac{\eps_{h_0}(k)+
     |\log h_0|^{1/2} + \int_{h_0}^1 \eps(h)\frac{dh}{h} }{\eta}+\frac{\|{\cal K}_{h_0}\|_{L^1}}{|\log \eta|^{1/2}}\right)
\end{split}
\]
recalling that 
\[
\eps_{h_0}(k) = \alpha_k \int^1_{h_0} \frac{dh}{h} h_0^{-2},
\]
and denoting 
\[
\bar \eps(h_0)=\frac{1}{|\log h_0|}\int_{h_0}^1 \eps(h)\,\frac{dh}{h}
\]
Remark that $\bar\eps(h_0)\rightarrow 0$ since $\eps(h)\rightarrow 0$:
 For instance if $\eps(h)=h^\theta$ then $\bar \eps(h_0)\sim |\log h_0|^{-1}$.
  The estimate then reads
  \[\begin{split}
\int_{\Pi^{2d}} {\cal K}_{h_0}(x-y)\,\chi(\delta\rho_k)\,dx\,dy\leq
   C\,&\left(|\log h_0|\frac{\alpha_k h_0^{-2} +
     |\log h_0|^{-1/2} + \bar \eps(h_0) }{\eta}\right.\\
&\left.+\frac{\|{\cal K}_{h_0}\|_{L^1}}{|\log \eta|^{1/2}}\right).\\
\end{split}
\] 
As  $\|{\cal K}_{h_0}\|_{L^1}\sim |\log h_0|$,  by optimizing in $\eta$, the following 
estimate is obtained
  \[\begin{split}
\int_{\Pi^{2d}} {\cal K}_{h_0}(x-y)\,\chi(\delta\rho_k)\,dx\,dy&\leq
 \frac{C\|{\cal K}_{h_0}\|_{L^1}}{|\log (\alpha_k h_0^{-2} + |\log h_0|^{-1/2}+\bar \eps(h_0))|^{1/2}}.
\end{split}
\] 
Per Prop. \ref{propcomp}, this gives the compactness of $\delta\rho_k$ as the r.h.s. is negligible against $\|{\cal K}_{h_0}\|_{L^1}$ as $h_0\rightarrow 0$. And it proves the case i) of Theorem \ref{maincompactness}.

\bigskip

\noindent {\it The case with $w_1(x)+w_1(y)$.}
Similarly, from \eqref{fw1}, one then proves that in the corresponding case 
\[\begin{split}
\int_{\Pi^{2d}} {\cal K}_{h_0}(x-y)\,\chi(\delta\rho_k)\,dx\,dy&\leq C\,|\log h_0|\,\left(\frac{|\log h_0|^{-1/2}+\bar\eps(h_0)}{\eta}+\frac{1}{|\log \eta|^\theta}\right)\\
&\leq \frac{C\,\|{\cal K}_{h_0}\|_{L^1}}{|\log (|\log h_0|^{-1/2}+\bar \eps(h_0))|^\theta},
\end{split}
\]
again using part ii) of Prop. \ref{weightprop}.

In both cases,
using Prop. \ref{kernelcompactness} together with Lemma \ref{lemmakernelaveraged} in the second case, one concludes that $\rho_k$ is locally compact in $x$ and then in $t,\,x$.  Thus we've shown the case ii) and concluded the proof of Theorem \ref{maincompactness}.

\bigskip

\noindent {\it The case with $w_1(y) w_1(x)$.}
The situation is more complicated for \eqref{fw2} and the product $w_1(y) w_1(x)$. Indeed  $w_0(x)+w_0(y)$ or $w_1(x)+w_1(y)$ are small only if both $w_0(x)$ and $w_0(y)$ are small (or the corresponding terms for $w_1$). But $w_1(x)\,w_1(y)$ can be small if either $w_1(x)$ or $w_1(y)$ is small. This was previously an advantage with then simpler computations but not here and \eqref{fw2} does not provide compactness.
  
This is due to the fact that one does not control the size of $\{w_w\leq \eta\}$ but only the mass of $\rho_k$ over that set. The difference between the two is the famous vacuum problem for compressible fluid dynamics which is still unsolved. 

The best that can be done by part ii) of Prop. \ref{weightprop} is for any $\eta,\;\eta'$ 
\[\begin{split}
\int_{\Pi^{2d}} \ind_{\rho_k(x)\geq \eta}\,\ind_{\rho_k(y)\geq \eta}\,K_h(x-y)\,\chi(\delta\rho_k)&\leq \frac{1}{\eta'^2}\,\int\,w_1(x)\,w_1(y)\,K_h(x-y)\,\chi(\delta\rho_k)\\
&+C\,\frac{\|K_h\|_{L^1}}{\eta^{1/2}\,|\log\eta'|^{\theta/2}},
\end{split}\]
using that $\rho_k\in L^2$ uniformly. Using \eqref{fw2} and optimizing in $\eta'$, one finds for some $\theta>0$
\[\begin{split}
\int_{\Pi^{2d}} \ind_{\rho_k(x)\geq \eta}\,\ind_{\rho_k(y)\geq \eta}\,K_h(x-y)\,\chi(\delta\rho_k)&\leq C\,\frac{\|K_h\|_{L^1}}{\eta^{1/2}\,
|\log (\eps(h)+h^\theta)|^{\theta/2}}.
\end{split}
\]
If $\mu_k$ and $F_k$ are uniformly in $W^{s,1}$ for $s>0$, then 
\[
\int_{\Pi^{2d}} \ind_{\rho_k(x)\geq \eta}\,\ind_{\rho_k(y)\geq \eta}\,K_h(x-y)\,\chi(\delta\rho_k)\leq C\,\frac{\|K_h\|_{L^1}}{\eta^{1/2}\,
|\log h|^{\theta/2}},
\]
which concludes the proof of Th. \ref{maincompactness2}. Note however that in many senses \eqref{fw2} is more precise than the final result.
%
\subsection{The coupling with the pressure in the anisotropic case}

\bigskip
In that case we need the weight $w_a$ and its regularization $w_{a,h}$, defined by \eqref{eqw} with \eqref{dk00} in order to compensate some terms coming from the anisotropic non-local part of the stress tensor.
\begin{lemma}
There exists $C_*>0$ s.t. assuming that $\rho_k$ solves \eqref{continuity} with $\alpha_k=0$, that \eqref{rhout},  \eqref{bounduk}, \eqref{boundrhok} with $\gamma>d/2$ and $p>\gamma+1+\ell=\gamma^2/(\gamma-1)$ hold where $\ell=1/(\gamma-1)$; assuming moreover that $P_k$ satisfying \eqref{monotone}, that $u$ solves \eqref{divu0noniso} with 
\begin{equation}
a_\mu\, \leq C_*.\label{smallamu}
\end{equation}
Then there exists $0<\theta<1$ s.t. for $\chi_a$ verifying \eqref{defchi0} for this choice of $\ell$
\[
\begin{split}
&\int_{h_0}^1 \int_{\Pi^{2d}} \frac{\overline K_{h}(x-y)}{h} (w_{a,h}(x) + w_{a,h}(y))\,  \chi_a(\delta\rho_k) (t)  \\
& \qquad 
\leq \int_{h_0}^1 \int_{\Pi^{2d}} \frac{\overline K_{h}(x-y)}{h} (w_{a,h}(x) + w_{a,h}(y))\,  \chi_a(\delta\rho_k)|_{t=0}+C\,(1+\ell)\,|\log h_0|^\theta.  
\end{split}
\]
\label{divanisotropic}
\end{lemma}

\begin{proof}
To simplify the estimate, we assume in this proof that $P_k(\rho_k)=\rho_k^\gamma$, the extension when $P_k$ satisfies instead \ref{monotone} being straightforward. We also recall that $\chi_a$ satisfies \eqref{defchi0}, meaning that for all practical purposes $\chi_a(\xi)\sim |\xi|^{1+\ell}$.

 We use the point iii) in Lemma 
\ref{boundQhAniso}
\[
\begin{split}
&\int_{h_0}^1 \int_{\Pi^{2d}} \frac{\overline K_{h}(x-y)}{h} (w_{a,h}(x) + w_{a,h}(y))\,  \chi_a(\delta\rho_k) (t)  \\
& \qquad 
- \int_{h_0}^1 \int_{\Pi^{2d}} \frac{\overline K_{h}(x-y)}{h} (w_{a,h}(x) + w_{a,h}(y))\,  \chi_a(\delta\rho_k)|_{t=0} \\ 
&\qquad\qquad \le  C|\log h_0|^\theta + I + II-\Pi_a  \\
\end{split}
\]
where $0<\theta<1$ and with the dissipation term by symmetry
\[
 \Pi_a=   \lambda \int_0^t \int_{h_0}^1 \int_{\Pi^{2d}} w_{a,h}(x)\,
     \chi_a(\delta \rho_k)\,\overline K_h \star (|\div u_k|+ |A_\mu P_k(\rho_k)|)(x)\,
    \bar{K_h} \frac{dh}{h} 
\]
while still by symmetry
\[\begin{split}
I= - \frac{1}{2}\int_{h_0}^1\int_0^t\int_{\Pi^{2d}} &\frac{\overline K_h(x-y)}{h}\,(\div u_k(x)-\div u_k(y))
        \chi_a'(\delta\rho_k) \,\bar \rho_k\,w_{a,h}(x),
\end{split}\]
and 
\[\begin{split}
II=- \frac{1}{2}\int_{h_0}^1 \int_0^t\int_{\Pi^{2d}} 
&\frac{\overline K_h(x-y)}{h}\,(\div u_k(x)+\div u_k(y))\\
& \hskip3cm  (\chi'_a(\delta\rho_k) \delta\rho_k - 2 \chi_a(\delta\rho_k))\, w_{a,h}(x).
\end{split}\]

\bigskip
\noindent {\bf I) The quantity $I$.}
We recall that in this case one has the formula \eqref{divu0noniso} on $\div u_k$ 
\begin{equation}\begin{split}
\div u_k=& \nu_k P_k(\rho_k) +\nu_k\, a_\mu A_\mu P_k(\rho_k) \\
&+
\nu_k(\Delta_\mu-a_\mu E_k)^{-1}\, \div(\partial_t(\rho_k\,u_k)+\div(\rho_k\,u_k \otimes u_k)),
\end{split}\label{decomposedivu}
\end{equation}
leading to the notation
\[
\tilde D\rho\,u_k=\nu_k(\Delta_\mu-a_\mu E_k)^{-1}\, \div(\partial_t(\rho_k\,u_k)+\div(\rho_k\,u_k \otimes u_k)).
\]
Therefore, one may decompose
\[
I=I_0+I^D+I^R,
\]
with
\[\begin{split}
I_0=-\frac{1}{2}\int_{h_0}^1 \int_0^t\int_{\Pi^{2d}} &\frac{\overline K_h(x-y)}{h}\,(\tilde D\rho\,u_k(x)-\tilde D\rho\,u_k(y))\\
&\quad  \chi_a'(\delta\rho_k) \,\bar \rho_k\,(w_{a,h}(x)+w_{a,h}(y)),
\end{split}
\]
while
\[\begin{split}
I^D=-\frac{\nu_k}{2}\int_{h_0}^1 \int_0^t\int_{\Pi^{2d}} &\frac{\overline K_h(x-y)}{h}\,(P_k(\rho_k(x))
   - P_k(\rho_k(y))\\
&\quad  \chi_a'(\delta\rho_k) \,\bar \rho_k\,w_{a,h}(x),
\end{split}
\]
and
\[\begin{split}
I^R=- \frac{a_\mu\,\nu_k}{2}
\int_{h_0}^1 \int_0^t\int_{\Pi^{2d}} &\frac{\overline K_h(x-y)}{h}\,(A_\mu P_k(\rho_k(x))
   -A_\mu P_k(\rho_k(y)))    \\
&\quad  \chi'_a(\delta\rho_k)  \,\bar \rho_k\,w_{a,h}(x).
\end{split}
\]

\bigskip

\noindent {\bf  I-1) The term $I_0$.}
This term is handled just as in the proof of Lemmas \ref{divnoviscosity}-\ref{divwithviscosity} by using Lemma \ref{Drhou} and for this reason we do not fully detail all the steps here. First note that Lemma \ref{Drhou} applies to $\tilde D\rho_k \,u_k$ as well as for $D\,\rho_k\,u_k$ as
\[
\tilde D\rho_k\,u_k= (\nu_k(\Delta_\mu-a_\mu E_k)^{-1}\,\Delta)\,D\rho_k\,u_k.
\]
Then as before, we first truncate by using some smooth function $I_k^L(t,x,y)= \phi(\rho_k(t,x)/L)
\phi(\rho_k(t,y)/L)$ with some smooth and compactly supported function $\phi$  leading to $I_0=I_0^L+I_0^{RL}$ with
\[
\begin{split}
I_0^L=-\frac{1}{2}\int_{h_0}^1 \int_0^t\int_{\Pi^{2d}} &\frac{\overline K_h(x-y)}{h}\,(\tilde D\rho\,u_k(x)-\tilde D\rho\,u_k(y))\\
&\quad  \chi_a'(\delta\rho_k) \,I_k^L\,\bar \rho_k\,(w_{a,h}(x)+w_{a,h}(y)),
\end{split}
\]
and
\[
\begin{split}
I_0^{RL}=-\frac{1}{2}\int_{h_0}^1 \int_0^t\int_{\Pi^{2d}} &\frac{\overline K_h(x-y)}{h}\,(\tilde D\rho\,u_k(x)-\tilde D\rho\,u_k(y))\\
&\quad  \chi_a'(\delta\rho_k) \,(1-I_k^L)\,\bar \rho_k\,(w_{a,h}(x)+w_{a,h}(y)).
\end{split}
\]
Remark that $\div u_k\in L^2_{t,x}$, $P_k(\rho_k)\in L^{p/\gamma}_{t,x}$ and since
$A_\mu$ is an operator of $0$ order, $A_\mu P_k(\rho_k)\in L^{p/\gamma}_{t,x}$. Therefore by Eq. \eqref{decomposedivu}
\[
\sup_k\|\tilde D\,\rho_k\,u_k\|_{L^{\min(2,p/\gamma)}_{t,x}}<\infty.
\]
On the other hand $|\chi_a'(\delta\rho_k)|\leq C\,(1+\ell)\,(|\rho_k(x)|^\ell+|\rho_k(y)|^\ell)$ and this lets us bound very simply $I^{RL}_0$ by H\"older estimates
\[\begin{split}
I_0^{RL}&\leq C\,(1+\ell)\,|\log h_0|\,\|\tilde D\,\rho_k\,u_k\|_{L^{\min(2,p/\gamma)}_{t,x}}\,\|(1-I_k^L)\,\rho_k^{1+\ell}\|_{L^{\max(2,q)}_{t,x}}\\
&\leq C\,(1+\ell)\,|\log h_0|\,\|(1-I_k^L)\,\rho_k^{1+\ell}\|_{L^{\max(2,q)}_{t,x}},
\end{split}\]
with $1/q+\gamma/p=1$. But $q\,(1+\ell)<p$ by the assumption $p>\gamma+1+\ell$ and similarly $2\,(1+\ell)<p$. As a consequence for some exponent $\theta_1>0$
\begin{equation}
I_0^{RL}\leq C\,(1+\ell)\,|\log h_0|\,L^{-\theta_1}.\label{boundI0RL}
\end{equation}  
We now use Lemma \ref{Drhou} for $\tilde D\rho_k\,u_k$ and $\Phi=\chi_a'(\delta\rho_k)\,I_k^L\,\bar\rho_k\,w_{a,h}(x)$. We note that $\|\Phi\|_{L^\infty}\leq C\,(1+\ell)\,L^{1+\ell}$. Moreover just as in the proof of 
Lemmas \ref{divnoviscosity}-\ref{divwithviscosity}, we can show that $\Phi$ satisfy a transport equation giving that
\[\begin{split}
C_\Phi&=\left\|\int_{\Pi^d} \overline K_h(x-y)\Phi(t,x,y)\,dy\right\|_{W^{1,1}_t W_x^{-1,1}}\\
&\qquad+\left\|\int_{\Pi^d} \overline K_h(x-y)\Phi(t,x,y)\,dx\right\|_{W^{1,1}_t W_x^{-1,1}}
\leq C\,(1+\ell)\,L^{1+\ell}.
\end{split}\]
By Lemma \ref{Drhou}, we obtain that for some $\theta_2>0$
\begin{equation}
I_0^L\leq C\,(1+\ell)\,L^{1+\ell}\,\int_{h_0}^1 h^{\theta_2}\,\frac{dh}{h}\leq C\,(1+\ell)\,L^{1+\ell}.\label{boundI0L}
\end{equation}
By optimizing in $L$, this lets us conclude that again for some $0<\theta<1$  and provided that $p>\gamma+1+\ell$
\begin{equation}
I_0\leq C\,(1+\ell)\,|\log h_0|^\theta.\label{boundmoddrhou}
\end{equation}

\bigskip

\noindent {\bf 1-2) The term $I^D$.} This term has the right sign as
\[\begin{split}
&\int_{h_0}^1 \int_0^t\int_{\Pi^{2d}} \frac{\overline K_h(x-y)}{h}\,(\rho_k^\gamma(x)-\rho_k^\gamma(y))
   \,  \chi'_a(\delta\rho_k) \,\bar \rho_k\,(w_{a,h}(x)+w_{a,h}(y))\\
&\quad \geq C\,\int_{h_0}^1 \int_0^t\int_{\Pi^{2d}} \frac{\overline K_h(x-y)}{h}\,
\chi'_a(\delta\rho_k) \, \delta \rho_k\,\bar \rho_k^\gamma\, (w_{a,h}(x)+w_{a,h}(y)).
\end{split}
\]
We will actually give a more precise control on $I^D+II^D$ later on when the corresponding decomposition of $II = II _0 + II^D + II^R$ will be introduced.
 
 \bigskip
 
 \noindent {\bf I-3) The term $I^R$.} The difficulty is thus in this quantity. From its definition, $A_\mu$ is a convolution operator. With a slight abuse of notation, we denote by $A_\mu$ as well its kernel or
\[
A_\mu f=\int_{\Pi^d} A_\mu(x-y)\,f(y)\,dy,
\]
and note that $A_\mu$ corresponds to an operator of $0$ order, {\em i.e.} it for instance satisfies the property $\int_A A_\mu=0$ for any annulus $A$ centered at the origin, $|A_\mu(x)|\leq C\,|x|^{-d}$.
Decompose
\[
A_\mu=L_h+R_h,
\qquad \mbox{supp}\,L_h\subset \{|x|\leq \delta_h\},
\]
such that both $L_h$ and $R_h$ remain bounded on any $L^p$ space, $1<p<\infty$, and moreover $R_h$ is a regularization of $A_\mu$ that is $R_h=A_\mu\star N_{\delta_h}$ for some smooth kernel $N_{\delta_h}$. 
The scale $\delta_h$ has to satisfy that
\[
\delta_h<<h,\quad \log \frac{h}{\delta_h}<< |\log h|.
\]
 For simplicity we choose here $\delta_h= {h}/{|\log h|}$.

\medskip
\noindent {\it Contribution of the $R_h$ part.} 

The first step is to decompose $R_h$ into dyadic blocks in Fourier. Introduce a decomposition of identity $\Psi_l$ as in sections \ref{usefulsection} and \ref{Besov} s.t. $1=\sum_l \hat\Psi_l$ and write
\begin{equation}
R_h=\sum_{l=|\log_2 h|}^{|\log_2 \delta_h|}  \Psi_l\star R_h+\tilde R_h,\quad \tilde R_h=\sum_{l<|\log_2 h|} \Psi_l\star R_h=\tilde N_h\star N_{\delta_h}\star A_\mu.
\end{equation}
Note that we of course require of the $\Psi_l$ to satisfy all the assumptions specified in section \ref{Besov} for the definition of Besov spaces. Define now $\overline N_h=\tilde N_h\star N_{\delta_h}$, this kernel 
$\overline N_h$ therefore satisfies that for any $s>0$
\begin{equation}
\|\overline N_h\|_{W^{s,1}}\leq C\,h^{-s},\label{regNh}
\end{equation}
and moreover by the localization property of the $\Psi_k$, one has that for $s>0$ and any $|\omega|\leq 1$
\begin{equation}
\int_{\Pi^d} |z|^s\,|\overline N_h(z)+\overline N_h(z+\omega\,r)|\,dr\leq C\,h^s.\label{momNh}
\end{equation}
Fix $t$ for the moment and decompose accordingly
\[\begin{split}
&\int_{h_0}^1 \int_{\Pi^{d}} \frac{\overline K_h(z)}{h}\, \|R_h\star \rho_k^\gamma(t,.)-R_h\star\rho_k^\gamma(t,.+z)\|_{L^q_x} \\
&\qquad\leq\int_{h_0}^1 \int_{\Pi^{d}} \frac{\overline K_h(z)}{h}\, \|\tilde R_h\star \rho_k^\gamma(t,.)-\tilde R_h\star\rho_k^\gamma(t,.+z)\|_{L^q_x}\\
&\qquad+ \int_{h_0}^1\sum_{l=|\log_2 h|}^{|\log_2 \delta_h|} \int_{\Pi^{d}} \frac{\overline K_h(z)}{h}\, \|\Psi_l\star R_h\star \rho_k^\gamma(t,.)-\Psi_l\star R_h\star\rho_k^\gamma(t,.+z)\|_{L^q_x}.
\end{split}\]
By \eqref{regNh} and \eqref{momNh}, the kernel $\overline N_h$ satisfies the assumptions of Lemma \ref{shiftNlemma} and thus with $U_{k}=A_\mu\star \rho_k^\gamma$, applying Lemma \ref{shiftNlemma}, for any $q>1$
\[\begin{split}
&\int_{h_0}^1 \int_{\Pi^{d}} \frac{\overline K_h(z)}{h}\, \|\tilde R_h\star \rho_k^\gamma(t,.)-\tilde R_h\star\rho_k^\gamma(t,.+z)\|_{L^q_x}\\
&\qquad=\int_{h_0}^1 \int_{\Pi^{d}} \frac{\overline K_h(z)}{h}\, \|\overline N_h\star U_k(t,.)-\overline N_h\star U_k(.+z)\|_{L^q_x}\\
&\qquad \leq C\,|\log h_0|^{1/2}\, \|U_k(t,.)\|_{L^q_x}. 
\end{split}
\]
Recalling that $A_\mu$ is continuous on every $L^p$ space, one has that $\|U_k\|_{L^p_x}\leq C\,\|\rho_k^\gamma\|_{L^q_x}$ hence
\[
\int_{h_0}^1 \int_{\Pi^{d}} \frac{\overline K_h(z)}{h}\, \|\tilde R_h\star \rho_k^\gamma(t,.)-\tilde R_h\star\rho_k^\gamma(t,.+z)\|_{L_x^q}\leq C\,|\log h_0|^{1/2}\,\|\rho_k(t,.)\|_{L^{q\,\gamma}_{x}}^{\gamma}.
\]
On the other hand simply by bounding 
$$|\Psi_l\star R_h\star \rho_k^\gamma(x)-\Psi_l\star R_h\star\rho_k^\gamma(y)|^q\leq C\,|\Psi_l\star R_h\star \rho_k^\gamma(x)|^q+|\Psi_l\star R_h\star\rho_k^\gamma(y)|^q$$
we write
\[\begin{split}
&\int_{h_0}^1\sum_{l=|\log_2 h|}^{|\log_2 \delta_h|} \int_{\Pi^{d}} \frac{\overline K_h(z)}{h}\, \|\Psi_l\star R_h\star \rho_k^\gamma(t,.)-\Psi_l\star R_h\star\rho_k^\gamma(t,.+z)\|_{L^q_x}\\
&\qquad\leq C\,\sum_{l\leq |\log_2 h_0|+\log_2|\log_2 h_0|}\|\Psi_l\star R_h\star \rho_k^\gamma(t,.)\|_{L^q_x}\,\int_{2^{-l}}^{l\,2^{-l}} \frac{dh}{h}.
\end{split}
\]
recalling that $\delta_h=h/|\log_2 h|$. This leads to
\[\begin{split}
&\int_{h_0}^1\sum_{l=|\log_2 h|}^{|\log_2 \delta_h|} \int_{\Pi^{d}} \frac{\overline K_h(z)}{h}\, \|\Psi_l\star R_h\star \rho_k^\gamma(t,.)-\Psi_l\star R_h\star\rho_k^\gamma(t,.+z)\|_{L_x^q}\\
&\qquad \leq C\,\sum_{l\leq 2\,|\log_2 h_0|}\log l\, \|\Psi_l\star R_h\star \rho_k^\gamma(t,.)\|_{L^q_x},\\
\end{split}
\]
and can in turn be directly bounded by
\[\begin{split}
&\leq C\,\log|\log h_0|\,\sum_{l\leq 2\,|\log_2 h_0|} \|\Psi_l\star R_h\star \rho_k^\gamma(t,.)\|_{L^q_{x}}\\
&\leq C\,\log|\log h_0|\,|\log h_0|^{1/2}\|R_h\star \rho_k^\gamma(t,.)\|_{L^q_{x}}\leq C\,|\log h_0|^{\theta}\,\|\rho_k(t,.)\|_{L^{q\gamma}_{tx}}^{\gamma}, 
\end{split}
\]
with $0<\theta<1$ by Lemma \ref{truncatedbesov}. Combining with the previous estimate, we deduce that
\begin{equation}\begin{split}
\int_{h_0}^1 \int_{\Pi^{d}} \frac{\overline K_h(z)}{h}\, 
 \|R_h\star \rho_k^\gamma(t,.)
 & -R_h\star\rho_k^\gamma(t,.+z)\|_{L_x^q} \\
\hskip6cm
& \leq C\,|\log h_0|^{\theta}\,\|\rho_k(t,.)\|_{L^{q\gamma}_{x}}^{\gamma}.
\end{split}\label{smallRh1}
\end{equation}
with $0<\theta<1$. And therefore since $\chi_a'(\xi)\leq (1+\ell)\,|\xi|^\ell$, by H\"older's inequality with the relation 
$1/q+ {(1+\ell)}/{(1+\gamma+\ell)}=1$, that is $q= {(1+\ell+\gamma)}/{\gamma}$
\[
\begin{split}
&\int_{h_0}^1\int_0^t \int_{\Pi^{2d}} \frac{\overline K_h(x-y)}{h}\,(R_h\star\rho_k^\gamma (x)-R_h\star \rho_k^\gamma(y))\,\chi'_a(\delta\rho_k) \,\bar \rho_k\,(w_{a,h}(x)+w_{a,h}(y))\\
& \geq -C\,(1\!+\!\ell)\,\int_0^t\|\rho_k(t,.)\|_{L^{1+\ell+\gamma}_x}^{1+\ell}\,\int_{h_0}^1 \int_{\Pi^{d}} \frac{\overline K_h(z)}{h}\,\|R_h\star\rho_k^\gamma (t,.)-R_h\star \rho_k^\gamma(t,.+z)\|_{L^q_x}.\\
\end{split}
\]
Finally by \eqref{smallRh1} there exists $0<\theta<1$ s.t.
\begin{equation}
\begin{split}
&\int_{h_0}^1\int_0^t \int_{\Pi^{2d}} \frac{\overline K_h(x-y)}{h}\,(R_h\star\rho_k^\gamma (x)-R_h\star \rho_k^\gamma(y))\, \\
&\hskip5cm  \chi'_a(\delta\rho_k) \,\bar \rho_k\,(w_{a,h}(x)+w_{a,h}(y))\\
&\qquad\geq -C\,(1+\ell)\,|\log h_0|^{\theta}\,\|\rho_k\|_{L^{\gamma+\ell+1}}^{\gamma+\ell+1}.
\end{split}\label{boundIRh}
\end{equation}

\medskip
\noindent {\it Contribution of the $L_h$ part.}
It remains the term with $L_h$ where we symmetrize the position of the weight with respect to the convolution with $L_h$ by
\[\begin{split}
&\int_{h_0}^1\int_0^t \int_{\Pi^{2d}} \frac{\overline K_h(x-y)}{h}\,(L_h\star\rho_k^\gamma (x)-L_h\star \rho_k^\gamma(y))\,\chi'_a(\delta\rho_k) \,\bar \rho_k\,w_{a,h}(x)\\
&= I^L_h-\mbox{Diff},
\end{split}
\]
with
\[\begin{split}
&I^L_h=\int_{h_0}^1\int_0^t \int_{\Pi^{3d}} \frac{\overline K_h(x-y)}{h}\,L_h(z)(\rho_k^\gamma (x-z) -\rho_k^\gamma(y-z))\\
&\qquad\qquad\qquad\qquad\qquad \chi'_0(\delta\rho_k) \,\bar \rho_k\,w_h^{1-n}(x)\,w_h^{n}(x-z),\\
\end{split}
\]
for $n=1-1/\gamma$.
   Recall that since $w_{a,h}=\overline K_h\star w_a$,  $|w_{a,h}(x)-w_{a,h}(x-z)|\leq h^{-1}\,|z|$ while $|z|\sim \delta_h$ on the support of $L_h$.
Thus using that $|\chi_a'(\xi)|\leq C\,|\xi|^{\ell}$ from \eqref{defchi0} and that $|L_h(z)|\leq C\,|z|^{-d}$, one has 
\[\begin{split}
&\mbox{Diff}=\int_{h_0}^1 \int_0^t\int_{\Pi^{3d}} \frac{\overline K_h(x-y)}{h}\,L_h(z)(\rho_k^\gamma (x-z)- \rho_k^\gamma(y-z))\\
&\qquad\qquad
\chi'_a(\delta\rho_k) \,\bar \rho_k\,w_{a,h}^{1-n}(x)\,(w_{a,h}^n(x)-w_{a,h}^n(x-z))\\
&\ \leq C\,(1+\ell)\,\int_{h_0}^1\int_0^t \int_{\Pi^{3d}} \ind_{|z|\leq \delta_h}\,\frac{\overline K_h(w)}{h\,|z|^d}\,  |\rho_k^\gamma(x-z)+\rho_k^\gamma(x-z+w)|\,\bar \rho^{\ell+1}\,h^{-n}\,|z|^{n}\\
&\ \leq C\,(1+\ell)\,\int_{h_0}^1\int_0^t \int_{\Pi^{3d}} \ind_{|z|\leq \delta_h}\,\frac{\overline K_h(w)}{h^{1+n}\,|z|^{d-n}}\,  \big(\rho_k^{\gamma+\ell+1}(x-z)+\rho_k^{\gamma+\ell+1}(x-z+w)\\
&\qquad\qquad\qquad\qquad+\rho_k^{\gamma+\ell+1}(x)+\rho_k^{\gamma+\ell+1}(x-w)\big).\\
\end{split}
\]
Using $|z|\le \delta_h=\frac{h}{|\log_2 h|}$, we obtain on the other hand that
\[\begin{split}
\int_{h_0}^1\frac{dh}{h^{1+n}}\int_{|z|\leq \delta_h} \frac{dz}{|z|^{d-n}}&\leq C\,\int_{h_0}^1\frac{\,\delta_h^n dh}{h^{1+n}}=C\,\int_{h_0}^1\frac{dh}{h\,|\log h|^n}\\
&\leq C\,|\log h_0|^{1-n}.
\end{split}
\]
As $n=1-1/\gamma$, this leads to  
\begin{equation}
\mbox{Diff}\leq C\,(1+\ell)\,\|\rho_k\|_{L^{\gamma+\ell+1}}^{\gamma+\ell+1}\,|\log h_0|^{1/\gamma}.\label{boundDiff}
\end{equation}
As for the first term by H\"older inequality, using again that $|\chi_a'(\xi)|\leq C\,\ell\,|\xi|^{\ell}$
\[
\begin{split}
&I^L_h\leq C\,(1+\ell)\,\int_{h_0}^1\int_0^t \int_{\Pi^{d}}\,\left(\int_{\Pi^d} |\delta\rho_k|^{\ell+1}\,\bar \rho_k^{(\ell+1)/\ell}\,w_{a,h}^{(1-\theta)\,(\ell+1)/\ell}(x)\,dx\right)^{\ell/(\ell+1)}\\
&\quad \left(\int_{\Pi^d} \left(L_h\star\left((\rho_k^\gamma(.)  -\rho_k^\gamma(.+w))\,w_{a,h}^{\theta}(.)\right)\right)_{|x}^{(\gamma+\ell+1)/\gamma}\,dx\right)^{\gamma/(\gamma+\ell+1)}
\, \frac{\overline K_h(w)}{h}\,dw,\\
\end{split}
\]
provided that $l$ is chosen s.t.
\[
\frac{\gamma}{\gamma+\ell+1}+\frac{\ell}{\ell+1}=1,\quad\mbox{or}\qquad \ell+1=\frac{\gamma}{\gamma-1},
\]
implying for instance that
\[
\frac{\ell+1}{\ell}=\gamma, \quad \frac{\gamma+\ell+1}{\gamma}=\ell+1\dots
\]
Given those algebraic relations and recalling that $L_h\star$ is continuous on every $L^p$ for any $1<p<\infty$
\[
\begin{split}
&I^L_h\leq C\,(1+\ell)\,\int_{h_0}^1 \int_0^t \int_{\Pi^{d}} \frac{\overline K_h(w)}{h}\,\left(\int_{\Pi^d} 
 |\rho_k^\gamma (x) -\rho_k^\gamma(x+w)|^{\frac{\gamma+\ell+1}\gamma}\,w_{a,h} \,dx\right)^{\frac\gamma{\gamma+\ell+1}}\\
&\qquad\qquad\left(\int_{\Pi^d}
      |\chi'_0(\delta_k)|^{\gamma}\,\bar \rho_k^{\gamma}\,w_{a,h}\,dx\right)^{\ell/(\ell+1)}\,dw.\\
\end{split}
\]
Since, using the definition of $\ell$, 
\[
|\rho_k^\gamma (x) -\rho_k^\gamma(x+w)|^{(\gamma+\ell+1)/\gamma}
\leq \gamma\,\bar\rho^\gamma\,|\delta\rho_k|^{(\gamma+\ell+ 1)/\gamma}
=\gamma\,\bar\rho^\gamma\,|\delta\rho_k|^{\ell+1},
\]
one has
\begin{equation}
I^L_h\leq C\,(1+\ell)\,\gamma \int_{h_0}^1 \int_{\Pi^{2d}} \frac{\overline K_h(w)}{h}\,|\delta\rho_k|^{\ell+1}\,\bar \rho_k^{\gamma}\,w_{a,h}(x)\,dx\,dw,\label{boundILh}
\end{equation}
which multiplied by $- a_\mu\,\nu_k/2$  will be bounded by $I^D+II^D$ provided $|a_\mu|$ is small enough.

\bigskip

\noindent {\bf II) The quantity II.}
Let us turn to $II$ and decompose it as for $I$
\[
II=II_0+II^D+II^R,
\]
where
\[\begin{split}
II_0=-\frac{1}{2}\int_{h_0}^1 \int_0^t\int_{\Pi^{2d}} &\frac{\overline K_h(x-y)}{h}\,(\tilde D\rho\,u_k(x)+\tilde D\rho\,u_k(y))\\
&\quad (\chi'_a(\delta\rho_k)\delta\rho_k - 2\chi_a(\delta\rho_k))\,(w_{a,h}(x)+w_{a,h}(y)),
\end{split}
\]
while
\[\begin{split}
II^D=-\frac{\nu_k}{2}\int_{h_0}^1 \int_0^t\int_{\Pi^{2d}} &\frac{\overline K_h(x-y)}{h}\,(P_k(\rho_k(x))+P_k(\rho_k(y)))\\
&\quad (\chi'_a(\delta\rho_k)\delta\rho_k - 2\chi_a(\delta\rho_k))\,(w_{a,h}(x)+w_{a,h}(y)),
\end{split}
\]
and
\[\begin{split}
II^R= -\frac{a_\mu\,\nu_k}{2} \,\int_{h_0}^1 \int_0^t\int_{\Pi^{2d}} &\frac{\overline K_h(x-y)}{h}\,
    (A_\mu P_k(\rho_k(x)) +A_\mu P_k(\rho_k(y)))\\
&\quad  (\chi'_a(\delta\rho_k)\delta\rho_k -2\chi_a(\delta\rho_k))\,(w_{a,h}(x)+w_{a,h}(y)).
\end{split}
\]

\bigskip

\noindent {\bf II-1) Term $II_0$.} For the term $II_0$, using Lemma \ref{Drhou} in a manner identical to $I_0$
\[\begin{split}
II_0\leq -\int_{h_0}^1 \int_0^t\int_{\Pi^{2d}} &\frac{\overline K_h(x-y)}{h}\,\overline K_h\star \tilde D\rho\,u_k(x) (\chi'_a(\delta\rho_k)\delta\rho_k - 2\chi_a(\delta\rho_k))\,w_{a,h}(x)\\
& + C\,(1+\ell)\,\|\rho_k\|_{L^{\gamma+\ell+1}}^{1+\ell}\,|\log h_0|^{\theta},
\end{split}
\]
for some $0<\theta<1$. Using formula \eqref{divu0noniso} or \eqref{decomposedivu}, one has that
\[
\div u_k - a_\mu A_\mu P_k(\rho_k) \geq \tilde D\rho\,u_k,
\]
and hence since $- \chi'_a\,\xi + 2\chi_a\geq -C\,(1+\ell)\,\chi_a$
\begin{equation}
\begin{split}
&II_0\leq C\,\ell\,\|\rho_k\|_{L^{\gamma+\ell+1}}^{1+\ell}\,|\log h_0|^{\theta}\\
& +C\,(1\!+\!\ell)\,\int_{h_0}^1 \int_0^t\int_{\Pi^{2d}} \frac{\overline K_h(x-y)}{h}\, 
  (\overline K_h\star (|\div u_k| + a_\mu |A_\mu P_k(\rho_k)|)\,  \chi_a(\delta\rho_k)\,w_{a,h}(x),\\
\end{split}\label{boundII0}
\end{equation}
and the first integral will be bounded by $\Pi_a/2$ for $\lambda$ large enough.

\bigskip

\noindent {\bf II-2) Term $II^D$.}  The term $II^D$ is controlled by $I^D$:
For $a\geq b$, by \eqref{defchi0}
\[
(a^\gamma+b^\gamma)\,
(- \chi_a'(a-b)(a-b) + 2\chi_a(a-b)) \geq- (a^\gamma-b^\gamma)\,\frac{\ell-1}{\ell}\chi_a'(a-b)(a+b).
\]
Therefore 
\begin{equation}
I^D+II^D\leq -C\,\gamma\,\frac{\nu_k}{2}\,\int_0^t\int_{\Pi^{2d}} \frac{\overline K_h(x-y)}{h}\,|\delta\rho_k|^{\ell+1}\,\bar \rho_k^\gamma\,(w_h(x)+w_h(y)),
\label{boundIID}
\end{equation}
for some $C$ independent of $\ell$ and $\gamma$.

\bigskip

\noindent {\bf  II-3) Term $II^R$.}
  The control on the last term, $II^R$, requires the use of the penalization $\Pi_a$
\[\begin{split}
II^R+\frac{1}{2}\Pi_a\leq &-a_\mu\,\nu_k \,\int_{h_0}^1 \int_0^t\int_{\Pi^{2d}} \frac{\overline K_h(x-y)}{h}\,(A_\mu\rho_k^\gamma (x)-A_\mu\overline K_h\star \rho_k^\gamma(x))\\
&\hskip4.5cm 
 (\chi'_a(\delta\rho_k)\delta\rho_k -2\chi_a(\delta\rho_k))\,w_{a,h}(x).\\
\end{split}
\]
We use the same decomposition of $A_\mu=L_h+R_h$ as for $I^R$.

\medskip

\noindent {\it Contribution of the $R_h$ part.}

Note that, as $\chi_a(\xi)\leq  C\,|\xi|^{1+\ell}$ and $|\chi_a'|\leq C\,(1+\ell)\,|\xi|^\ell$, for $q= {(1+\ell+\gamma)}/{\gamma}$ or $1/q+ {(1+\ell)}/{(1+\ell+\gamma)}=1$ 
\[\begin{split}
&-a_\mu\,\nu_k \,\int_{h_0}^1 \int_0^t\int_{\Pi^{2d}} \frac{\overline K_h(x-y)}{h}\,(R_h\star\rho_k^\gamma (x)-\overline K_h\star R_h\star\rho_k^\gamma(x))\\
&\qquad\qquad
 (\chi'_a(\delta\rho_k)\delta\rho_k -2\chi_a(\delta\rho_k))\,w_{a,h}(x)\\
&\ \leq C\,(1+\ell)\int_0^t\|\rho_k(t,.)\|_{L^{1+\ell+\gamma}_x}^{1+\ell}\, \int_{h_0}^1 \int_{\Pi^{d}} \frac{\overline K_h(z)}{h}\,\|R_h\star\rho_k^\gamma (t,.)-R_h\star\rho_k^\gamma (t,.+z)\|_{L^q_x}.
\end{split}
\]
Now by estimate \eqref{smallRh1}, we have that
\begin{equation}\begin{split}
&-a_\mu\,\nu_k \,\int_{h_0}^1 \int_0^t\int_{\Pi^{2d}} \frac{\overline K_h(x-y)}{h}\,(R_h\star\rho_k^\gamma (x)-\overline K_h\star R_h\star\rho_k^\gamma(x))\\
&\hskip6cm 
 (\chi'_a(\delta\rho_k)\delta\rho_k -2\chi_a(\delta\rho_k))\,w_{a,h}(x)\\
&\ \leq C\,(1+\ell)\,|\log h_0|^{3/4}\int_0^t \|\rho_k(t,.)\|_{L^{1+\ell+\gamma}_x}^{1+\ell}\,\|\rho_k(t,.)\|_{L^{\gamma\,q}_x}^\gamma\\
&\ \leq C\,(1+\ell)\,|\log h_0|^{3/4}\,\|\rho_k\|_{L^{1+\ell+\gamma}_{t,x}}^{1+\ell+\gamma}.
\end{split}\label{boundIIRh}
\end{equation}

\medskip

\noindent {\it Contribution of the $L_h$ part.} Similarly as for $I^R$, we symmetrize the weights leading to the following decomposition
\[
\begin{split}
&-a_\mu\,\nu_k \,\int_{h_0}^1 \int_0^t\int_{\Pi^{2d}} \frac{\overline K_h(x-y)}{h}\,(L_h\star\rho_k^\gamma (x)-\overline K_h\star L_h\star\rho_k^\gamma(x)) \\
&\hskip6cm  (\chi'_a(\delta\rho_k)\delta\rho_k -2\chi_a(\delta\rho_k))\,w_{a,h}(x)\\
&\qquad\qquad=II^L_h+\mbox{Diff}_2,
\end{split}
\]
where
\[
\begin{split}
II^L_h=a_\mu\,\nu_k\, \int_{h_0}^1 \int_0^t\int_{\Pi^{2d}} &\frac{\overline K_h(x-y)}{h}\,L_h\star\bigl(w_{a,h}^n(x)\,(\rho_k^\gamma (x)-\overline K_h\star \rho_k^\gamma(x))\bigr)\\
&\hskip2.8cm  (-\chi_a'(\delta\rho_k)\delta\rho_k +2\chi_a(\delta\rho_k))\,w_{a,h}^{1-n}(x),
\end{split}
\]
with still $n=1-1/\gamma$. The term $\mbox{Diff}_2$ is controlled as the term $\mbox{Diff}$ in $I^R$ using the regularity of $w_{a,h}$ and yielding
\begin{equation}
\mbox{Diff}_2\leq C\,(1+\ell)\,|\log h_0|^\gamma\,\|\rho_k\|_{L^{\gamma+1+\ell}}^{\gamma+1+\ell}.\label{boundDiff2}
\end{equation}

\smallskip

We handle $II^L_h$ with H\"older estimates quite similar to the ones used for the term~$I^L_h$, recalling that $L_h\star$ is bounded on any $L^q$ space for $1<q<\infty$
\[\begin{split}
&II^L_h\leq C\, a_\mu\,\nu_k\,(1+\ell)\,\int_{h_0}^1 \int_0^t\int_{\Pi^{d}} \frac{\overline K_h(w)}{h}\,\|w_{a,h}^\theta\,(\rho_k^\gamma (.)-\overline K_h\star \rho_k^\gamma(.)))\|_{L^{\ell+1}}\\
&\qquad\qquad\qquad\qquad\qquad \||\delta\rho_k|^{\ell+1}\,w_{a,h}^{1-\theta}(x)\|_{L^{(\ell+1)/\ell}}\,dw\\
&\ \leq C\, a_\mu\,\nu_k\,(1+\ell)\,\int_{h_0}^1 \int_0^t\int_{\Pi^{d}} \frac{\overline K_h(w)}{h} \int_{\Pi^d} w_{a,h}(x)\,|\rho_k^\gamma(x)-\overline K_h\star \rho_k^\gamma(x)|^{\ell+1}\,dx\,dw\\
&\ +C\,  a_\mu\,\nu_k\,(1+\ell)\,\int_{h_0}^1 \int_0^t\int_{\Pi^{d}} \frac{\overline K_h(w)}{h} \int_{\Pi^d}
w_{a,h}(x)\,|\delta\rho_k|^{(\ell+1)^2/\ell}\,dx\,dw.
\end{split}\]
One has immediately that
\[\begin{split}
&\int_{h_0}^1 \int_0^t\int_{\Pi^{2d}} \frac{\overline K_h(w)}{h}\,
w_{a,h}(x)\,|\delta\rho_k|^{(\ell+1)^2/\ell}\,dx\,dw\\
&\qquad\leq\int_{h_0}^1 \int_0^t\int_{\Pi^{2d}} \frac{\overline K_h(w)}{h}\,
w_{a,h}(x)\,|\delta\rho_k|^{\ell+1}\,\bar\rho_k^{\gamma}\,dx\,dw\,
\end{split}\]
as $|\delta\rho_k|\leq\bar\rho_k$ and again $(\ell+1)/\ell=\gamma$. 

\medskip

\noindent Finally as $(\ell+1)\,(\gamma-1)=\gamma$
\[\begin{split}
&\int_{h_0}^1 \int_0^t\int_{\Pi^{2d}} \frac{\overline K_h(w)}{h}\,  w_{a,h}(x)\,|\rho_k^\gamma(x)-\overline K_h\star \rho_k^\gamma(x)|^{\ell+1}\,dx\,dw\\
&\ \leq \gamma\,\int_{h_0}^1 \int_0^t\int_{\Pi^{3d}} \frac{\overline K_h(w)}{h}\,\overline K_h(z)\,w_{a,h}(x)\,|\rho_k(x)-\rho_k(x+z)|^{\ell+1}\, (\rho_k(x)+\rho_k(x+z))^\gamma\\
&\ \leq \gamma\,\int_{h_0}^1 \int_0^t\int_{\Pi^{2d}} \frac{\overline K_h(z)}{h}\,w_{a,h}(x)\,|\rho_k(x)-\rho_k(x+z)|^{1+\ell}\,(\rho_k(x)+\rho_k(x+z))^\gamma,\\
\end{split}\]
as $\overline K_h(w)$ is the only term depending on $w$ and is of integral $1$. 
Therefore 
\begin{equation}
II_h^L\leq C\,a_\mu\,\nu_k\,\gamma\,(1+\ell)\,\int_{h_0}^1 \int_0^t\int_{\Pi^{2d}} \frac{\overline K_h(z)}{h}\,w_{a,h}(x)\,|\delta\rho_k|^{1+\ell}\,\bar\rho_k^\gamma.
\label{boundIILh}\end{equation}

\bigskip

To conclude, we sum all the contributions, more precisely \eqref{boundmoddrhou}, \eqref{boundIRh}, \eqref{boundDiff}, \eqref{boundILh}, \eqref{boundII0}, \eqref{boundIID}, \eqref{boundIIRh}, \eqref{boundDiff2}, and \eqref{boundIILh}, to find for some $0<\theta<1$ and provided $p>\gamma+1+\ell$
\[
\begin{split}
&\int_{h_0}^1 \int_{\Pi^{2d}} \frac{\overline K_{h}(x-y)}{h} (w_{a,h}(x) + w_{a,h}(y))\,  \chi_a(\delta\rho_k) (t)  \\
& \ \leq 
\int_{h_0}^1 \int_{\Pi^{2d}} \frac{\overline K_{h}(x-y)}{h} (w_{a,h}(x) + w_{a,h}(y))\,  \chi_a(\delta\rho_k)|_{t=0}+C\,(1+\ell)\,|\log h_0|^\theta\\
&\qquad+C\,\left(a_\mu\,\nu_k\,(1+\ell)-C\,\frac{\nu_k}{2}\right)\,\gamma\,\int_{h_0}^1 \int_0^t\int_{\Pi^{2d}} \frac{\overline K_h(z)}{h}\,w_{a,h}(x)\,|\delta\rho_k|^{1+\ell}\,\bar\rho_k^\gamma.
\end{split}
\]
This finishes the proof of the lemma: As $1+\ell={\gamma}/{(\gamma-1)}$ is bounded (we recall that $\gamma>d/2$), if $a_\mu\leq C_*$ for $C_*>0$ well chosen, the last term in the r.h.s. is non positive.
\end{proof}
\subsection{Conclusion of the proof of Theorem \ref{maincompactness3}}
We combine Lemmas \ref{boundQhAniso} and \ref{divanisotropic} to get
 the following estimate
\[\begin{split}
\int_{h_0}^1 \int_{\Pi^{2d}}  
& \overline K_h(x-y) (w_{a,h}(x)+w_{a,h}(y)) \chi_a(\delta\rho_k)(t)
    dx dy \frac{dh}{h}  \\ 
& \le C\, |\log h_0|^\theta + \mbox{initial value},
\end{split}\]
with $0<\theta<1$. 
 We now follow the same steps as in the proof of Theorem \ref{maincompactness} with the weight $w_0(t,x)+w_0(t,y)$. We define $\omega_\eta=\{w_{a,h}(t,x)\leq \eta\}$ and note that
\[\begin{split}
\int_{\Pi^{2d}} {\cal K}_{h_0}(x-y)\,\chi_a(\delta\rho_k)&=\int_{h_0}^1\int_{\Pi^{2d}} {\overline K}_h(x-y)\,\chi_a(\delta\rho_k)\,\frac{dh}{h}\\
&=\int_{h_0}^1\int_{x\in\omega_\eta^c\ or\ y\in\omega_\eta^c} {\overline K}_h(x-y)\,\chi_a(\delta\rho_k)\,\frac{dh}{h}\\
&+\int_{h_0}^1\int_{x\in\omega_\eta\ and\ y\in\omega_\eta} {\overline K}_h(x-y)\,\chi_a(\delta\rho_k)\,\frac{dh}{h}.\\
\end{split}\]
Now
\[\begin{split}
&\int_{h_0}^1\int_{x\in\omega_\eta^c\ or\ y\in\omega_\eta^c}  \overline K_h(x-y)\,\chi_a(\delta\rho_k)\,\frac{dh}{h}\\
&\qquad\leq \frac{1}{\eta}\, \int_{h_0}^1 \int_{\Pi^{2d}}  \overline K_h(x-y) (w_{a,h}(x)+w_{a,h}(y)) \chi_a(\delta\rho_k)
\frac{dh}{h}\leq C\, |\log h_0|^\theta,
\end{split}
\]
while by point iii) in Prop. \ref{weightprop} and using the $L^p$ bound on $\rho$, for some $\theta>0$
\[
\begin{split}
\int_{h_0}^1\int_{x\in\omega_\eta\ and\ y\in\omega_\eta} {\overline K}_h(x-y)\,\chi_a(\delta\rho_k)\,\frac{dh}{h}&\leq 2\,\int_{h_0}^1\int_{\Pi^{2d}} {\overline K}_h(x-y)\,\rho_k^{1+\ell}\,\ind_{\overline K_h\star w_a\leq \eta}\,\frac{dh}{h}\\
&\leq \frac{C\,|\log h_0|}{|\log \eta|^{\theta}}.
\end{split}
\]
Hence we have
\[
\begin{split}
\int_{\Pi^{2d}} {\cal K}_{h_0}(x-y)\,\chi_a(\delta\rho_k)(t)
&\leq C\,|\log h_0|\,\left(\frac{|\log h_0|^{\theta-1}}{\eta}+\frac{1}{|\log \eta|^{\theta}}\right)\\
&\leq \frac{C\,\|{\cal K}_{h_0}\|_{L^1}}{|\log |\log h_0||^{\theta}},
\end{split}
\]
by optimizing in $\eta$ and recalling that $\|{\cal K}_{h_0}\|_{L^1}=|\log h_0|$.

 Using Prop. \ref{kernelcompactness} together with Lemma \ref{lemmakernelaveraged}, one concludes that $\rho_k$ is locally compact in $t,\;x$.  Thus we
conclude the proof of Theorem \ref{maincompactness3}.
%
\section{Proof of Theorems \ref{MainResultPressureLaw} and \ref{MainResultAniso}: Approximate sequences \label{sec-proof-mainresults}}
In this section, we construct approximate systems that allow
to use Theorems \ref{maincompactness} and \ref{maincompactness3} to prove  Theorems \ref{MainResultPressureLaw} and \ref{MainResultAniso}.

We do not need to use here pressure laws $P$ which depend explicitly on $t$ or $x$, which simplifies the form of the assumptions on the behavior of $P$, either \eqref{monotone} or \eqref{nonmonotone}.
\subsection{From regularized systems with added viscosity to no viscosity}
Our starting point for global existence is the following regularized system  
\begin{equation}\label{Barocompeps}
\left\{
\begin{array}{rl}
& \partial_t \rho_{k} + {\rm div} (\rho_{k} u_{k}) 
     = \alpha_k \Delta \rho_{k},\\
& \partial_t (\rho_k u_k) 
      + {\rm div}(\rho_k u_k\otimes u_k)  
      - \mu \Delta u_k - (\lambda+\mu) \nabla {\rm div } u_k -{\cal A}_\eps\star u_k 
      \\
& \hskip3cm  + \nabla P_\eps(\rho_k) + \alpha_k \nabla \rho_k\cdot \nabla u_k
                                    =\rho_k f,
\end{array}
\right. 
\end{equation}
with the fixed initial data
\begin{equation}
\rho_k\vert_{t=0} =\rho^0,\qquad \rho_k\, u_k\vert_{t=0} =\rho^0\,u^0.\label{datainitial}
\end{equation} 
The pressure $P_\eps$ satisfies the bound
 \eqref{gammacontrol} with $\gamma>3d/(d+2)$ uniformly in $\eps$, that is 
\[
C^{-1}\,\rho^\gamma-C\leq P_\eps(\rho)\leq C\,\rho^\gamma+C,
\]
implying that $e(\rho)\geq C^{-1}\,\rho^{\gamma-1}-C$.
 In addition we ask that $P_\eps$ satisfies the quasi-monotone property \eqref{monotone} but possibly depending on $\eps$, {\em i.e.} there exists $\rho_{0,\eps}$ s.t.
\[
(P_\eps(s)/s)'\geq 0\ \mbox{for all}\ s\geq \rho_{0,\eps}.
\]
And finally we impose an $\eps$ dependent bound \eqref{gammacontrol} on $P_\eps$ for some $\gamma_\eps>d$
\begin{equation}
C^{-1}_\eps\,\rho^{\gamma_\eps}-C\leq P_\eps(\rho)\leq C_\eps\,\rho^{\gamma_\eps}+C,\label{gamma_eps}
\end{equation}
Similarly ${\cal A}_\eps$ is assumed to be a given smooth function, possibly depending on $\eps$ but such that the operator defined by
\[
{\cal D}_\eps\,f=- \mu \Delta f - (\lambda+\mu) \nabla {\rm div } f-{\cal A}_\eps\star f
\]   
satisfies \eqref{generalstress} and \eqref{generalstress2} uniformly in $\eps$. 

  As usual the equation of continuity is regularized by means of an artificial viscosity term
and the momentum balance is replaced by a Faedo-Galerkin approximation to eventually reduce the problem
on $X_n$, a finite-dimensional vector space of functions. 
 
This approximate system can then be solved by a standard procedure: The velocity $u$ of the approximate momentum equation is looked as a fixed point of a suitable integral operator. Then 
given $u$, the approximate continuity equation is solved directly by means of the standard theory of linear parabolic equations. This methodology concerning the compressible Navier--Stokes equations is well explained and described in the reference books \cite{Fei}, \cite{FeNo}, \cite{NoSt}. We omit the rest of this classical (but tedious) procedure and we assume that we have well posed and global weak solutions to \eqref{Barocompeps}.

We now use the classical energy estimates detailed in subsection \ref{basicaprioriest}. Note that they remain the same in spite of the added viscosity in the continuity equation because in particular of the added term $\alpha_k \nabla \rho_k\cdot \nabla u_k$ in the momentum equation. Let us summarize the {\it a priori} estimates that are obtained
\[
\sup_{k,\eps}\,\sup_t \int_{\Pi^{d}} (\rho_k\,|u_k|^2+\rho_k^\gamma)\,dx<\infty,\qquad
\sup_{k,\eps}\int_0^T\int_{\Pi^d} |\nabla u_k|^2\,dx\,dt<\infty,
\] 
from \eqref{energyestimates}. The estimate \eqref{gainintegrability} may actually require the $\eps$ dependent bound from \eqref{gamma_eps} with $\gamma_\eps>d$ to control the additional term $\alpha_k\,\nabla\rho_k\cdot\nabla u_k$. It provides
\[
\sup_{k}\int_0^T\int_{\Pi^d} \rho_k^p(t,x)\,dx\,dt<\infty,
\]
with $p_\eps=\gamma_\eps+2\,\gamma_\eps/d-1$ or $p=\gamma+2\,\gamma/d-1$ which means $p>2$ as $\gamma>3\,d/(d+2)$. This bound may not be independent of $\eps$ because it requires \eqref{gamma_eps}. However since $\alpha_k$ vanishes at the limit, it still implies that any weak limit $\rho$ of $\rho_k$ satisfies
\[
\sup_{\eps}\int_0^T\int_{\Pi^d} \rho_k^p(t,x)\,dx\,dt<\infty,
\]
for $p=\gamma+2\,\gamma/d-1$.

From these a priori estimates, we obtain \eqref{boundrhok} and \eqref{bounduk}.
And from those bounds it is straightforward to deduce that $\rho_k\,u_k$ and $\rho_k\,|u_k|^2$ belong to $L^q_{t,x}$ for some $q>1$, uniformly in $k$. Therefore using the continuity equation in \eqref{Barocompeps}, we deduce \eqref{rhorho_t}. Using the momentum equation, we obtain \eqref{rhout} but this bound (and only this bound) is not independent of $\eps$ because of ${\cal A}_\eps$.
 
Finally taking the divergence of the momentum equation and inverting $\Delta$
\[\begin{split}
(\lambda+2\mu)\,\div u_k=&P_\eps(\rho_k)+\Delta^{-1}\div(\partial_t(\rho_k\,u_k)+\div(\rho_k\,u_k\otimes u_k))\\
&-\Delta^{-1}\,\div (\rho_k\,f+{\cal A}_\eps \star u_k)+\alpha_k\Delta^{-1}\,\div (\nabla\rho_k\cdot\nabla u_k),
\end{split}\]
which is exactly \eqref{divu0} with $\mu_k=\lambda+2\,\mu$ satisfying \eqref{elliptic} and compact, while
\[
F_k=-\Delta^{-1}\,\div (\rho_k\,f+{\cal A}_\eps \star u_k)+\alpha_k\Delta^{-1}\,\div (\nabla\rho_k\cdot\nabla u_k).
\]
The first term in $F_k$ is also compact in $L^1$ since ${\cal A}_\eps$ is smooth for a fixed $\eps$. On the other hand
\[
\alpha_k\Delta^{-1}\,\div (\nabla\rho_k\cdot\nabla u_k)
\]
converges to $0$ in $L^1$ since $\sqrt{\alpha_k}\nabla \rho_k$ is uniformly bounded in $L^2$ and $\nabla u_k$ is as well in $L^2$. Therefore $F_k$ is compact in $L^1$.
  We may hence apply point i) of Theorem \ref{maincompactness} to obtain the compactness of $\rho_k$ in $L^1$. Then extracting converging subsequences, we can pass to the limit in every term (see subsection \ref{classicalapproach} for instance) and obtain the existence of weak solutions to
\begin{equation}\label{Barocompepslim}
\left\{
\begin{array}{rl}
& \partial_t \rho + {\rm div} (\rho u) 
     = 0,\\
& \partial_t (\rho u) 
      + {\rm div}(\rho u\otimes u)  
      - \mu \Delta u - (\lambda+\mu) \nabla {\rm div } u -{\cal A}_\eps\star u 
  + \nabla P_\eps(\rho) =\rho f.
\end{array}
\right. 
\end{equation} 
\subsection{General pressure laws: End of proof  (Theorem \ref{MainResultPressureLaw})}
    Consider now a nonmonotone pressure $P$ satisfying \eqref{gammacontrol} and \eqref{nonmonotonehyp}.  
     Let us fix $c_{0,\eps}=1/ \eps$ and define
\[
P_\eps(\rho)=P(\rho)\quad \mbox{if}\ \rho\leq c_{0,\eps},\qquad 
P_\eps(\rho)=  P(c_{0,\eps}) + C (\rho-c_{0,\eps})^\gamma,
  \quad \mbox{if}\ \rho\geq c_{0,\eps}. 
\] 
If $\gamma\leq d$ then we also add to $P_\eps$ a term in $\eps\,\rho_\eps^{\gamma_\eps}$ to satisfy \eqref{gamma_eps}.
 Note that $P_\eps$ is Lipschitz, converges uniformly to $P$ on any compact interval. 
Due to \eqref{gammacontrol}  there exists $\rho_{0,\eps}$ with $\rho_{0,\eps}\to +\infty$ as $\eps\to +\infty$ 
 such that for $\rho\ge \rho_{0,\eps}$,
$$(P_\eps(s)/s)'= (P_\eps'(s)s - P(s))/s^2
       \ge \bigl(C (\gamma-1) (\rho -c_{0,\eps})^\gamma - P(c_{0,\eps})\bigr)/s^2 \ge 0.$$
  The approximate pressure $P_\eps$ still satisfies \eqref{gammacontrol} with $\gamma$ and due
  to the previous inequality it satisfies \eqref{monotone} for $\rho\geq \rho_{0,\eps}$ and 
  \eqref{nonmonotone} in the following sense: For all $s\ge0$
 $$ |P'_\eps(s)| \le \overline P s^{\tilde\gamma-1} \ind_{s\le c_{0,\eps}} 
                                 + C (\gamma-1) s^{\gamma-1} \ind_{s\ge c_{0,\eps}}.$$
     As a consequence, we have existence of weak solutions $(\rho,\; u)$ to \eqref{Barocompeps} for this $P_\eps$ (assuming ${\cal A}_\eps=0$) for any $\eps>0$. 
 Consider a sequence $\eps_k\rightarrow 0$ and the corresponding sequence $(\rho_k,\;u_k)$ of weak solutions to \eqref{Barocompeps}.  

Because the previous {\it a priori} estimates were uniform in $\eps$ for the limit $\rho$ and $u$, (including \eqref{rhout} since ${\cal A}_\eps=0$), then the sequence $(\rho_k,\;u_k)$
satisfies all the bounds \eqref{boundrhok}, \eqref{bounduk}, \eqref{rhout}, \eqref{rhorho_t} and \eqref{nonmonotone}. 
 
Moreover the representation \eqref{divu0} still holds with $\mu_k=2\mu+\lambda$, compact in $L^1$ and satisfying \eqref{elliptic}.
 Finally the exponent $p$ in \eqref{boundrhok} can be chosen up to $\gamma+2\,\gamma/d-1$. Since $\gamma>3\,d/(d+2)$ then $p>2$ and since $\gamma>(\tilde\gamma+1)\,d/(d+2)$ then one has $p>\tilde\gamma$ as well. 

Therefore all the assumptions of point ii) of Theorem \ref{maincompactness} are satisfied and one has the compactness of $\rho_k$. Extracting converging subsequences of $\rho_k$ and $u_k$, one passes to the limit in every term. Note in particular that $P_{\eps_k}(\rho_k)$ converges in $L^1$ to $P(\rho)$ by the compactness of $\rho_k$, the uniform convergence of $P_{\eps_k}$ to $P$ on compact intervals and by truncating $P_{\eps_k}(\rho_k)$ for large values of $\rho_k$ since the exponent $p$ in \eqref{boundrhok} is strictly larger than $\gamma$. 

This proves the global existence in Theorem \ref{MainResultPressureLaw}. The regularity of $\rho$ follows from Theorem \ref{maincompactness2}, which concludes the proof of Theorem \ref{MainResultPressureLaw}.   
\subsection{Anisotropic viscosities: End of proof (Theorem \ref{MainResultAniso})}
For simplicity, we take $f=0$. Consider now a ``quasi-monotone'' pressure $P$ satisfying \eqref{strictmonotone}.  Observe that $P$ then automatically satisfies \eqref{gammacontrol}
since $P(0)=0$.
  To satisfy \eqref{monotone}, we have to modify $P$ on an interval $(c_{0,\eps}, +\infty)$
  with $c_{0,\eps} \to + \infty$ when $\eps\to +\infty$.
More precisely we consider $P_\eps$ as defined in the previous subsection
  \[
P_\eps(\rho)=P(\rho)\quad \mbox{if}\ \rho\leq c_{0,\eps},\qquad 
P_\eps(\rho)=  P(c_{0,\eps}) + C (\rho-c_{0,\eps})^\gamma,
  \quad \mbox{if}\ \rho\geq c_{0,\eps}. 
\] 
Remark that since $\gamma>d$ here, we never need to add a term with $\gamma_\eps$.

  Now given any smooth kernel, for instance $\overline K$, we define
\[
{\cal A}_\eps\star\,u=\div(\delta A(t)\,\nabla\,\overline K_\eps\star u).
\]
Because of the smallness assumption on $\delta A(t)$, the operator ${\cal D}_\eps$ satisfies \eqref{generalstress} and \eqref{generalstress2} uniformly in $\eps$. Therefore we have existence of global weak solutions to \eqref{Barocompepslim} with this choice of $P_\eps$ and ${\cal A}_\eps$.
  We again consider a sequence of such solutions $(\rho_k,\;u_k)$ corresponding to some sequence $\eps_k\rightarrow 0$.
 Because the estimates are uniform in $\eps$ for \eqref{Barocompepslim}, we have again that this sequence satisfies the bounds \eqref{boundrhok}, \eqref{bounduk}, \eqref{rhorho_t}.  
We now assume that 
\[
\gamma>\frac{d}{2}\,\left[\left(1+\frac1d\right)+\sqrt{1+\frac1{d^2}}\right],
\]
implying that $p$ in \eqref{boundrhok} is strictly larger than $\gamma^2/(\gamma-1)$. Moreover observe that
\[
\|{\cal A}_{\eps_k} u_k\|_{L^2_{t}\,H^{-1}_x}\leq C\,\|\nabla u_k\|_{L^2_{t,x}},
\]
such that \eqref{rhout} is also satisfied.

For simplicity, we assume that $\delta A$ has a vanishing trace (otherwise just add the corresponding trace to $\mu$).
Denote $a_\mu=2\,\|\delta A\|_{L^\infty}/(2\,\mu+\lambda)$ and the operator $E_k$
\[
E_k\,u=-\div\left(\frac{\delta A(t)}{2\,\|\delta A\|_{L^\infty}}\,\nabla\,\overline K_\eps\star u\right)
=-\sum_{ij}\frac{\delta A_{ij}(t)}{2\,\|\delta A\|_{L^\infty}}\,\partial_{ij} \overline K_{\eps_k}\star u.
\]
For $a_\mu$ small enough, $\Delta-a_\mu\, E_k$ is a uniform elliptic operator so that $(\Delta-a_\mu\,E_k)^{-1}\,\Delta$ is bounded on every $L^p$ space, uniformly in $k$. For the same reason, $A_\mu=(\Delta-a_\mu\,E_k)^{-1}\,E_k$ is bounded on every $L^p$ space with norm less than $1$ and can be represented by a convolution with a singular integral.

Taking the divergence of the momentum equation in \eqref{Barocompeps}, one has
\[
(2\mu+\lambda)\,\big(\Delta\,\div u_k-a_\mu\,E_k\,\div u\big)=\Delta\,P(\rho_k)+\div(\partial_t(\rho_k\,u_k)+\div(\rho_k\,u_k\otimes u_k)).
\]
Just write $\Delta\,P_\eps(\rho_k)=(\Delta-a_\mu\,E_k)\,P_\eps(\rho_k)+a_\mu\,E_k\,P_\eps$, take the inverse of $\Delta-a_\mu\,E_k$ to obtain
\[\begin{split}
(2\mu+\lambda)\,\div u=&P(\rho_k)+a_\mu\,(\Delta-a_\mu\,E_k)^{-1}\,E_k\, P(\rho_k)\\
&+(\Delta-a_\mu\,E_k)^{-1}\, \div(\partial_t(\rho_k\,u_k)+\div(\rho_k\,u_k\otimes u_k)),
\end{split}\]
which is exactly \eqref{divu0noniso} with $\nu_k=(2\,\mu+\lambda)^{-1}$.  As a consequence, if $a_\mu\leq C_*$, which is implied by $\|\delta A\|_{L^\infty}$ small enough, then Theorem \ref{maincompactness3} applies and $\rho_k$ is compact. Passing to the limit again in every term proves Theorem \ref{MainResultAniso}. Note that $P_{\eps_k}(\rho_k)$ converges
in $L^1$ to $P(\rho)$ for the same reason than in the previous subsection.

\bigskip

\noindent {\it The case with $D(u)$ instead of $\nabla u$.}
Let us finish this proof by remarking on the different structure in the case with symmetric stress tensor $\div(A\,D\,u)$. In that case, one cannot find $\div u_k$ by taking the divergence of the momentum equation but instead we have to consider the whole momentum equation. Let us write it as
\[
{\cal E}\,u_k=\nabla P(\rho_k)+\partial_t(\rho_k\,u_k)+\div(\rho_k\,u_k\otimes u_k),
\] 
with ${\cal E}$ the elliptic vector-valued operator
\[
{\cal E}\,u=\mu\,\Delta u+(\mu+\lambda)\,\nabla \div u+\div(\,\delta A\,D\,u).
\] 
The operator ${\cal E}$ is invertible for $\delta A$ small enough as one can readily check in Fourier for instance where $-\hat{\cal E}$ becomes a perturbation of $\mu\,|\xi|^2\,I+(\mu+\lambda)\,\xi\otimes \xi$. Its inverse has most of the usual properties of inverses of scalar elliptic operator (with the exception of the maximum principle for instance). Therefore, one may still write
\[
\div u_k=\div\,{\cal E}^{-1}\,\nabla P(\rho_k)+\div\,{\cal E}^{-1}\,(\partial_t(\rho_k\,u_k)+\div(\rho_k\,u_k\otimes u_k)),
\]
leading to the variant \eqref{divu0nonisosym} of the simpler formula \eqref{divu0noniso}.
%
\section{Extension to the Navier--Stokes-Fourier system\label{withtemperature}}
We present here a direct application of our compactness results to compressible Navier-Stokes systems with temperature. Those systems are considered to be more physically realistic. They also offer many examples of non monotone pressure laws thus being especially relevant to our approach without any thermodynamic stability assumption.
 
This chapter is in many respects a preliminary work, with non optimal assumptions. It is likely that further works would require estimates specifically tailored to the case under consideration, and for the same reason would require to specifically tailor the compactness Theorem \ref{maincompactness}. But we believe that the result shown here can provide a good general basis or first step for such future works.
\subsection{Model and estimates}
The heat-conducting compressible Navier--Stokes equations  read
 \begin{equation}\label{HeatCNS}
\left\{
\begin{array}{rl}
& \partial_t \rho + {\rm div} (\rho u) =0,\\
& \partial_t (\rho u) + {\rm div}(\rho u\otimes u) - {\rm div} \, {\cal S}  + \nabla P(\rho,\vartheta)  = 0,\\
& \partial_t (\rho E) + {\rm div}(\rho u E)  + {\rm div} (P(\rho,\vartheta) u) 
                                   =  {\rm div} (  {\cal S} u) + {\rm div} (\kappa(\vartheta) \nabla\vartheta),
\end{array}
\right. 
\end{equation} 
where $E= |u|^2/2 + e$ is the total energy with $P=P(\rho,\vartheta)$ and $e=e(\rho,\vartheta)$ respectively  stand for the  pressure and the  (specific) internal energy. As usual the system is supplemented with initial conditions. 

We denote the stress tensor ${\cal D}={\rm div}\, {\cal S}$ and as in the barotropic case,  we always assume that for some function $\bar \mu$
\[
\int \nabla u\,:\,{\cal S}\,dx\sim \int \bar \mu\,|\nabla u|^2\,dx.
\]
Instead of the temperature $\vartheta$, one could also choose as a third unknown the internal energy (or the entropy as defined below). Some formulas are easier when using the set of variables $(\rho,\vartheta)$ and some with the set $(\rho, e)$. 
For this reason and simplicity we here follow the classical notations of Thermodynamics for partial derivatives, denoting for instance $\partial_\vartheta f|_{\rho}$ if $f$ is the function $f(\rho,\vartheta)$ and $\partial_e f|_{\rho}$ 
if instead one considers the composition $f(\rho, \vartheta(\rho,e))$.
 
The total mass as well as the total energy of the system are constants of 
motion namely
\[
\int \rho(t,\cdot) \, dx = \int \rho_0 \, dx
\]
and 
\[
\int (\frac{1}{2} \rho |u|^2 + \rho e(\rho,\vartheta))(t,\cdot) \, dx 
= \int (\frac{1}{2} \rho_0 |u_0|^2 + \rho_0 e(\rho_0,\vartheta_0)) \, dx. 
\]
From the mathematical viewpoint, the heat-conducting compressible Navier--Stokes
system suffers the deficiency of strong  {\it a priori} bounds. For instance in comparison with the previous barotropic Navier--Stokes, the energy bound does not yield any $L^2_t H^1_x$ bound on the velocity.

\smallskip

 In order to be consistent with the second principle of Thermodynamics which implies
 the existence of the entropy as a closed differential form in the energy balance, the following
 compatibility condition, called "Maxwell equation'' between $P$ and $e$ has to be satisfied
\[
P = \rho^2{ \frac{\partial e}{\partial\rho}}\Bigl\vert_{\vartheta} 
    + \vartheta \frac{\partial P}{\partial\vartheta}\Big\vert_{\rho}.
\]
The specific entropy $s=s(\rho,\vartheta)$ is defined up to an additive constant by
\[
\frac{\partial s}{\partial \vartheta}\Big\vert_{\rho}= \frac{1}{\vartheta} \frac{\partial e}{\partial \vartheta}\Big\vert_{\rho}
    \qquad  \hbox{ and } \qquad
     \frac{\partial s}{\partial \rho}\Big\vert_{\vartheta}= - \frac{1}{\rho^2} \frac{\partial P}{\partial\vartheta}
      \Big\vert_{\rho}.
\]
An other important assumption on the entropy function is made,
$$\hbox{ the entropy } s \hbox{ is a concave function of } (\rho^{-1}, e).$$
This property ensures in particular the non negativity of the so-called $C_v$ coefficient
given by 
\[
C_v = \frac{\partial e}{\partial\vartheta}\Big\vert_{\rho} 
             = - \frac{1}{\vartheta^2} \frac{\partial^2 s}{\partial e^2}\Big\vert_{\rho}^{-1}. 
\]
If $(\rho,\vartheta)$ are smooth and bounded
from below away from zero and if the velocity field is smooth, then the total energy balance
can be replaced by the thermal energy balance
\[
C_v \rho (\partial_t\vartheta + u\cdot \nabla\vartheta)  - {\rm div}(\kappa(\vartheta)\nabla\vartheta) 
   = {\cal S} : \nabla u 
   - \vartheta \frac{\partial P(\rho,\vartheta)}{\partial\vartheta} {\rm div} u.
\]
Furthermore, dividing by $\vartheta$, we arrive at the entropy equation
\begin{equation}
\partial_t(\rho s) + {\rm div}(\rho s u) 
   - {\rm div} \Bigl(\frac{\kappa(\vartheta)\nabla\vartheta}{\vartheta}\Bigr) = 
   \frac{1}{\vartheta} \Bigl({\cal S} : \nabla u + \frac{\kappa|\nabla\vartheta|^2}{\vartheta} \Bigr).\label{entropyEq}
\end{equation}
\subsection{The entropy estimate through thermodynamical stability}
The first difficulty in the Navier--Stokes--Fourier system is to obtain a $L^2$ estimate on $\nabla u$, since, contrary to the barotropic case, it does not follow from the dissipation of energy.

Instead a first approach is to use the entropy. Coupling the equation on the total energy and the equation on $s$, we get
\begin{equation}\begin{split}
& \int (\frac{1}{2} \rho |u|^2 + \rho [e(\rho,\vartheta) - \vartheta_\star s])(t,\cdot) \, dx 
+ \int_0^t\int  \frac{\vartheta_\star}{\vartheta} \Bigl({\cal S}:\nabla u 
     +  \frac{\kappa |\nabla\vartheta|^2}{\vartheta}\Bigr) \\
& = \int (\frac{1}{2} \rho_0 |u_0|^2 + \rho_0 [e(\rho_0,\vartheta_0)- \vartheta_\star s_0]) \, dx, 
\end{split}\label{energytotal}
\end{equation}
for any constant temperature $\vartheta_\star$.  This is exactly the type of modified energy estimate that we are looking for as it provides a control on $\kappa(\vartheta)\,|\nabla \vartheta|^2/\vartheta^2$ and with any reasonable choice of stress tensor, a control on $|\nabla u|^2/\vartheta$. 

However it is useless unless one can bound
\[
H_{\vartheta_\star}(\rho,\vartheta) = \rho [e(\rho,\vartheta) -  \vartheta_\star s].
\]
Observe that (using the mass conservation) 
\[
\frac{d}{dt} \int H_{ \vartheta_\star}(\rho,\vartheta) = 
   \frac{d}{dt}  \int \Bigl[H_{ \vartheta_\star}(\rho,\vartheta)  
     - \frac{\partial H_{ \vartheta_\star}( \rho_\star,  \vartheta_\star)}{\partial \rho} (\rho-  \rho_\star)
     - H_{ \vartheta_\star}( \rho_\star,  \vartheta_\star)\Bigr].
\]
Remark that, using the Maxwell equation and the definition of the entropy, we have
\[
  \frac{\partial H_{ \vartheta_\star}(\rho,\vartheta)} {\partial \vartheta} =
    \frac{\rho}{\vartheta} (\vartheta-  \vartheta_\star) \frac{\partial e(\rho,\vartheta) }{\partial \vartheta}, \qquad
    \frac{\partial^2 H_{ \vartheta_\star}(\rho, \vartheta_\star)} {\partial \rho^2} 
      =  \frac{1}{\rho} \frac{\partial P(\rho, \vartheta_\star)}{\partial\rho}.
\]
Those estimates were developed in the works by E. {\sc Feireisl} and collaborators, under the assumption of thermodynamic stability
\begin{equation}
\frac{\partial e(\rho,\vartheta)}{\partial \vartheta} >0, \qquad 
\frac{\partial P(\rho,\vartheta)}{\partial \rho} > 0. \label{thermostability}
\end{equation}
The meaning of such condition is that both the specific heat at constant volume $C_v$
and the compressibility of the fluid are positive.
However as we have already seen, the latter condition is violated by several physical state law, such as the standard Van der Waals equation of state.

Under \eqref{thermostability}, one obtains that $H_{\vartheta_\star}$ is increasing in $\vartheta$ for $\vartheta>\vartheta_\star$ and decreasing for $\vartheta<\vartheta_\star$, that is that  $H_{\vartheta_\star}$ has a minimum at $\vartheta=\vartheta_\star$. The second part of \eqref{thermostability} implies that  $H_{\vartheta_\star}$ is convex in $\rho$. 

One chooses accordingly $\rho_\star$ as the minimum of $H_{\vartheta_\star}(\rho,\vartheta_\star)$ and decompose the quantity linked to $H_{\overline\vartheta}(\rho,\vartheta)$ into two parts
\begin{equation}\begin{split}
& H_{ \vartheta_\star}(\rho,\vartheta)  
     - \frac{\partial H_{ \vartheta_\star}( \rho_\star,  \vartheta_\star)}{\partial \rho} (\rho-  \rho_\star)
     - H_{ \vartheta_\star}( \rho_\star,  \vartheta_\star) \\
& = \bigl[H_{ \vartheta_\star}(\rho,\vartheta)  - H_{ \vartheta_\star}(\rho, \vartheta_\star) \bigr] \\
&   \hskip1cm  + \bigl[H_{ \vartheta_\star}(\rho, \vartheta_\star) 
     - \frac{\partial H_{ \vartheta_\star}( \rho_\star, \vartheta_\star)}{\partial \rho} (\rho-  \rho_\star)
     - H_{ \vartheta_\star}( \rho_\star,  \vartheta_\star)\bigr] \geq 0
\end{split}\label{Hsign}
\end{equation}
by the thermodynamic stability assumptions on $\partial e(\rho,\vartheta)/\partial\vartheta$
and $\partial P(\rho,\vartheta)/\partial\rho$.
 As a consequence under \eqref{thermostability}, one deduces that
\begin{equation}
\int_0^t\int  \Bigl(\bar\mu\,\frac{|\nabla u|^2}{\vartheta} 
     +  \kappa(\vartheta)\,\frac{\mu\, |\nabla\vartheta|^2}{\vartheta^2}\Bigr)\,dx\,dt \leq constant.\label{entropydissipation}
\end{equation}

\subsection{The pressure laws covered by previous works.}
In every previous work, the viscous stress tensor is assumed to be isotropic
\[
{\cal S}= \mu\,(\nabla u+\nabla u^T)+\lambda\,\div u\, {\rm Id},
\]
with coefficients $\mu,\;\lambda$ either constant or depending only on $\vartheta$.

\medskip

\noindent {1) \it The pressure law as a perturbation of the barotropic case.} It is due  to {\sc E. Feireisl} who considered  pressure laws under the form
\[
P(\rho,\vartheta) = P_c(\rho) + \vartheta P_\vartheta(\rho),
\]
where
\begin{equation}\begin{split}
& P_c(0)=0, \quad P'_c(\rho) \ge a_1 \rho^{\gamma-1} - b \hbox{ for } \rho >0,\\
& P_c(\rho) \le a_2 \rho^\gamma + b \hbox{ for all } \rho \ge 0,\\
& P_\vartheta(0) = 0, \qquad P'_\vartheta(\rho) \ge 0 \hbox{ for all } \rho \ge 0,\\
& P_\vartheta(\rho) \le c(1+\rho^\beta),
\end{split}\label{pressureheat}
\end{equation}
and
\[
\gamma >d/2, \qquad \beta < \frac{\gamma}{2} \hbox{ for } d=2, \quad
    \beta=  \frac{\gamma}{3} \hbox{ for } d=3
\]
with constants $a_1>0$, $a_2$, $b$ and $P_c$, $P_\vartheta$ in
${\cal C}[0,+\infty) \cap {\cal C}^1(0,+\infty)$.
   In agreement with Maxwell law and the entropy definition, it implies the following form
on the internal energy 
\[
e(\rho,\vartheta) = \int_{ \rho_\star}^\rho \frac{P_c(s)}{s^2} ds 
                               + Q(\vartheta),
\]
where $Q'(\vartheta)= C_v(\vartheta)$ (specific heat at constant volume). The entropy is given by
\[
s(\rho,\vartheta) = \int_{ \rho_\star}^\vartheta \frac{C_v(s)}{s} ds -  S_\vartheta(\rho),
\]
where $S_\vartheta$ is the thermal pressure potential given through 
$\displaystyle S_\vartheta(\rho) = \int_{ \rho_\star}^\rho {P_\vartheta(s)}/{s^2} ds$.
   The heat conductivity coefficient $\kappa$ is assumed to satisfy
  \[
\kappa_1(\vartheta^\alpha + 1)  \le \kappa(\vartheta) \le \kappa_2(\vartheta^\alpha+1) 
     \hbox{ for all } \vartheta \ge 0,
\]
 with constants $\kappa_1>0$ and $\alpha\ge 2$.
  The thermal energy $Q=Q(\vartheta) = \int_0^\vartheta C_v(z) dz$ has not yet been determined and is assumed
 to satisfy $\int_{z\in [0,+\infty)} C_v(z)>0$ and $C_v(\vartheta) \le c(1+\vartheta^{\alpha/2-1})$.
    Because the energy and pressure satisfy \eqref{thermostability}, the estimate on $H_\vartheta$ gives a control on $\rho^\gamma$ in $L^\infty(0,T;L^1(\Omega))$
and through \eqref{entropydissipation} a control on $\vartheta$ in $L^2(0,T; L^6(\Omega))$ in dimension $3$ and in $L^2(0,T;L^p(\Omega))$
for all $p<+\infty$ in dimension $2$.

\medskip

Because \eqref{entropydissipation} does not provide an $H^1_x$ bound on $u$, {\sc E. Feireisl}  
combines it with a direct energy estimate (see below).
   Therefore one obtains the exact equivalent of estimates \eqref{energyestimates} in this case. Using similar techniques, one then proves \eqref{gainintegrability} for $0<a<\min (1/d, 2d/\gamma -1)$. 
  The proof of compactness and existence follows the main steps shown in subsection \ref{classicalapproach}.


 \bigskip

 \noindent {2) \it  Pressure laws with large radiative contribution.} It is due to {\sc E. Feireisl} and {\sc A.~Novotny }  who consider   pressure laws exhibiting  both coercivity of type $\rho^\gamma$ and $\vartheta^4$ for large densities and temperatures namely 
\[
P(\rho,\vartheta) = \vartheta^{\gamma/(\gamma-1)} Q(\frac{\rho}{\vartheta^{1/(\gamma-1)}}) +\frac{a}{3}  \vartheta^4 \hbox{ with }a>0, \quad \gamma >3/2,
\]
with 
\[
Q\in {\cal C}^1([0,+\infty)), \qquad Q(0)=0, \qquad Q'(Z) >0 \hbox{ for all } Z\ge 0,
\]
and
\[
\lim_{Z\to +\infty} \frac{Q(Z)}{Z^\gamma} = Q_\infty >0.
\]
In agreement to Maxwell law and the definition of entropy, it implies the following
form on the internal energy and the entropy
\[
e(\rho,\vartheta) = \frac{1}{(\gamma-1)} \frac{\vartheta^{\gamma/(\gamma-1)}}{\rho}
     Q(\frac{\rho}{\vartheta^{1/(\gamma-1)}}) + a \frac{\vartheta^4}{\rho},
\]
and 
\[
s(\rho,\vartheta) = S\left(\frac{\rho}{\vartheta^{1/(\gamma-1)}}\right) + \frac{4a}{3}\frac{\vartheta^3}{\rho}.
\]
They
impose
\[
0 < -S'(Z) =\frac{1}{\gamma-1} \frac{\gamma Q(Z) - Q'(Z)Z}{Z} < c  < +\infty \hbox{ for all } Z>0,
\]
with $\lim_{Z\to +\infty} S(Z) = 0$ so that thermodynamical stability \eqref{thermostability} holds.
Therefore the energy provides uniform bounds in $L^\infty_t L^1_x$ for $\vartheta^4$ and 
 $\rho^\gamma$.
   One assumes in this case that the viscosities and heat conductivity satisfy
 \[
\mu, \lambda \in {\cal C}^1([0,+\infty)) \hbox{ are Lipschitz and } \quad
     \mu\, (1+\vartheta) \le \mu(\vartheta), \quad 0\le \lambda(\vartheta), \quad  \mu_0>0,
 \]
 and 
\[
\kappa \in {\cal C}^1([0,+\infty), \qquad 
    \kappa_0 (1+\vartheta^3) \le \kappa(\vartheta) \le \kappa_1 (1+\vartheta^3), 
    \qquad 0<\kappa_0\le \kappa_1.
 \]
Almost everywhere convergence of the temperature is obtained using
the radiation term. Extra integrability on $P(\rho,\vartheta)$ can be derived just as in the  barotropic case. 
Finally the same procedure as in the barotropic case is followed to have compactness on the density, relying heavily on the monotonicity of the pressure ${\partial P(\rho,\vartheta)}/{\partial\rho}>0$. This gives global existence of weak solutions (in a sense that we precise later).

\bigskip

  With respect to these previous works, we focus here, as in the barotropic case, in removing the assumption of monotonicity on the pressure law, which is consistent with the discussion of such laws in subsection \ref{discusslaws}.

\subsection{The direct entropy estimate \label{secenergytemp}}
We explain here the general framework for our result on the Navier--Stokes--Fourier system. The estimates here closely follow the ones pioneered  by {\sc P.--L. Lions}, and {\sc E. Feireisl} and {\sc A. Novotny}; just as in the barotropic case, our contribution is the new compactness argument not the energy estimates. With respect to the previous discussion, we only present them here in a more general context as in particular we will not need the monotonicity of $P$.  

If one removes the monotonicity assumption on $P$ then thermodynamic stability \eqref{thermostability} does not hold anymore. Following {\sc  P.--L. Lions}, it is however possible to obtain the estimate \eqref{entropydissipation} directly by integrating the entropy equation \eqref{entropyEq}
\[
\int_0^t\int  \Bigl(\mu\,\frac{|\nabla u|^2}{\vartheta} 
     +  \kappa{(\vartheta)}\frac{ |\nabla\vartheta|^2}{\vartheta^2}\Bigr)\,dx\,dt \leq C\,\int \rho\,s(t,x)\,dx.
\]
Therefore the bound \eqref{entropydissipation} holds under the general assumption that there exists $C$ s.t.
\begin{equation}
s(\rho,\vartheta)\leq C\,e(\rho,\vartheta)+\frac{C}{\rho}.\label{boundentropyenergy}
\end{equation}
Recall that
\[
e=m(\vartheta)+\int_{\rho^\star}^\rho \frac{(P(\rho',\vartheta) - \vartheta\partial_{\vartheta}P(\rho',\vartheta))}{\rho'^2}\,d\rho',
\]
and
\[
\partial_\rho s= 
-  \frac{\partial_\vartheta P}{\rho^2}.
\]
We also have that $\partial_\vartheta s =\partial_\vartheta e / \vartheta$, therefore as long as $m(\vartheta)\geq 0$ with 
\[
\int_{\vartheta^*}^\vartheta \frac{m'(s)}{s}\,ds \le C(1+m(\vartheta)),
\]
and
\[
 -\int_{\rho^*}^\rho \frac{\partial_\vartheta P}{\rho'^2}\,d\rho' \le
  C+C\,\int_{\rho^*}^\rho \frac{P -\vartheta \partial_\vartheta P}{\rho'^2}\,d\rho',
\]
for some $C>0$  then \eqref{boundentropyenergy} is automatically satisfied  
and one obtains the bound \eqref{entropydissipation}.
Moreover if $e(\vartheta,\rho)\geq \rho^{\gamma-1}/C$ then one also has that 
$\rho\in L^\infty_t\,L^\gamma_x$.
  Assuming now that
\[
\kappa_1\,(\vartheta^\alpha+1)\leq \kappa(\vartheta)\leq \kappa_2\,(\vartheta^\alpha+1),
\]
with $\alpha \ge 2$, one deduces from \eqref{entropydissipation} that $\log \vartheta\in L^2_t H^1_x$ and $\vartheta^{\alpha/2}\in L^2_t\,H^1_x$ or 
by Sobolev embedding $\vartheta\in L^{\alpha}_t\,L^{\alpha/(1-2/d)}_x$. 
   By a H\"older estimate, it is also possible to obtain a Sobolev-like, $L^{p_1}_t W^{1,p_2}_x$, bound on $u$
\[\begin{split}
&\int_0^T\left(\int |\nabla u|^{p_2}\,dx\right)^{p_1/p_2}\,dt
\leq \left(\int_0^T\int \frac{|\nabla u|^2}{\vartheta}\,dx\,dt\right)^{p_1/2}\\
&\qquad\left(\int_0^T\left(\int \vartheta^{p_2/(2-p_2)}\,dx\right)^{p_1(2-p_2)/(p_2(2-p_1))}dt\right)^{(2-p_1)/2}<\infty,
\end{split}\] 
provided that $p_2/(2-p_2)=\alpha/(1-2/d)$ and $p_1/(2-p_1)=\alpha$ or
\begin{equation}
p_1=\frac{2\,\alpha}{1+\alpha},\quad p_2=\frac{2\,\alpha\,d}{d\,(\alpha+1)-2}.\label{sobolevexputemp}
\end{equation}
Unfortunately this Sobolev estimate does not allow to derive an equivalent of \eqref{gainintegrability}. Instead one requires a $L^2_t H^1_x$ estimate on $u$ (the critical point is in fact the $L^2_t$ with value in some Sobolev in $x$).
   Instead one can easily extend the argument by {\sc E. Feireisl} and {\sc A. Novotny}: For any $\phi(\rho)$, one can write
\[
\frac{1}{2}\frac{d}{dt} \int \rho |u|^2 + \frac{d}{dt} \int \phi(\rho)
    + \int {\cal S}:\nabla u 
    = \int (P(\vartheta,\rho)-\phi'(\rho)\,\rho+\phi(\rho))\,{\rm div u}.
\]
This leads to the assumption that there exists some $\phi$ s.t.
\begin{equation}
\begin{aligned}
& C^{-1} \rho^\gamma - C   \le \phi(\rho) \le C \rho^\gamma + C, \\
&  \left|P(\vartheta,\rho)-\phi'(\rho)\,\rho+\phi(\rho)\right|\leq C
\Bigl( \rho^{\beta_1}+\vartheta^{\beta_2}+\sqrt{\rho e(\vartheta,\rho)}\Bigr),
\label{pressureminusbarotrope}
\end{aligned}
\end{equation}
 with
\begin{equation}
\beta_1\leq \frac\gamma 2,\quad \beta_2\leq \frac\alpha 2. \label{assumptionbetas}
\end{equation}
Indeed, with \eqref{pressureminusbarotrope}, one has
\[
\int_0^T\int {\cal S}:\nabla u\,dx\,dt\leq E(\rho^0,u^0,\vartheta^0)
    +C\,\int_0^T\int \Bigl(\rho^{\beta_1}+\vartheta^{\beta_2}+\sqrt{\rho e(\vartheta,\rho)}\Bigr)\,|\div u|\,dx\,dt.
\] 
By \eqref{generalstress}, this leads to
\begin{equation}
\int_0^T\int |\nabla u|^2\,dx\,dt\leq C\,E(\rho^0,u^0,\vartheta^0)+C\,\|\nabla u\|_{L^2_{t,x}}\,\|\rho^{\beta_1}+\vartheta^{\beta_2}\|_{L^2_{t,x}},\label{H1withtemp}
\end{equation}
and the desired $H^1$ bound follows from \eqref{assumptionbetas}. It is now possible to follow the same steps to obtain an equivalent of \eqref{gainintegrability} if $\gamma>d/2$
\begin{equation}
\int_0^T \int_\Omega \rho^{\gamma+a}\,dx\,dt\leq C(T,E(\rho^0,u^0,\vartheta^0)),\quad\mbox{for all}\ a< 1/d \label{gainintegrabilitytemp}
\end{equation}
Note here that the assumptions \eqref{pressureminusbarotrope}-\eqref{assumptionbetas} are likely not optimal and will be improved in future works. They nevertheless already cover the two examples presented above.
\subsection{Main result in the heat-conducting case}
    For convenience, we repeat here all the assumptions presented above and the concluding
mathematical result which may be obtained.  Let $P(\rho,\vartheta)$ a positive pressure law such that 
 that for some $C>0$ and $\gamma>d$
\begin{equation}
\left\{\begin{aligned}
&P(\rho,\vartheta) \hbox{ such that } -\int_{\rho^*}^\rho \frac{\partial_\vartheta P}{\rho'^2}\,d\rho' \le
  C+C\,\int_{\rho^*}^\rho \frac{P -\vartheta \partial_\vartheta P}{\rho'^2}\,d\rho' , \\
& e(\rho,\vartheta)= m(\vartheta) + \int_{\rho^\star}^\rho \frac{(P(\rho',\vartheta) -
     \vartheta \partial_\vartheta P(\rho',\vartheta))}{\rho'^2}\,d\rho'\geq \frac{\rho^{\gamma-1}}{C} +\frac{\vartheta^{\gamma_\vartheta}}{C\,\rho},\\
&\quad   \hbox{ with }  m(\vartheta)\ge 0 \\
& \hskip1cm  \hbox{ and }  \int_{\vartheta^*}^\vartheta \frac{m'(s)}{s}\,ds \le C(1+m(\vartheta))\leq C\,(1+\vartheta^{\alpha\,(\gamma+a-1)/2(\gamma+a)}), \\
&\kappa_1\,(\vartheta^\alpha+1)\leq \kappa(\vartheta)\leq \kappa_2\,(\vartheta^\alpha+1),\quad \mu,\;\lambda \ \mbox{constant and } \alpha\ge 4,\\
& C^{-1} \rho^\gamma- C\leq \phi(\rho)\leq C\,\rho^\gamma + C, \\ 
&|P(\rho,\vartheta)-\phi'(\rho)\,\rho+\phi(\rho)|
  \leq C\,\rho^{\beta_1}+C\,\vartheta^{\beta_2}+C\,\sqrt{\rho e(\rho,\vartheta)},  \\
  & |\partial_\vartheta P(\rho,\vartheta)| \le C \rho^{\beta_3} + C \vartheta^{\beta_4}
\end{aligned}\right.\label{assumptionswithtemp}
\end{equation}
for
\begin{equation}\left\{\begin{aligned}
&\beta_1\leq \frac{\gamma}{2},\quad \beta_2< \frac{\alpha}{2},\\
& \beta_3 < \frac{\gamma+ a+1}{2}, \qquad \beta_4 < \frac{\alpha}{2},
\\
&\frac{2}{d}\,\mu+\lambda>0, \quad \gamma_{\vartheta} \ge 0\\
\end{aligned}\right.\label{coefftemp}
\end{equation}
where we recall that $a<\min\left(1/d, 2\gamma/d -1\right)=1/d$ since $\gamma>d$ here.
  We also assume that the specific heat is positive (as is necessary for the physics) {\em i.e.}
\begin{equation}
C_v=\partial_\vartheta e(\rho,\vartheta)>0,\quad \forall \rho,\;\vartheta,\label{specificheat>0}
\end{equation}
and that the pressure contains a radiative part
\begin{equation}
\partial_\vartheta^2 P(\rho=0,\vartheta)>0.\label{radiative}
\end{equation}
We do not need to impose monotony on $P$ and it is enough that
\begin{equation}
\begin{split}
&\bigl|\partial_\rho P(\rho,\vartheta)\bigr|\leq C\,\rho^{\gamma-1}+C\,\vartheta^b
\qquad \hbox{ with }  0 \le  b<\frac{\alpha}{2}.\\
\end{split}\label{nonmonotonewithtemp}
\end{equation}
Finally the initial data has to satisfy
\begin{equation}\label{initemp1}
\begin{aligned}
&\rho_0 \in L^\gamma(\Pi^d), \quad
\vartheta_0 \in L^{\gamma_{\vartheta}}(\Pi^d)  \\
& \hskip2cm \hbox{ with } \rho_0\ge 0,  \quad 
      \vartheta_0 >0 \hbox{ in } \Pi^d   \qquad \hbox{ and }
       \int_{\Pi^d} \rho_0 = M_0>0,  
      \end{aligned}
      \end{equation}
 and
 \begin{equation}\label{initemp2}
      E_0 = \int_{\Pi^d} \Bigl(\frac{1}{2}\frac{|(\rho u)_0|^2}{\rho_0} + \rho_0 e(\rho_0,\vartheta_0)\Bigr) <+\infty.
\end{equation}

\bigskip

Then,  on the condition of approximate solutions construction, we can prove the following result.
\begin{theorem}
Assume that \eqref{assumptionswithtemp}--\eqref{nonmonotonewithtemp} are satisfied with $\gamma>d$
and that the initial data satisfy \eqref{initemp1} and \eqref{initemp2}.
Then there exists
\begin{equation}\begin{split}
&\rho\in L^\infty([0,\ T],\ L^\gamma(\Pi^d))\cap L^{\gamma+a}([0,\ T]\times\Pi^d),\
 \quad \hbox{ for all } \  a< 1/d,\\
&u\in L^2([0,\ T],\ H^1(\Pi^d))\cap L^\infty([0,\ T],\ L^2_\rho(\Pi^d)),\\
&\vartheta\in L^\alpha([0,\ T],\ L^{\alpha/(1-2/d)}(\Pi^d)),\ \log \vartheta\in L^2 ([0,\ T],\ H^1(\Pi^d)),
\end{split}\label{estimateswithtemp}
\end{equation}
with $\vartheta>0$ \hbox{a.a. on } $(0,T)\times\Pi^d$ that is solution in the sense of distributions to
\begin{equation}
\begin{split}
& \partial_t \rho + {\rm div} (\rho u) =0,\\
& \partial_t (\rho u) + {\rm div}(\rho u\otimes u) - \mu\,\Delta u-(\lambda+\mu)\nabla \div u  + \nabla P(\rho,\vartheta)  = 0,
\end{split} \label{cont+momentwithtemp}
\end{equation}
while the equation on the temperature is satisfied in the following sense: In the sense of distribution, the entropy solves
\begin{equation}
\partial_t(\rho\,s(\rho,\vartheta))+\div(\rho\,s(\rho,\vartheta)\,u)-\div\left(\frac{\kappa(\vartheta)\,\nabla \vartheta}{\vartheta}\right)\geq \frac{1}{\vartheta}\left({\cal S}:\nabla u+\frac{\kappa\,|\nabla\vartheta|^2}{\vartheta}\right),
\label{localenergyineq}
\end{equation}
with ${\cal S}=\mu\,(\nabla u+\nabla u^T)+\lambda\,\div u\,{\rm Id} $ and recalling that
\[
s(\rho,\vartheta)=-\int_{\rho^\star}^\rho \frac{\partial_\vartheta P}{\rho'^2}(\rho',\vartheta)\,d\rho'+\int_0^\vartheta \frac{m'(\vartheta')}{\vartheta'}\, d\vartheta'. 
\]
This is supplemented by the total energy property
\begin{equation}
\int(\rho\frac{|u|^2}{2}+\rho\,e(\rho,\vartheta))\,dx  =   \int(\rho^0\frac{|u^0|^2}{2}+\rho^0\,e(\rho^0,\vartheta^0))\,dx.
\label{energyconswithtemp}
\end{equation}\label{Mainresultwithtemp}
and the initial data conditions satisfied by $(\rho,\rho u, \rho s)$ in a weak sense
$$\rho\vert_{t=0} = \rho_0, \qquad \rho u\vert_{t=0} = (\rho u)_0, \qquad \rho s\vert_{t=0^+} \ge \rho_0
     s(\rho_0, \vartheta_0).$$
\end{theorem}

\begin{remark}
   An example of pressure law included in Theorem \ref{Mainresultwithtemp} is a perturbation of the truncated virial expansion as described in subsection \ref{discusslaws} plus a radiative part namely
\[
P(\rho,\vartheta)=\rho^\gamma
                               +\vartheta\,\sum_{n=0}^{N} B_n(\vartheta)\,\rho^n 
\]
with $\gamma>2N\geq 4$. Choosing $m=constant$ for simplicity in this example, this leads to 
\[
e(\rho,\vartheta)=m+\frac{\rho^{\gamma-1}}{\gamma-1}
   -\sum_{n\ge 2}^N \vartheta^2\,B_n'(\vartheta)\,\frac{\rho^{n-1}}{n-1}
   - \vartheta^2 B'_1(\vartheta) \log \rho
   + \vartheta^2 B'_0(\vartheta) \frac{1}{\rho} ,
\]

\noindent For simplicity,  let us assume that $B_1= constant$, which is the normal virial assumption, see {\rm  \cite{Fei}} for example, so that this term vanishes. 
The entropy reads
\[
s(\rho,\vartheta)=-\sum_{n\ge 2}^N (\vartheta\,B_n'(\vartheta)+B_n(\vartheta))\,\frac{\rho^{n-1}}{n-1}
  +B_1\,\log \rho + (\vartheta B_0'(\vartheta)+ B_0(\vartheta)) \frac{1}{\rho}
\]

\medskip

\noindent {\rm 1)} We assume that the pressure contains a radiative part, namely that $B_0$ is convex in $\vartheta$ with $ C^{-1}\,\vartheta^{b-1}\leq B_0(\vartheta) \leq \vartheta^{b-1}$ and $ C^{-1}\,\vartheta^{b-2}\leq B_0'(\vartheta) \leq \vartheta^{b-2}$ where $2\le b \le \alpha/2$. This already satisfies \eqref{radiative}. 

\medskip

\noindent {\rm 2)} For $n\geq 2$ , the coefficients $B_n$ can have any sign but we require a concavity assumption: 
\[
\frac{d}{d\vartheta}(\vartheta^2\,B_n') \leq 0.
\]
 This ensures, with the assumption on $B_0$, that the specific heat $C_v = \partial_\vartheta e (\rho, \vartheta)$ satisfies {\rm \eqref{specificheat>0}}. This is again a classical assumption for the virial, see still {\rm \cite{Fei}}.  Note that it would be enough to ask this concavity of some of them and moreover that this is automatically satisfied if $B_n\sim -\vartheta^\nu$, that is precisely for the coefficients contributing to the non-monotony of $P$ in $\rho$.

\medskip

\noindent {\rm 3)} We also require some specific bounds on the $B_n$ namely that there exist $\bar B_n$ and $\eps>0$ s.t.
\begin{eqnarray}
&& \displaystyle 
\nonumber |\vartheta\, B_n'(\vartheta )|+|B_n(\vartheta)|\leq C\,\vartheta^{
    \frac{\gamma-n}{\gamma}\,b-1-\eps},\\
&&  \displaystyle 
|B_n(\vartheta)|+|\vartheta\, B_n(\vartheta )-\bar B_n|\leq C \vartheta^{\frac{\alpha}{2}\left(1-\frac{2n}{\gamma}\right)-\eps}.\label{assumeBn}
\end{eqnarray}
First of all this gives us a bound from below on $e$
\[\begin{split}
e(\rho,\vartheta)&\geq m+\frac{\rho^{\gamma-1}}{\gamma-1}+C^{-1}\,\frac{\vartheta^{b}}{\rho}-\sum_{n=2}^N \vartheta\,\vartheta^{\frac{\gamma-n}{\gamma}\,b-1-\eps} \,\frac{\rho^{n-1}}{n-1}\\
&\geq m+\frac{\rho^{\gamma-1}}{\gamma-1}+C^{-1}\,\frac{\vartheta^{b}}{\rho}-C\,\sum_{n=2}^N \left(\rho^{\gamma-1-\eps'} +\frac{\vartheta^{b-\eps'}}{\rho}\right),
\end{split}\]
by Young's inequality, so that this implies \eqref{assumptionswithtemp}$_{2}$. Assumption \eqref{nonmonotonewithtemp} is proved with an identical calculation. The same calculation also proves Assumption \eqref{assumptionswithtemp}$_{1}$ by showing that $s\leq C\,(\rho^{\gamma-1}+\vartheta^b+1)$.

\medskip

\noindent {\rm 4)}  Then choosing 
$$\phi(\rho)=\frac{\rho^\gamma}{\gamma-1}+\sum_{ 2\leq n\leq N} \bar B_n\,\frac{\rho^n}{n-1}
      + \bar B_0,$$ 
and using again the second part of \eqref{assumeBn}
 we have that
\[\begin{split}
|P-\phi'(\rho)\,\rho+\phi(\rho)|&\leq C\,\sum_{n\leq N} |\vartheta\, B_n(\vartheta )-\bar B_n|\,\rho^n\leq C\,\sum_{n\leq N} \vartheta^{\frac{\alpha}{2}\left(1-\frac{2n}{\gamma}\right)-\varepsilon }\,\rho^n\\
&\leq C\,N\,(\rho^{\gamma/2}+\vartheta^{\alpha/2-\eps}),
\end{split}\]
still by Young's inequality. This yields \eqref{assumptionswithtemp}$_{4,5}$ with the right inequalities on $\beta_1$ and $\beta_2$ in \eqref{coefftemp}.
 The same calculation also proves that $|\partial_\vartheta P|\leq C\,(\rho^{\gamma/2+\eps'}+\vartheta^{\alpha/2-\eps'})$ thus proving \eqref{assumptionswithtemp}$_{6}$ with bounds \eqref{coefftemp} on $\beta_3,\;\beta_4$.

\medskip

We will treat more general virial expansions in a future work. Note for a fixed $\vartheta$ then $P(\rho,\vartheta)$ is indeed increasing with respect to $\rho$ after a critical $\bar\rho_\vartheta$ which depends on $\vartheta$ and can be arbitrarily large where $\vartheta>>1$. This is the reason why $P$ does not satisfy any of the classical monotonicity assumption such as \eqref{strictmonotone} or \eqref{monotone} and why our new approach is needed.
  Our pressure law has two important parts: the radiative term (corresponding to $n=0$) to get compactness on the temperature and the $\phi(\rho)$ term to get compactness on the density 
in time and space. 
\end{remark}

\begin{remark} The assumption \eqref{nonmonotonewithtemp} is more demanding in the Navier-Stokes-Fourier setting (though satisfied by the example just above). It would be more natural to have instead
\begin{equation}
|\partial_\rho P(\rho,\vartheta)| \leq \bar P(\vartheta)\,\rho^{\gamma-1}
      +  C(\vartheta), \label{nonmonotonewithtempandcoeff}
\end{equation}
for some unbounded function $\bar P$ of $\vartheta$. This should be possible and require some modifications in the proof of Theorem \ref{maincompactness}, combined with appropriate estimates on $\vartheta$.  This  will be the subject of the forthcoming work {\rm \cite{BrFeJaNo}}.
    In general obtaining more optimal results for the Navier-Stokes-Fourier system seems to depend on adapting precisely the estimates and Theorem \ref{maincompactness} to the specific model under consideration. 
\end{remark}

\begin{remark}
With appropriate additional growth assumption, it is possible to show that the combination of \eqref{localenergyineq} and \eqref{energyconswithtemp} implies the usual energy equation (and not just an inequation as \eqref{localenergyineq}. The reason for the formulation \eqref{localenergyineq}-\eqref{energyconswithtemp} is that it cannot be proved without those additional assumptions. We refer to {\rm \cite{Fei}} and {\rm \cite{FeNo}}. 
\end{remark}

\begin{remark}
In Theorem \ref{Mainresultwithtemp}, the viscosity coefficients $\mu,\;\lambda$ are independent of the temperature $\vartheta$. Instead several models use temperature dependent coefficients $\mu(\vartheta),\;\lambda(\vartheta)$. To handle that case, the proof given below would have to be modified; the compactness of the temperature would have to be established first. 
\end{remark}

\noindent{\bf Sketch of the proof.} 
 The proof of Theorem \ref{Mainresultwithtemp} follows the same steps as the proof of Theorem \ref{MainResultPressureLaw}. For this reason we only sketch briefly the points which are similar and insist more on the differences induced by the presence of the temperature. 
    The general method is taken from \cite{FeNo} and is adapted to take advantage of our new compactness estimates.
   As usual one starts from an approximate model 
\begin{equation}
\begin{split}
& \partial_t \rho_k + {\rm div} (\rho_k u_k) =\alpha_k\,\Delta\rho_k,\\
& \partial_t (\rho_k u_k) + {\rm div}(\rho_k u_k\otimes u_k) - \mu\,\Delta u_k-(\lambda+\mu)\nabla \div u_k  + \nabla P_{\eps,\delta}(\rho_k,\vartheta_k)\\
&\hspace{5cm}+\alpha_k\,\nabla\rho_k\cdot\nabla u_k  = 0,\\
&\partial_t(\rho_k\,s_\eps(\rho_k,\vartheta_k)) + \div(\rho_k\,s_\eps(\rho_k,\vartheta_k)\,u_k) 
- \div (\frac{\kappa_\eps(\vartheta_k)}{\vartheta_k} \,\nabla \vartheta_k ) \\
&\hspace{2cm}= \sigma^1_{\varepsilon,\delta, \alpha_k} +   \frac{1}{\vartheta_k}\left({\cal S}_k:\nabla u_k
+\frac{\kappa_\eps(\vartheta_k)}{\vartheta_k}|\nabla \vartheta_k|^2\right)\\ 
&\hspace{3cm} + \alpha_k \frac{\Delta \rho_k}{\vartheta_k} \bigl(\vartheta_k s_\eps(\rho_k,\vartheta_k) 
- e_{\eps}(\rho_k,\vartheta_k)  - \frac{P_{\eps}(\rho_k,\vartheta_k) }{\rho_k} \bigr),
\end{split} \label{NSFapprox}
\end{equation}
with again ${\cal S}_k=\mu\,(\nabla u_k+\nabla u_k^T)+\lambda\,\div u_k\,Id$ and
          where $\sigma^1_{\varepsilon, \delta, \alpha_k} \in M^1([0,T]\times \Pi^d)$ vaguely  converges to $\sigma^1_{\varepsilon, \delta} \in M^1([0,T]\times \Pi^d)$.
The entropy actually only satisfies an inequality since $\sigma^1_{\eps,\delta,\alpha_k}$ is unknown and $s_\eps$ is hence not fully defined from \eqref{NSFapprox}. However we supplement this by imposing the conservation of the total energy
\[
\int \rho\,\left(\frac{|u|^2}{2}+e_{\eps,\delta}(\rho,\vartheta)\right)\,dx=
\int \rho^0\,\left(\frac{|u^0|^2}{2}+e_{\eps,\delta}(\rho^0,\vartheta^0)\right)\,dx.
\]
  Note that the last quantity in the entropy relation was already introduced to compensate some terms in the energy equality. Imposing the conservation of energy is a simple way of fixing the remaining additional terms (which are all positive) instead of writing them explicitly.

  Remark  that contrary to the barotropic case, we always need here two steps of approximations in the pressure (and hence energy), indexed by $\eps$ and $\delta$.
  The existence of such solutions may be done adapting the procedure for instance 
that we can find in \cite{FeNo} using appropriate Galerkin scheme, penalization terms
and precise estimates. This important part is out of the scope of our paper but will be fully described 
with more complicated  pressure law in \cite{BrFeJaNo} (replacing the assumption \eqref{nonmonotonewithtemp} by \eqref{nonmonotonewithtempandcoeff}). 

\bigskip

The first level of approximation is obtained by adding a barotropic correction
\begin{equation}
P_\delta(\rho,\vartheta)=P(\rho,\vartheta)+\delta\,(\rho^{\gamma}+ \rho^2). \label{barotropcorr}
\end{equation}
 From the approximated pressure, one defines as usual the approximate energy
\begin{eqnarray}
e_\delta(\rho,\vartheta) 
&& \nonumber = m(\vartheta)
+\int_{\rho^*}^\rho \frac{(P_\delta (\rho',\vartheta)  
    - \vartheta \partial_\vartheta P_\delta (\rho',\vartheta)}{\rho'^2}\,d\rho'  \\
&&     =  e(\rho,\vartheta) +
    \frac{\delta}{\gamma -1} \,\rho^{\gamma-1} 
    +  \delta \rho,
\label{energydelta}
\end{eqnarray} 
and the entropy also has the straightforward expression
\begin{equation}
s_\delta(\rho,\vartheta)=\int_0^\vartheta \frac{m'(s)}{s}\,ds
-\int_{\rho^*}^\rho \frac{\partial_\vartheta P_\delta (\rho',\vartheta)}{\rho'^2}\,d\rho'=s(\rho,\vartheta).\label{entropydelta}
\end{equation}

In the second stage of approximation, indexed by $\eps$, we first replace $\kappa(\vartheta)$ by $\kappa_\eps(\vartheta)$. We assume that $\kappa_\eps=\kappa+\tilde \kappa_\eps$ with $\tilde \kappa_\eps$ a positive function which vanishes for $2\,\eps<\vartheta<1/2\,\eps$ and satisfies
\begin{equation}
\tilde \kappa_\eps(\vartheta)\longrightarrow +\infty\ \mbox{as}\ \vartheta\rightarrow \eps\ \mbox{or}\ \eps^{-1},\qquad \int_{\eps}^{\eps^{-1}} \tilde\kappa_\eps(\vartheta)\,d\vartheta<\infty.
\label{truncatetemp}
\end{equation} 
This is slightly different from \cite{FeNo} but gives similar bounds on the temperature while being simpler to handle in our case.
 Then we define $P_{\eps}$ as an approximate pressure law of $P_\delta$ truncated for large densities, namely for $\bar\rho_\eps=\eps^{-1}$, we define
\begin{equation}
P_{\eps}(\rho,\vartheta)=P(\rho,\vartheta)\quad\mbox{if}\ \rho\leq\bar\rho_\eps,\quad P_{\eps}(\rho,\vartheta)=P(\bar\rho_\eps,\vartheta)   
 \>   \mbox{if}\ \rho>\bar\rho_\eps\label{truncatepressure}
\end{equation}
Note that since the temperature will be bounded by \eqref{truncatetemp}, it is not necessary to truncate it in \eqref{truncatepressure}. This leads to
\[
P_{\eps,\delta}(\rho,\vartheta)=P_{\eps}(\rho,\vartheta)+\delta\,(\rho^\gamma + \rho^2).
\]
We also use $\bar \rho_\eps$ to truncate the initial density $\rho^0$: We define $\rho^{0}_\eps=\rho^0$ if $\rho^0\leq \bar \rho_\eps$ and $\rho^{0}_\eps=\bar\rho_\eps$ otherwise. We impose $\rho^0_\eps,\;u^0,\vartheta^0$ as initial data for \eqref{NSFapprox}.

The energy density is still defined from the pressure $P_{\eps,\delta}$ or $P_\eps$, 
\begin{equation}
\begin{split}
&e_{\eps}(\rho,\vartheta)=m(\vartheta)
+\int_{\rho^*}^\rho \frac{(P_\varepsilon (\rho',\vartheta)  
    - \vartheta \partial_\vartheta P_\varepsilon (\rho',\vartheta)}{\rho'^2}\,d\rho',\\
& \nonumber e_{\eps,\delta}(\rho,\vartheta)=m(\vartheta)
+\int_{\rho^*}^\rho \frac{(P_{\varepsilon,\delta} (\rho',\vartheta)  
    - \vartheta \partial_\vartheta P_{\varepsilon,\delta} (\rho',\vartheta)}{\rho'^2}\,d\rho' \\
  &
  \hskip1.4cm  = e_\eps(\rho,\vartheta)+ \frac{\delta}{\gamma-1} \,\rho^{\gamma-1}
  + \delta \rho
\end{split}\label{energyeps}
\end{equation}
and so is the entropy which is actually independent of $\delta$ (just as $s_\delta$ is independent of $\delta$ in \eqref{entropydelta})
\begin{equation}
s_{\eps,\delta}(\rho,\vartheta)=s_\eps(\rho,\vartheta)=\int_0^\vartheta \frac{m'(s)}{s}\,ds
-\int_{\rho^*}^\rho \frac{\partial_\vartheta P_\varepsilon (\rho',\vartheta)}{\rho'^2}\,d\rho'.\label{entropyeps}
\end{equation}
Remark that for $\rho<\rho_\eps$ then $e_\eps(\rho,\vartheta)=e(\rho,\vartheta)$ and similarly $s_\eps(\rho,\vartheta)=s(\rho,\vartheta)$. This trivially implies in particular that $e_\eps(\rho^0_\eps,\vartheta^0)\geq 0$ and that
\[
\int_{\Pi^d} \rho_\eps^0\,e_\eps(\rho^0_\eps,\vartheta^0)\,dx\longrightarrow \int_{\Pi^d} \rho^0\,e(\rho^0,\vartheta^0)\,dx\quad\mbox{as}\ \eps\rightarrow 0,
\] 
which enables to propagate the energy. At later times, one may not have $\rho_k\leq \bar\rho_\eps$ and $e_\eps(\rho_k,\vartheta_k)$ may be negative at some points. Nevertheless from \eqref{assumptionswithtemp}$_{6}$, one has that for $\rho>\bar\rho_\eps$
\begin{equation}
\left|e_\eps(\rho,\vartheta)- e_\eps(\bar\rho_\eps,\vartheta)\right|\leq C\,\vartheta\,\frac{\bar\rho_\eps^{\beta_3}}{\rho} +C\,\frac{\vartheta^{\beta_4+1}}{\rho}.\label{diffe_eps}
\end{equation}

\medskip

The first step is to pass to the limit $\alpha_k\rightarrow 0$ in the system \eqref{NSFapprox} to obtain a global weak solution to  
\begin{equation}
\begin{split}
& \partial_t \rho + {\rm div} (\rho u) =0,\\
& \partial_t (\rho u) + {\rm div}(\rho u\otimes u) - \mu\,\Delta u-(\lambda+\mu)\nabla \div u  + \nabla P_{\eps,\delta}(\rho,\vartheta)=0,\\
& \partial_t(\rho\,s_\eps(\rho,\vartheta))+\div(\rho\,s_\eps(\rho,\vartheta)\,u)-\div\left(\frac{\kappa_\eps(\vartheta)\,\nabla \vartheta}{\vartheta}\right) \\
& \hskip5cm = \sigma^2_{\eps, \delta} +   \frac{1}{\vartheta}\left({\cal S}:\nabla u+\frac{\kappa_\eps\,|\nabla\vartheta|^2}{\vartheta}\right)
\end{split} \label{NSFapproxeps}
\end{equation}
where $ \sigma^2_{\varepsilon,\delta}\in M_+([0,T]\times \Pi^d )$ with a total mass bounded
uniformly with respect to $\varepsilon$ and $\delta$. The system is still supplemented by $\rho^0_\eps,\;u^0,\;\vartheta^0$ as initial data.

  Note that the entropy still only satisfies an inequality, given the unknown positive measure, however this is supplemented by the preservation of total energy
\[
\int \rho\,\left(\frac{|u|^2}{2}+e_{\eps,\delta}(\rho,\vartheta)\right)\,dx=
\int \rho^0\,\left(\frac{|u^0|^2}{2}+e_{\eps,\delta}(\rho^0,\vartheta^0)\right)\,dx.
\]
As we remarked above, this identity occurs also for $\alpha_k$ fixed as mentioned in \cite{FeNo}; this is the reason for the appropriate corrections in $\alpha_k$ in the momentum and entropy equations. 

     Then one considers a sequence $\eps_k\to 0$ and the corresponding sequences $\rho_k,\;u_k, \;\vartheta_k$ of solution to \eqref{NSFapproxeps}. Passing again to the limit, we obtain global weak solutions to 
\begin{equation}
\begin{split}
& \partial_t \rho + {\rm div} (\rho u) =0,\\
& \partial_t (\rho u) + {\rm div}(\rho u\otimes u) - \mu\,\Delta u-(\lambda+\mu)\nabla \div u  + \nabla P_{\delta}(\rho,\vartheta)=0,\\
& \partial_t(\rho\,s(\rho,\vartheta))+\div(\rho\,s(\rho,\vartheta)\,u)-\div\left(\frac{\kappa(\vartheta)\,\nabla \vartheta}{\vartheta}\right) \\
& \hskip5cm 
= \sigma^3_\delta  +   \frac{1}{\vartheta}\left({\cal S}:\nabla u+\frac{\kappa\,|\nabla\vartheta|^2}{\vartheta}\right),
\end{split} \label{NSFapproxdelta}
\end{equation}
with $\sigma^3_\delta  \in M_+([0,T]\times \overline \Omega)$ with total mass bounded 
uniformly with respect to $\delta$ and with now $\rho^0,\;u^0,\;\vartheta^0$ as initial data.. We still  conserve the energy equality
\[
\int \rho\,\left(\frac{|u|^2}{2}+e_{\delta}(\rho,\vartheta)\right)\,dx=
\int \rho^0\,\left(\frac{|u^0|^2}{2}+e_{\delta}(\rho^0,\vartheta^0)\right)\,dx.
\]
Finally, we take a sequence $\delta_k\rightarrow 0$ and obtain global weak solutions
 as announced by Theorem \ref{Mainresultwithtemp}.

\bigskip

The three limits are somewhat similar with slightly different {\it a priori} estimates and in the case of \eqref{NSFapprox} to \eqref{NSFapproxeps} additional difficulties in obtaining those a priori estimates.  This limit uses point i) of Theorem \ref{maincompactness} and the other two point ii). For this reason, we sketch the limit \eqref{NSFapprox} to \eqref{NSFapproxeps} separately and both limits, \eqref{NSFapproxeps} to \eqref{NSFapproxdelta} and \eqref{NSFapproxdelta} to \eqref{cont+momentwithtemp}-\eqref{localenergyineq} at the same time.  

\bigskip

\noindent {\em The limit from \eqref{NSFapprox} to \eqref{NSFapproxeps}.} We use the a priori estimates described above, in particular in subsection \ref{secenergytemp}. The entropy estimate though requires more care for System \eqref{NSFapprox}. We emphasize here that most of the complications in this limit were already present in \cite{FeNo} and are solved here in a very similar manner. The novel features of our work are mostly present in the other two limits, from \eqref{NSFapproxeps} to \eqref{NSFapproxdelta} and from \eqref{NSFapproxdelta} to the limit system. 

However because we handle pressure terms that are more general, we can unfortunately not simply use the result of \cite{FeNo} to skip this first limit. Instead one has to check carefully that indeed at this level, estimates work the same.

First of all, given assumption \eqref{truncatetemp}, standard parabolic estimates imply that any solution to \eqref{NSFapprox} is either constant in $x$ or satisfies
\begin{equation}
\eps\leq\vartheta_k(t,x)\leq\frac{1}{\eps}.\label{minmaxtemp}
\end{equation}
Remark that since $\displaystyle \int_\eps^{\eps^{-1}} \tilde\kappa_\eps<\infty$, we do not have strict inequalities in \eqref{minmaxtemp}.
  In addition, the conservation of energy yields
\begin{equation}
\sup_{t,k}\int_{\Pi^d} \rho_k^\gamma(t,x)\,dx<\infty.\label{firstenergycons}
\end{equation}
It also useful to write a partial energy estimate, based only on the barotropic part of the potential energy
\begin{equation}\begin{split}
&\int \rho_k\,\left(\frac{|u_k|^2}{2}+\delta\,\left(\rho^2 +\frac{\rho_k^{\gamma}}{\gamma-1}\right)\right)\,dx  \\
&  =
\int \rho^0\,\left(\frac{|u^0|^2}{2}+\delta\,\left((\rho^0_k)^2+\frac{(\rho_k^0)^{\gamma}}{\gamma-1}\right)\right)\,dx\\
&\quad +\int_0^t\int \left(P_\eps(\rho_k,\vartheta_k)\,\div u_k-{\cal S}_k:\nabla u_k\right) \\
&\hskip2cm  -\alpha_k\,\delta\int_0^t\int \left(2+\gamma\,\rho_k^{\gamma-2}\right)\,|\nabla\rho_k|^2.\end{split}\label{energybarotrope}
\end{equation}
Since $P_\eps\leq C_\eps$ ($\vartheta_k$ is bounded and $P_\eps$ is truncated in $\rho_k$), we deduce that
\begin{equation}\begin{split}
&\int \rho_k\,\left(\frac{|u_k|^2}{2}+\delta\,\left(\rho^2 +\frac{\rho_k^{\gamma}}{\gamma-1}\right)\right)\,dx   \\
& 
 \hskip2cm \leq C_\eps-\alpha_k\,\delta\int_0^t\int \left(2+\gamma\,\rho_k^{\gamma-2}\right)\,|\nabla\rho_k|^2.\end{split}\label{energydissipk}
\end{equation}

\noindent Now observe that 
 \begin{equation}\label{relationimportante}
\begin{split}
&\frac{\Delta \rho_k}{\vartheta_k} \bigl(\vartheta_k\, s_{\eps} - e_{\eps} - \frac{P_{\eps}}{\rho_k}\bigr) \\
& = \div \Bigl[\bigl(\vartheta_k\, s_{\eps} - e_{\eps} -\frac{P_{\eps}}{\rho_k}  \bigr) \frac{\nabla\rho_k}{\vartheta_k}\Bigr]
+ {\partial_\rho P_{\eps}} \frac{|\nabla\rho_k|^2}{\rho_k\,\vartheta_k}  \\
& \hskip5.5cm - \Bigl(e_{\eps} + \rho\,\partial_\rho e_{\eps})\Bigr) \frac{\nabla\rho_k\cdot \nabla\vartheta_k}{\vartheta^2_k}
\end{split} 
\end{equation}
where the functions $P_{\eps}$, $e_{\eps}$ and $s_\eps$ are all taken on the points $\rho_k(t,x)$, $\vartheta_k(t,x)$. The identity \eqref{relationimportante} is obtained through the general relations between $s$, $e$  and $P$ and in particular
\[
\frac{\partial}{\partial\rho}  (\vartheta s - e - \frac{P}{\rho}) 
= -\frac{1}{\rho}\frac{\partial P}{\partial\rho} 
\]
and
\[
\frac{\partial}{\partial\theta}  
( s - \frac{e}{\vartheta} - \frac{P}{\rho\vartheta}) = \frac{1}{\vartheta^2}\Bigl(e+\rho 
\frac{\partial e}{\partial\rho}).
\]
Integrating in space and time the entropy equation in \eqref{NSFapprox}, we therefore get
 \begin{equation}\label{entropyestdelta}
\begin{split}
\int_0^t\int_{\Pi^d}   \frac{1}{\vartheta_k}\left({\cal S}_k:\nabla u_k
 +\frac{\kappa_\eps(\vartheta_k)}{\vartheta_k}|\nabla \vartheta_k|^2\right)
+ \alpha_k \int_0^t\int_{\Pi^d}  \partial_\rho P_{\varepsilon}  \frac{|\nabla\rho_k|^2}{\rho_k\vartheta_k} \\ 
 \le \alpha_k \int_0^t\int_{\Pi^d} \Bigl(e_{\eps} + \rho_k\,{\partial_\rho e_{\eps}})\Bigr) \frac{\nabla\rho_k\cdot \nabla\vartheta_k}{\vartheta_k^2}
+ \int_{\Pi^d} \rho_k s_\eps (t) - \int_{\Pi^d} (\rho_k s_\eps)(0).
\end{split}
\end{equation}
Therefore for $\alpha_k$ fixed we have extra
 terms in the right-hand side with respect to the case of System \eqref{NSFapproxeps} or \eqref{NSFapproxdelta} as described in subsection \ref{secenergytemp}, namely
\[
\alpha_k \int_{\Pi^d}
     (e_{\eps} + \rho_k\,\partial_\rho e_{\eps}) \frac{\nabla \rho_k\cdot \nabla\vartheta_k}{\vartheta_k^2} 
\]
 and an additional term on the left-hand side
\[
 \alpha_k\int_{\Pi^d}  \partial_\rho P_{\eps}\, \frac{|\nabla\rho_k|^2}{\vartheta_k}.
 \]
    On the other hand by the bounds \eqref{minmaxtemp} on $\vartheta_k$ and 
the truncation in $P_\eps$ imply that $s_{\eps}$ is bounded by some constant depending on $\eps$, hence trivially in that case
\[
\int_{\Pi^d} \rho_k\,s_\eps(\rho_k,\vartheta_k)\,dx\leq C_\eps.
\]
From the definition \eqref{truncatepressure} of $P_\eps$, we obtain that  
\begin{equation}
\partial_\rho P_{\eps}(\rho,\vartheta)=0,\qquad \forall\ \rho>\bar\rho_\eps,\quad\forall \eps\leq\vartheta_k(t,x)\leq\frac{1}{\eps}.\label{monotonetemp}
\end{equation}
From the bounds \eqref{minmaxtemp} on $\vartheta_k$, we deduce that $P_{\eps}$ satisfies
\[
 \alpha_k\int_{\Pi^d}  \partial_\rho P_{\eps}\, \frac{|\nabla\rho_k|^2}{\vartheta_k}\geq 
-C\,\alpha_k\int_{\rho_k\leq \bar\rho_\eps}  |\nabla\rho_k|^2.
 \]
On the other hand, by the dissipation \eqref{energydissipk}, we get a control on
$$\alpha_k \delta  \int_0^T \int_{\Pi^d}  |\nabla \rho_k|^2 \le C_{\epsilon}$$
and thus
\begin{equation} \label{rhsP}
 \alpha_k\int_{\Pi^d}  \partial_\rho P_{\eps}\, \frac{|\nabla\rho_k|^2}{\vartheta_k}\geq 
-C_{\delta, \varepsilon} .
\end{equation}
Turning to the last term in the entropy estimate, 
we  remark that $e_\eps$ and $\rho\partial_\rho e_\eps$ are bounded still from the truncation of $P_\eps$, so that one obtains
\begin{equation}\begin{split}
& \alpha_k \int_{\Pi^d}
     (e_{\eps} + \rho_k\,\partial_\rho e_{\eps}) \frac{\nabla \rho_k\cdot \nabla\vartheta_k}{\vartheta_k^2}       \\
    &\hskip2cm \leq \alpha_k\,C_\eps\int_{\Pi^d}\nabla\rho_k \cdot \nabla \vartheta_k\\
&\hskip2cm  
\leq \alpha_k^{3/2}\,\frac{\delta\,\eps}{2} \int_{\Pi^d} |\nabla\rho_k|^2+C_{\eps,\delta} \,\alpha_k^{1/2} \int_{\Pi^d} |\nabla\vartheta_k|^2 \\
& \hskip2cm
\leq  \alpha_k^{1/2} C_{\varepsilon, \delta}
\end{split}\label{rhse}
\end{equation}
Combining \eqref{rhsP}-\eqref{rhse} with \eqref{entropyestdelta} concludes the bound on the entropy.

\medskip

The rest of the a priori estimates described in subsection \ref{secenergytemp} follow in a straightforward manner. We hence summarize the uniform bounds here
\begin{equation}\begin{split}
& \eps\leq \vartheta_k\leq \frac{1}{\eps},\\
&\sup_{t,k} \int_{\Pi^d} \left(\rho_k^\gamma+\rho_k\,|u_k|^2+\vartheta_k^{\gamma_\vartheta}\right)\,dx
      +  \alpha_k \delta \int_0^T \int_{\Pi^d} |\nabla\rho_k|^2 < \infty,\qquad  \\
& \sup_k \int_0^T\int_{\Pi^d} \left(\rho_k^{\gamma + a}+|\nabla u_k|^2\right)\,dx\,dt<\infty,\quad a<1/d,\\
&\sup_k \int_0^T\int_{\Pi^d} \left(|\nabla \vartheta_k^{\alpha/2}|^2+|\nabla \log \vartheta_k|^2 \right)\,dx\,dt<\infty,\\
& \sup_k\,\|\partial_t (\rho_k\,s_\eps(\rho_k,\vartheta_k))\|_{M^1([0,T]; W^{-1,1}(\Pi^d))}<\infty.
\end{split}\label{boundwithalphak}
\end{equation}
Combining those bounds with the continuity and momentum equations shows that \eqref{boundrhok}, \eqref{bounduk}, \eqref{rhout}, and \eqref{rhorho_t}
are satisfied. 
  Moreover taking the divergence of the momentum equation in \eqref{NSFapprox} and inverting the Laplacian, one obtains
\[\begin{split}
(\lambda+2\,\mu)\,\div u_k=&P_{\eps,\delta}(\rho_k,\vartheta_k)+\Delta^{-1}\,\div\left(\partial_t(\rho_k\,u_k)+\div(\rho_k\,u_k\otimes u_k)\right)\\
&+\alpha_k\,\Delta^{-1}\,\div\,(\nabla\rho_k\cdot\nabla u_k).
\end{split}\]
This is exactly the identity \eqref{divu0} with $P_k(\rho,t,x)=P_{\eps,\delta}(\rho,\vartheta_k(t,x))$ and $\mu_k=\lambda+2\,\mu$ which satisfies \eqref{elliptic}. 
   Observe that as a consequence if $r$ is strictly larger than $\bar\rho_\eps$ then $P_k(t,x,\rho)$ is only $\delta(\rho^\gamma+\rho^2)$ plus a given function of $\vartheta_k$ thus satisfying the first part of \eqref{monotone}. In addition if $r$ and $s$ are less than $\bar\rho_\eps$
\[\begin{split}
|P_k(t,x,r)-P_k(t,y,s)|&\leq \max_{\vartheta\in [\vartheta_k(t,x),\ \vartheta_k(t,y)],\;\rho\leq\bar\rho_\eps} \Big(|\partial_\rho P_{\eps,\delta}(\rho,\vartheta)|\,|r-s|\\
&\qquad+|\partial_\vartheta P_{\eps,\delta}(\rho,\vartheta)|\,|\vartheta_k(t,x)-\vartheta_k(t,y)|\Big)\\
&\leq C_\eps\,(|r-s|+|\vartheta_k(t,x)-\vartheta_k(t,y)|).
\end{split}\]
Since $\vartheta_k$ belongs uniformly to an appropriate Sobolev space, that means that the last part of \eqref{monotone} is also satisfied.
We may hence apply the variant of Theorem \ref{maincompactness} yielding the compactness of $\rho_k$ in $L^1$.
 
Since $u_k\in L^2_t H^1_x$, we get the compactness on $\sqrt\rho_k u_k$ in $L^2_t L^2_x$ from the momentum equation.  The passage to the limit in the continuity equation is therefore as usual. 
  We also have compactness in space for $u_k$ and $\vartheta_k$ respectively from the viscosity and conductivity.  

   The next step is to obtain the compactness of the temperature. This relies on the entropy inequation in \eqref{NSFapproxeps} but follows the now classical approach and for this reason we only sketch the procedure:
\begin{itemize}
\item Observe that $\rho_k\,s_\eps(\rho_k,\vartheta_k)$ in some $L^p_{t,x}$ with $p>1$ from \eqref{boundwithalphak}. In fact $\rho_k s_\eps(\rho_k,\vartheta_k)$ is in $L^\infty_t L^\gamma_x$
uniformly using the definition of $s_\eps(\rho_k,\vartheta_k)$.
\item Observe that $\partial_t (\rho_k \,s_\eps(\rho_k,\vartheta_k))$ is uniformly in
$M^1([0,T]; W^{-1,1}(\Pi^d))$. Rewrite the entropy equation as
\[
\partial_t (\rho_k \,s_\eps(\rho_k,\vartheta_k))
   +{\rm div} (R_{\eps,k}^1) + R^2_{\eps,k}=\sigma^1_{\eps,\delta,\alpha_k},
\]
where $R^1_{\eps,k}$, $R^2_{\eps,k}$ contains all the other terms. By our previous estimates and in particular from \eqref{boundwithalphak}, we know that $\|R^i_{\eps,k}\|_{L^1_t L^1_x}$ is uniformly bounded in $k$ for $i=1,2$.  Hence integrating the previous equation
\[
\int_0^t\int_{\Pi^d} \sigma^1_{\eps,\delta,\alpha_k}
=  \int_{\Pi^d} \rho_k \,s_\eps(\rho_k,\vartheta_k)+\int_0^t \int_{\Pi^d} R^2_{\eps,k}
   - \int_{\Pi^d} [\rho_k \,s_\eps(\rho_k,\vartheta_k)]\vert_{t=0} 
\leq C_\eps. 
\]
Recalling that $\sigma^1_{\eps,\delta,\alpha_k}\geq 0$ then this implies that $\|\sigma^1_{\eps,\delta,\alpha_k}\|_{M^1_{t,x}}\leq C_\eps$. Finally this bounds $\partial_t (\rho_k \,s_\eps(\rho_k,\vartheta_k))$ in $M^1_{t,x}+L^1_t W^{-1,1}_x\subset M^1_t W^{-1,1}_x$. 
\item Obtain the $a.e.$ convergence of $\rho_k\, s_\eps(\rho_k,\vartheta_k)(t,x)$. From the first point and the compactness in space of $\vartheta_k$ and $\rho_k$, we deduce that $\rho_k\,s_\eps(\rho_k,\vartheta_k)$ is compact in space. From the second point, we deduce that it is compact in time. Therefore after extraction $\rho_k\,s_\eps(\rho_k,\vartheta_k)$ converges strongly in $L^1_{t,x}$ and after possibly further extraction, one has that there exists $l(t,x)$ s.t.
\[
\rho_k(t,x)\,s_\eps(\rho_k(t,x),\vartheta_k(t,x))\longrightarrow l(t,x), \quad \mbox{for}\ a.e.\ t,\;x.
\]
\item Obtain the $\rho$ a.e. convergence of $\vartheta_k(t,x)$. Extracting again, we have from the compactness of $\rho_k(t,x)$, that $\rho_k(t,x)$ converges $a.e.$ to $\rho(t,x)$. 
By the assumption on the specific heat \eqref{specificheat>0}, $s_\eps$ is invertible in $\vartheta_k$ ($P_\eps$ and so $s_\eps$ is not truncated in $\vartheta$ only in $\rho$). Therefore this implies that, $\rho$ a.e.,  $\vartheta_k(t,x)$ converges. 
\item Obtain the a.e. (not just $\rho$ a.e.) convergence of $\vartheta_k$. The integral defining $s$ and $s_\eps$ is singular at $\rho=0$ (because of the $\rho'^{-2}$ factor). In fact $\rho\,s(\rho,\vartheta)|_{\rho=0}=\partial_\vartheta P(\rho=0,\vartheta)$. Assumption \eqref{radiative} on the radiative part of $P$ guarantees that $\rho\,s$ is still invertible in $\vartheta$ even if $\rho=0$ and hence that $\vartheta_k$ converges a.e. even on the set where $\rho$ vanishes. 
\item Dominated convergence then implies the compactness of $\vartheta_k$ in $L^p([0,\ T]\times \Pi^d)$ for any $p<\infty$ (recall that for fixed $\epsilon$, we have $\vartheta_k$ 
 uniformly bounded in time and space).
\end{itemize}
This allows us to pass to the limit in every term of the momentum equation, including the barotropic term $\delta\, \rho^\gamma$ using the extra-integrability  $\rho_k\in L^{\gamma+a_\delta}$ in \eqref{boundwithalphak}. We can also simply pass to the limit in the entropy $s_\eps(\rho_k,\vartheta_k)$. 

Define $\bar\kappa_\eps$ by $\bar\kappa_\eps'=\kappa_\eps$. By \eqref{truncatetemp}, we know that $\bar\kappa_\eps(\vartheta)$ is bounded on $[\eps,\ \eps^{-1}]$. This lets us pass to the limit in $\kappa_\eps(\vartheta_k)\,\nabla \vartheta_k=\nabla \bar\kappa_\eps(\vartheta_k)$ by the compactness of $\vartheta_k$.
 Let us  pass to the limit next  in 
\[
\alpha_k\, \frac{\Delta\rho_k}{\vartheta_k}\,\left(\vartheta_k\,s_\eps-e_{\eps} -\frac{P_{\eps}}{\rho_k}\right).
\]
By relation \eqref{relationimportante}, the bounds on $s_\eps$, $e_{\eps}$, $P_{\eps}$, the a priori estimates in \eqref{rhse} and \eqref{boundwithalphak} and the dissipation term in \eqref{energydissipk}, this term will converge to $0$ if we can prove that
\[
\alpha_k\,\int_0^T\int_{\Pi^d} \frac{|\nabla \rho_k|^2}{\rho_k}\,dx\,dt\longrightarrow 0.
\]
On the other hand from the continuity equation, one has that
\[\begin{split}
&\int_{\Pi^d} \rho_k\,\log \rho_k(t,x)\,dx-\int_{\Pi^d} \rho_k^0\,\log \rho_k^0(x)\,dx+\int_0^t\int_{\Pi^d} \rho_k\,\div u_k\,dx\,dt\\
&\qquad=-\alpha_k\,\int_0^T\int_{\Pi^d} \frac{|\nabla \rho_k|^2}{\rho_k}\,dx\,dt.
\end{split}\]
By the compactness of $\rho_k$, we can pass to the limit in every term in the l.h.s. so
\[\begin{split}
&\int_{\Pi^d} \rho\,\log \rho(t,x)\,dx-\int_{\Pi^d} \rho^0\,\log \rho^0(x)\,dx+\int_0^t\int_{\Pi^d} \rho\,\div u\,dx\,dt\\
&\qquad=-\lim\alpha_k\,\int_0^T\int_{\Pi^d} \frac{|\nabla \rho_k|^2}{\rho_k}\,dx\,dt.
\end{split}\]
On the other hand, $\rho$ solves the continuity equation without any diffusion and since it belongs to $L^2$ by \eqref{boundwithalphak}, one also has that
\[
\int_{\Pi^d} \rho\,\log \rho(t,x)\,dx-\int_{\Pi^d} \rho^0\,\log \rho^0(x)\,dx+\int_0^t\int_{\Pi^d} \rho\,\div u\,dx\,dt=0,
\]
which has for consequence the required property
\[
\lim\alpha_k\,\int_0^T\int_{\Pi^d} \frac{|\nabla \rho_k|^2}{\rho_k}\,dx\,dt=0.
\] 
  Note now that is not possible to pass to the limit in the r.h.s. of the entropy equation
\eqref{NSFapprox}$_3$ without the use of appropriate defect measure.
   For instance the r.h.s. contains $\mu\,|\nabla u_k|^2$ which does not in general converge to $\mu\,|\nabla u|^2$, as this would require the compactness of $\nabla u_k$ in $L^2$.
     Instead one uses convexity properties  and compactness on the temperature $\vartheta_k$ to prove that
\[
w-\lim \frac{1}{\vartheta_k}\left({\cal S}_k:\nabla u_k+\frac{\kappa_\eps\,|\nabla\vartheta_k|^2}{\vartheta_k}\right)
=
\tilde\sigma^2_{\eps, \delta} +
 \frac{1}{\vartheta}\left({\cal S}:\nabla u+\frac{\kappa_\eps\,|\nabla\vartheta|^2}{\vartheta}\right)
\]
where $\tilde \sigma^2_{\eps, \delta}\in M_+([0,T]\times \overline \Omega)$.
     This leads to the equality in the entropy equation in \eqref{NSFapproxeps}$_3$ with a positive (and a priori unknown)  measure $\sigma^2_{\eps, \delta}$ which is the sum of $\tilde \sigma^2_{\eps, \delta}$ and the vague limit of $\sigma^1_{\varepsilon, \delta,\alpha_k}$ when $\alpha_k$ tends to zero. 
   Note we can also write an inequality ignoring the positive measure.
   Finally we conclude the proof of this limit by noticing that we may pass to the strong limit in all the terms of the energy conservation. Therefore we indeed obtain the equality in \eqref{energyconswithtemp}. 

Now use \eqref{assumptionswithtemp}$_1$ (which implies \eqref{boundentropyenergy}) to write the inequalities 
\[
s(\rho,\vartheta)  \le C [ e(\rho,\vartheta) + \rho^{-1}],\qquad
s_\varepsilon (\rho,\vartheta)  \le C [ e_\varepsilon (\rho,\vartheta) + f_\varepsilon(\rho)], \]
where $f_\varepsilon(\rho) = 1/\overline{\rho}_\varepsilon+1/\rho$.

We can hence integrate in $t,\;x$ the entropy equations to find
\begin{equation}
\begin{split}
& \int_0^T \int_{\Pi^d} \sigma^3_{\delta}\leq \sup_t \int_{\Pi^d} \rho_k\,s(\rho_k,\vartheta_k)\leq  C\,\sup_t \int_{\Pi^d}\left(\rho_k\,e(\rho_k,\vartheta_k)+1\right)\leq C,\\
& \int_0^T \int_{\Pi^d} \sigma^2_{\varepsilon, \delta}\leq \sup_t \int_{\Pi^d} \rho_k\,s_\eps(\rho_k,\vartheta_k) \\
&\hskip1.9cm  \leq  C\,\sup_t \int_{\Pi^d}\left(\rho_k\,e_\eps(\rho_k,\vartheta_k)+1+\rho_k\right)\leq C.
\end{split}\label{defectmeasurecontrol}
\end{equation}
This ensures that $\sigma^2_{\varepsilon,\delta},\;\sigma^3_\delta \in M^1([0,T]\times \overline \Omega)$ uniformly with respect to $\varepsilon$ and $\delta$. We again refer to \cite{Fei, FeNo} for the details of similar procedures. 

It only remains to pass to the limit in the energy equality: For all $t$
\[
\int_{\Pi^d}\left(\rho_k\,e_{\eps,\delta}(\rho_k,\vartheta_k) +\rho_k\,\frac{|u_k|^2}{2}\right)\,dx= \int_{\Pi^d}\left(\rho_\eps^0\,e_{\eps,\delta}(\rho_\eps^0,\vartheta^0) +\rho_\eps^0\,\frac{|u^0|^2}{2}\right)\,dx.
\]
It is straightforward to show the convergence of $\rho_k\,e_{\eps,\delta}(\rho_k,\vartheta_k)$ to $\rho\,e_{\eps,\delta}(\rho,\vartheta)$ in $L^1_{loc, t,x}$, from the compactness of $\rho_k,\;\vartheta_k$, the a priori estimates and the bounds on $e_{\eps,\delta}$ that can be derived from \eqref{assumptionswithtemp}$_{6}$, \eqref{nonmonotonewithtemp} (see \eqref{diffe_eps} for instance). Similarly the compactness of $\rho_k,\;\rho_k\,u_k$ and the a priori estimates imply the convergence of $\rho_k\,|u_k|^2$ in $L^1_{loc, t,x}$. Since this convergence is only $L^1$ in time and not $L^\infty$, we need one more step: Extract another subsequence s.t. for $a.e.$ $t$
\[
\int_{\Pi^d}\left(\rho_k\,e_{\eps,\delta}(\rho_k,\vartheta_k) +\rho_k\,\frac{|u_k|^2}{2}\right)\,dx\longrightarrow \int_{\Pi^d}\left(\rho\,e_{\eps,\delta}(\rho,\vartheta) +\rho\,\frac{|u|^2}{2}\right)\,dx,
\] 
which finally gives the conservation of energy $a.e.$ $t$ since the left-hand side is constant.

\bigskip

\noindent{\em The limits \eqref{NSFapproxeps} to \eqref{NSFapproxdelta} and \eqref{NSFapproxdelta} to the final system}. 
First we consider a sequence $\eps_k\rightarrow 0$ and corresponding sequences of solutions $\rho_k$, $u_k$, $\vartheta_k$ to \eqref{NSFapproxeps}. Some of the a priori estimates for this sequence of solutions to \eqref{NSFapproxeps} are obtained by simply keeping only the estimates in \eqref{boundwithalphak} uniform in $\eps$. The others, such as the entropy estimate, have to be derived again. But one proceeds exactly as described in subsection \ref{secenergytemp}, leading to  
\begin{equation}\eps\rightarrow0\left\{\begin{aligned}
&\sup_{t,k} \int_{\Pi^d} \left(\rho_k^\gamma +\rho_k\,|u_k|^2+\vartheta_k^{\gamma_\vartheta} \right)\,dx<\infty,\\
& \sup_k \int_0^T\int_{\Pi^d} \left(\rho_k^{\gamma+a_\delta}+|\nabla u_k|^2\right)\,dx\,dt<\infty
 \qquad \hbox{ for all } \> a_\delta<  \frac{1}{d},\\
&\sup_k \int_0^T\int_{\Pi^d} \left(|\nabla \vartheta_k^{\alpha/2}|^2+|\nabla \log \vartheta_k|^2 \right)\,dx\,dt<\infty,
\end{aligned}\right.\label{boundwitheps}
\end{equation}
Similarly for \eqref{NSFapproxdelta}, we obtain
\begin{equation}\delta\rightarrow0\left\{\begin{aligned}
&\sup_{t,k} \int_{\Pi^d} \left(\rho_k^\gamma+\rho_k\,|u_k|^2+\vartheta_k^{\gamma_\vartheta} \right)\,dx<\infty,\\
& \sup_k \int_0^T\int_{\Pi^d} \left(\rho_k^{\gamma+a}+|\nabla u_k|^2\right)\,dx\,dt<\infty
   \quad \hbox{ for all }  a< \frac{1}{d},\\
&\sup_k \int_0^T\int_{\Pi^d} \left(|\nabla \vartheta_k^{\alpha/2}|^2+|\nabla \log \vartheta_k|^2 \right)\,dx\,dt<\infty.
\end{aligned}\right.\label{boundwithdelta}
\end{equation}
The bounds \eqref{boundwithdelta} will of course imply \eqref{estimateswithtemp} after passage to the limit. 
   As before combining those bounds with the continuity and momentum equations shows that \eqref{boundrhok}, \eqref{bounduk}, \eqref{rhout}, and \eqref{rhorho_t}
are satisfied with $p=\gamma+a_\delta$ or $p=\gamma+a$ for System \eqref{NSFapproxeps} and for System \eqref{NSFapproxdelta}. 

\bigskip

     We now have to check that the pressure defined by $P_k(t,x,\rho_k)=P(\vartheta_k(t,x),\rho_k)$ if $\rho \le \epsilon_k^{-1}$
or $P_k(t,x,\rho_k)=P(\vartheta_k(t,x), \varepsilon_k^{-1})  $ if $\rho \ge \epsilon_k^{-1}$ satisfies
 the stability assumptions, namely  \eqref{nonmonotone}. 

\medskip

\noindent {\it First property.} 
  First of all by the bounds in \eqref{assumptionswithtemp}
$$
|P(\vartheta_k(x),\rho_k(x))-P(\vartheta_k(y),\rho_k(y))|\leq Q_k(t,x,y) + I 
$$
with 
$$ Q_k(t,x,y)
    =C\,(\rho_k(x)^{\beta_3}+\rho_k(y)^{\beta_3}+ (\vartheta_k(t,x))^{\beta_4}+(\vartheta_k(t,y))^{\beta_4})\,|\vartheta_k(t,x)-\vartheta_k(t,y)|$$
 and
\[
I= |P(\vartheta_k(t,x),\rho_k(x))-P(\vartheta_k(t,x),\rho_k(y))|.
\]
Note that $\vartheta$ is compact in $x$ in $L^1_{t,x}$ and uniformly bounded in $L^\alpha_t L^{\alpha/(1-2/d)}_x$. Since $\beta_4\leq \alpha-1\leq \alpha/2$ as $\alpha\geq 2$, one has that for some $\theta<1$
\[\begin{split}
&\int_0^T\int_{\Pi^{2d}} K_h(x-y)\, (\vartheta_k(t,x))^{\beta_4}+(\vartheta_k(t,y))^{\beta_4})\,|\vartheta_k(t,x)-\vartheta_k(t,y)|\\
&\quad\leq C\,\|K_h\|_{L^1}^\theta\,\left(\int_0^T\int_{\Pi^{2d}} K_h(x-y)\,|\vartheta_k(t,x)-\vartheta_k(t,y)|\right)^{1-\theta}.
\end{split}\]
Similarly $\rho_k$ is uniformly bounded in $L^{\gamma+a}_{t,x}$ so that for example
\[
\int \rho_k^{\beta_3}\,\vartheta_k\leq \|\vartheta_k\|_{L^\alpha_{t,x}}\,\|\rho_k\|^{\beta_3\,\alpha/(\alpha-1)}\leq C,
\]
as $\beta_3\,\alpha/(\alpha-1)\leq \gamma+a$. Indeed $\beta_3<(\gamma+a+1)/2$ and $\alpha\geq 4$.
 This implies that
\[
\sup_k\int_0^T\int_{\Pi^{2d}} \frac{K_h(x-y)}{\|K_h\|_{L^1}}\,Q_k(t,x,y)\longrightarrow 0,\quad \mbox{as}\ h\rightarrow 0.
\]
Note that if $\partial_\rho P_k(t,x,s)  = 0$ if $s> \epsilon_k^{-1}$. Hence we can use \eqref{nonmonotonewithtemp} on to write 
\[\begin{split}
|P(\vartheta_k(t,x),\rho(x))-P(\vartheta_k(t,x),\rho_k(y))|
    \le &C( \rho_k(x)^{\gamma -1}  + \rho_k(y)^{\gamma-1} + \vartheta_k(t,x)^b)\\ &\ |\rho_k(x) - \rho_k(y)|.
\end{split}\]
Denote $\tilde P_k(t,x)=\vartheta_k(t,x)^b$.
Let us now use again that
$\vartheta$ is compact in $x$ in $L^1_{t,x}$ and uniformly bounded in 
$L^\alpha_t L^{\alpha/(1-2/d)}_x$ and the assumption on $b\leq\alpha/2$ in \eqref{nonmonotonewithtemp} to get the convergence
\[ 
\sup_k \int_0^T\int_{\Pi^{2d}} \frac{K_h(x-y)}{\|K_h\|_{L^1}\|}\, |\tilde P_k(t,x)-\tilde P_k(t,y)| \to 0 \hbox{ as } h\to 0,
\]
together with the uniform bound
\[
\sup_k\|\tilde P_k\|_{L^2_{t,x}}\leq \sup_k\|\vartheta_k(t,x)^b\|_{L^2_{t,x}}\leq \sup_k\|\vartheta_k(t,x)\|_{L^\alpha_{t,x}}<\infty.
\]

\medskip

\noindent {\it Second property.} Following the same procedure as in the previous bounds, we obtain from \eqref{assumptionswithtemp} and \eqref{nonmonotonewithtemp}
\[\begin{split}
P_k(t,x,\rho_k(t,x))\leq C\,\Big(&\rho_k(t,x)^\gamma+(\vartheta_k(t,x))^b\,\rho_k(t,x) +(\vartheta_k(t,x))^{1+\beta_4}\\
&+\rho_k(x)^{\beta_3}\,\vartheta_k(t,x)\Big).
\end{split}\]
By Young's inequalities with $1/\gamma^*+1/\gamma=1$
\[\begin{split}
P_k(t,x,\rho_k(t,x))\leq C\,\Big(&\rho_k(t,x)^\gamma+(\vartheta_k(t,x))^{b\,\gamma^*} +(\vartheta_k(t,x))^{1+\beta_4}\\
&+(\vartheta_k(t,x))^{\gamma/(\gamma-\beta_3)}\Big).
\end{split}\]
Define
\[
R_k(t,x,y)=(\vartheta_k(t,x))^{b\,\gamma^*} +(\vartheta_k(t,x))^{1+\beta_4}
+(\vartheta_k(t,x))^{\gamma/(\gamma-\beta_3)},
\]
and observe that $R_k$ is compact in $x$ in $L^1$ since $\vartheta_k$ is compact in $x$ and $b\,\gamma^*<\alpha$ by \eqref{nonmonotonewithtemp}, $1+\beta_4<\alpha$ since $\alpha>2$, with finally $ \gamma/(\gamma-\beta_3)<\alpha$ as $\beta_3<(\gamma+a+1)/2$, $a< 2\gamma/d-1$ and $\alpha\geq 4>d$. 

\medskip

Thus we conclude that the pressure law satisfies \eqref{nonmonotone} for both limits with $\tilde\gamma=\gamma$. We can again apply Theorem \ref{maincompactness}, point ii in this case. We again need that $p>\tilde\gamma$.
   Therefore we obtain the compactness of $\rho_k$ and from the momentum equation, the compactness of $\sqrt{\rho_k}\,u_k$. As in the limit of \eqref{NSFapprox} to \eqref{NSFapproxeps}, the next step is to obtain the compactness  in time of $\vartheta_k$. We follow the same procedure which however now requires more work on the first two points:
\begin{itemize}
\item Bound $\rho_k\,s_{\eps_k}(\rho_k,\vartheta_k)$ or $\rho_k\,s(\rho_k,\vartheta_k)$ uniformly in some $L^p_{t,x}$ with $p>1$. This is not immediate anymore since $s$ is not bounded and $s_{\eps_k}$ is not uniformly bounded. One uses the expression of the entropy in \eqref{entropyeps} and the bounds on $|\partial_\vartheta P|$ and $m$ in \eqref{assumptionswithtemp} to find
\begin{equation}
\rho_k\,s_{\eps_k}(\rho_k,\vartheta_k)\leq C\,(\rho_k^{\beta_3}+\vartheta_k^{\beta_4}+ \rho_k\,\vartheta^{\alpha(\gamma+a-1)/(\gamma+a)}),
\end{equation}
and similarly for $\rho_k\,s$. 
Then the a priori estimates on $\rho_k\in L^{\gamma+a}$ or $\rho_k\in L^{2+a_\delta}$ and $\vartheta_k\in L^\alpha_{t,x}\cap L^\infty_t L^{\gamma_\vartheta}_x$ are enough provided $\beta_3<\gamma+a$ and $\beta_4<\alpha$ or $\beta_4<\gamma_\vartheta$; both are ensured by the stronger condition in \eqref{coefftemp}. 
\item Bound $\partial_t (\rho_k \,s_{\eps_k}(\rho_k,\vartheta_k))$, $\partial_t (\rho_k \,s(\rho_k,\vartheta_k))$ uniformly in $M^1([0,T], W^{-1,1}(\Pi^d))$. 
One first has to bound $\rho_k s_{\eps_k}\,u_k$ in $L^1_{t,x}$. For instance
\[
\int_0^T\int_{\Pi^d} \rho_k^{\beta_3}\,u_k\leq \int_0^T\int_{\Pi^d} (\rho_k\,|u_k|^2+\rho_k^{2\beta_3-1}),
\] 
which is bounded provided $2\beta_3-1<\gamma+a$ and $2\,\beta_3-1<2+a_\delta$ leading to the assumption on $\beta_3$ in \eqref{coefftemp}. As for $\int \vartheta^{\beta_4}_k\,u_k$, we have by Sobolev embedding that $u_k\in L^2_t L^{2d/(d-2)}$ and therefore need that $\vartheta_k^{\beta_4}\in L^2_t L^{2d/(d+2)}$. Given that $\vartheta_k\in L^\alpha_{t,x}\cap L^\infty_t\,L^{\gamma_\vartheta}_x$, this is ensured by the condition on $\beta_4$ (note that the condition could in fact be improved by using that $\vartheta_k\in L^\alpha_t L^{\alpha/(1-2/d)}_x$). The last term is treated in a similar manner. From the entropy dissipation 
$\kappa_\eps(\vartheta_k)\,|\nabla\vartheta_k|^2/\vartheta_k$ is uniformly in $L^1_{t,x}$ as well.  Therefore, we just have now
to use that $\sigma^2_{\varepsilon,\delta}$ is uniformly in $M^1([0,T]\times \overline \Omega)$ 
proved  previously through relation \eqref{defectmeasurecontrol}.  
  Thus we conclude the bound on $\partial_t (\rho_k \,s_{\eps_k}(\rho_k,\vartheta_k))$.
  Similar calculations may be done concerning $\partial_t (\rho_k \,s(\rho_k,\vartheta_k))$
 with the defect measures appearing from the limit $\varepsilon_k\to 0$.  

\end{itemize}
 
   To conclude and pass to the limit in every term of the  momentum equation,  we use  the integrability properties on $\rho_k$, $\vartheta_k$ \eqref{assumptionswithtemp}${}_4$, \eqref{assumptionswithtemp}${}_5$, \eqref{assumptionswithtemp}${}_6$ and \eqref{nonmonotonewithtemp} to provide an $L^p_t L^p_x$ integrability with $p>1$ of the pressure and more precisely
\[
\int_{\rho_k>L\;\mbox{or}\;\vartheta_k>L} |P(\rho_k,\vartheta_k)|\leq C\,\int_{\rho_k>L\;\mbox{or}\;\vartheta_k>L} \left(\vartheta_k\,\rho_k^{\beta_3}+\rho_k^\gamma+\vartheta_k^{b}\,\rho_k\right)
\leq C\,L^{-\theta},
\] 
for some $\theta>0$ since $\rho_k$ is uniformly in $L^{\gamma+a}$, $\vartheta_k$ in $L^\alpha$ and from the assumptions \eqref{coefftemp} on the $\beta_i$. Since $P_{\eps,\delta}$ contains the additional barotropic part $\delta\,(\rho^2_k+\rho_k^\gamma)$ which is also controlled, one has that
\[
\int_{\rho_k>L\;\mbox{or}\;\vartheta_k>L} \left(|P_{\eps,\delta}(\rho_k,\vartheta_k)| +|P_{\delta}(\rho_k,\vartheta_k)|\right) \leq C\,L^{-\theta}.
\] 
This lets us pass to the limit in $P_{\delta}$ and $P_{\eps,\delta}$ by using the pointwise convergence of $\rho_k$ and $\vartheta_k$. 

The type same control applies to $\rho_k\,s_{\eps,\delta}(\rho_k,\vartheta_k)$ or $\rho_k\,s_{\delta}(\rho_k,\vartheta_k)$ so that we may pass to the limit in the entropy equation. 
  The right--hand side requires the use of convexity just as for the limit of System \eqref{NSFapprox} to \eqref{NSFapproxeps} and the properties of the total mass defect measures control follow the same calculations than \eqref{defectmeasurecontrol}.

\section{Models occurring in other contexts\label{otherextensions}}
Macroscopic models in various biological settings involve a density $\rho$ that
is transported by a velocity vector field $u$ with source term, such as 
$$\partial_t \rho + {\rm div} (\rho u) = \rho\,  G(P(\rho), c)$$
for some functions $G$ and $P$. The function $G$ may include birth and death
terms and it could also depend on other quantities such as nutrients concentration
denoted $c$, for example Oxygen in cancer modeling.  This concentration is typically governed by
a parabolic equation with right--hand side modeling the consumption of the resource(s)
\begin{equation}
\partial_t c - \Delta c = -\rho \, H(P(\rho),c)\label{biochemical}
\end{equation}
with 
$$\partial_P H \le 0, \qquad \partial_c H \ge 0, \qquad H(P,0)=0.$$
 The velocity field is described through a constitutive law
 for instance
\begin{equation}\label{SimpleStreamFunc}
-\nu \Delta \Psi + \alpha \Psi = P(\rho)  - S, \qquad u=-\nabla \Psi
\end{equation}
where $S$ is a given source term.

There have been several studies of such systems with applications 
to crowd motion, traffic jams, cancerology using specific reformulations:
gradient flow or kinetic descriptions and appropriate choices for $f$ and~$G$.

\medskip

 For instance the special case $c=0$, $\alpha=0$, $G\equiv 0$  has been studied by {\sc B.~Maury}, {\sc A. Roudneff--Chupin} and F.  {\sc Santambrogio} (see \cite{MaRoSa}) through the framework of optimal transportation reformulating the problem as a gradient flow in the Wasserstein space of measures. 
  Other examples concern the reformulation through a kinetic formulation. For instance, very recently, {\sc B.~Perthame} and {\sc N. Vauchelet} (see \cite{PeVa}) have studied the case $c=0$ in the whole space  with the pressure law $P(\rho)= (\gamma +1)  \rho^\gamma/\gamma$ with $\gamma> 1$, $G$ satisfying $$ G \in {\cal C}^1(\R), \qquad G'(\cdot) \le - \eta <0, \qquad G(P_M) = 0 \hbox{ for some } P_M>0$$ and $\alpha, \nu >0$.

  The main result of this last paper is  the ''stiff pressure law'' limit, namely the limit $\gamma \to +\infty$, leading to a free boundary model which generalizes the classical Hele-Shaw equation. Such kind of limit has also been performed for the compressible Navier--Stokes  equations 
by {\sc P.--L. Lions} and  {\sc N.~Masmoudi} (see \cite{LiMa}) with $P(\rho) = a \rho^\gamma$ with $G\equiv 0$. 

Recently   {\sc C. Perrin} and {\sc E. Zatorska} (see \cite{PeZa}) have studied the singular limit $\varepsilon \to 0$ for a  singular pressure law $P(\rho) = \varepsilon \rho^\gamma / (1-\rho)^\beta$ with $\gamma, \beta >3$. The advantage of such pressure law is that $0\le \rho \le 1$ for a fixed $\varepsilon$ which is important for some applications as mentioned by B. {\sc Maury} in his review paper \cite{Ma}.

Pressure laws which blows-up for a critical density are of course the exact analogous of the Van der Waals equation of state for compressible fluid dynamics.  They are also encountered in other setting such as crowd motion, granular flow, sedimentation problems. 

\medskip

  An other possibility to describe the velocity field is to consider the Brinkman equation instead of Eq. \eqref{SimpleStreamFunc},
namely
\[
-\nu \Delta u + \alpha u + \nabla P(\rho)= S,
\]
or with a Stokes viscosity term
\[
- \nu \Delta u - (\lambda + \nu) \nabla {\rm div u} + \alpha u + \nabla P(\rho)= S. 
\]
This type of correction accounts for flow through medium where the grains of the media are porous themselves and has been justified in \cite{Al}.   If the velocity is irrotational then this model is exactly reduced to Eq.  {\rm (\ref{SimpleStreamFunc})}. Note however that viscoelastic models for tumor growth may allow for instance to observe a lemon-like shape tumor, whereas with a Newtonian model an ovoid  is obtained. Even limited, this difference can eventually lead to bigger ones as the outer rim is composed of proliferating cells with exponential  growth. The kind of shape obtained for instance
 in \cite{BrCoGrRiSa} is observed in {\it in--vitro} experiments.

\bigskip

\noindent{\bf What our method can bring:}

We do not try to state a theorem here given the large variety of possible models. Instead we give a few elements for which the method introduced here could prove crucial
\begin{itemize}
\item More complex pressure laws, attractive and repulsive, could be considered. This would be the exact equivalent of Theorem \ref{MainResultPressureLaw}. Note that biological systems frequently exhibit preferred ranged of densities for instance with attractive interactions for low densities and repulsive at higher ones.
\item More importantly the transition from attractive to repulsive interactions may depend on the concentration $c$ of nutrients or other bio-chemicals. This is similar to the dependence on the temperature in the state laws for the Navier-Stokes-Fourier system. For example if the pressure blows-up at some thresholds, enforcing a maximal density, then this threshold and the maximal density will depend in general on $c$. Because $c$ is not necessarily uniformly bounded, the range of attractive interactions (where $G(P(\rho),c)$ is decreasing in $\rho$) is not compactly supported and classical approaches may fail.
\item There can be several nutrients or bio-chemicals. That means that in general one has several $c_1,\;c_2,\;\ldots$ with several equations \eqref{biochemical} (or a vector-valued one if the diffusion speeds are the same). If chemotaxis is considered, some of those bio-chemicals may be attractive while other are repulsive. This may lead to a complicated pattern of interactions which again cannot be handled by classical approaches.   
\item Many of this models are posed in porous media which are inherently anisotropic. In biology for instance the tissue or the porous matrix is heterogeneous. Therefore the equation for $u$ should read
$$-{\rm div} (A(t,x)D(u)) + \alpha (t,x) \, u +  \nabla P(\rho)= S$$
for some $A$ and $\alpha$ and with in general a non-monotone pressure law.  Our new approach could for instance help to enrich the model mathematically studied recently in
 \cite{DoTr}.

\end{itemize}

\section{Appendix: Notations \label{notations}}
For the reader convenience, we repeat and summarize here some of our main notations.

\bigskip

\noindent {\em Physical quantities.}

\begin{itemize}
\item$\rho(t,x)$, or $\rho_k(t,x)$ denotes the density of the fluid.

\item $u(t,x)$, or $u_k(t,x)$ denotes the velocity field of the fluid.

\item $P(.)$, or $P_k(.)$ denotes the pressure law.

\item $e(\rho)$ is the internal energy density.\\  In the barotropic case, $e(\rho)=\int_{\rho_{ref}}^\rho P(s)/s^2\,ds$.

\item $E(\rho,u)=\int \rho\,(|u|^2/2+e(\rho))$ is the total energy of the fluid.

\item $\mu$, $\lambda$ and $\mu_k$ denote various viscosity coefficients or combination thereof.

\item ${\cal S}$ denotes the viscous stress tensor. \\
          In the simplest isotropic case: ${\cal S} = 2 \mu D(u) + \lambda \div u {\rm Id}$. 
          
\item ${\cal D}$ is the diffusion  term related to the viscous stress tensor
          by ${\cal D} u =\div {\cal S}$.

\end{itemize}

\noindent In the Navier--Stokes--Fourier case: \\we have the additional notations, appearing only in section \ref{withtemperature}.

\begin{itemize}

\item  $\vartheta(t,x)$ is the temperature field of the fluid.
         
\item $s(\rho,\vartheta)$ is the entropy of the fluid. 

\item $\kappa(\vartheta)$ is the heat conductivity coefficient.

\item $C_v$ is the specific heat of the fluid.
\end{itemize}

\bigskip

\noindent{\em Technical notations.}

\begin{itemize}

\item $d$ is the dimension of space.

\item $k$ as an index always denotes the index of a sequence.

\item $h$ and $h_0$ are scaling parameters used to measure oscillations of certain quantities such as the density.

\item $K_h$ is a convolution kernel on $\Pi^d$, \\ $K_h(x)=(h+|x|)^{-a}$ for $x$ small enough and with $a>d$.

\item $\overline{K}_h$ is equal to $K_h/\|K_h\|_{L^1}$.

\item $ {\cal K}_{h_0}=\int_{h_0}^1 \overline{K}_h(x)\,\frac{dh}{h}$ is the weighted average of $K_h$. Note that $\|{\cal K}\|_{L^1}\sim |\log h_0|.$

\item $w_0$, $w_1$ and $w_a$ are the weights and 
         $w_{i,h}= \overline K_h\star w_i$ their regularization with $i=0,1,a$.

\item $C$ is a constant whose exact value may change from one line to another but which is always independent of $k$, $h$ or other scaling parameters.

\item $\eps(h)$ is a smooth function with $\eps(0)=0$.

\item $\theta$ is an exponent whose exact value may change as for $C$ but in $(0,1)$.

\item The exponent $p$ is most of the time such that $\rho\in L^p_{t,x}$.

\item $q$ and $r$ are other exponents for $L^p$ type spaces that are used when needed.

\item $I,\;II,\ldots$ and $A,\;B,\;D,\;E\ldots$ are notations for some intermediary quantities used in the proofs. Their definitions may change from one proof to another.

\item $x,\;y,\;w,\;z$ are typically variables of integration over the space domain.

\item $\delta\rho_k(x,y)=\rho_k(x)-\rho_k(y)$ the difference of densities.

\item $\bar \rho_k(x,y)=\rho_k(x)+\rho_k(y)$ the sum of densities.

\item $D\,\rho u_k=\Delta^{-1}\div(\partial_t (\rho_k\,u_k)+\div(\rho_k\,u_k\otimes u_k))$ denotes the effective flux.

\item The individual weights $w(t,x)=w_0,\;w_1,\;w_a$ satisfy Eq. \eqref{eqw} or
\[
\partial_t w+u_k\cdot \nabla w= -D\,w+\alpha_k\,\Delta_x w,
\]
where $D=D_0,\; D_1, \;D_a$ are the penalizations in \eqref{dk0}, \eqref{dk1}, \eqref{dk00}.
\item The weights $w_0$ or $w_a$ may be convolved to give $w_h=\overline K_h\star w_0$, $w_{a,h}=\overline K_h\star w_a$.
\item The weights are then added or multiplied to obtain the composed $W(t,x,y)= W_0,\;W_1,\;W_2,\;W_a$ with
\[\begin{split}
&W_0=w_0(x)+w_0(y),\quad W_1=w_1(x)+w_1(y),\\
& W_2=w_1(x)\,w_1(y),\quad W_a=w_a(x)+w_a(y).
\end{split}
\]
The main properties of the weights are given in Prop. \ref{weightprop}.
\end{itemize}

\section{Appendix:  Besov spaces and Littlewood-Paley decomposition\label{Besov}}
We only recall some basic definitions and properties of Besov spaces for use in Lemma \ref{shiftDulemma}. We start with the classical Littlewood-Paley decomposition and refer
to the readers for instance to \cite{BaChDa}, \cite{Ab} and \cite{BeSe} for details and applications
to fluid mechanic. 
   Choose any family $\Psi_k\in {\cal S}(\Pi^d)$ s.t.
\begin{itemize}
\item Its Fourier transform $\hat \Psi_k$ is positive and compactly supported in the annulus $\{2^{k-1}\leq |\xi|\leq 2^{k+1}\}$.
\item It leads to a decomposition of the identity in the sense that there exists $\Phi$ with $\hat \Phi$ compactly supported in $\{|\xi|\leq 2\}$ s.t. for any $\xi$
\[
1=\hat\Phi(\xi)+\sum_{k\geq 1} \hat \Psi_k(\xi).
\]
\item The family is localized in $\Pi^d$ in the sense that for all $s>0$
\[
\sup_k \|\Psi_k\|_{L^1}<\infty,\quad \sup_k 2^{ks}\,\int_{\Pi^d} |z|^s\,|\Psi_k(z)|\,dz<\infty.
\] 
\end{itemize}
Note that in $\R^d$, one usually takes $\Psi_k(x)=2^{kd}\,\Psi(2^k\,x)$ but in the torus, it can be advantageous to use a more general family. It is still necessary to take it smooth enough for the third assumption to be satisfied (it is for instance the difference between the Dirichlet and Fejer kernels).

For simplicity, we then denote $\Psi_0=\Phi$ for $k=0$ and for $k\geq 1$, $\Psi_k(x)=2^{kd}\,\Psi(2^{-k}\,x)$. For any $f\in {\cal S}'(\R^d)$, we also write $f_k=\Psi_k\star f$ and then obtain the decomposition
\begin{equation}
f=\sum_{k=0}^\infty f_k.\label{littlewoodpaley}
\end{equation}
From this decomposition one may easily define the Besov spaces
\begin{defi}
The Besov space $B^s_{p,q}$ is the space of all $f\in L^1_{loc}\cap{\cal S}'(\R^d)$ for which 
\[
\|f\|_{B^{s}_{p,q}}=\left\| 2^{s\,k}\,\|f_k\|_{L^p_x}\right\|_{l^q_k}=\left(\sum_{k=0}^\infty 2^{s\,k\,q}\,\|f_k\|_{L^p_x}^q\right)^{1/q}<\infty.
\]
\end{defi}
The main properties of the Littlewood-Paley decomposition that we use in this article can be summarized as
\begin{prop}
For any $1<p<\infty$ and any $s$, there exists $C>0$ s.t. for any $f\in L^1_{loc}\cap{\cal S}'(\R^d)$ 
\[\begin{split}
\frac{2^{s\,k}}{C}\,\|f_k\|_{L^p}\leq &\|\Delta^{s/2}\,f_k\|_{L^p}\leq C\, {2^{s\,k}}\,\|f_k\|_{L^p},\\
 C^{-1}\,\left\| \left(\sum_{k=0}^\infty 2^{2\,k\,s}\,|f_k|^2\right)^{1/2}\right\|_{L^p}\leq &\|f\|_{W^{s,p}}\leq C\,\left\| \left(\sum_{k=0}^\infty 2^{2\,k\,s}\,|f_k|^2\right)^{1/2}\right\|_{L^p}.
\end{split}
\]
And as a consequence for $1< p\leq 2$
\[
C^{-1}\,\|f\|_{B^s_{p,2}}\leq \|f\|_{W^{s,p}}\leq C\,\|f\|_{B^{s}_{p,p}}.
\]
\label{propLP}
\end{prop}
In particular a consequence of Prop. \ref{propLP} is the following bound on truncated Besov norm
\begin{lemma} For any $1<p\leq 2$, there exists $C>0$ s.t. for any $f\in L^1_{loc}\cap{\cal S}'(\R^d)$ and any $K\in \N$
\[
\sum_{k=0}^K 2^{s\,k}\,\|f_k\|_{L^p_x}\leq C\,\sqrt{K}\,\|f\|_{W^{s,p}}
\]
\label{truncatedbesov}
\end{lemma}
\begin{proof}
By a simple Cauchy-Schwartz estimate
\[
\sum_{k=0}^K 2^{s\,k}\,\|f_k\|_{L^p_x}\leq \sqrt{K}\,\left(\sum_{k=0}^\infty 2^{2\,s\,k}\,\|f_k\|_{L^p_x}^2\right)^{1/2}=\sqrt{K}\,\|f\|_{B^s_{p,2}},
\]
which concludes by Prop. \ref{propLP}.
\end{proof}

\bigskip

\noindent {\it Acknowledgments.} The authors would like to thank E. {\sc Feireisl} for 
constructive comments that contributed to improving the quality of redaction.

\end{document}